\documentclass[11pt]{amsart} 

\usepackage[numbers,comma,square,sort&compress]{natbib}
\RequirePackage{color}
\RequirePackage[colorlinks,urlcolor=my-blue,linkcolor=my-red,citecolor=my-green]{hyperref}

\input{AsymOptMacros.tex}
\usepackage{datetime}

\usepackage{ulem}

\makeatletter
\@namedef{subjclassname@2020}{%
  \textup{2020} Mathematics Subject Classification}
\makeatother

\begin{document}

  \title[First-passage percolation]{Geodesic length and shifted weights\\ in first-passage percolation}

\author[A.~Krishnan]{Arjun Krishnan}
\address{Arjun Krishnan\\ University of Rochester\\  Mathematics Department\\ Hylan 817\\   Rochester, NY 14627\\ USA.}
\email{arjun@shirleyarjun.net}
\urladdr{https://people.math.rochester.edu/faculty/akrish11/}
\thanks{A.\ Krishnan was partially supported by a Wiley Assistant Professorship at University of Utah, an AMS-Simons Travel Grant, and Simons Collaboration grant 638966.}

\author[F.~Rassoul-Agha]{Firas Rassoul-Agha}
\address{Firas Rassoul-Agha\\ University of Utah\\  Mathematics Department\\ 155S 1400E\\   Salt Lake City, UT 84112\\ USA.}
\email{firas@math.utah.edu}
\urladdr{https://www.math.utah.edu/~firas}
\thanks{F.\ Rassoul-Agha was partially supported by National Science Foundation grants DMS-1811090 and DMS-2054630.}

\author[T.~Sepp\"al\"ainen]{Timo Sepp\"al\"ainen}
\address{Timo Sepp\"al\"ainen\\ University of Wisconsin-Madison\\  Mathematics Department\\ Van Vleck Hall\\ 480 Lincoln Dr.\\   Madison WI 53706-1388\\ USA.}
\email{seppalai@math.wisc.edu}
\urladdr{https://people.math.wisc.edu/~seppalai/}
\thanks{T.\ Sepp\"al\"ainen was partially supported by  National Science Foundation grants  DMS-1854619 and DMS-2152362    
  and by the Wisconsin Alumni Research Foundation.} 


\let\thefootnote\relax\footnotetext{Submitted August 19, 2021. Revised January 23, 2023. Accepted February 9, 2023.} 

\begin{abstract}We study first-passage percolation through related optimization problems over paths of restricted length. The path length variable is in duality with a shift of the weights. This puts into a convex duality framework old observations about the convergence of the normalized Euclidean length of geodesics due to Hammersley and Welsh, Smythe and Wierman, and Kesten, and leads to new results about geodesic length and the regularity of the shape function as a function of the weight shift. For points far enough away from the origin, the ratio of the geodesic length and the $\ell^1$ distance to the endpoint is uniformly bounded away from one. The shape function is a strictly concave function of the weight shift. Atoms of the weight distribution generate singularities, that is, points of nondifferentiability, in this function. We generalize to all distributions, directions  and dimensions an old singularity result of Steele and Zhang for the planar Bernoulli case. When the weight distribution has two or more atoms, a dense set of shifts produce singularities. The results come from a combination of the convex duality, the shape theorems of the different first-passage optimization problems, and 
modification arguments.
\end{abstract}
\maketitle

\setcounter{tocdepth}{2}
\tableofcontents


 
 \section{Introduction}

\subsection{Stochastic growth models} 

Irregular and stochastic growth surrounds us, for example in tumors, bacterial colonies, infections, spread of fluid  in a porous medium, and propagating flame fronts.  These phenomena attract the attention of mathematicians, scientists and engineers in various disciplines.  Simplified mathematical models of stochastic growth have been studied in probability theory for over half a century.   This work has inspired some of the central innovations of modern probability, such as the subadditive ergodic theorem, and created new connections between probability and other parts of mathematics,  such as representation theory, integrable systems, and partial differential equations.    



A class of much-studied stochastic growth models possess a metric-like structure where growth progresses along paths that optimize an energy functional defined in terms of a random environment. Depending on whether the optimal path is chosen through minimization or maximization, these models are called first-passage percolation and last-passage percolation.  

A variety of settings for first- and last-passage percolation are studied. The admissible paths can be general or they can be restricted to be directed along some spatial directions.   The underlying space can be a graph, the continuum, or a mixture of the two.  In the graph case, the environment is given by random weights attached to the vertices or the edges.  The most typical choice of graph is the $d$-dimensional integer lattice $\Z^d$.     The one-dimensional case usually reduces to classical probability so the real work begins from the planar case $d=2$.  

Much progress in the planar case has taken place over the past 25 years under the rubric \textit{Kardar-Parisi-Zhang universality}.    A universal planar continuum limit, the directed landscape, has recently been constructed \cite{Dau-Ort-Vir-22-}.  It is expected to be the scaling limit of a wide class of planar first- and last-passage percolation models, but this remains conjectural at present.  Evidence for the universality comes from proofs that certain special exactly solvable directed models converge to the directed landscape \cite{Dau-Vir-21-}.  We refer the reader to articles \cite{Cor-16, Dam-Ras-Sep-16} and the monograph \cite{Auf-Dam-Han-17}   for general introductions to the field. 

Our paper studies first-passage percolation with undirected paths on the integer lattice in arbitrary dimension. This  has proved to be, in a sense, the most challenging model, as no exactly solvable version has been discovered. A proof that this model lies in the KPZ class, while universally expected, appears well beyond reach in the current state of the field.  Our results concern properties of the geodesics and {the regularity of the limiting norm} 
as {we perturb the {random} weights by a common additive constant.}
    We turn to discuss the background.  
   
 \subsection{First-passage percolation and its limit shape}   
 \label{sec:fpp and its limit shape}
 In {\it first-passage percolation} (FPP) a random pseudometric is defined on $\Z^d$ by $T_{x,y}=\inf_\pi\sum_{e\tsp\in\tsp\pi} t(e)$ where the $\{t(e)\}$ are nonnegative, independent and identically distributed (i.i.d.) random weights on the nearest-neighbor edges between vertices of $\Z^d$ and the infimum  is over  self-avoiding paths  $\pi$ between the two points  $x$ and $y$.  A minimizing path is called a {\it geodesic} between $x$ and $y$.  FPP was introduced by Hammersley and Welsh \cite{Ham-Wel-65}  in 1965 as a simplified model of fluid flow in an inhomogeneous medium.     A precise technical definition of the model comes in Section \ref{sec:rfpp}.

 The fundamental questions of FPP concern the behavior of the passage times  $T_{x,y}$  and the geodesics as the distance between $x$ and $y$ grows. 
 At the level of the law of large numbers,  under suitable hypotheses, normalized passage times converge with probability one:  $n^{-1} T_{\zevec, x_n} \to \gly(\xi)$ as $n\to\infty$,  whenever  $n^{-1}x_n\to\xi\in\R^d$. 
 The special case  $\mu(\evec_1)=\lim_{n\to\infty} n^{-1}T_{\zevec, n\evec_1}$  of the limit is also called the  {\it time constant}.

 The limiting {\it shape function} $\gly$ is a norm that characterizes the asymptotic shape of a large ball.    Define the  randomly growing ball in $\R^d$   for $t\ge 0$ by 
 $B(t)=\{ x\in\R^d: T_{\zevec, \fl x}\le t\}$ where $\fl x\in\Z^d$ is obtained from $x\in\R^d$ by taking integer parts coordinatewise.  
Under the right  assumptions,  as $t\to\infty$ the normalized ball $t^{-1}B(t)$ converges to the unit ball  $\cB=\{\xi\in\R^d: \gly(\xi)\le 1\}$ defined by the norm $\gly$.  

The shape function $\gly$ is not explicitly known in any nontrivial example. Soft properties  such as convexity, continuity, positive homogeneity, and $\gly(\xi)>0$ for $\xi\ne\zevec$ when zero-weight edges are subcritical,  are readily established.  But anything beyond that, such as strict convexity or differentiability, remain conjectural.    The only counterexample to this state of affairs is the classic  Durrett-Liggett \cite{Dur-Lig-81} planar flat edge result,  sharpened by Marchand \cite{Mar-02}, and then extended by Auffinger and Damron  \cite{Auf-Dam-13} to include differentiability at the boundary  of the flat edge.

 The FPP shape theorem    occupies    a venerable position  as one of the fundamental results of the subject of random growth models and as an early motivator of subadditive ergodic theory.  
 The reader is referred to the monograph \cite{Auf-Dam-Han-17} for a recent overview of the known results  and open problems. 
 
 \subsection{Differentiability and length of geodesics} 
 The success  
 of the shape theorem  contrasts sharply with the situation of another natural limit question, namely the behavior  of the normalized Euclidean length (number of edges)  of a  geodesic as one endpoint is taken to infinity.  No useful subadditivity or other related property has been found. This issue has been addressed only a few times over the 55 years of FPP study   and the results remain incomplete. 
 
 The fundamental observation due to Hammersley and Welsh is the connection between (i)  the limit  of the normalized length of the geodesic and (ii)  the derivative of the shape function  as a function of a weight shift.   
 For $h\in\R$ let  $\gly^{(h)}(\xi)$ denote the shape function  for the shifted weights $\{t(e)+h\}$.    Let $\underline L_{\tspb\zevec, x}^{(h)}$ be   the minimal Euclidean length  of a geodesic from the origin to the point $x$ for the shifted weights $\{t(e)+h\}$. 
Then the important fact is that when $n^{-1}x_n\to\xi$,  
\be\label{ham-w} 
\lim_{n\to\infty}  n^{-1}\underline L_{\tspb\zevec, x_n}^{(h)} \;=\;  \frac{\partial \gly^{(s)}(\xi)}{\partial s}\bigg\vert_{s=h}
\ee
provided the derivative at $h$  on the right-hand side exists. 

The shape function  $\gly^{(h)}(\xi)$ is a concave function of $h$ and hence   the derivative in \eqref{ham-w}  exists and the  limit holds for all but countably many shifts $h$.  
But since the time constant itself remains a mystery, not a single specific nontrivial case where  this identity holds 
has been identified.  
 The first results on the size of the set of exceptional $h$ at which  the derivative on the right fails are proved in the present paper and summarized in Sections \ref{sec:int-geod} and \ref{sec:int-reg}  below.  

Here is a brief accounting of the history of \eqref{ham-w}. 

Hammersley and Welsh (Theorem 8.2.3 in \cite{Ham-Wel-65}) gave the first version of \eqref{ham-w}. It was proved for the time constant of planar FPP, so for $d=2$ and $\xi=\evec_1$, and for the particular sequence $x_n=(n,0)$.   Their result applied to the geodesic of the so-called cylinder passage time from $(0,0)$ to $(n,0)$, and the mode of convergence in \eqref{ham-w}  was convergence in probability. 
 
The   limit \eqref{ham-w}  was improved   in 1978  by Smythe and Wierman  (Theorem 8.2 in \cite{Smy-Wie-78})   and in 1980 by  Kesten   \cite{Kes-80}, in particular from convergence in probability to almost sure convergence.  The ultimate version has recently  been established by Bates (Theorem 1.25 in \cite{Bat-20-}):  almost sure convergence  in \eqref{ham-w}  without any moment assumptions on the weights, in all directions $\xi$, provided the derivative on the right exists.


 A handful of precise results related to \eqref{ham-w} exist in specific situations defined by criticality in percolation.   Let $p_c$ denote  the critical probability of Bernoulli bond percolation on $\Z^d$.  When $\P(t(e)=0)\ge p_c$ the FPP problem becomes in a sense degenerate.  Geodesics to far-away points can take advantage of long paths of zero-weight edges and  the shape function $\mu$ becomes identically zero. 
 
   Zhang \cite{Zha-95} proved in 1995 that  in the supercritical case defined by  $\P(t(e)=0)> p_c$, for $\xi=\evec_1$ and $h=0$,  the limit on the left in \eqref{ham-w} exists and equals a nonrandom constant.  
   In the planar critical case, that is, $d=2$, $\P(t(e)=0)=1/2= p_c$ and $h=0$, Damron and Tang \cite{Dam-Tan-19}  proved that the 
 left-hand side in \eqref{ham-w} blows up in all directions $\xi$.  
 
In  2003 
Steele and Zhang \cite{Ste-Zha-03} proved the first, and before the present paper the only,  precise result about the derivative  in \eqref{ham-w}, valid  for subcritical planar FPP  with Bernoulli weights.  When the distribution is   $\P(t(e)=0)=p=1-\P(t(e)=1)$, there exists $\delta>0$ such that, if $\tfrac12-\delta\le p<\tfrac12$, $d=2$ and $\xi=\evec_1$, then the derivative in \eqref{ham-w} fails to exist at $h=0$. Thus the 
Hammersley-Welsh differentiability criterion for the convergence  of normalized geodesic length  faces a limitation.

%




\subsection{Duality of path length and weight shift} 
We move on to describe the contents of our paper.   To investigate \eqref{ham-w} and more broadly properties of geodesic length, 
we develop a convex duality between the weight shift $h$ and a parameter that captures the asymptotic length of a path.    This puts the limit  \eqref{ham-w} into a convex-analytic framework.    
To account for the possibility of nondifferentiability in \eqref{ham-w}, we enlarge the class of paths considered from genuine geodesics to {\it $o(n)$-approximate geodesics}. These are paths whose endpoints are order $n$ apart and whose  passage times are within $o(n)$ of the optimal passage time. Through these we can capture the entire superdifferential of the shape function as a function of the shift  $h$. 
 
 To be able to work explicitly with the path-length parameter, we 
   introduce    a version of FPP that minimizes over paths with a given number of steps but drops the requirement that paths be self-avoiding (Section \ref{sec:Gfpp1}).  A further useful variant  of the restricted path length FPP process  allows zero-length steps that do not increase the passage time.    The shape functions $\gpp$ and $\zgpp$ of these altered models are no longer positively homogeneous, but they turn out to be  continuously differentiable  along rays from the origin (Theorem \ref{thm:gdiff}). 
 
 The restricted path length shape functions $\gpp$ and $\zgpp$ are connected with the FPP  shape function $\mu$ in several ways. 
  A key fact  is that $\gpp$ and $\zgpp$ agree  with  $\mu$ on certain subsets of $\R^d$ described by positively homogeneous functions that are connected with geodesic length (Theorems \ref{thm:fpp10} and \ref{thm:gdiff}).   Second,  $\gpp$ and $\zgpp$ generate $\mu$  as the maximal positively homogeneous convex function dominated by  $\gpp$ and $\zgpp$    (Remark \ref{rm:co}).   Third, $\gpp$ and $\zgpp$ contain the information for generating all the shifts $\gly^{(h)}$ through convex duality (Theorem \ref{thm:fppb3} and Remark \ref{rmk:dual}).  

From this setting we derive two types of main results for FPP: results on the Euclidean length of geodesics and on the regularity of the shape function as a function of the weight shift, briefly summarized in the next two paragraphs.  The proofs come through a combination of  
\begin{enumerate} [label=\rm(\roman{*}), ref=\rm(\roman{*})] \itemsep=2pt 
\item   versions of the van den Berg-Kesten modification arguments \cite{Ber-Kes-93},  
\item 
the convex duality (Theorem \ref{thm:fppb3}), and 
\item a shape theorem for the altered FPP models (Theorem \ref{thm:Gpp9} and Theorem \ref{thm:shape-G2} in Appendix \ref{a:gpp}).  
\end{enumerate} 
 Our results are valid on $\Z^d$  in  all dimensions $d\ge 2$, under the standard moment bound needed for the shape theorem and the assumption that the minimum of the edge weight $t(e)$ has probability strictly below $p_c$.

\subsection{Euclidean length of geodesics} \label{sec:int-geod} 
One of our fundamental  results is that with probability one, all geodesics from the origin  to far enough lattice points $x$ have length at least $(1+\delta)\abs{x}_1$ for a fixed constant $\delta>0$ (Theorem \ref{thm:geod4}).   The equality in \eqref{ham-w} between the limiting normalized  length of the geodesic and the derivative of the shape function, which is conditional on the existence of these quantities,   is generalized to an unconditional identity between the entire interval  of the asymptotic normalized lengths of the $o(n)$-approximate geodesics and the superdifferential of the shape function as a function of the weight shift (Theorem \ref{thm:fppb3}).  When the random weight $t(e)$ has an atom at zero or at least two atoms that satisfy suitable linear relations with integer coefficients, there are multiple geodesics whose lengths vary on the same scale as the distance between the endpoints (Theorem \ref{thm:ndi3}).  For any weight distribution with at least two atoms, this happens on a countable dense set of shifts (Theorem \ref{thm:ndi9}). 

\subsection{Regularity of the shape function as a function of the weight shift} 
\label{sec:int-reg}
A second suite of  main results  
concerns the regularity 
of the shape function $\gly^{(h)}(\xi)$ as a function of the weight shift $h$, in a fixed spatial direction $\xi\in\R^d\tspa\setminus\{\zevec\}$.
This function is {\it strictly concave} in $h$ (Theorem \ref{thm:fppb2}). 
In the situations where the atoms of $t(e)$ bring about geodesics whose asymptotic normalized lengths vary, the concave function $h\mapsto\gly^{(h)}(\xi)$ acquires points of nondifferentiability.    In particular, there is a countable dense set of these singularities whenever the edge weight has two atoms (Theorems \ref{thm:ndi3} and \ref{thm:ndi9}).     We   extend the Steele-Zhang nondifferentiability result \cite{Ste-Zha-03} mentioned above to all dimensions, all directions $\xi$,  and all distributions with an atom at the origin.  Furthermore,  we disprove their conjecture that $h=0$ is the only nondifferentiability point in the Bernoulli case (Remark \ref{rm:ber}).


 \subsection{Organization of the paper} 
Section \ref{sec:rfpp} describes the models and the main results.  Section \ref{sec:open} describes open problems that arise from this work.  

The proofs are divided into four sections.  Section \ref{sec:tech1}  develops soft  results about the relationships between the different shape functions and the Euclidean lengths of optimal paths. The main technical Sections \ref{sec:conc} and \ref{sec:ndiff} contain the modification arguments.  The final Section \ref{sec:m-pf} combines the results from Sections \ref{sec:tech1}, \ref{sec:conc} and \ref{sec:ndiff} to prove the main theorems. 

Four appendixes contain   auxiliary results that rely on standard material. Appendix \ref{a:fpp<0} extends the FPP shape function to weights that are allowed small negative values. Appendix \ref{a:gpp}  proves a shape theorem  for the restricted path length versions of FPP. Appendix \ref{a:Pei} contains the Peierls argument that sets the stage for the modification proofs.  Appendix \ref{a:conv} presents a lemma about the  subdifferentials of convex functions. 

\subsection{Further literature: convergence of empirical measures} 
 We close this introduction with a mention of a  significant recent extension to the 
 differentiability approach to limits along geodesics,  due to  Bates \cite{Bat-20-}.   By representing the weights as functions $t(e)=\tau(U_e)$ of uniform random variables, one can consider  perturbations $\wt t(e)=\tau(U_e)+\psi(U_e)$ of the weights and differentiate the shape function 
in   directions $\psi$ in infinite dimensions.   This way the limit in \eqref{ham-w} can be upgraded to convergence of the empirical distribution of weights along a geodesic, again whenever the required derivative exists.  This holds for various uncountable dense collections of weight distributions, exactly as \eqref{ham-w} holds for an uncountable  dense set of shifts $h$.  

These more general limit results continue to share the fundamental  shortcoming of the limit in \eqref{ham-w},   namely, that no particular nontrivial case can be identified where  the limit  holds.    If  $\P(t(e)=0)\ge p_c$ the empirical measure along a geodesic converges trivially to a pointmass at zero. 

   Finding extensions of our results to the  general perturbations  of \cite{Bat-20-}  presents an interesting open problem.

 \subsection{Notation and conventions}  Here is  notation that the reader may wish quick access to.  
 $\Z_+=\{0,1,2,3,\dotsc\}$, $\N=\{1,2,3,\dotsc\}$,  and $\R_+=[0,\infty)$.   For $n\in\N$, $[n]=\{1,2,\dotsc,n\}$.  
Standard basis vectors in $\R^d$ are $\evec_1=(1,0,\dotsc,0)$, $\evec_2=(0,1,0,\dotsc,0), \dotsc, \evec_d=(0,\dotsc,0,1)$ and $\zevec$ is the origin of $\R^d$.  
  The $\ell^1$ norm of $x=(x_1,\dotsc,x_d)\in\R^d$  is  $\abs{x}_1=\sum_{i=1}^d\abs{x_i}$.   Particular subsets of $\R^d$  that recur are $\range=\{\pm \evec_1,\dotsc,\pm \evec_d\}$,  $\zrange=\range\cup\{\zevec\}$, 
$\Uset=\conv\range=\{\xi\in\R^d:  \abs{\xi}_1\le 1\}$, and the topological interior  $\intr\Uset$. 

    A finite or infinite path or sequence is denoted by $x_{\parng{m}{n}}=(x_m,\dotsc,x_n)$ for $-\infty\le m\le n\le\infty$.   Other notation for lattice paths are $\path$ and $\pi$.   The steps of a path are the nearest-neighbor edges  $e_i=\{x_{i-1}, x_i\}$.  A finite path $x_{\parng{m}{n}}$ that satisfies $\abs{x_n-x_m}_1=n-m$   is called an {\it $\ell^1$-path}. 
    
  A {\it positively homogeneous} function $f$ satisfies $f(cx)=cf(x)$ for $c> 0$ whenever both $cx$ and $x$ are in the domain of $f$ \cite[p.~30]{Roc-70}.      One-sided derivatives of a  function defined around $s\in\R$ are defined by 
  $f'(s+)=\lim_{h\searrow0} h^{-1}[f(s+h)-f(s)]$ and 
  $f'(s-)=\lim_{h\searrow0} h^{-1}[f(s)-f(s-h)]$.  
    
The diamond  $\wild$  is a wild card for three superscripts  $\langle{\tt empty}\rangle$ (no superscript at all), $o$ (zero steps allowed), and $\text{sa}$ (self-avoiding) that distinguish different FPP processes with restricted path length.   

A real number $r$ is an {\it atom}  of the random edge weight $t(e)$ if $\P\{t(e)=r\}>0$.  $\est=\esssup t(e)$ and  $\eit=\essinf t(e)$.   Superscript $(b)$ on any quantity means that it is computed with  weights  shifted by $b$:  $t^{(b)}(e)=t(e)+b$.   

The symbol $\triangle$ marks the end of a numbered remark. 
 
\subsection{Acknowledgements} The authors would like to thank Michael Damron for sketching a route to prove the lower bound in Theorem \ref{thm:fpp10}\ref{thm:fpp10.ii}, and an anonymous referee for spotting a mistake in the proof of Theorem \ref{thm:fppb3}.
 

\section{The models and the main results}  
\label{sec:rfpp}

\subsection{Setting}  \label{sec:set} 
  Fix an arbitrary dimension $d\ge 2$.    Let $\cE_d=\{\{x,y\}:  x,y\in\Z^d, \abs{x-y}_1=1\}$ denote the set of {\it undirected nearest-neighbor edges}  between vertices  of $\Z^d$.
  $\OSP$ is the probability space of an environment $\w=(t(e):e\in\cE_d)$  such that the  edge weights $\{t(e):e\in\cE_d\}$ are  independent and  identically distributed  (i.i.d.) real-valued random variables.   Translations  $\{\theta_x\}_{x\tsp\in\tsp\Z^d}$  act  on $\Omega$ by $(\theta_x\w){\{u,v\}}=t(\{x+u,x+v\})$ for a nearest-neighbor edge $\{u,v\}$. 

  A nearest-neighbor path $\pi=x_{\parng{0}{n}}=(x_i)_{i=0}^n$ is any finite  sequence of  vertices $x_0, x_1, \dotsc, x_n\in\Z^d$ that satisfy $\abs{x_{i+1}-x_i}_1=1$ for each $i$.  The  steps of $\pi$ are the nearest-neighbor edges  $e_i=\{x_{i-1},x_i\}$.      The {\it Euclidean length} $\abs\pi$ of     $\pi$ is the number of edges, so in this case $\abs\pi=n$.  Then we call $\pi$ an {\it$n$-path}. 
   The {\it passage time}  of $\pi$ 
   is the sum of the weights of its edges: 
 \be\label{def-fpp.1} \tpath(\pi)=\sum_{i=1}^n t(e_i). 
  \ee
  These definitions apply even if the path repeats vertices or edges, as will be allowed at times in the sequel. 
  For notational consistency we also admit the zero-length path  $\pi=x_{\parng{0}{0}}=(x_0)$ that has no edges and has zero passage time and length: $\tpath(\pi)=\abs\pi=0$.   

The main results are described next in three parts: results for standard FPP in Section \ref{sec:fpp1}, results for restricted path-length FPP in Section \ref{sec:Gfpp1}, including the connections between the two types of FPP, and finally in 
Section \ref{sec:dual}  the duality between weight shift and geodesic length. 
 
\subsection{Standard first-passage percolation} \label{sec:fpp1} 

In {\it standard first-passage percolation} (FPP) the passage time between two points  is defined as the minimal passage time over all self-avoiding paths.   A path  $\pi=x_{\parng{0}{n}}=(x_i)_{i=0}^n$ is {\it self-avoiding} if $x_i\ne x_j$ for all pairs $i\ne j$.
Let $\saPaths_{x,y}$ 
  denote the collection of all self-avoiding paths from $x$ to $y$, of arbitrary but finite length. Define the passage time between $x$ and $y$ as 
 \be\label{def-fpp.2}
T_{x,y}= \inf_{\pi \,\in \, \saPaths_{x,y}} \tpath(\pi) . 
 \ee  
 This definition gives $T_{x,x}=0$ because the only  self-avoiding path from $x$ to $x$ is the zero-length path.       
  A {\it geodesic} is a self-avoiding  path $\pi$ that minimizes in \eqref{def-fpp.2}. 

 When $t(e)\ge0$  the restriction to self-avoiding paths  is superfluous in the definition of $T_{x,y}$.   
Let  $p_c$ denote   the critical probability of Bernoulli bond percolation on $\Z^d$.  
A frequently used assumption in FPP  is that zero-weight edges are subcritical: 
 \be\label{pc-ass0} 
\P\{t(e)=0\}<p_c. 
\ee
 For nonnegative weights, the assumption \eqref{pc-ass0} guarantees the existence of a geodesic (Prop.~4.4 in \cite{Auf-Dam-Han-17}).

  For $b\in\R$, define  {\it $b$-shifted weights}  by 
  \be\label{b-shift} 
  \w^{(b)}=(t^{(b)}(e): e\in\cE_d) \quad\text{with}\quad 
  t^{(b)}(e)=t(e)+b \quad\text{for} \ \ e\in\cE_d. 
\ee  
   All the quantities associated with weights $\w^{(b)}$ acquire the superscript. For example,   $T^{(b)}_{x,y}$ is the passage time in \eqref{def-fpp.2} under weights $\w^{(b)}$.
Let 
\be\label{eit7}  
\eit=\P\text{-}\essinf_\w t(e)   \ee
denote  the (essential) lower  bound of the weights.  So in particular, $\w^{(-\eit)}$ is the weight configuration shifted so that the lower bound is at zero.
Since we shift weights, most of the time we have to 
replace   \eqref{pc-ass0}  with  this assumption:  
 \be\label{pc-ass}\begin{aligned}
\P\{t(e)=\eit\} < p_c .  
\end{aligned}\ee

 Let 
$\{t_i\}$ denote  i.i.d.\ copies of the edge weight $t(e)$.   The following  moment assumption will be employed for various values of $p$. 
\be\label{lin-ass5} 
\E[ \, (\min\{t_1,\dotsc, t_{2d}\})^p\,] <\infty . 
\ee

We record the Cox-Durrett shape theorem (\cite{Cox-Dur-81}, Thm.~2.17 in \cite{Auf-Dam-Han-17}), with a small extension to weights that can take negative values. This theorem is proved as Theorem \ref{thm:mu-b} in Appendix \ref{a:fpp<0}. 
  
  \begin{theorem}\label{thm:fpp1}   Assume  $\eit\ge 0$,  \eqref{pc-ass}, and the moment bound  \eqref{lin-ass5} with $p=d$.  
 Then there exists a constant $\eet>0$, determined by the dimension $d$ and the distribution of the shifted weights $\w^{(-\eit)}$,  and a full-probability event $\Omega_0$ such that   the following statements hold.  
 For each real $b>-\eit-\eet$  
 there exists a continuous, convex, positively homogeneous shape function $\gly^{(b)}:\R^d\to\R_+$ such that the limit 
 \be\label{fpp-lim}  \gly^{(b)}(\xi)=\lim_{n\to\infty}    n^{-1} T^{(b)}_{\zevec, x_n}  
 \ee
holds for each $\w\in\Omega_0$, whenever  $\{x_n\}\subset\Z^d$ satisfies 
 $x_n/n\to\xi$. 
   We have  $\gly^{(b)}(\zevec)=0$ and $\gly^{(b)}(\xi)>0$ for $\xi\ne\zevec$.  
 \end{theorem} 
 
If we require the shape function only for a single nonnegative weight distribution without the shifts, then \eqref{pc-ass} can be replaced with  the weaker assumption  \eqref{pc-ass0}, and we will occasionally do so.   
 The shape function  of unshifted weights is denoted  by $\gly=\gly^{(0)}$. 

To emphasize dependence on $b$ with  $\xi\ne\zevec$ fixed, we write 
\be\label{fppb-def} 
\fppb_\xi(b)=\gly^{(b)}(\xi) 
\qquad\text{for } \ b>-\eit-\eet   .    
\ee
 Several of our main results concern  the regularity of $\fppb_\xi$ and its connections with geodesic length.   The reason for allowing negative weights by extending the shift $b$ below $-\eit$ is to enable us to talk about the regularity of $\fppb_\xi(b)$ at $b=-\eit$.   Throughout this paper, $\eet$ is the constant specified in Theorem \ref{thm:fpp1}. 
 

 \begin{theorem}\label{thm:fppb2}  Assume $\eit\ge0$,  \eqref{pc-ass}, and the moment bound  \eqref{lin-ass5} with $p=d$.    Fix $\xi\in\R^d\tspa\setminus\{\zevec\}$. 
 \begin{enumerate}     [label=\rm(\roman{*}), ref=\rm(\roman{*})]  \itemsep=3pt 
 \item\label{thm:fppb2.i} The function  $\fppb_\xi$ of \eqref{fppb-def} is a continuous, strictly increasing, concave function on the open interval  $(-\eit-\eet, \infty)$.   
 \item\label{thm:fppb2.ii} Strict concavity holds on $[-\eit,\infty)\tspa${\rm:}
$\fppb_\xi'(a+)>\fppb_\xi'(b-)$ for $-\eit\le a<b<\infty$. Furthermore,  $\fppb_\xi'(b+)>\fppb_\xi'((-\eit)+)$  for $b\in(-\eit-\eet, -\eit)$.  
\end{enumerate} 
 \end{theorem}



  Concavity implies that one-sided derivatives   $\fppb_\xi'(b\pm)$ for $b>-\eit-\eet$   exist,  $\fppb_\xi'(b-)\ge\fppb_\xi'(b+)$,  and as functions of $b$, they are nonincreasing,   $\fppb_\xi'(b-)$ is left-continuous, and  $\fppb_\xi'(b+)$ is right-continuous. 
 {\it Strict concavity}    is the novel part of the theorem. This property is proved in Section \ref{sec:m-pf}, based on the modification argument of Section \ref{sec:concave}. 
 
   Introduce the  notation  
\be\label{L-max}\begin{aligned}
\underline L_{\tspb\zevec,x} &=\text{minimal Euclidean  length of a geodesic for $T_{\zevec,x}$}\\
\text{and} \qquad  \overline L_{\tspb\zevec,x} &=\text{maximal Euclidean  length of a geodesic for $T_{\zevec,x}$,}
\end{aligned}\ee
with the superscripted variants $\underline L^{(b)}_{\tspb\zevec,x}=\underline L_{\tspb\zevec,x}(\w^{(b)})$ and $\overline L^{\tspa(b)}_{\tspb\zevec,x}=\overline L_{\tspb\zevec,x}(\w^{(b)})$ for shifted weights $\w^{(b)}$. 
 For a continuous weight distribution $\underline L_{\tspb\zevec,x}=\overline L_{\tspb\zevec,x}$ almost surely because in that case  geodesics are unique almost surely.    This is not the case for all shifts because as $b$ increases the geodesic jumps occasionally  and at the jump locations there are two geodesics.  
 
Recall  that a geodesic for standard FPP is by definition self-avoiding.   Under the assumptions of Theorem \ref{thm:fpp1}, 
Theorem \ref{thm:mu-b} in Appendix \ref{a:fpp<0} proves that the following holds  on an event $\Omega_0$ of full probability:  
$\overline L^{\tspa(b)}_{\tspb\zevec,x}<\infty$  for all $x\in\Z^d$ and $b>-\eit-\eet$,  and
there exist  a finite  deterministic constant $c$ and a finite  random constant $K$ such that 
\be\label{L-b88.1}   
\overline L^{\tspa(b)}_{\zevec, x} \le \frac{c \abs{x}_1}{(b+\eit)\wedge 0+\eet} 
\qquad   \forall  b>-\eit-\eet \ \ 
 \text{whenever } \    \ \abs{x}_1\ge K. 
\ee



 We justify part \ref{thm:fppb2.i} of   Theorem \ref{thm:fppb2}.  This sets the stage for further discussion.   Let $ b>-\eit-\eet$. Take   $0<\delta\le b+\eit+\eet$ and $\eta>0$.  Considering the shifted weights on the   minimal and maximal length  geodesics of $ T^{(b)}_{\zevec,x}$ leads to 
\be\label{fppb49}  
   T^{(b-\delta)}_{\zevec,x} \le T^{(b)}_{\zevec,x}-  \delta \overline L^{\tspa(b)}_{\tspb\zevec,x}
 \quad\text{and}\quad   
 T^{(b+\eta)}_{\zevec,x} \le T^{(b)}_{\zevec,x}+ \eta   \underline L^{(b)}_{\tspb\zevec,x}.
 \ee 
Rearrange to 
\be\label{fppb56} 
 \frac{ T^{(b+\eta)}_{\zevec,x} - T^{(b)}_{\zevec,x}}\eta \le    \underline L^{(b)}_{\tspb\zevec,x}  \le     \overline L^{\tspa(b)}_{\tspb\zevec,x} \le \frac{ T^{(b)}_{\zevec,x}-T^{(b-\delta)}_{\zevec,x} }\delta.   
\ee 
Here are the arguments for the properties of $\gly_\xi$ claimed in part \ref{thm:fppb2.i} of   Theorem \ref{thm:fppb2}. 
 \begin{enumerate} 
  [label=\rm(i.\alph{*}), ref=\rm(i.\alph{*})]  \itemsep=3pt


\item  {\it Strict increasingness.}  In \eqref{fppb49}  take $x=x_n$ such that  $x_n/n\to\xi$.   Since $\overline L^{\tspa(b)}_{\tspb\zevec,x} \ge\abs{x}_1$,  the inequality  $\fppb_\xi(b-\delta)\le\gly_\xi(b)-\delta\abs\xi_1$ follows by taking the limit \eqref{fpp-lim} in \eqref{fppb49}.

\item  {\it Concavity}  follows by taking  the same limit in \eqref{fppb56}.


\item {\it Continuity}  of $\fppb_\xi$ on the open interval $(-\eit-\eet, \infty)$ follows from concavity.   
\end{enumerate} 

  Since $\underline L^{(b)}_{\tspb\zevec,x} \ge\abs{x}_1$,  \eqref{fppb56} and the monotonicity of the  derivatives give the easy bound 
  \be\label{fppb62}    \fppb_\xi'(b\pm) \ge \abs{\xi}_1 .\ee
 A corollary of the strict concavity given in Theorem \ref{thm:fppb2}\ref{thm:fppb2.ii} is the strict inequality  $ \fppb_\xi'(b\pm) > \abs{\xi}_1$. 
The next theorem records a slight strengthening of this and consequences of  \eqref{L-b88.1} and \eqref{fppb56}.  A precise proof is given in Section \ref{sec:m-pf}.

 \begin{theorem}\label{thm:fppb4}   Assume $\eit\ge0$,  \eqref{pc-ass}, and the moment bound  \eqref{lin-ass5} with $p=d$.      Let $\eet$ be the constant 
 specified in Theorem \ref{thm:fpp1} and  let  $c$ be the constant  in \eqref{L-b88.1}.  
 Then there exists a full-probability event $\Omega_0$ such that the following holds: 
 for all shifts $b>-\eit-\eet$,  directions  $\xi\in\R^d\tspa\setminus\{\zevec\}$, weight configurations  $\w\in\Omega_0$, and sequences $x_n/n\to\xi$, we have the bounds 
 \be\label{fppb5.24}\begin{aligned}  (1+D(b)) \abs{\xi}_1 \le   \fppb_\xi'(b+) 
 &\le 
 \varliminf_{n\to\infty} \frac{\underline L^{(b)}_{\tspb\zevec,\tspb x_n}(\w)}{n} \\
 & \le  \varlimsup_{n\to\infty} \frac{\overline L^{\tspa(b)}_{\tspb\zevec,\tspb x_n}(\w)}{n}
 \le   \fppb_\xi'(b-) \le \frac{2c}{(b+\eit)\wedge 0+\eet}\,\abs{\xi}_1 . 
 \end{aligned} \ee
 $D(b)$ is a nonincreasing function of $b$ such that $D(b)>0$ for all $b>-\eit-\eet$.  
\end{theorem} 

The  first inequality in \eqref{fppb5.24} says that the strict concavity gap $  \fppb_\xi'(b+)>\abs{\xi}_1  $  is uniform across all directions $\abs\xi_1=1$. 
This point is further strengthened to a uniformity for fixed  weight configurations $\w$  in Theorem \ref{thm:geod4}.

\begin{remark} \label{rm:geod3} Here are   points that follow Theorems \ref{thm:fppb2} and \ref{thm:fppb4}.  Let $\xi\ne\zevec$. 

\smallskip 

(i)  The inequalities in \eqref{fppb5.24} imply  the limit  of  Hammersley-Welsh, Smythe-Wierman and Kesten simultaneously for all sequences.   Under the assumptions of Theorem \ref{thm:fppb4}, 
 suppose $\fppb_\xi$ is differentiable at $b\in(-\eit-\eet, \infty)$.  Then  \eqref{fppb5.24}  implies 
 that for all   $\w\in\Omega_0$ and sequences $x_n/n\to\xi$,
\be\label{ham-w2}  \lim_{n\to\infty} \frac{\underline L^{(b)}_{\tspb\zevec,\tspb x_n}(\w)}{n}  = \lim_{n\to\infty} \frac{\overline L^{\tspa(b)}_{\tspb\zevec,\tspb x_n}(\w)}{n}= \fppb_\xi'(b).  \ee
By concavity, this happens at  all but countably many $b$.    In particular,    if $\fppb_\xi$ is a differentiable function then  geodesic lengths converge with probability one, simultaneously in all directions and  at all weight shifts.  Presently there is no proof of differentiability under any hypotheses.  Further below we show  failures of differentiability under assumptions on the atoms of the weight distribution.

 Suppose $\fppb_\xi'(b+) < \fppb_\xi'(b-)$.  Then \eqref{fppb5.24}  tells us that all the possible asymptotic normalized lengths of geodesics that go in direction $\xi$ form a subset of the interval $[\tspb\fppb_\xi'(b+) , \fppb_\xi'(b-)\tspb]$.  
Presently there is no description of this subset.   

%
For a characterization  of $[\tspb\fppb_\xi'(b+) , \fppb_\xi'(b-)\tspb]$ in terms of path length, given  below  in Theorem \ref{thm:fppb3}, we expand the class of  paths considered  to allow {\it $o(n)$-approximate geodesics}. These are paths from the origin to $n\xi+o(n)$ whose passage times are in the range $n\gly_\xi(b)+o(n)$, without necessarily being geodesics between their endpoints. 

\smallskip 

(ii)   The strict concavity of $\fppb_\xi$ given in Theorem \ref{thm:fppb2}  and the inequalities in \eqref{fppb5.24} imply that, for all $\w, \wt\w\in\Omega_0$ and sequences $x_n/n\to\xi$ and $\wt x_n/n\to\xi$,   
\be\label{685}   
\varlimsup_{n\to\infty} \frac{\overline L^{\tspa(b)}_{\tspb\zevec,\tspb x_n}(\w)}{n}
\le  \fppb_\xi'(b-) <    \fppb_\xi'(a+) \le \varliminf_{n\to\infty} \frac{\underline L^{(a)}_{\tspb\zevec,\tspb \wt x_n}(\wt\w)}{n}    
\quad \text{for all } b>a> -\eit-\eet .    
\ee
In other words,  distinct  shifts of a given weight distribution  cannot share any possible asymptotic geodesic lengths, even under distinct but typical environments $\w$ and $\wt\w$.   


\smallskip 

 (iii)   There is a corresponding monotonicity for geodesics at fixed $\w$. 
  Namely, when all the weights increase  by a common  constant, geodesics can only shrink in length. Let $\pi^{(a)}$ and $\pi^{(b)}$ be arbitrary geodesics for $T^{(a)}_{\zevec,x}$ and $T^{(b)}_{\zevec,x}$, respectively.   Then  
 \be\label{686} \begin{aligned}
\text{$\abs{\pi^{(b)}}\le \abs{\pi^{(a)}}$  for fixed   $a<b$ and $\w$. }   
\end{aligned}\ee 
 This follows from 
\be\label{687}  \begin{aligned} 
 \tpath^{(b)}(\pi^{(a)}) -(b-a)\abs{\pi^{(a)}} &=\tpath^{(a)}(\pi^{(a)}) \le \tpath^{(a)}(\pi^{(b)}) =   \tpath^{(b)}(\pi^{(b)}) -(b-a)\abs{\pi^{(b)}}\\
 & \le   \tpath^{(b)}(\pi^{(a)}) -(b-a)\abs{\pi^{(b)}}.
\end{aligned}\ee    
Furthermore, suppose a unique geodesic is chosen, for example   by  taking the  minimal one according to  some ordering of geodesics.  Then as $a$ increases to $b$, the geodesic cannot change without its length strictly shrinking: 
 \be\label{688} \begin{aligned}
\text{ for fixed   $a<b$ and $\w$, $\abs{\pi^{(b)}}=\abs{\pi^{(a)}}$  implies $\pi^{(b)}=\pi^{(a)}$. }   
\end{aligned}\ee 
This follows because the string of inequalities \eqref{687}  together with $\abs{\pi^{(b)}}=\abs{\pi^{(a)}}$ implies that 
$\tpath^{(b)}(\pi^{(a)})\le \tpath^{(b)}(\pi^{(b)})$, so $\pi^{(a)}$ is still at least as good as $\pi^{(b)}$ for weights $\{t^{(b)}(e)\}$.

\medskip

 (iv)   We establish below in Theorem \ref{thm:fppb3} that $\fppb_\xi'(b\pm)\to\abs{\xi}_1$ as $b\to\infty$. 
   Naturally, as the weight shift grows very large, it   pays less    to search for smaller weights at the expense of a longer geodesic.  
\phantom{pointlesstext} \hfill$\triangle$\end{remark}

The first inequality in \eqref{fppb5.24} implies that   asymptotically the lengths of geodesics in a particular  direction $\xi$ exceed the $\ell^1$-distance.    The next theorem strengthens this to a uniformity across all sufficiently faraway lattice endpoints.  Its proof   in Section \ref{sec:m-pf} relies on  the convex duality described in Section \ref{sec:dual}, the restricted path length  shape theorem of Appendix \ref{a:gpp},  and the modification arguments of Section \ref{sec:conc}.

\begin{theorem}\label{thm:geod4}  Assume $\eit\ge0$,   
\eqref{pc-ass},   
  and the moment bound \eqref{lin-ass5} with $p=d$.    There exist a deterministic constant $\delta>0$ and an almost surely  finite random constant $K$ such that 
$\underline L_{\tspb\zevec,x}\ge (1+\delta)\abs{x}_1$ whenever  $x\in\Z^d$ satisfies $ \abs{x}_1\ge K$.  
\end{theorem}

\medskip 
 
  We turn to nondifferentiability results for $\fppb_\xi$.   
  An {\it atom} of the weight distribution is a  value $r\in\R$ such that $\P\{t(e)=r\}>0$.

 \begin{theorem} \label{thm:ndi3} Assume $\eit\ge0$,   
  \eqref{pc-ass}, and the moment bound \eqref{lin-ass5} with $p=d$.  
Additionally, assume   that the weight  distribution
 satisfies at least one of the assumptions {\rm(a)} and {\rm(b)} below:
  \begin{enumerate} [label=\rm(\alph{*}), ref=\rm(\alph{*})]  \itemsep=3pt
 \item  zero is an atom; 
 \item there are two strictly positive atoms $r<s$ such that $s/r$ is rational. 
 \end{enumerate} 
 Then there exist constants $0<D,\delta,M<\infty$    such that 
\be\label{ndi3.i.1}   \P\bigl(\, \overline L_{\tspb\zevec,x}-\underline  L_{\tspb\zevec,x} \ge D\abs{x}_1\bigr) \ge \delta  
\qquad\text{for   $\abs{x}_1\ge M$. }  \ee
Furthermore,  for all $\xi\in\R^d\tspa\setminus\{\zevec\}$,  $\gly_\xi'(0-)-\gly_\xi'(0+)\ge D\abs\xi_1$ and so the   function  $\fppb_\xi(a)=\gly^{(a)}(\xi)$  is not differentiable at $a=0$. 
%
%
 \end{theorem}


For unbounded weights the  result above can be proved under more general assumptions on the atoms  (see Theorem \ref{thm:ndi} in Section \ref{sec:ndiff}).  

 \begin{theorem} \label{thm:ndi9} Assume $\eit\ge0$,   
  \eqref{pc-ass}, and the moment bound \eqref{lin-ass5} with $p=d$.  
 Additionally, assume  that the weight  distribution has at least two atoms. 
Then there exists a countably infinite   set $B\subset[-\eit, \infty)$ with these properties. 
  \begin{enumerate} [label=\rm(\roman{*}), ref=\rm(\roman{*})]  \itemsep=3pt
 \item  \label{thm:ndi9.i}  
   $B$ is dense in $[-\eit, \infty)$. 

\item\label{thm:ndi9.ii}  For each $b\in B$,  conclusion  \eqref{ndi3.i.1} of Theorem \ref{thm:ndi3} holds for the shifted weights $\w^{(b)}$ with constants  $D^{(b)},\delta^{(b)},M^{(b)}$ that depend on $b$. 

\item \label{thm:ndi9.iii}
For each $\xi\in\R^d\tspa\setminus\{\zevec\}$ and $b\in B$,  $\fppb_\xi(a)=\gly^{(a)}(\xi)$ is not differentiable at $a=b$.  
\end{enumerate} 
\end{theorem}

The proof of Theorem \ref{thm:ndi9} in Section \ref{sec:pf-nd} constructs the singularity set $B$ explicitly from two atoms of $t(e)$  as a countably infinite union of  arithmetic sequences.  
 
\begin{remark} \label{rm:ber} 
Standard {\it Bernoulli weights} satisfy   $\P\{t(e)=0\}+\P\{t(e)=1\}=1$.  
 In the subcritical planar Bernoulli case (that is, $d=2$, $t(e)\in\{0,1\}$ and $\P\{t(e)=0\}<\tfrac12$),  Steele and Zhang \cite{Ste-Zha-03} proved  that $\fppb_{\evec_1}(a)$ is not differentiable at $a=0$, as long as  $\P\{t(e)=0\}$ is close enough to $\tfrac12$.  Furthermore, they 
 conjectured that  $\fppb_{\evec_1}(a)$ is differentiable at all $a$ such that $\fppb_{\evec_1}(a)>-\infty$ except at $a=0$ (page 1050 in \cite{Ste-Zha-03}).    
 
 Theorem \ref{thm:ndi3} above  extends the nondifferentiability at $a=0$  to all directions $\xi$, all  dimensions, and all weight distributions that have  an atom at zero.  
   Theorem \ref{thm:ndi9} above   disproves the Steele-Zhang conjecture by showing that, in all dimensions, in  the subcritical  Bernoulli case  the nondifferentiability points form a countably infinite dense subset of $(0,\infty)$. 
\qedex\end{remark} 

\smallskip

\subsection{Restricted path-length first-passage percolation} \label{sec:Gfpp1}

 Next we discuss  FPP models that restrict the length of the   paths over which the optimization takes place but give up the self-avoidance requirement.  Remark \ref{rm:co} below characterizes  the FPP shape function $\gly$ as the positively homogeneous convex function generated by the restricted path shape functions.  In the next Section \ref{sec:dual} this leads to the convex duality of $\fppb_\xi$ and a sharpening of Theorem \ref{thm:fppb4}, and  further conceptual understanding of the previous results.  
 
      It turns out convenient to consider also a version whose paths are allowed zero steps. In this case  the set 
$\range=\{\pm \evec_1,\dotsc,\pm \evec_d\}$ of admissible steps  is  augmented to  $\zrange=\range\cup\{\zevec\}$.  
  For $x,y\in\Z^d$ and $n\in\N$  define three classes  of  paths $x_{\parng{0}{n}}=(x_i)_{i=0}^n$  from $x$ to $y$ of length $n$, presented here from largest to smallest:  
\be\label{wPdef}\begin{aligned}
   \zPaths_{\tsp x,(n),y}&=\{ x_{\parng{0}{n}}\in(\Z^d)^{n+1}:     x_0=x, x_n=y, \text{ each }   x_i-x_{i-1}\in\zrange\},   \\
 \Paths_{\tsp x,(n),y}&=\{ x_{\parng{0}{n}} \in(\Z^d)^{n+1}:     x_0=x, x_n=y, \text{ each }   x_i-x_{i-1}\in\range\},   \\[3pt] 
\text{and}\qquad   
 \saPaths_{\tsp x,(n),y}&=\{ x_{\parng{0}{n}} \in\Paths_{\tsp x,(n),y}:   
  \text{ points  $x_0, x_1, \dotsc, x_n$ are distinct}\} . 
  \end{aligned}\ee 
  
The superscript in  $\saPaths$ is for self-avoiding.  
Paths in $\Paths_{\tsp x,(n),y}$ and $\zPaths_{\tsp x,(n),y}$ are  allowed to repeat both vertices and edges.    Paths in $\Paths_{\tsp x,(n),y}$ are called {\it $\range$-admissible}, and those in $\zPaths_{\tsp x,(n),y}$ {\it $\zrange$-admissible}.    An $n$-path  $x_{\parng{0}{n}}$ from $x_0=x$ to $x_n=y$  is  an {\it $\ell^1$-path}   if $n=\abs{y-x}_1$.   
 For $n=0$ and $\wild\in\{\langle{\tt empty}\rangle ,o,\text{sa}\}$ we define each  collection $ \wPaths_{\tsp x,(0),x}$ as consisting only of the zero-length path $(x)$.   
 For $x\ne y$,  $\Paths_{\tsp x,(n),y}$ and   $\saPaths_{\tsp x,(n),y}$ are nonempty if and only if $n-\abs{y-x}_1$ is a nonnegative even integer,    while $\zPaths_{\tsp x,(n),y}$ is nonempty if and only if $n\ge\abs{y-x}_1$.  

With the three classes of paths go   three collections of points reachable by an  admissible path of length $n$   from the origin:  for the three superscripts $\wild\in\{\langle{\tt empty}\rangle ,o,\text{sa}\}$,  define 
\be\label{wDset1}
\wDset_n=\{x\in\Z^d:   \wPaths_{\tsp\zevec,(n),x}\ne\varnothing\} . 
\ee
If $0\le k<n$,  any $k$-path can be augmented to an $n$-path by adding $n-k$  zero steps,  and hence we have  $\zDset_n=\cup_{0\le k\le n}\Dset_k$. 

The environment $\w=(t(e):e\in\cE_d)$ is extended to 
  zero steps by stipulating that zero steps  always have zero weight, even when weights are shifted:   $t^{(b)}(\{x,x\})=0$ $\forall x\in\Z^d$ and $b\in\R$.    
 


 Define three 
point-to-point first-passage times between two points $x,y\in\Z^d$  with  restricted path lengths:     for   $\wild\in\{\langle{\tt empty}\rangle ,o,\text{sa}\}$, 
\begin{align} \label{wGdef} 
\wGpp_{x,(n),y} &=\min_{x_{\parng{0}{n}}\,\in\,\wPaths_{\tsp x,(n),y}}\;\sum_{k=0}^{n-1}t(\{x_k,x_{k+1}\}) \quad\text{for } \ y-x\in \wDset_n. 
\end{align} 
If $\wPaths_{\tsp x,(n),y}=\varnothing$,    set $\wGpp_{x,(n),y}=\infty$.  
Obvious relations hold between these passage times and the standard FPP from \eqref{def-fpp.2}: 
\be\label{GzG4}     \zGpp_{x,(n),y} =\min _{k:\,\abs{y-x}_1\le k\le n}\Gpp_{x,(k),y}, \ee 
\[  
  \saPaths_{x,y} =  \bigcup_{n\ge\abs{y-x}_1}  \saPaths_{\tsp x,(n), y}  
 \]
 and 
 \be\label{def-fpp}
\tpath_{x,y}= \inf_{\pi\tsp\in\tsp\saPaths_{x,y}} \tpath(\pi)=  \inf_{n:\,n\ge\abs{y-x}_1}\saGpp_{x,(n),y}. 
 \ee
For nonnegative weights  the restriction to self-avoiding paths  is superfluous  for $T_{x,y}$ and hence 
 \be\label{def-fpp2}  
 \text{if  \ $\eit\ge0$ \  then } \ \ T_{x,y}=\inf_{n:\,n\ge\abs{y-x}_1}\zGpp_{x,(n),y}=\inf_{n:\,n\ge\abs{y-x}_1}\Gpp_{x,(n),y}.\ee  

These identities point to the usefulness of $\Gpp$ and $\zGpp$. Namely, they capture the FPP passage time when the path length parameter $n$ coincides with a geodesic length.  After taking this connection to the limit, the discrepancies between the shape function of $\Gpp$ and the FPP shape function $\gly$ reveal which asymptotic path lengths are too short and which are too long to be asymptotic geodesic lengths.  
  
 The reader may wonder about the purpose of $\zGpp$ and the zero-weight zero step.  We shall see that $\zGpp$ is a convenient   link between standard FPP and   restricted path length  FPP because it  is monotone:
 \be\label{zG-mon}  \text{if }  \eit\ge0 \text{ and } m\le n \text{ then }   \zGpp_{x,(m),y}\ge \zGpp_{x,(n),y}\ge T_{x,y}. \ee
 The monotonicity is simply a consequence of the fact that any $m$-path can be augmented to an $n$-path by adding zero steps. 
 
 
The self-avoiding version  $\saGpp_{x,(n),y}$ is mentioned here to complete the overall picture but will not be used in the sequel.   Open problem \ref{sec:o-p-real} points the way to an extension of this work that requires a study of $\saGpp_{x,(n),y}$.

We state a shape theorem for restricted path length  FPP, but only on the open set   $\inter\Uset=\{\xi\in\R^d:  \abs{\xi}_1< 1\}$.   Its closure, the compact  $\ell^1$ ball  $\Uset$,   is  the  convex hull of both $\range$ and $\zrange$ and the  set of possible asymptotic velocities of admissible paths in $\wPaths_{\tsp\zevec,(n),\tspb\bbullet}$ as $n\to\infty$.    In this theorem we introduce the  parameter $\alpha$  as a variable that controls asymptotic path length.



 \begin{theorem}\label{thm:Gpp9}  Assume  $\eit>-\infty$  and that the moment bound  \eqref{lin-ass5}   holds with $p=d$ for the nonnegative weights $t_i^+=t_i\vee 0$. 
  Then there exist 
 \begin{enumerate} [label=\rm(\alph{*}), ref=\rm\alph{*}] \itemsep=3pt  
 \item  nonrandom  continuous convex  functions 
 $\gpp:\inter\Uset\to[\eit,\infty)$ and $\zgpp:\inter\Uset\to[\eit\wedge0,\infty)$
and 
 \item  an event $\Omega_0$  of $\P$-probability one 
\end{enumerate}
 such that the following statement holds for any fixed $\w\in\Omega_0$:    for any  $\xi\in\R^d$, any   real $\alpha>\abs{\xi}_1$, and any sequences $k_n\to\infty$ in $\N$, $x_n\in\Dset_{k_n}$ and  $y_n\in\zDset_{k_n}$  such that $k_n/n\to\alpha$, $x_n/n\to\xi$ and $y_n/n\to\xi$,   we have the laws of large numbers 
 \be\label{Gpp9}
 \alpha\tsp\gpp\biggl(\frac{\xi}{\alpha}\biggr)=\lim_{n\to\infty}  \frac{\Gpp_{\zevec,(k_n), x_n}}n
 \quad\text{and}\quad 
  \alpha\tsp\zgpp\biggl(\frac{\xi}{\alpha}\biggr)=\lim_{n\to\infty}  \frac{\zGpp_{\tspb\zevec,(k_n), y_n}}n. 
 \ee
 
Furthermore,  $\gpp(\zevec)=\eit$ and  $\zgpp(\zevec)=\eit\wedge 0$.   In general $ \zgpp\le\gpp$ on $\inter\Uset$.  
If $\eit\le0$ then $\gpp=\zgpp$ on all of $\inter\Uset$.   If $\eit>0$ then $\gpp>\zgpp$ in a neighborhood of the origin.  

   \end{theorem}
   
   
 The laws of large numbers \eqref{Gpp9}   come from  Theorem \ref{thm:shape-G2}   in Appendix \ref{a:gpp}. 
  The soft properties of $\gpp$ and $\zgpp$ stated in the last paragraph of Theorem \ref{thm:Gpp9} are proved  in    Lemma \ref{lm:g1} in Section \ref{sec:tech1}.      Figure \ref{fig:agxa} illustrates the limit functions in \eqref{Gpp9}. 
 
It is convenient to have $\wgpp$ defined on the whole of $\Uset$.  An attempt to do this through the laws of large numbers \eqref{Gpp9} would divert attention from the main points of this paper.  Furthermore, without stronger moment assumptions there cannot be a finite limit, as can be observed by considering $\xi=\evec_1$.  Since there is a unique $n$-path from $\zevec$ to $n\evec_1$, we see that a finite limit is possible only if $t(e)\in L^1(\P)$: 
\be\label{gpp(e1)} 
\lim_{n\to\infty} n^{-1} \wGpp_{\zevec,(n),n\evec_1}=  \lim_{n\to\infty} n^{-1} \sum_{k=1}^n t(\{(k-1)\evec_1, k\evec_1\}) =  \E[t(e)].  
\ee

 Instead of limiting passage times, we take radial limits of the shape functions from the interior as stated in the next theorem.   
 The proof of this  theorem comes in Lemma \ref{lm:g1}\ref{lm:g1-5}.   
 
 \begin{theorem}\label{thm:gpp-ext}  
Under the  assumptions of  Theorem \ref{thm:Gpp9}
 we can extend both shape functions to all of $\Uset$ via  limits along rays: for  $\wild\in\{\langle{\tt empty}\rangle ,o\}$ and  $\abs{\xi}_1=1$ define  $\wgpp(\xi)=\lim_{t\nearrow 1}\wgpp(t\xi)$.   The resulting functions 
$\gpp:\Uset\to[\eit,\infty]$ and $\zgpp:\Uset\to[\eit\wedge0,\infty]$     
 are both convex and lower semicontinuous. 
 \end{theorem} 
 
  With this theorem we can extend the functions $\wgpp$ to lower semicontinuous proper convex functions on  all of $\R^d$  by setting 
\be\label{gpp-ext5}    \wgpp(\xi)=+\infty \qquad \text{for } \ \xi\notin\Uset. \ee 
If  $\wgpp$ is finite on $\Uset$, then   $\wgpp$ is automatically 
upper semicontinuous on $\Uset$ \cite[Theorem 10.2]{Roc-70}, and hence  continuous on $\Uset$. 
 
 \medskip

 The next theorem clarifies   the relationship of  $\gpp$ and $\zgpp$  with $\gly$, beyond the obvious $\gly\le  \zgpp\le\gpp$, and links their connection with  the asymptotic geodesic lengths from Theorem \ref{thm:fppb4}.   In particular, we introduce here two functions $\lain\le\lasu$ that play several roles in our asymptotic results. In the theorem below they are first introduced as the boundaries of the regions where $\gly$  coincides with $\gpp$ and $\zgpp$. Part \ref{thm:fpp10.ii} indicates that $\lain$ and $\lasu$ are also related to the derivatives of $\gly_\xi$ and geodesic length. 
 
 These properties are then elaborated on as we proceed. The interval $[\lain(\xi), \lasu(\xi)]$ captures all the asymptotic lengths of geodesics in direction $\xi$, while the full interval is exactly the set of all asymptotic lengths of approximate geodesics (Remark \ref{rm:path}). In Theorem \ref{thm:gdiff} we see that $\lain$ and $\lasu$ describe ranges where $\gpp$ and $\zgpp$ are affine and where these two functions disagree. The macroscopic description is completed in Theorem \ref{thm:fppb3}: as the weight shift $b$ increases,  the interval  $[\, \aalain{b}(\xi),\aalasu{b}(\xi)\, ]$ shifts to the left and always   equals  the superdifferential $\partial\gly_\xi(b)$ of the concave function $\gly_\xi$.   Then we have reached the desired generalization of the Hammersley-Welsh connection \eqref{ham-w}:   the assumptions of differentiability and existence of limiting geodesic length have been dropped, and the correct identity equates  the superdifferential with the set of asymptotic lengths of approximate geodesics.  


Set 
\be\label{mgly}  \mgly=\sup_{\abs{\xi}_1=1}\gly(\xi) .  \ee
In part \ref{thm:fpp10.ii} of the theorem, on both lines of  \eqref{lainsu1}  the first inequality depends on the modification arguments and hence the subcriticality assumption is strengthened to  \eqref{pc-ass}.   To capture the complete picture  we include in  \eqref{lainsu1}   the inequalities from \eqref{fppb5.24}. 

\begin{theorem}\label{thm:fpp10} 
Assume $\eit\ge0$,   
  \eqref{pc-ass0}, and the moment bound \eqref{lin-ass5} with $p=d$.

 \begin{enumerate} [label=\rm(\roman{*}), ref=\rm(\roman{*})]  \itemsep=3pt
\item\label{thm:fpp10.i} 
There exist 
  two positively homogeneous functions $ \lain:\R^d\to\R_+$ and   $\lasu:\R^d\to[0,\infty]$
  such that $\lain\le\lasu$, and for all $\xi\in\Uset$,  
\be\label{lain2.5}   
 \zgpp(\xi)=\gly(\xi)\;\Longleftrightarrow\;\lain (\xi)\le1
\ee 
and 
\be\label{lain2.7}   
 \gpp(\xi)=\gly(\xi)\;\Longleftrightarrow\;\lain (\xi)\le1\le \lasu(\xi). 
\ee 
Furthermore,  $\lain$ is lower semicontinuous and $\lasu$ is upper semicontinuous.  If   $\eit=0$ then $\lasu(\xi)\equiv\infty$, while $\lasu$ is finite  in the case $\eit>0$. 

\item\label{thm:fpp10.ii}     Strengthen the subcriticality assumption to \eqref{pc-ass}. 
   There exists a nonrandom constant $D>0$ and a full-probability event $\Omega_0$ such that,   for all   $\xi\in\R^d\tspa\setminus\{\zevec\}$, sequences $x_n/n\to\xi$, and $\w\in\Omega_0$, 
\be\label{lainsu1}    \begin{aligned} 
 (1+D)  \abs{\xi}_1 \le \lain(\xi) &=\gly_\xi'(0+)  \le
 \varliminf_{n\to\infty} n^{-1}{\underline L_{\tspb\zevec,\tspb x_n}(\w)} \\[2pt] 
 & \le  \varlimsup_{n\to\infty} n^{-1}{\overline L_{\tspb\zevec,\tspb x_n}(\w)}
  \le \gly_\xi'(0-) < \lasu(\xi) =  \infty  \quad  \text{if } \eit=0\\[2pt]  
\text{and}\quad 
 (1+D)   \abs{\xi}_1 \le \lain(\xi) &=\gly_\xi'(0+)  \le 
 \varliminf_{n\to\infty} n^{-1}{\underline L_{\tspb\zevec,\tspb x_n}(\w)} \\[2pt] 
 & \le  \varlimsup_{n\to\infty} n^{-1}{\overline L_{\tspb\zevec,\tspb x_n}(\w)}
    \le \gly_\xi'(0-)= \lasu(\xi) \le (\mgly/\eit)\abs{\xi}_1  \quad  \text{if } \eit>0.  
\end{aligned} \ee
\end{enumerate} 

\end{theorem}

\medskip 

 We spell out some of the consequences of the theorems.  

\begin{remark}[Coincidence of shape functions]  

There exists a finite constant $\kappa$  such that    $\lain (\xi)\le \kappa\abs{\xi}_1$  $\forall \xi\in\R^d$. This follows from lower semicontinuity and homogeneity, but is also proved directly from Kesten's fundamental bound in Lemma \ref{lm:mu=g} below.    Hence  the set  $ \{\gly=\zgpp\}=\{\lain\le 1\}$ contains  the nondegenerate neighborhood $\{\xi\in\R^d: \abs{\xi}_1\le \kappa^{-1}\}$ of the origin. 

  If $\eit=0$ then  $ \{\gly=\gpp\}=\{\gly=\zgpp\}$ because $\gpp=\zgpp$.  If $\eit>0$ the equality $ \gly(\xi)=\gpp(\xi)$ holds for at least one nonzero point $\xi$ along each ray from the origin.  
  With all of the above, the first inequality of \eqref{lainsu1} implies that   $\{\gly=\gpp\}$ and $\{\gly=\zgpp\}$ are both nonempty closed subsets of $\inter\Uset$. 
 \qedex\end{remark} 
 
\begin{remark}[$o(n)$-approximate geodesics]   \label{rm:path}  For  $\alpha>\abs{\xi}_1>0$, \eqref{lain2.7}   gives the equivalence $\gly(\xi)= \alpha\tsp\gpp({\xi}/{\alpha})$ if and only if $\alpha \in[\tspb\lain(\xi), \lasu(\xi)\tspb]$.  (This is illustrated in Figure \ref{fig:agxa} below.)   By the law of large numbers \eqref{Gpp9}, this happens if and only if, with probability one, there are lattice points $x_n$ and paths $\pi^n$ from $\zevec$ to $x_n$ such that $x_n/n\to\xi$, $\abs{\pi^n}/n\to\alpha$ and $\tpath(\pi^n)/n\to\gly(\xi)$. These paths $\pi^n$ do not have to be self-avoiding or geodesics between their endpoints. But $\tpath(\pi^n)/n\to\gly(\xi)$ does imply that $\tpath(\pi^n)$ is within $o(n)$ of the passage time of the geodesic between $\zevec$ and $x_n$. 
 The asymptotic normalized  lengths  of  true self-avoiding geodesics  for $\gly(\xi)$ are a subset of the interval $[\tspb\lain(\xi), \lasu(\xi)\tspb]$ of asymptotic normalized  lengths  of $o(n)$-approximate geodesics, as indicated in \eqref{lainsu1}. 
\qedex\end{remark}

\begin{remark}[Convergence of geodesic length]    We now see the connection 
 between  the convergence of the normalized geodesic length and  the coincidence of shape functions.   In the case $\eit>0$, \eqref{lainsu1} shows that  convergence   in direction $\xi\ne\zevec$ follows from    $\lain (\xi)=\lasu(\xi)$, which is equivalent to the condition that   the set 
$\{\gly=\gpp\}$  has empty   relative interior on the $\xi$-directed  ray. 
\qedex\end{remark}

\begin{remark}[Convexity] \label{rm:co}  
Fix $\xi\in\R^d\tspa\setminus\{\zevec\}$.  For   $\wild\in\{\eee, o\}$,     the convexity and continuity  of $\wgpp$ on $\inter\Uset$  imply the convexity and continuity  of the function $\alpha\mapsto \alpha\wgpp({\xi}/{\alpha})$ defined    for  $\alpha\in(\tspb\abs{\xi}_1, \infty)$.   By Theorem \ref{thm:gpp-ext},  $\alpha\wgpp({\xi}/{\alpha})$ extends to $\alpha=\abs{\xi}_1$ by letting $\alpha\searrow\abs{\xi}_1$.  By \eqref{gpp-ext5},  we extend $\alpha\wgpp({\xi}/{\alpha})$ to $\alpha\in[0,\abs{\xi}_1)$ by setting its value equal to $+\infty$.  Thereby $\alpha\mapsto \alpha\wgpp({\xi}/{\alpha})$ is a  lower semicontinuous proper convex function on $\R_+$. 

For $\zgpp$,  monotonicity \eqref{zG-mon}   implies  further that 
\be\label{zgpp76}   \text{$\alpha\mapsto \alpha\tsp\zgpp({\xi}/{\alpha})$ is 
nonincreasing for $\alpha\in(\tspb\abs{\xi}_1, \infty)$. }  \ee

 A consequence of Theorem \ref{thm:fpp10} is that for  $\xi\in\R^d\tspa\setminus\{\zevec\}$, 
\be\label{mu-gpp} 
\gly(\xi)=\inf_{\alpha\ge\abs{\xi}_1} \alpha\tsp\wgpp\biggl(\frac{\xi}{\alpha}\biggr) =\inf_{\alpha\ge0} \alpha\tsp\wgpp\biggl(\frac{\xi}{\alpha}\biggr)
=\begin{cases}   \alpha\tsp\zgpp({\xi}/{\alpha}) &\forall\alpha \in[\tspb\lain(\xi), \infty), \\[3pt]
 \alpha\tsp\gpp({\xi}/{\alpha}) &\forall\alpha \in[\tspb\lain(\xi), \lasu(\xi)\tspb]\cap[\tspb\lain(\xi), \infty). 
 \end{cases} 
\ee
 In the language of convex analysis \cite[p.~35]{Roc-70},  the identity above characterizes  the standard FPP shape function $\gly$ as the {\it positively homogeneous convex function generated by $\wgpp$}. This means that   $\gly$ is the greatest positively homogeneous convex function such that $\gly(\zevec)\le 0$ and $\gly\le\wgpp$.   Figure \ref{fig:agxa} illustrates \eqref{mu-gpp}. 
\qedex\end{remark}

\begin{figure}[t]
	\begin{center}
		\begin{tikzpicture}[>=latex,  font=\small,scale=.55]
		\begin{scope}
			\draw[<->](0,9)--(0,0)node[below]{$0$}--(14,0)node[right]{$t$};
			\draw[dotted,line width=1pt] (0,0)--(14,-14*1.5*2/20+14*0.43);
			\draw(13,-13*1.5*2/20+13*0.43)node[rotate=15.64,below]{$\gly^{(-\eit)}(t\xi)=t\gly^{(-\eit)}(\xi)$};
			\draw[dashed] (8,3.96)--(14,6.96);
			\draw(14,7)node[rotate=28,below]{$\gly(t\xi)=t\gly(\xi)$};
			\draw[line width=0.5pt,nicosred] (0,0)--(6,2.97);
			\draw[line width=1.2pt] (6,3)--(8,4);
			\draw[dashed] (2,0)node[below]{\text{$\displaystyle\frac1{\lain^{(-\eit)}(\xi)}$}}--(2,9);
			\draw[dashed] (6,0)node[below]{$\displaystyle\frac1{\lasu(\xi)}$}--(6,9);
			\draw[dashed] (8,0)node[below]{$\displaystyle\frac1{\lain(\xi)}$}--(8,9);
			\draw[dashed] (12,0)node[below]{$\displaystyle\frac1{\abs{\xi}_1}$}--(12,-13*1.5*2/20+13*0.43-1.5);
			\draw[dashed] (12,-13*1.5*2/20+13*0.43-.5)--(12,9);
			\draw[domain=8:11.5,variable=\t,samples=10,line width=1.2pt]plot[smooth]({\t},{8-2*(12-\t)^(1/2)+.1*(\t-8)});
			\draw[domain=11.5:12,variable=\t,samples=20,line width=1.2pt]plot[smooth]({\t},{8-2*(12-\t)^(1/2)+.1*(\t-8)});
			\draw[domain=2:6,variable=\t,samples=10,line width=1.2pt,my-blue]plot[smooth]({\t},{3+(6-\t)^1.5/20+0.43*(\t-6)});
			\draw[line width=1.2pt,my-blue] (0,3+8/20-4*0.43+2*1.5*2/20-2*0.43)node[left]{\textcolor{black}{$\eit$}}--(2,3+8/20-4*0.43);
			 \draw(4,-2.8)node[below]{\textcolor{nicosred}{thin: $t\mapsto\zgpp(t\xi)$}};
			 \draw(10.5,-2.8)node[below]{\textcolor{my-blue}{thick: $t\mapsto\gpp(t\xi)$}};
			\draw(-3,6)node[above]{\textcolor{my-blue}{thick} affine};
			\draw[->](-3,5)node[above]{slope $\gly^{(-\eit)}(\xi)$} to [out=270,in=90](1,1.5);
			\draw(7,9.6)node[above]{\ \ \ \textcolor{my-blue}{thick} = \textcolor{nicosred}{thin} = $\gly(t\xi)$};
			\draw[->](7,9)node[above]{affine} to [out=-120,in=90](7,3.7);
			\draw[->](12,9.5)node[right]{\textcolor{my-blue}{thick} = \textcolor{nicosred}{thin}} to [out=180,in=120](10.5,5.9);
			\draw[->](14,8.5)node[right]{slope $=\infty$}--(12.1,8.05);
			\draw(4,6)node[above]{\textcolor{nicosred}{thin} affine};
			\draw[->](4,5)node[above]{slope $\gly(\xi)$}--(3,1.5);
		\end{scope}
		\end{tikzpicture}
	\end{center}
	\caption{\small Illustration of Theorem \ref{thm:gdiff} in the case $\eit>0$.  On the $t$-axis it is possible that the two middle points $\frac1{\lasu(\xi)}$ and $\frac1{\lain(\xi)}$ coincide. The separation illustrated here is the case where  $\lain(\xi)=\gly_\xi'(0+)<\gly_\xi'(0-)=\lasu(\xi)$, which can happen when $\eit>0$ for example in the situation described in Theorem \ref{thm:ndi3}.  Strict concavity of $\gly_\xi$ implies that the middle points    are necessarily separated from $\frac1{\lain^{(-\eit)}(\xi)}$ and $\frac1{\abs{\xi}_1}$ (see \eqref{lainsu3} below).   
	} 
	\label{fig:gdiff}
	\medskip 
\end{figure}

\medskip 

The last theorem of this section   records  further properties of $\wgpp$, illustrated in Figure \ref{fig:gdiff}.   Part  \ref{thm:gdiff.iii}   can be  proved only in Section \ref{sec:m-pf} after the modification results and hence requires the stronger subcriticality assumption \eqref{pc-ass}. 
 
 \begin{theorem}\label{thm:gdiff}   
Assume $\eit\ge0$,   
  \eqref{pc-ass0}, and the moment bound \eqref{lin-ass5} with $p=d$.  
   Fix $\xi\in\R^d\tspa\setminus\{\zevec\}$.   For   $\wild\in\{\langle{\tt empty}\rangle, o\}$, 
 the shape functions $\wgpp$ of Theorem \ref{thm:Gpp9}  have the following properties  along the $\xi$-directed ray from the origin. 
  \begin{enumerate} [label=\rm(\roman{*}), ref=\rm(\roman{*})]  \itemsep=3pt
  \item\label{thm:gdiff.i}    The function $t\mapsto\wgpp(t\xi)$ is continuous, convex and strictly increasing  for  $t\in[0, \abs{\xi}_1^{-1})$.   
  Both functions are affine at least in one  nondegenerate interval with one endpoint at the origin: for  $t\in[0, \abs{\xi}_1^{-1}]$, 
  \be\label{gdiff39} \begin{aligned} 
  t\in[\tspa 0, (\aalain{-\eit}(\xi))^{-1}\tspa] 
  \ &\iff\  \gpp(t\xi)=\eit+ t\gly^{(-\eit)}(\xi) \\
   t\in[\tspa 0,(\lain(\xi))^{-1}\tspa]   \ &\iff \   \zgpp(t\xi)= t\gly(\xi). 
  \end{aligned} 
  \ee
  
    \item\label{thm:gdiff.ii}     For  $t\in[0, \abs{\xi}_1^{-1}]$, 
   \be\label{gdiff43} \begin{aligned} 
   t\in\bigl[\tspa 0,(\tspa\lasu(\xi))^{-1}\tspa\bigr)   \ &\iff \    \gpp(t\xi) > \zgpp(t\xi) \\
      t\in[\tspa (\lasu(\xi))^{-1}, \abs{\xi}_1^{-1}] \ &\iff \     \gpp(t\xi) = \zgpp(t\xi) . 
  \end{aligned} \ee


   \item\label{thm:gdiff.iii}   Strengthen the subcriticality assumption to \eqref{pc-ass}. 
  The function $t\mapsto\wgpp(t\xi)$ is continuously differentiable on the open interval  $(0, \abs{\xi}_1^{-1})$ and $\lim_{t\nearrow\abs{\xi}_1^{-1}} (\wgpp)'(t\xi)=+\infty$.   If $\wgpp(\xi/\abs{\xi}_1)<\infty$ then 
the left derivative of $t\mapsto\wgpp(t\xi)$ at $t=\abs{\xi}_1^{-1}$  exists and equals $+\infty$.

 \end{enumerate} 

 \end{theorem} 
 
Notice that the right-hand sides in \eqref{gdiff39} agree if and only if $\eit=0$, as is consistent with the agreement $\gpp=\zgpp$ when $\eit=0$.   From \eqref{gdiff43} and \eqref{lainsu1}  we read that if $\eit>0$, the set $\{\gpp>\zgpp\}$ is an open neighborhood of $\zevec$ that consists of finite rays from the origin, while its complement  $\{\gpp=\zgpp\}$ contains the nonempty  annulus $\{\zeta\in\Uset:   (1+D)^{-1} \le\abs{\zeta}_1\le 1\}$, where $D$ is the constant in \eqref{lainsu1}.   Another consequence of \eqref{gdiff39} and \eqref{gdiff43} is that $\zgpp$ is never strictly between $\gly$ and $\gpp$ but always agrees with at least one of them. 

 By Lemma \ref{lm:co1} in Appendix \ref{a:conv}, the differentiability property in part \ref{thm:gdiff.iii}  can be equivalently stated in geometric terms  as follows:   for $\xi\in(\intr\Uset)\setminus\{\zevec\}$, the subdifferential $\partial\wgpp(\xi)$ lies on a hyperplane perpendicular to $\xi$.

\medskip 

\subsection{Duality of the weight shift and geodesic length} \label{sec:dual} 
This section  develops the duality between the weight shift  variable   $b$ in $\w\mapsto\w^{(b)}$ and the  path-length variable $\alpha$ in the limit shapes \eqref{Gpp9}.   Nonnegative weights ($\eit\ge0$) are assumed throughout. 


   Fix  $\xi\in\R^d\tspa\setminus\{\zevec\}$ for the duration of this section.  
    We  restrict  the shape function $\gly_\xi(b)$  of \eqref{fppb-def} to shifts $b\ge-\eit$  that preserve the nonnegativity of the  weights and then extend it   to an upper semicontinuous concave function on all of $\R$ by setting 
 \be\label{fppb103}     \fppc_\xi(b)=\begin{cases} \gly_\xi(b)= \gly^{(b)}(\xi), &b\ge-\eit \\
 -\infty,  &  b<-\eit.
 \end{cases}  \ee
To emphasize, the function  $ \fppc_\xi(b)$ drops   the extension to $b\in(-\eit-\e_0, -\eit)$ done in Theorem \ref{thm:fpp1}.  The reason for this choice  is that developing the duality for shifts $b<-\eit$ requires a study of the shape function of the self-avoiding version $\saG_{\zevec,(n),x}$ of restricted path length FPP. This  is not undertaken in the present paper and is left as open problem \ref{sec:o-p-real}.  
 
By definition,  the  concave dual $\fppc_\xi^*:\R\to[-\infty,\infty)$  is another  upper semicontinuous   concave function, and together $\fppc_\xi$ and $\fppc_\xi^*$ satisfy 
 \be\label{nustar9}  \fppc_\xi^*(\alpha)=\inf_{b\tsp\in\tsp\R}  \{ \alpha b-\fppc_\xi(b)\} 
 \quad\text{and}\quad  
  \fppc_\xi(b)=\inf_{\alpha\tsp\in\tsp\R}  \{ \alpha b-\fppc_\xi^*(\alpha)\}. 
  \ee
  The {\it superdifferential}  of the concave function $\fppc_\xi$ at $b$  is by definition the set 
\[ \partial\fppc_\xi(b)=\{ \alpha\in\R: \fppc_\xi(b') \le \fppc_\xi(b)+ \alpha(b'-b) \ \forall b'\in\R\} .\]
By the definition $\partial\fppc_\xi(b)=\varnothing$ for $b<-\eit$.  
For $b>-\eit$, $\partial\fppc_\xi(b)$ is the bounded  closed interval   $[\tspb\fppc_\xi'(b+), \fppc_\xi'(b-)\tspb]$ and so   $\partial\fppc_\xi(b)=\{\alpha\}$  if and only if $\fppc_\xi'(b)=\alpha$.   These general equivalences hold: 
\[  \forall \alpha, b\in\R: \quad \alpha\in\partial\fppc_\xi(b) \ \Longleftrightarrow\ 
\fppc_\xi^*(\alpha)+\fppc_\xi(b)= \alpha b  \ \Longleftrightarrow\  b\in\partial\fppc_\xi^*(\alpha). \] 

 
Theorem \ref{thm:fppb3} below establishes the convex duality. The qualitative nature of the  (negative of the) dual function in \eqref{nustar6} is illustrated in Figure \ref{fig:agxa}, on the left in the case $\eit=0$ and on the right in the case $\eit>0$.   In particular, on the left the affine portion of $\alpha\mapsto\alpha\gpp(\xi/\alpha)$ on the interval $[\tspb\aalain{b}(\xi), \aalasu{b}(\xi)\tspa]$ is the dual of the superdifferential  $\partial\gly_\xi(b)$ in \eqref{supd-1}.  The infinite slope at the left edge $\abs{\xi}_1+$  is the dual of the limit \eqref{fppb8}.  

A convenient feature of the restricted  path length shape function without zero steps is that it  transforms trivially under the weight shift:
\be\label{b-46}    \aagpp{b}(\xi) = \gpp(\xi) + b. \ee 
This and \eqref{mu-gpp} applied to $\gly^{(b)}(\xi)$ give \eqref{fppb4} below for $b>-\eit$,  which is the basis for the duality.

\begin{figure}[t]
\smallskip 
	\begin{center}
		\begin{tikzpicture}[>=latex,  font=\footnotesize,scale=.45]
		\begin{scope}
			\draw(7.5,9.5)node{\fbox{Original weights with $\eit=0$}};
			\draw[<->](0,8)--(0,0)node[below]{$\abs{\xi}_1$}--(14,0)node[right]{$\alpha$};
			\draw[dashed] (0,0.45)node[left]{$\gly(\xi)$}--(11,0.45);
			\draw[domain=0:0.5,variable=\a,samples=20,line width=1.2pt]plot[smooth]({\a},{0.5+3.5^1.5/12+2.2*1.5*3.5^0.5/12+3*1.5*3.5^0.5/12*5.3*(2/3)*(5.3^(1/3)-\a^(1/3))});
			\draw[domain=0.5:5.3,variable=\a,samples=10,line width=1.2pt]plot[smooth]({\a},{0.5+3.5^1.5/12+2.2*1.5*3.5^0.5/12+3*1.5*3.5^0.5/12*5.3*(2/3)*(5.3^(1/3)-\a^(1/3))});
			\draw[domain=5.3:7.5,variable=\a,samples=2,line width=1.2pt]plot[smooth]({\a},{-2.2*1.5*3.5^0.5/12*(\a-7.5)/(7.5-5.3)+0.5+3.5^1.5/12});
			\draw[domain=7.5:11,variable=\a,samples=5,line width=1.2pt]plot[smooth]({\a},{0.5+(11-\a)^1.5/12)});
			\draw[domain=11:14,variable=\a,samples=2,line width=1.2pt]plot[smooth]({\a},{0.5});
			\draw[dashed] (5.3,0)node[below]{$\lain^{(b)}\!(\xi)$}--(5.3,4);
			\draw[dashed] (7.5,0)node[below]{$\ \ \ \ \lasu^{(b)}\!(\xi)$}--(7.5,4);
			\draw[dashed] (11,0)node[below]{$\ \ \ \lain(\xi)$}--(11,4);
			\draw[->](6.4,5)node[above]{affine, slope $-b$}--(6.4,1.4);
			\draw[->](12.5,4)node[above]{constant=$\gly(\xi)$}--(12.5,0.7);
 			 \draw(6,-2)node[below]{$\alpha\mapsto\alpha\zgpp(\xi/\alpha)=\alpha\gpp(\xi/\alpha)$};
		\end{scope}
		\begin{scope}[shift={(18,0)}]
			 \draw(7,9.5)node{\fbox{Shifted weights $\w^{(b)}$ with $b>-\eit=0$}};
			\draw[<->](0,8)--(0,0)node[below]{$\abs{\xi}_1$}--(14,0)node[right]{$\alpha$};
			\draw[dashed] (0,0.5+1.5*3.5^0.5/12*7.5+1+3.5^1.5/12-0.04)node[left]{$\gly^{(b)}(\xi)$}--(5.3,0.5+1.5*3.5^0.5/12*7.5+1+3.5^1.5/12-0.04); 
			\draw[domain=0:0.5,variable=\a,samples=20,line width=1.2pt]plot[smooth]({\a},{1.5+3.5^1.5/12+2.2*1.5*3.5^0.5/12+3*1.5*3.5^0.5/12*5.3*(2/3)*(5.3^(1/3)-\a^(1/3))+1.5*3.5^0.5/12*\a});
			\draw[domain=0.5:5.3,variable=\a,samples=10,line width=1.2pt]plot[smooth]({\a},{1.5+3.5^1.5/12+2.2*1.5*3.5^0.5/12+3*1.5*3.5^0.5/12*5.3*(2/3)*(5.3^(1/3)-\a^(1/3))+1.5*3.5^0.5/12*\a});
			\draw[domain=5.3:7.5,variable=\a,samples=2,line width=1.2pt]plot[smooth]({\a},{1.5*3.5^0.5/12*7.5+1.5+3.5^1.5/12});
			\draw[domain=7.5:11,variable=\a,samples=5,line width=1.2pt,color=my-blue]plot[smooth]({\a},{1.5+(11-\a)^1.5/12)+(1.5*3.5^0.5/12+0.1)*\a-0.1*7.5});
			\draw[domain=11:14,variable=\a,samples=2,line width=1.2pt,color=my-blue]plot[smooth]({\a},{1.5+(1.5*3.5^0.5/12+0.1)*\a-.1*7.5});
			\draw[line width=0.5pt,color=nicosred](7.5,1.5*3.5^0.5/12*7.5+1.5+3.5^1.5/12-0.03)--(14,1.5*3.5^0.5/12*7.5+1.5+3.5^1.5/12-0.03);
			\draw[dashed] (5.3,0)node[below]{$\lain^{(b)}\!(\xi)$}--(5.3,5);
			\draw[dashed] (7.5,0)node[below]{$\ \ \ \ \aalasu{b}\!(\xi)$}--(7.5,5);
			\draw[dashed] (11,0)node[below]{$\ \ \ \lain(\xi)$}--(11,5);
						\draw[->](3,5.6)node[above]{\textcolor{my-blue}{thick}=\textcolor{nicosred}{thin}} to [out=-60,in=30](1.8,4.5);
			\draw[->](6.4,6.6)node[above]{\textcolor{my-blue}{thick}=\textcolor{nicosred}{thin}=$\gly^{(b)}(\xi)$}--(6.4,4);
			\draw[->](12.2,6.1) to [out=270,in=130](12.8,5.1);
			\draw(13,5.9)node[above]{\textcolor{my-blue}{thick} affine, slope $b$};
			\draw[->](12.5,2.1)--(12,3.5);
			\draw(11.2,1.8)node[right]{\textcolor{nicosred}{thin}=$\gly^{(b)}(\xi)$};
			 \draw(6,-1.5)node[below]{\textcolor{nicosred}{\ \ thin: $\alpha\mapsto\alpha(\zgpp)^{(b)}(\xi/\alpha)$}};
			 \draw(6,-2.5)node[below]{\textcolor{my-blue}{thick: $\alpha\mapsto\alpha\gpp^{(b)}(\xi/\alpha)$}};
		\end{scope}
		\end{tikzpicture}
	\end{center}
	\caption{\small  Fix $\xi\in\R^d\tsp\setminus\{\zevec\}$.  {\it Left}: Graphs of  the  functions $\alpha\mapsto\alpha\gly(\xi/\alpha)=\gly(\xi)$ and $\alpha\mapsto\alpha\zgpp(\xi/\alpha)=\alpha\gpp(\xi/\alpha)$ in the case $\eit=0$.  All three agree from $\lain(\xi)$ onwards to  $\lasu(\xi)=\infty$. 
	{\it Right}: Graphs of  $\alpha\gly^{(b)}(\xi/\alpha)=\gly^{(b)}(\xi)$, 
	$\alpha(\zgpp)^{(b)}(\xi/\alpha)$ and $\alpha\gpp^{(b)}(\xi/\alpha)$ for the weights shifted by $b>-\eit=0$. 
	The labeling of the $\alpha$-axis is the same in both figures.  As the weights shift to higher values, the shape functions move up.  In particular,  the thick graph $\alpha\mapsto\alpha\gpp^{(b)}(\xi/\alpha)$ on the right is obtained by adding the function $\alpha\mapsto b\alpha$ to the graph on the left.   On the possibly degenerate interval  $[\tspb\aalain{b}(\xi), \aalasu{b}(\xi)\tspa]$  we have the triple coincidence $\alpha(\zgpp)^{(b)}(\xi/\alpha)=\alpha\gpp^{(b)}(\xi/\alpha)=\gly^{(b)}(\xi)$ and after that $\alpha\gpp^{(b)}(\xi/\alpha)$ separates from the other two.   As $b$ increases, the  interval  $[\tspb\aalain{b}(\xi), \aalasu{b}(\xi)\tspa]$ moves to the left, without overlaps, approaching $\abs\xi_1$ as $b\nearrow\infty$.  In both pictures, at the left endpoint $\abs{\xi}_1+$ the graphs coming from $\zgpp$ and $\gpp$ have slope $-\infty$. The three regions $[\tspb\abs\xi_1, \tspb\aalain{b}(\xi))$, $[\tspb\aalain{b}(\xi), \aalasu{b}(\xi)\tspa]$ and $(\,\aalasu{b}(\xi), \infty)$ of qualitatively distinct behavior in the diagram on the right are described in Proposition \ref{pr:g2}.}
	\label{fig:agxa}
	\medskip
\end{figure}

\begin{theorem}\label{thm:fppb3}  
Assume $\eit\ge0$,   
  \eqref{pc-ass}, and the moment bound \eqref{lin-ass5} with $p=d$.  
 Fix  $\xi\in\R^d\tspa\setminus\{\zevec\}$. 
\begin{enumerate} [label=\rm(\roman{*}), ref=\rm(\roman{*})]  \itemsep=3pt
\item \label{thm:fppb3-2}    The concave dual of $\fppc_\xi$  is 
\be\label{nustar6}
\fppc_\xi^*(\alpha)=   \begin{cases} -\alpha\tsp\gpp(\xi/\alpha), &\alpha\ge\abs{\xi}_1 \\ -\infty,   &\alpha<\abs{\xi}_1. \end{cases}
\ee 
In particular, we have the identities 
\be\label{fppb4}
    \fppc_\xi(b)=\inf_{\alpha\ge\abs{\xi}_1} \alpha\aagpp{b}(\xi/\alpha) = 
\inf_{\alpha\ge\abs{\xi}_1} \{ \alpha\tsp\gpp(\xi/\alpha) +\alpha b\} \quad\text{ for } \  b\in\R,   
\ee
and 
\be\label{fppb5}
 \alpha\tsp\gpp(\xi/\alpha) =\sup_{b\ge-\eit} \{  \fppc_\xi(b)-\alpha b\} \quad \text{ for } \  \alpha\ge\abs{\xi}_1.  
 \ee

\item  \label{thm:fppb3-3}  For $b>-\eit$,   the superdifferential   $\partial\fppc_\xi(b)$ is the compact interval 
\be\label{supd-1}   \partial\fppc_\xi(b) =  [\tspa\fppc_\xi'(b+), \fppc_\xi'(b-)]  =   [\, \aalain{b}(\xi),\aalasu{b}(\xi)\, ]   \ee
while 
\be\label{supd-2}  \partial\fppc_\xi(-\eit)=[\, \fppc_\xi'((-\eit)+),\infty)=\bigl[\, \aalain{-\eit}(\xi), \aalasu{-\eit}(\xi)\bigr). \ee
  Furthermore, 
\be\label{fppb8} 
  \lim_{b\to\infty} \fppc_\xi'(b\pm) =\abs{\xi}_1.    \ee

\end{enumerate} 

\end{theorem}

\smallskip 

\begin{remark} \label{rmk:dual}   
%

(a) Let us make explicit the conversion back to the original FPP shape function $\gly_\xi(b)= \gly^{(b)}(\xi)$ in Theorem \ref{thm:fppb3}.  In \eqref{fppb4} $\fppc_\xi(b)$ can be replaced by $\gly_\xi(b)$ for $b\ge-\eit$.    In each of \eqref{fppb5}, \eqref{supd-1} and \eqref{fppb8},  $\fppc_\xi$ can be replaced by $\gly_\xi$.   \eqref{supd-2} cannot be valid for $\partial\gly_\xi(-\eit)$ because $\gly_\xi(b)>-\infty$  for  some $b<-\eit$.  We do have 
\be\label{supd-3} 
\gly_\xi'((-\eit)+)=\fppc_\xi'((-\eit)+)=\aalain{-\eit}(\xi)
\quad\text{ but }\quad 
\gly_\xi'((-\eit)-)<\infty=\aalasu{-\eit}(\xi).
\ee

\medskip 

(b) The strict concavity of $\fppc_\xi$ that was stated in Theorem \ref{thm:fppb2} was purposely  left out of Theorem \ref{thm:fppb3} so that this latter theorem  can be proved easily at the end of  Section \ref{sec:tech1}, before we turn to the modification arguments.   
 Combining Theorem \ref{thm:fppb3} with Theorems \ref{thm:fppb2} and \ref{thm:fppb4} and \eqref{lainsu1} gives the following. 
 There  exists a constant $\kappa<\infty$  that depends on the dimension and the weight distribution  such that, for all  $b>a> -\eit$, 
\be\label{lainsu3} 
  \abs{\xi}_1 < \aalain{b}(\xi) \le  \aalasu{b}(\xi) < \aalain{a}(\xi) \le  \aalasu{a}(\xi) <  \aalain{-\eit}(\xi)\le \kappa\abs{\xi}_1<\infty=\aalasu{-\eit}(\xi). 
  \ee
The strict inequalities above are due to the strict concavity of $\gly_\xi$. 

\medskip 

(c)   When the infimum $\eit$ of the support of the weights is zero,     $\gpp$ and $\zgpp$ coincide (Theorem \ref{thm:Gpp9} and Lemma \ref{lm:g1}\ref{lm:g1-3}).  Through  \eqref{b-46}  we get an alternative representation of the concave dual in \eqref{nustar6} in terms of the restricted path FPP shape that admits zero steps: 
\be\label{gpp3.2}  
\alpha\tsp\gpp(\xi/\alpha) =\alpha\aagpp{-\eit}(\xi/\alpha) +\alpha\eit 
=\alpha(\zgpp)^{(-\eit)}(\xi/\alpha) +\alpha\eit.  
\ee
We can combine \eqref{fppb4} and \eqref{gpp3.2} into a statement that shows that both $\gpp$ and $(\zgpp)^{(-\eit)}$ contain full information for retrieving all the shifts of $\gly$ among nonnegative weights: 
\be\label{fppb4.8}  
\gly_\xi(b)=  \inf_{\alpha\ge\abs{\xi}_1} \{ \alpha\tsp\gpp(\xi/\alpha) +\alpha b\}
=  \inf_{\alpha\ge\abs{\xi}_1} \{ \alpha\tspa(\zgpp)^{(-\eit)}(\xi/\alpha) +\alpha (\eit+b)\}
\quad\text{ for } b\ge-\eit. 
\ee

(d) Equations   \eqref{lain2.5}, \eqref{gpp3.2}, and  positive homogeneity of $\lain$ and $\gly$
  show   that $\alpha\mapsto\alpha\tsp\gpp(\xi/\alpha)$
   is affine for large $\alpha$: 
\be\label{gpp3.4}  
\alpha\tsp\gpp(\xi/\alpha) 
= \gly^{(-\eit)}(\xi)+\alpha\eit \quad\text{for $\alpha\ge\aalain{-\eit}(\xi)$.}  
\ee
The reader can recognize this statement as the dual version of 
$\partial\fppc_\xi(-\eit) =[\, \aalain{-\eit}(\xi),\infty)$ from the theorem above, and an immediate consequence of \eqref{gdiff39}.   This affine portion of $\alpha\tsp\gpp(\xi/\alpha)$ is visible in both diagrams of Figure \ref{fig:agxa}. 

Identities  \eqref{gpp3.2} and  \eqref{gpp3.4} suggest that, for $\alpha\ge\aalain{-\eit}(\xi)$,  the recipe for an  optimal path of length approximately $n\alpha$ from $\zevec$ to a point close to $n\xi$ is  this:   
shift the weights so that their infimum  is zero and take the optimal path for the shifted weights $\w^{(-\eit)}$. In particular, once $\alpha$ is above the FPP geodesic length, we can follow   the   FPP geodesic of the shifted weights $\w^{(-\eit)}$    and extend the path to length $n\alpha$ by finding  and  repeating an edge whose weight is close to the minimum $\eit$. 
\qedex\end{remark}
  


\section{Open problems}\label{sec:open}


We list here  open problems raised by the results. 

\subsection{Asymptotic length of geodesics} 

Does the Hammersley-Welsh limit  generalize in some natural way when   $\fppb_\xi'(b+) < \fppb_\xi'(b-)$?   For example,  are there  weight configurations 
$\w$ and $\wt\w$  and sequences $x_n/n\to\xi$ and $\wt x_n/n\to\xi$ such that 
\be\label{ham-w3}  \varliminf_{n\to\infty} \frac{\underline L^{(b)}_{\tspb\zevec,\tspb x_n}(\w)}{n}  = \fppb_\xi'(b+) 
\qquad\text{and}\qquad 
\varlimsup_{n\to\infty} \frac{\overline L^{\tspa(b)}_{\tspb\zevec,\tspb \wt x_n}(\wt\w)}{n}= \fppb_\xi'(b-)\,? \ee
If so, can these statements be strengthened to limits, and are they valid for all sequences and almost surely?   Even if one cannot know the limits, are the random variables $\varliminf_{n\to\infty} n^{-1}{\underline L_{\tspb\zevec,\tspb x_n}} $ and $ \varlimsup_{n\to\infty} n^{-1}{\overline L_{\tspb\zevec,\tspb x_n}}$ almost surely constant?

%
%
%
%
 
\subsection{Properties of the shape functions} 
Is $\gly_\xi$ differentiable when the weight distribution is continuous?  What about the case of a single positive atom which is not covered by Theorems \ref{thm:ndi3}--\ref{thm:ndi9}?  Is any comparison between $\gly_\xi$ and $\gly_{\wt\xi}$ possible for two distinct directions $\xi$ and $\wt\xi$?  
 Is the function $\lain$ defined in \eqref{lain2.5} 
 a norm on $\R^d$?   Do $\lain$ and $\lasu$ possess more regularity than given in Theorem \ref{thm:fpp10}\ref{thm:fpp10.ii}?

\subsection{Duality of the weight shift and geodesic length for real-valued  weights}   \label{sec:o-p-real} 
The duality described in Section \ref{sec:dual} restricted the shape function $\fppb_\xi(b)$ to nonnegative weights through definition \eqref{fppb103}. 
This leaves open the duality of  $\fppb_\xi(b)$ for  $b<-\eit$.  To capture the full convex duality over all shifts $b$  requires a study of the process $\saGpp_{\zevec,(n),x}$, restricted path length FPP that optimizes over self-avoiding paths, in a manner analogous to our study of $\Gpp_{\zevec,(n),x}$  and its shape function. 
 
The present shortcoming can be seen for example in the case $\eit=0$ of  \eqref{lainsu1} where $\lasu(\xi)$ blows up and cannot capture the left derivative $\gly_\xi'(0-)$. Graphically this same phenomenon appears in the left diagram of Figure \ref{fig:agxa} where the graph of $\alpha\gpp(\xi/\alpha)$ never separates from $\gly(\xi)$ after $\lain(\xi)$.  The graph of   the function $\alpha\sagpp(\xi/\alpha)$ of the self-avoiding version will separate from $\gly(\xi)$ for large enough $\alpha$ and capture $\gly_\xi'(0-)$.  

 \subsection{Modification arguments for real weights}  Do the van den Berg-Kesten modification arguments \cite{Ber-Kes-93} extend to weights that can take negative values?  Such an extension would permit the extension of the strict concavity of $\gly_\xi(b)$ to $b<-\eit$. 
 
  \subsection{General perturbations of weights} Develop versions of our  results for other perturbations of the weights, besides the simple shift $t^{(h)}(e)= t(e)+h$, such as the perturbations considered in \cite{Bat-20-}. 
 
 \medskip 


\section{The shape functions and lengths of optimal paths} 
 \label{sec:tech1}

This section develops soft auxiliary  results required for the main results of Section \ref{sec:rfpp}.   Along the way we prove Theorem \ref{thm:gpp-ext}, part \ref{thm:fpp10.i}  of Theorem \ref{thm:fpp10}, parts \ref{thm:gdiff.i}--\ref{thm:gdiff.ii} of Theorem  \ref{thm:gdiff}, and Theorem \ref{thm:fppb3}. 
To begin,  assume  $\eit>-\infty$  and   the moment bound  \eqref{lin-ass5}     with $p=d$ for the nonnegative weights $t^+(e)=t(e)\vee 0$. 
 Take the existence of the continuous, convex functions $\gpp, \zgpp:\inter\Uset\to[\eit\wedge 0,\infty)$ that satisfy  the laws of large numbers \eqref{Gpp9}  from  Theorem \ref{thm:shape-G2}   in Appendix \ref{a:gpp}.  The limit implies $\gpp\ge\zgpp$.    Extend the shape functions   $\gpp$ and  $\zgpp$  
to all of $\Uset$ 
  through radial limits: for   $\wild\in\{\langle{\tt empty}\rangle, o\}$,   define 
\be\label{g-ext} \wgpp(\xi)=\lim_{t\nearrow 1}\wgpp(t\xi) \;\in\;[\eit\wedge 0, \infty]
\qquad\text{for } \ \abs{\xi}_1=1.  \ee
The limit exists because $t\mapsto\wgpp(t\xi)$ is a convex function on the interval $[0,1)$.  Monotonicity \eqref{zgpp76},   $\gpp\ge\zgpp$, and the limit  combine to give,  for $\abs{\xi}_1\le \tau\le\alpha$, 
\be\label{g1-92}    \alpha\tsp\zgpp({\xi}/{\alpha})  \le  \tau\zgpp({\xi}/{\tau})\le   \tau\gpp({\xi}/{\tau}) .\ee

   Part \ref{lm:g1-5} of the  next lemma  proves Theorem \ref{thm:gpp-ext}.  
  
  \medskip 

\begin{lemma}    \label{lm:g1} 
Assume  $\eit>-\infty$  and   the moment bound  \eqref{lin-ass5}     with $p=d$ for the nonnegative weights $t^+(e)=t(e)\vee 0$.  
 The restricted path  shape functions have the following properties. 
\begin{enumerate} [label=\rm(\roman{*}), ref=\rm(\roman{*})]  \itemsep=5pt
\item \label{lm:g1-1}  $\gpp(\zevec)=\eit$ and $\zgpp(\zevec)=\eit\wedge 0$.  
\item  \label{lm:g1-3} If $\eit\le0$ then $\gpp=\zgpp$ on all of $\Uset$.  If $\eit>0$ then $\gpp>\zgpp$ in an open  neighborhood of the origin. 
\item  \label{lm:g1-4}   For all $\xi\in\R^d\tspa\setminus\{\zevec\}$ and $\alpha\ge\abs{\xi}_1$, 
\be\label{g1-42}   
 \alpha\tsp\zgpp\biggl(\frac{\xi}{\alpha}\biggr) = \inf_{\tau:\, \abs{\xi}_1\le\tau\le\alpha}  \tau\gpp\biggl(\frac{\xi}{\tau}\biggr)  
 \ee
 and the infimum on the right is attained at some $\tau\in[\,\abs{\xi}_1, \alpha]$.  In particular,  $\abs{\xi}_1=1$ implies  $\zgpp(\xi)=\gpp(\xi)$. 
 
\item  \label{lm:g1-5}  For   $\wild\in\{\langle{\tt empty}\rangle, o\}$, 
  the extended function $\wgpp$ is convex and lower semicontinuous on $\Uset$.   
\end{enumerate} 

\end{lemma} 

\begin{proof} 
Part \ref{lm:g1-1}.  The lower bounds $\gpp\ge \eit$ and $\zgpp\ge\eit\wedge 0$ on all of $\inter\Uset$ follow from $t(e)\ge\eit$ and $t(\{x,x\})=0$. Also immediate is   $\zgpp(\zevec)\le0$.  
Given $\e>0$, we can fix as measurable functions of almost every $\w$, 
\be\label{g1-e4} 
\text{an edge $e=\{x,y\}$ such that $t(e)<\eit+\e$,  and a path $\pi$ from $\zevec$ to $x$.}
\ee
For large enough $n$ consider paths $x_{\parng{0}{n}}$  that follow $\pi$ and then repeat edge  $e$  $n-\abs\pi$ times.  Then $x_n/n\to0$ and in the limit $\zgpp(\zevec)\le \gpp(\zevec)\le\eit+\e$. 



\medskip

Part \ref{lm:g1-3}.   The claim for $\eit\le0$ is true because the zero steps of a path $\pi_n\in\zPaths_{\tsp\zevec,(n),x}$ can be replaced by repetitions of an edge with weight close to $\eit$.  Here is a detailed proof. 

 Fix $\xi\in\inter\Uset$ and a sequence $x_n\in\Dset_n$ such that $x_n/n\to\xi$.  Let $\pi_n$ be an optimal path for $\zGpp_{\zevec,(n),x_n}$ and let $k_n$ be the number of zero steps in $\pi_n$.  Let $e$ and $\pi$ be as in \eqref{g1-e4}.   We construct an $\cR$-admissible  path $\pi_n'$  of length $n$ from $\zevec$ to $x_n$ or $x_n+\evec_1$ that repeats edge $e$ as many times as possible, as follows. 
\begin{itemize} 
\item First, if $k_n$ is even,  let $y_n=x_n$, and if $k_n$ is odd, let $y_n=x_n+\evec_1$.  Then $\pi_n$ (plus the $y_n-x_n$ step if necessary) goes from $\zevec$ to $y_n$ in $n-2\fl{k_n/2}$ nonzero steps.  
\item The remaining $2\fl{k_n/2}$ steps are spent in an initial segment    from $\zevec$ back  to $\zevec$ by first following $\pi$ to $x$, then back and forth across $e$ altogether $2(\fl{k_n/2}-\abs\pi)^+$ times, and then from $x$ back to $\zevec$ along $\pi$ (in reverse direction).   If $\fl{k_n/2}\le\abs\pi$ then the initial segment does not go all the way to $x$ but turns back towards $\zevec$ after $\fl{k_n/2}$ steps along $\pi$. 
\end{itemize} 

Let $e_1,\dotsc,e_m$ denote the edges of $\pi$.  We get the following bound: 
\be\label{g1-78} \begin{aligned} 
\Gpp_{\zevec,(n), y_n}
&\le \tpath(\pi_n') =  2\sum_{i=1}^{m\wedge\fl{k_n/2}} t(e_i) 
+   2\bigl(\fl{k_n/2}-\abs\pi\bigr)^+ t(e) +   \tpath(\pi_n)  + t(\{x_n,y_n\})  \\
&\le 2\sum_{i=1}^{m} t^+(e_i) 
 +  2\bigl(\fl{k_n/2}-\abs\pi\bigr)^+(\eit+\e) +  \zGpp_{\zevec,(n),x_n}  + t^+(\{x_n,x_n+\evec_1\}). 
\end{aligned}\ee
Divide through by $n$ and let $n\to\infty$ along a suitable subsequence, utilizing 
  $\eit\le0$ and $y_n/n\to\xi$.   We obtain
  \[   \gpp(\xi)\le \e+ \zgpp(\xi) +\varliminf_{n\to\infty} n^{-1}t^+(\{x_n,x_n+\evec_1\}). 
  \]
    The last term vanishes almost surely because  $n^{-1}t^+(\{x_n,x_n+\evec_1\})\to 0$ in probability.     Since $\gpp\ge\zgpp$ always,  letting $\e\searrow0$ establishes the   equality $\gpp=\zgpp$ under $\eit\le0$. 

The   statement for $\eit>0$ in Part \ref{lm:g1-3} follows from Part \ref{lm:g1-1} and continuity.  
 
\medskip

Part \ref{lm:g1-4}.  For $\eit\le0$ \eqref{g1-42} follows from $\zgpp=\gpp$ 
and  \eqref{g1-92}.  

Assume $\eit>0$.   The inequalities in \eqref{g1-92} imply that   $\le$ holds  in \eqref{g1-42}. 
To prove  the opposite inequality $\ge$  in \eqref{g1-42}, consider  first $\alpha>\abs\xi_1$ so that we can take advantage of the laws of large numbers.  Choose $k_n\to\infty$ and  $x_n\in\zDset_{k_n}$ so that $k_n/n\to\alpha$,  $\abs{x_n}_1\to\infty$ and $x_n/n\to\xi$. Begin with 
\[   \zGpp_{\zevec,(k_n),x_n}=\min_{j:\, \abs{x_n}_1\le j\le  k_n} \Gpp_{\zevec,(j), x_n}. \] 
 
Let $\e>0$ and choose a partition $\abs{\xi}_1=\tau_0<\tau_1<\dotsm<\tau_m=\alpha$ such that $\tau_{i}-\tau_{i-1}<\e$.   Choose integers $\ell_{n,i}$ such that 
$\abs{x_n}_1=\ell_{n,0}<\ell_{n,1}<\dotsm< \ell_{n,m}$,  $\ell_{n,m}\ge k_n$, 
$\ell_{n,i}/n\to\tau_i$ and $x_n\in\Dset_{\ell_{n,i}}$. (When $\ell_{n,i}>\abs{x_n}_1$,  $x_n\in\Dset_{\ell_{n,i}}$ only requires  $\ell_{n,i}$ to have the right parity.)      Then  
\begin{align*}
 \zGpp_{\zevec,(k_n),x_n} \;\ge \;\min_{1\le i\le m}  \; \min_{\ell_{n,i-1}\le j\le \ell_{n,i}} \Gpp_{\zevec,(j),x_n}
 \;\ge\;  \min_{1\le i\le m}   \Gpp_{\zevec,(\ell_{n,i} ),x_n} -  2\tpath(\pi)-  2n\e (\eit+\e) 
\end{align*}
where we again utilize \eqref{g1-e4}: 
  for $\ell_{n,i-1}\le j\le \ell_{n,i}$  whenever $x_n\in\Dset_j$,  construct an 
$\ell_{n,i}$-path from $\zevec$ to $x_n$   by first going from $\zevec$ to one endpoint of $e$, repeating $e$ as many times as needed, returning to $\zevec$, and then following an optimal $j$-path from $\zevec$ to $x_n$. (If  $\ell_{n,i}-j$ is too small to allow travel  all the way to $e$, then proceed part of the way and return to $\zevec$.  $\ell_{n,i}-j$ is even because  $x_n\in\Dset_{\ell_{n,i}}\cap\Dset_j$.)  The number of repetitions of $e$ is at most $2n\e$ when $n$ is large enough.   
 
In the limit 
\[   \alpha\tsp\zgpp({\xi}/{\alpha})  \ge \min_{1\le i\le m}  \tau_{i} \tsp\gpp(\xi/\tau_{i})  -2\e (\eit+\e) \ge  \inf_{\tau:\, \abs{\xi}_1\le\tau\le\alpha}  \tau\gpp({\xi}/{\tau})   -2\e (\eit+\e).  \]   
Let $\e\searrow 0$ to complete the proof of \eqref{g1-42} in the case $\alpha>\abs\xi_1$.     The infimum in \eqref{g1-42} is attained because on the right either  the extended function is continuous down to $\tau=\abs{\xi}_1$ or then it blows up to $\infty$. 
 
 To complete the proof of \eqref{g1-42} we show that $\zgpp(\xi)=\gpp(\xi)$ when $\abs\xi_1=1$.  Only  $\zgpp(\xi)\ge\gpp(\xi)$ needs proof.  Let $c<\gpp(\xi)$.  Since $\gpp\ge\eit>0$ we can assume $c>0$.  Pick $u<1$ so that $\gpp(s\xi)>c$ for $s\in[u,1]$.   Then by \eqref{g1-42} applied to the case  $\alpha>1$,  for $t\in[u,1)$ we have  
 \[  \zgpp(t\xi) = t\cdot  \inf_{s\tspa\in\tspa[\tspb t\tspa,\tspb 1]} \frac{\gpp(s\xi)}s \ge t\tspa c. 
 \]
 Letting $t\nearrow1$ gives $\zgpp(\xi)\ge c$.

\medskip

Part \ref{lm:g1-5}.   Convexity extends readily to all of $\Uset$.  If $\xi=\alpha\xi'+(1-\alpha)\xi''$ in $\Uset$, then for $0<t<1$ convexity on $\inter\Uset$ gives 
 $\wgpp(t\xi)\le\alpha\wgpp(t\xi')+(1-\alpha)\wgpp(t\xi'')$ and we can let $t\nearrow 1$.

 
We check the lower semicontinuity  of the extension $\zgpp$ on $\Uset$.  Since $\zgpp$ is continuous in the interior, we need to consider only limits to the boundary.  
  Let $\abs\xi_1=1$, $\zgpp(\xi)>c$  and $\xi_j\to\xi$ in $\Uset$.   By the limit in \eqref{g-ext} we can pick  $t<1$ so that $t^{-1}\zgpp(t\xi)> c$.  By the continuity of $\zgpp$ on $\inter\Uset$, $\zgpp(t\xi_j)\to\zgpp(t\xi)$.   Pick $j_0$ so that 
  $t^{-1}\zgpp(t\xi_j)> c$ for $j\ge j_0$.  Apply \eqref{g1-92} to $\xi_j$ with $\alpha=t^{-1}$ and $\tau=1$ to get 
  $\zgpp(\xi_j)\ge t^{-1}\zgpp(t\xi_j)> c$,  again for all $j\ge j_0$.  
  Lower semicontinuity  of $\zgpp$ has been established.  
  
   Lower semicontinuity  of $\gpp$ follows from $\gpp\ge\zgpp$ and the equality $\gpp=\zgpp$   on the boundary:  when $\abs\xi_1=1$  and $\xi_j\to\xi$ in $\Uset$, 
 $  \varliminf_{j\to\infty} \gpp(\xi_j) \ge    \varliminf_{j\to\infty} \zgpp(\xi_j) \ge 
\zgpp(\xi) = \gpp(\xi). 
$ 
\end{proof} 



\bigskip 

     In the remainder of this section we investigate the connections of $\zgpp$ and $\gpp$  with standard FPP and assume $\eit\ge0$ and either \eqref{pc-ass0} or  \eqref{pc-ass}.   We begin  with  the fact that  $\gly$ and $\zgpp$ coincide in a neighborhood of the origin.  
Since the next lemma  considers   nonnegative weights without any shifts,  the weaker  subcriticality assumption     \eqref{pc-ass0} is sufficient.  

\begin{lemma}\label{lm:mu=g}  Assume $\eit\ge 0$,  \eqref{pc-ass0}, and 
the moment bound  \eqref{lin-ass5} with $p=d$.  
Then there exists a  constant $\kappa\in(1,\infty)$ and a  positively homogeneous function $\lain :\R^d\to\R_+$ such that    $\abs{\xi}_1\le\lain (\xi)\le \kappa\abs{\xi}_1$  $\forall \xi\in\R^d$ and
	\begin{align}\label{lain1}
	 \text{for $\xi\in\Uset$,} \quad  \gly(\xi)=\zgpp(\xi)\;\Longleftrightarrow\;\lain (\xi)\le1.
	\end{align}
In particular,  $ \gly(\xi)=\zgpp(\xi)$ in the neighborhood $\{\xi\in\R^d: \abs{\xi}_1\le \kappa^{-1}\}$ of the origin.  	
\end{lemma}

\begin{proof}   We claim that there exists a constant $\kappa\in(1,\infty)$ such that 
\be\label{mug6} 
\forall\xi\in\R^d\tspa\setminus\{\zevec\}: \  \gly(\xi/\alpha)=\zgpp(\xi/\alpha) \text{ for } \alpha\ge\kappa\abs{\xi}_1. 
\ee
  It suffices to prove that a constant  $\kappa$ works  for all $\abs{\xi}_1=1$. 
Towards this end we show   the existence of a  deterministic  constant $\kappa$ and a random constant $M_1$ such that 
\be\label{L105} \overline L_{\tspb\zevec,x}\le \tfrac12\kappa\abs{x}_1 \quad
\text{ for all } \quad  \abs{x}_1\ge M_1.  
\ee

By Kesten's foundational estimate   (Proposition 5.8 in \cite{Kes-86-stflour}, also Lemma 4.5 in \cite{Auf-Dam-Han-17}), valid under the subcriticality assumption \eqref{pc-ass0},   there are positive  constants $\delta,c_1$ such that, for all $k\in\N$, 
\be\label{kesten1}
\P\bigl( \exists \text{ self-avoiding path $\gamma$ such that $\zevec\in\gamma$, $\abs\gamma\ge k$, and $\tpath(\gamma)\le k\delta$} \bigr) \le e^{-c_1k}. 
\ee
By adding these probabilities over the cases $\abs\gamma=k\ge n$ we get 
 \be\label{kesten5}
\P\bigl\{\tspa\text{$\exists$ self-avoiding path $\gamma$ from the origin with $\abs\gamma\ge n$ and $\tpath(\gamma)\le \delta\abs\gamma$}\tspa\bigr\} \le Ce^{-c_1n}.  
\ee 
 Thus there exists a random constant $M_1$ such that any self-avoiding path $\gamma$ from the origin of length  $\abs\gamma\ge M_1$ satisfies $\tpath(\gamma)\ge \delta\abs\gamma$. 
 
%

Since the FPP shape function $\gly$ is positively homogeneous, by   the  FPP shape theorem (\cite[p.~11]{Auf-Dam-Han-17}  
or  \eqref{shape-mu} in Appendix \ref{a:fpp<0})  we can increase  $M_1$ if necessary so that, for    a deterministic constant $c_2$, 
\be\label{sh-thm8}  T_{\zevec,x}\le c_2\abs{x}_1 
\quad \text{for all $\abs x_1\ge M_1$.}
\ee

Let $\abs{x}_1\ge M_1$ and let $\pi$ be a geodesic for $T_{\zevec, x}$.  Then  
\[   \delta\abs\pi \le \tpath(\pi) =  T_{\zevec,x}\le c_2\abs{x}_1 \]
from which  $\abs\pi\le (c_2/\delta) \abs{x}_1$.   \eqref{L105} has been verified. 


\bigskip 

Given $\xi$ such that $\abs{\xi}_1=1$, let  
$x_n/n\to\xi$.  Then for   all 
  large enough $n$,   $\overline L_{\tspb\zevec,\tspb x_n}\le n\kappa$.  Hence, recalling \eqref{GzG4}, 
\[  T_{\zevec,x_n} =\min_{\abs{x_n}_1\le k\le n\kappa} \Gpp_{\tspb\zevec,(k),x_n}=  \zGpp_{\tspb\zevec,(\fl{n\kappa}),x_n}. \]
In the limit 
$ \gly(\xi) = \kappa\zgpp(\xi/\kappa)$.  (The requirement $\kappa>1$ was imposed precisely to    justify   the limit  $n^{-1}\zGpp_{\tspb\zevec,(\fl{n\kappa}),x_n}\to \kappa\zgpp(\xi/\kappa)$.) 
 By the   lower bound $\zgpp\ge\gly$
and the monotonicity in \eqref{g1-92},  $ \gly(\xi) = \alpha\tsp\zgpp(\xi/\alpha)$ for $\alpha\ge\kappa$.   \eqref{mug6} has been verified. 

Define 
\be\label{lain8}  \lain(\xi)=\inf\{\alpha\ge\abs{\xi}_1:  \gly(\xi/\alpha)=\zgpp(\xi/\alpha)\}.  \ee
The claimed properties of the function $\lain$ follow.
\end{proof}

Later in the paper (Corollary \ref{cor:lain2}) after much more work we can show that $\lain(\xi)\ge (1+D)\abs{\xi}_1$.


In the next lemma we strengthen  the subcriticality assumption to \eqref{pc-ass} so that we can   apply Lemma \ref{lm:mu=g} to  the shifted weights $\w^{(-\eit)}$ and $\gly^{(-\eit)}(\xi)>0$.


\begin{lemma} \label{lm:gpp9}    Assume $\eit\ge 0$,  \eqref{pc-ass},  and 
the moment bound  \eqref{lin-ass5} with $p=d$.   
  For   $\wild\in\{\langle{\tt empty}\rangle, o\}$, 
 the shape functions $\wgpp$    have the following properties for a fixed $\xi\in\R^d\tspa\setminus\{\zevec\}$.  
  \begin{enumerate} [label=\rm(\roman{*}), ref=\rm(\roman{*})]  \itemsep=3pt
  \item\label{lm:gpp9.i}   On the $\xi$-directed ray  these  functions   are affine  in a nondegenerate interval started from zero:   for  $0\le t\le \abs{\xi}_1^{-1}$, 
  \be\label{gpp39} \begin{aligned} 
    t\in[\tspa 0,(\lain(\xi))^{-1}\tspa]   \ &\iff \   \zgpp(t\xi)= t\gly(\xi)\\
 \text{and } \quad  t\in[\tspa 0, (\aalain{-\eit}(\xi))^{-1}\tspa] 
  \ &\iff\  \gpp(t\xi)=\eit+ t\gly^{(-\eit)}(\xi) . 
  \end{aligned}   \ee

    \item\label{lm:gpp9.ii}    
 The function $t\mapsto\wgpp(t\xi)$ is continuous, convex, and strictly increasing  for  $t\in[0, \abs{\xi}_1^{-1})$.   
  \end{enumerate} 

 \end{lemma} 
 
 \begin{proof} 
 
 Part \ref{lm:gpp9.i}. 
The first line of \eqref{gpp39}   is exactly  \eqref{lain1}. 
Shifting weights gives $\gpp(\zeta)=\eit+\aagpp{-\eit}(\zeta)$ and Lemma \ref{lm:g1}\ref{lm:g1-3} gives $\aagpp{-\eit}(\zeta)=(\zgpp)^{(-\eit)}(\zeta)$. 
    Then the first line of \eqref{gpp39}   applied to $\w^{(-\eit)}$ gives the second line. 
    
  \medskip 

 Part \ref{lm:gpp9.ii}.  
  Continuity and convexity on $\inter\Uset$ are already in the construction of the functions $\wgpp$.   Since $\gly(\xi)\ge \gly^{(-\eit)}(\xi)>0$ (Theorem \ref{thm:fpp1}), 
  $t\mapsto\wgpp(t\xi)$ is  strictly increasing on a nondegenerate  interval from $0$.  By convexity, it has to be strictly increasing on the entire interval  $[0, \abs{\xi}_1^{-1})$.  
\end{proof} 

  Since the functions $\alpha\mapsto\alpha\wgpp(\xi/\alpha)$ are central to our treatment,   we rewrite \eqref{lain1} in this form: 
	\begin{align}\label{lain3}
	 \text{for $\xi\in\R^d\tspa\setminus\{\zevec\}$ and $\alpha\ge\abs{\xi}_1$,} \quad    \alpha\tsp\zgpp\biggl(\frac{\xi}{\alpha}\biggr) = \gly(\xi) \;\Longleftrightarrow\; \alpha\ge \lain (\xi).
	\end{align}
Together  with  \eqref{g1-42} the above implies that  some $\tau\ge\abs{\xi}_1$ satisfies $\tau\gpp(\xi/\tau)=\gly(\xi)$.  By the $\gly\le \zgpp\le\gpp$ inequalities, any such $\tau$ must satisfy  $\tau\ge\lain(\xi)$.   Now we have 
\be\label{mu-gpp.1} 
\text{  for  }\xi\in\R^d\tspa\setminus\{\zevec\} ,  \ \ \gly(\xi)=\inf_{\alpha:\,\alpha\ge\abs{\xi}_1} \alpha\tsp\gpp\biggl(\frac{\xi}{\alpha}\biggr). 
\ee
Furthermore,   for $\xi\in\R^d\tspa\setminus\{\zevec\}$, 
\be\label{lasu1}
\lasu(\xi)=\sup\Bigl\{\alpha\ge\abs{\xi}_1:    \alpha\tsp\gpp\biggl(\frac{\xi}{\alpha}\biggr) = \gly(\xi) \Bigr\}  \; \in \; [\tspa\lain(\xi),\infty]
\ee
is a meaningful definition as the supremum of a nonempty set.    
Positive homogeneity of $\lasu$ on $\R^d\tspa\setminus\{\zevec\}$   follows from the positive homogeneity  of  $\gly$.  
By Lemma \ref{lm:g1}\ref{lm:g1-3} and \eqref{lain3},  
\be\label{mgly53}  \eit=0 \quad\text{implies} \quad 
  \lasu (\xi) =\infty.  \ee
Recall $\mgly=\sup_{\abs{\xi}_1=1}\gly(\xi)$.   Let $\alpha$ be such that $\alpha\tsp\gpp(\xi/\alpha)=\gly(\xi)$.  Then 
\[   \alpha\eit \le  \alpha\tsp\gpp(\xi/\alpha)=\gly(\xi)  \le \mgly\abs{\xi}_1.  
\]  
Thus 
\be\label{mgly55}  \eit>0 \quad\text{implies} \quad 
  \lasu (\xi)  \le  (\mgly/\eit) \abs{\xi}_1.  \ee

  Since $\eit>0$ implies that  $\gpp(\zevec)=\eit>0=\gly(\zevec)$,  \eqref{lasu1} is not a meaningful definition of $\lasu(\zevec)$.  Cued by  \eqref{mgly53} and \eqref{mgly55}, we can retain positive homogeneity by defining 
  \be\label{lasu56} 
\lasu(\zevec) =\begin{cases} 0, &\eit>0 \\  \infty, &\eit=0.  \end{cases}   
  \ee



The next proposition  collects properties of the   functions $\alpha\mapsto\alpha\wgpp(\xi/\alpha)$ for  $\wild\in\{\langle{\tt empty}\rangle ,o\}$.  These properties are implicit in  the  definitions and  previously established facts.   Note that part \ref{pr:g2-1} below  is still conditional for we have not yet proved that $\abs{\xi}_1< \lain (\xi)$.  The trichotomy in the proposition is illustrated in Figure \ref{fig:agxa}.  


\begin{proposition}\label{pr:g2}  Assume $\eit\ge 0$,  \eqref{pc-ass0}, 
and the moment bound  \eqref{lin-ass5} with $p=d$.
%
  Fix  $\xi\in\R^d\tspa\setminus\{\zevec\}$.  
Then for  $\alpha\in[\,\abs{\xi}_1,\infty )$,  the functions  $\alpha\mapsto\alpha\tsp\gpp(\xi/\alpha)$  and $\alpha\mapsto\alpha\tsp\zgpp(\xi/\alpha)$  have the following  properties.  
\begin{enumerate} [label=\rm(\roman{*}), ref=\rm(\roman{*})]  \itemsep=3pt
\item \label{pr:g2-1}  For $\abs{\xi}_1\le \alpha<  \lain (\xi)$,  \  $\alpha\tsp\zgpp(\xi/\alpha)=\alpha\tsp\gpp(\xi/\alpha)$ are strictly decreasing, convex, and strictly above $\gly(\xi)$. 
\item \label{pr:g2-2}    For $\lain(\xi) \le \alpha\le \lasu(\xi)$,  \  $\alpha\tsp\zgpp(\xi/\alpha)=\alpha\tsp\gpp(\xi/\alpha)=\gly(\xi)$.  
\item  \label{pr:g2-3}  For $\alpha> \lasu(\xi)$,  $\alpha\tsp\zgpp(\xi/\alpha)=\gly(\xi)$, while  $\alpha\tsp\gpp(\xi/\alpha)>\gly(\xi)$  and $\alpha\tsp\gpp(\xi/\alpha)$  is  convex and  strictly increasing.  This case is  nonempty if and only if $\eit>0$.
\end{enumerate} 
\end{proposition}

\begin{proof} 
The inequalities 
\be\label{lainsu}    
\lain(\xi) < \infty  \quad\text{and}\quad \abs{\xi}_1 \le \lain(\xi) \le \lasu(\xi) \le \infty.  
\ee
are built into the definitions and Lemma \ref{lm:mu=g}.  

 Part \ref{pr:g2-1}.  Assume $\abs{\xi}_1< \lain (\xi)$ so there is something to check.  Since  $\alpha\mapsto\alpha\tsp\zgpp(\xi/\alpha)$ is nonincreasing, convex, and reaches its minimum $\gly(\xi)$ at $\alpha=\lain(\xi)$ but not before,  it must be strictly decreasing for  $\abs{\xi}_1\le \alpha< \lain (\xi)$.   

Suppose $\alpha_0 \zgpp(\xi/\alpha_0) < \alpha_0\gpp(\xi/\alpha_0)$ for some $\alpha_0>\abs{\xi}_1$.   (Equality holds at $\alpha_0=\abs{\xi}_1$ by Lemma \ref{lm:g1}\ref{lm:g1-4}.)   We show that $\lain (\xi) <\alpha_0$.  
By   \eqref{g1-42},  for some $\tau_0\in[\tspb\abs{\xi}_1, \alpha_0)$ and all $\alpha\in[\tau_0,\alpha_0]$ 
\[ 
\alpha_0 \zgpp(\xi/\alpha_0) =   \tau_0\gpp(\xi/\tau_0) = \inf_{\abs{\xi}_1\le \tau\le \alpha_0} \tau\gpp(\xi/\tau) 
= \inf_{\abs{\xi}_1\le \tau\le \alpha} \tau\gpp(\xi/\tau) = \alpha\tsp\zgpp(\xi/\alpha).  \]
Thus $\alpha\mapsto \alpha\tsp\zgpp(\xi/\alpha)$ is  constant on  $[\tau_0,\alpha_0]$  with $\tau_0<\alpha_0$.  It must be that   $\lain (\xi)\le \tau_0<\alpha_0$.     

\medskip

Part \ref{pr:g2-2}.   From Part \ref{pr:g2-1} and \eqref{lain3},  the behavior of $ \alpha\tsp\zgpp(\xi/\alpha)$ is completely determined.  Furthermore,   $ \alpha\tsp\gpp(\xi/\alpha)$ achieves its minimum $\gly(\xi)$ at $\alpha=\lain(\xi)$ by a combination of \eqref{g1-42} with Part \ref{pr:g2-1} and \eqref{lain3}.   Then 
$ \alpha\tsp\gpp(\xi/\alpha)$ must be nondecreasing  for $\alpha\ge\lain (\xi)$, and definition \eqref{lasu1} forces  $\alpha\tsp\gpp(\xi/\alpha)=\gly(\xi)$ for  $\lain(\xi) \le \alpha\le \lasu(\xi)$.

\medskip

Part \ref{pr:g2-3} follows from convexity and the definitions.
\end{proof}

The next lemma shows that $\lain$ is lower semicontinuous and   $\lasu$ upper semicontinuous.

\begin{lemma}\label{lm:sc}  Let $\xi_i\to\xi$ in $\R^d\tspa\setminus\{\zevec\}$.  Then 
\be\label{xi44}    \lain (\xi) \le \varliminf_{i\to\infty}  \lain (\xi_i)\le \varlimsup_{i\to\infty}  \lasu (\xi_i) \le  \lasu (\xi)   . 
\ee

\end{lemma}

\begin{proof}  If $\lain (\xi)=\abs{\xi}_1$, the first inequality of \eqref{xi44} is trivial.    Suppose $\abs{\xi}_1 < \alpha < \lain (\xi)$.  Then 
$\alpha\tsp\zgpp(\xi/\alpha)>\gly(\xi)$.   By continuity on $\inter\Uset$, 
$\alpha\tsp\zgpp(\xi_i/\alpha)>\gly(\xi_i)$   for large $i$, which implies  $\lain (\xi_i)>\alpha$.

 If $\lasu(\xi)=\infty$, the last inequality of \eqref{xi44} is trivial.  By \eqref{mgly53} and \eqref{mgly55}, the complementary case has  $\eit>0$ and therefore   $ \lasu(\xi_i)\le (\mgly/\eit)\abs{\xi_i}_1$.   Then 
 \[   \lasu(\xi_i) \tspb\gpp\biggl(\frac{\xi_i}{\lasu(\xi_i)}\biggr)=\mu(\xi_i) . \] 
Suppose a subsequence satisfies $\lasu(\xi_i)\to \tau> \lasu (\xi)\ge\abs{\xi}_1$. Then for all large enough $i$,  $\lasu(\xi_i) \ge (1+\delta)\abs{\xi_i}_1$ for some $\delta>0$.  Continuity of $\gpp$  on $\inter\Uset$ and of $\gly$ on $\R^d$  then leads to $\tau\gpp(\xi/\tau)=\gly(\xi)$, a contradiction.  
\end{proof}

  At this point we have covered everything needed to prove part \ref{thm:fpp10.i}  of Theorem \ref{thm:fpp10} and   parts \ref{thm:gdiff.i}--\ref{thm:gdiff.ii} of Theorem  \ref{thm:gdiff}. The proofs of these theorems will be completed in Section \ref{sec:pf-cdg} after the modification arguments.   As the last item of this section we prove the claims about the convex duality.

\begin{lemma}
    Assume \eqref{lin-ass5} with $p=1$. For all $\xi \in \R^d$, we have
    \begin{equation*}
        \lim_{b \to \infty} \frac{\gly^{(b)}(\xi)}{b} = \abs{\xi}_1.
    \end{equation*} 
    \label{lem:b->inf}
\end{lemma}

\begin{proof}
   We may assume that $\xi \in \Z^d\setminus\{\zevec\}$ for the following. 
   The extension to $\xi \in \R^d$ follows from the homogeneity and convexity of $\gly^{(b)}(\xi)$ in $\xi$.

\begin{claim}
   For each $\xi \in \Z^d$, there exist $2d$ edge-disjoint paths $\{\pi_i\}_{i=1}^{2d}$ from $\zevec$ to $\xi$ such that their Euclidean lengths satisfy $\abs{\pi_1} =|\xi|_1$ and  $\abs{\pi_i} \leq |\xi|_1 + 8$ for $i\ne2$.
   \label{claim:existence of disjoint paths from origin to xi}
\end{claim}
   The proof of this claim comes after the proof of the lemma, but it is intuitively clear that there exist $2d$ edge-disjoint paths from $\zevec$ to $\xi$ such that at least one path has length $|\xi|_1$ and the length of each path is at most $|\xi| _1+ C_{\xi}$ for some constant $C_{\xi}$ (see \cite[Fig.\ 2.1]{Kes-86-stflour} for the case $\xi = k \evec_i$). Then,
   \begin{equation}
       \left( 1 + \frac{r_0}{b} \right)\abs{\xi}_1 \leq \frac{\gly^{(b)}(\xi)}{b} 
       \leq \frac{\E[T^{(b)}_{\zevec, \xi}]}{b} 
       \leq \E\Bigl[b^{-1}\min_{i=1,\ldots,2d}T^{(b)}(\pi_i)\Bigr],
       \label{eq:definition of zb}
       \end{equation}
   where the first inequality follows from $T^{(b)}_{\zevec, \xi} \geq \abs{\xi}_1 (b + r_0)$, the second from subadditivity, and the third from the fact that $T^{(b)}_{\zevec, \xi}$ is an infimum over all paths from $\zevec$ to $\xi$. Denote the integrand on the right-hand side of \eqref{eq:definition of zb} by $Z_b$.

   Since $T^{(b)}(\pi_i) \leq (C_{\xi} + |\xi|_1) b + T(\pi_i)$ for $i = 1,\ldots,2d$, we have
   \begin{equation*}
       Z_b \leq  (C_{\xi} + |\xi|_1) + \min_{i=1,\ldots,2d} T(\pi_i) \quad \text{for all }b \geq 1.
   \end{equation*}
   Next, we show that $\min_{i=1,\ldots,2d} T(\pi_i)$ is integrable (see \citep[Theorem 2.2]{Auf-Dam-Han-17}) in preparation for the dominated convergence theorem. 
   A union bound over the edges of each path $\pi_i$ and independence of the edge weights in the paths implies
   \begin{equation*}
       \P\bigl\{ \min_{i=1,\ldots,2d} T(\pi_i) \geq s \bigr\} \leq \Bigl( \max_i \abs{\pi_i} \,\P \Bigl\{ t_e \geq \frac{s}{\abs{\pi_i}}  \Bigr\} \Bigr)^{2d}.
   \end{equation*}
Integrating over $s\ge0$ shows that for some constant $C_{\xi}$,
   \begin{equation*}
       \E\bigl[\min_{i=1,\ldots,2d} T(\pi_i)\bigr] \leq C_{\xi}\,\E[ \, \min\{t_1,\dotsc, t_{2d}\} \,] <\infty.
   \end{equation*}
   Since $Z_b$ can be written as
   \begin{equation*}
       Z_b = \min_{i=1,\ldots,2d} \Bigl(\abs{\pi_i}+ \frac{T(\pi_i)}{b}\Big), 
   \end{equation*}
   we see that $\lim_{b \to \infty} Z_b = \abs{\pi_1} = \abs{\xi}_1$. Therefore, by the dominated convergence theorem, we have
   \begin{equation*}
       \abs{\xi}_1 \leq \lim_{b \to \infty} \frac{\gly^{(b)}(\xi)}{b}    
       \leq \lim_{b \to \infty} \E[Z_b] = \abs{\xi}_1.\qedhere 
   \end{equation*}
\end{proof}

\begin{proof}[Proof of Claim~\ref{claim:existence of disjoint paths from origin to xi}]
   For general $\xi \in \Z^d$, let $k$ be the number of non-zero coordinates of $\xi$ and suppose $k > 1$. This is the effective dimension of the rectangle formed with the origin and $\xi$ as extreme opposing corners. We may assume without loss of generality that the first $k$ coordinates of $\xi$ are non-zero and the rest are $0$. So let $\xi = (a_1,a_2,\dotsc, a_k,0,\dotsc,0)$. 

   The first $k$ disjoint paths run along the edges of the rectangle.  Such a path   is encoded by a permutation $\sigma \in \bS_k$. For example, $\sigma = (1,2,\ldots,k)$ corresponds to the path $\zevec \to a_1 \evec_1 \to a_1 \evec_1 + a_2 \evec_2 \to \cdots$. Two paths encoded by permutations $\sigma = (\sigma_1,\ldots,\sigma_k)$ and $\mu = (\mu_1,\ldots,\mu_k)$ meet (share a vertex) before $\xi$ if and only if for some $j < k$, $\{\sigma_1,\ldots,\sigma_j\} = \{\mu_1,\ldots,\mu_j\}$. Consider the $k$ paths corresponding to the cyclic permutations:
   \begin{align*}
    \pi_1 = (1,2,\ldots,k), \,        \pi_2 = (2,3,\ldots,1) , \dotsc,   \pi_k  = (k,1,\ldots,k-1). 
   \end{align*}
  These $k$ paths are vertex disjoint, except for their first and last vertices, and have length $|\xi|_1$.

   The next $d-k$ paths are formed as follows.  For each  $j\in\{k+1,\dotsc,d\}$,  start with an $\evec_j$ step, follow the path $\pi_1$ to $\xi + \evec_j$, and  conclude with a $-\evec_j$ step to $\xi$. Get another $d-k$ paths by starting with $-\evec_j$  and finishing with $\evec_j$.  These paths have length $|\xi|_1+2$.
   
   Now we have altogether  $k+2(d-k)$ paths.   The final $k$ paths are a little trickier.

 For each $i = \{1,\ldots,k\}$, pair direction $\evec_i$ with path $p_{i+1\!\!\!\mod \!k}$. We  construct the path for $i=1$, and the rest are similar. The first step is $-\evec_1$.  Then follow   $\pi_2$ until $\pi_2$ is about to step in the $\evec_k$ direction (the last step before it steps in the $\evec_1$ direction). On the $\evec_k$ segment take $a_k + 1$ steps and then take $a_1 + 1$ steps in the $\evec_1$ direction (this avoids the $\pi_2$ path),  ending up at $\xi + \evec_k$. Finish at $\xi$ by taking a final $-\evec_k$ step. Replacing $\evec_1$ and $\pi_2$ by $\evec_j$ and $p_{j+1\!\!\! \mod\! k}$ for $j=2,\ldots,k$ gives us  $k$ such paths that are disjoint from each other and all  the previous paths  (except for their first and last vertices). All these have length $|\xi|_1 + 4$. Notice the crucial assumption of $k > 1$ for this construction.

   The $k=1$ case is covered in~\cite[Fig 2.1]{Kes-86-stflour}, as mentioned earlier. One can verify that this gives the worst case of $|\xi|_1 + 8$.
\end{proof}

\begin{proof}[Proof of Theorem \ref{thm:fppb3}]    


\medskip 

{\it Step 1. Identity \eqref{fppb4}.} 
   For $b\ge-\eit$,   \eqref{fppb4} is a combination of \eqref{mu-gpp.1} and \eqref{b-46}.      For large $\alpha$  
\be\label{gpp8} \alpha\tsp\gpp(\xi/\alpha)\le \gly(\xi)+\alpha\eit  
\ee 
because an $\fl{n\alpha}$-path from $\zevec$ to a point close to $n\xi$ can be created by following the strategy in the proof of Lemma \ref{lm:g1}\ref{lm:g1-3}:   repeat an edge close to the origin  with weight close to $\eit$ as many times as needed,  and then  follow a geodesic to a point close to $n\xi$.   Bound  \eqref{gpp8}   implies that the right-hand side of  \eqref{fppb4}  equals   $-\infty$ for $b<-\eit$.  Identity  \eqref{fppb4} has been verified for all $b\in\R$. 

\medskip 

{\it Step 2. The duality.}  
 The convexity and lower semicontinuity of $\alpha\mapsto \alpha\tsp\gpp(\xi/\alpha)$ for $\alpha\ge\abs{\xi}_1$  imply that the function defined by the right-hand side of \eqref{nustar6} is concave and upper semicontinuous.     Thus \eqref{fppb4} implies that  $\fppc_\xi$ is the concave dual of this  function. Then we can identify the dual $\fppc_\xi^*$ of $\fppc_\xi$ as \eqref{nustar6}, which gives \eqref{fppb5}.  
 

\medskip 

{\it Step 3. The superdifferentials.}  
Let $b>-\eit$.  Then  $\aalasu{b}(\xi)<\infty$ by \eqref{mgly55}.   By Proposition \ref{pr:g2} and the duality, 
\be\label{la801}  \begin{aligned} 
 { [\, \aalain{b}(\xi),\aalasu{b}(\xi)\, ] }&=\{\alpha\ge \abs{\xi}_1: \fppc_\xi(b)=\alpha\aagpp{b}(\xi/\alpha)\} \\
 &=\{\alpha\ge \abs{\xi}_1: \fppc_\xi(b)=\alpha\tsp\gpp(\xi/\alpha)+\alpha b\}\\
 &=\{\alpha\in\R:  \fppc_\xi(b)= \alpha b-\fppc_\xi^*(\alpha)\}  = \partial\fppc_\xi(b). 
 \end{aligned}
 \ee 
Similarly 
  \be\label{la804}  \begin{aligned} 
 { [\, \aalain{-\eit}(\xi),\infty) }&=\{\alpha\ge \abs{\xi}_1: \fppc_\xi(-\eit)= \alpha\aagpp{-\eit}(\xi/\alpha)\} \\
 &=\{\alpha\in\R:  \fppc_\xi(-\eit)= -\alpha\eit-\fppc_\xi^*(\alpha)\}  =\partial\fppc_\xi(-\eit). 
 \end{aligned}
 \ee 
  Fix $a>-\eit$ and let $b>a$.   From $\abs{\xi}_1\le \fppb_\xi'(b\pm)$ given in  \eqref{fppb62},  concavity,  and Lemma \ref{lem:b->inf}, 
\[  \abs{\xi}_1\le  \fppc_\xi'(b\pm) \le \frac{\fppc_\xi(b)-\fppc_\xi(a)}{b-a}    \;\to\; \abs{\xi}_1  
\quad\text{as } b\to\infty. \]
\qedhere
 \end{proof} 
 
\medskip 

\section{Modification proofs for strict concavity} \label{sec:conc}

The modification arguments provide the power to go beyond soft results. In particular, these give us the strict concavity of the shape function in the shift variable (Theorem \ref{thm:fppb2}\ref{thm:fppb2.ii}), the strict separation of  $\lain(\xi)$ from $\abs{\xi}_1$  (Theorem \ref{thm:fpp10}\ref{thm:fpp10.ii}), and the strict exceedance of $\ell^1$ distance by the geodesic length (Theorem \ref{thm:geod4}).

\subsection{Preparation for the modification arguments} 
We  adapt  to our goals the modification argument of van den Berg and Kesten \cite{Ber-Kes-93}.   Throughout this section $\eit=\essinf t(e)\ge0$.

An {\it $N\!$-box}   $B$ is by definition  a rectangular subset of $\Z^d$ of the form 
\be\label{B6}    B=\{ x=(x_1,\dotsc,x_d)\in\Z^d:   a_i\le x_i\le a_i+3N \text{ for } i\in[d]\setminus k,\, a_k\le x_k\le a_k+N\}  \ee
for some $a=(a_1,\dotsc,a_d)\in\Z^d$ and $k\in[d]$.  In other words,  one of the dimensions of  $B$ has size $N$ and the other $d-1$ dimensions are of size $3N$.    The two {\it large}  $3N\times\dotsm\times 3N$  faces of $B$ in \eqref{B6} are the subsets 
\[   \{x\in B:  x_k=a_k\}\quad \text{and}\quad \{x\in B: x_k=a_k+N\}  . \] 
The  {\it interior} of $B$  is defined by requiring $a_i< x_i< a_i+3N$ and $a_k< x_k< a_k+N$ in \eqref{B6}.  The boundary $\partial B$ of $B$ is the set of points of $B$ that have a nearest-neighbor vertex in the complement of $B$.   Our convention will be that an edge $e$ lies in   $B$ if {\it both} its endpoints lie in $B$, otherwise $e\in B^c$.  
A suitable $\ell^1$-neighborhood   around $B$ is defined by 
\be\label{B6.3}    \overline B=\{ x\in\Z^d:   \exists y\in B: \, \abs{x-y}_1\le 3N(d-1)+N\}.   \ee
The significance of the choice 
  $3N(d-1)+N$ is that   the $\ell^1$-distance from any point in $B$ to the boundary of $\overline B$ is at least as large as the distance between any two points in $B$.  

   Introduce two  parameters $0<s_0, \delta_0<\infty$ whose choices are made precise later.   Consider these conditions on the edge weights in $B$ and $\overline B$:   
\be\label{1.bl0}    \max_{e \in B} t(e) \leq s_0\,,  \ee
\be\label{1.bl1}    \sum_{e \in B} t(e) \leq s_0\,,  \ee
and
\be \label{1.bl2} \begin{aligned}  &\text{$\tpath(\pi) > (\eit + \delta_0)\abs{y-x}_1$ for every self-avoiding path $\pi$ that stays entirely in $\overline B$}\\
&\text{and whose endpoints $x$ and $y$  satisfy $\abs{y- x}_1 \geq N$.}
\end{aligned} \ee
The properties of a black box stated in the next definition depend on whether the weights are bounded or unbounded.   
We let   $\est=\P\text{-}\esssup t(e)$.

\begin{definition}[Black box] \label{d:bl}  $ $ 
 \begin{enumerate} [label=\rm(\roman{*}), ref=\rm(\roman{*})]  \itemsep=3pt
 \item  In the case of bounded weights {\rm($\est<\infty$)}, 
color a box $B$ {\it black}  if conditions \eqref{1.bl0} and \eqref{1.bl2} are satisfied.
 \item In the case of unbounded weights {\rm($\est=\infty$)}, 
color a box $B$ {\it black}  if conditions \eqref{1.bl1} and \eqref{1.bl2} are satisfied. 
 \end{enumerate} 
%
\end{definition} 

By choosing $s_0$ and $N$ large enough and $\delta_0$ small enough,  the probability  of a given $B$ being black can be made as close to $1$ as desired.    This is evident for conditions \eqref{1.bl0} and  \eqref{1.bl1}.  For  condition \eqref{1.bl2} it follows from  Lemma 5.5 in \cite{Ber-Kes-93} that we quote here:

\begin{lemma} {\rm\cite[Lemma 5.5]{Ber-Kes-93}}   Assume  \eqref{pc-ass}, that is,  that the infimum of the passage time is subcritical.   Then there exist constants $\delta_0>0$ and $D_0>0$ such that for all   $x,y\in\Z^d$, 
\be\label{vBK98} 
\P\{  T_{x,y} \le  (\eit + \delta_0)\abs{y-x}_1  \}  \le e^{-D_0\abs{y-x}_1} .  
\ee
\end{lemma} 

When $\eit>0$, Lemma 5.5 of \cite{Ber-Kes-93} requires the weaker assumption $P\{t(e)=\eit\} < \vec p_c$ where $\vec p_c$ is the critical probability of oriented bond percolation on $\Z^d$. However, since we  consider shifts of weights that can turn $\eit$ into zero, it is simpler to assume \eqref{pc-ass} for all $\eit\ge 0$ instead of keeping track when we might get by with the weaker assumption. 

The probability of the complement of \eqref{1.bl2} is then bounded by 
\begin{align*}
\P\{\text{\eqref{1.bl2} fails}\} \le \sum_{x,\,y\,\in \overline  B: \,\abs{y-x}_1\ge N} 
\P\{  T_{x,y} \le  (\eit + \delta_0)\abs{y-x}_1  \} \le  C_dN^{2d} e^{-D_0N} . 
\end{align*} 
The bound above decreases for large enough $N$ and hence gives us this conclusion: 
 \be\label{cross23} \begin{aligned} 
  &\text{There exists a fixed $\delta_0>0$ such that for any $\e>0$ there exist $N$ and $s_0$ } \\
  &\text{such that 
  $\P\{\text{box $B$ is black}\}\ge 1-\e$ while $\P\{t(e)\ge s_0\}>0$. Increasing $N$ and $s_0$ } \\
  &\text{while keeping $\delta_0$ fixed cannot violate this condition as long as $\P\{t(e)\ge s_0\}>0$.} 
\end{aligned}   \ee
Condition $\P\{t(e)\ge s_0\}>0$ is included  above simply to point out that $s_0$ is not chosen so large that property \eqref{1.bl0} becomes trivial for bounded weights. 



A nearest-neighbor path $\pi=(x_i)_{i=0}^n $ that lies in $B$  is  a {\it short crossing}  of $B$ if $x_0$ and $x_n$ lie on opposite large faces of $B$.      More generally, we say that  
\be\label{cross4} \begin{aligned} 
&\text{a path $\pi$ {\it crosses} $B$   if some segment $\pi_{x_k,x_m}=(x_i)_{i=k}^m$  of  $\pi$ is a short crossing of $B$}\\
&\text{and neither endpoint of $\pi$ lies in $B$.} 
\end{aligned} \ee
The second part of the definition ensures that $\pi$ genuinely ``goes through'' $B$. 

 Let $\cB$ be the countable set of all triples  $(B,v,w)$ where $B$ is an $N$-box  and   $v$ and $w$ are two distinct points on the boundary  of $B$.    A path $\pi$ has a {\it $(B,v,w)$-crossing}  if  \eqref{cross4} holds and  $v$ is the point where $\pi$  first enters $B$ and    
$w$ is the point through which $\pi$ last   exits $B$.  (Then the  short crossing of $B$ is some   segment $\pi_{v',w'}\subset\pi_{v,w}$.)    
  If  $\pi$   crosses $B$,  then  $\pi$  has  a $(B,v,w)$-crossing for some $(B,v,w)\in\cB$ with  $(v,w)$ uniquely determined by $\pi$ and $B$. 
  
    Partition  the set  $\cB$ of all elements $(B,v,w)$ into $K$ subcollections  $\cB_1,\dotsc, \cB_K$ 
  such that within each $\cB_j$ all boxes $B$ are separated by distance $N$.   Any  particular  box $B$ appears at most once in any particular $\cB_j$.   The number   $K$  of subcollections  depends only on the dimension $d$ and the size parameter $N$.     The particular  size $N$  of the separation  of boxes in $\cB_j$  is taken for convenience only.   In the end what matters is that the boxes are separated and that once $N$ is fixed,  $K$ is a constant.  

Let 
$  \bB(0,r)=\{ x\in\Z^d:  \abs{x}_1\le r \}  $ denote the $\ell^1$-ball (diamond) of radius $\fl{r}$ in $\Z^d$, with (inner) boundary  
$\partial \bB(0,r)=\{ x\in\Z^d:  \abs{x}_1=\fl r \}  $.   The lemma below is proved in Appendix \ref{a:Pei}. 

\begin{lemma}\label{pei-lm1}   By fixing $s_0$ and $N$ large enough and $\delta_0$ small enough as in \eqref{cross23}, we can ensure the existence of 
  constants $0<\delta_1, D_1, n_1 <\infty$ such that, for all $n\ge n_1$, 
\be\label{vw56}  \begin{aligned} 
 \P\bigl\{ &\text{every lattice path $\pi$  from the origin to $\partial \bB(0,n)$ has an index} \\
 &\quad\;  \text{  $j(\pi)\in[K]$ such that $\pi$ has at least $\fl{n\delta_1}$  $(B,v,w)$-crossings} \\
 &\quad\;  \text{  of black boxes $B$ such that    $(B,v,w)\in\cB_{j(\pi)}$}\bigr\}
 \ge 1-e^{-D_1n}. 
 \end{aligned}  \ee
\end{lemma}

We turn  to the modification argument for  the strict concavity of $\fppb_\xi$  claimed in Theorem \ref{thm:fppb2}.  

\subsection{Strict concavity} \label{sec:concave} 

 Let $\delta_0>0$ be the quantity  in   \eqref{1.bl2} in the definition of a black box.   In addition to $t(e)\ge0$ we consider two complementary assumptions on the weight distribution.  Either the weights are unbounded: 
 \be\label{16.ass3}
 \est=\infty 
\ee
and satisfy a moment bound, or the weights are bounded and have  a strictly  positive support point close enough to the lower bound: 
\be\label{ass78.1} \begin{aligned}  
&\text{the support of $t(e)$ contains a  point $r_1$ that satisfies} \\
&\qquad \qquad\qquad\qquad
 0<r_1 <\eit+\delta_0<\est<\infty. \\
\end{aligned}\ee
  If $\eit>0$ we can choose $r_1=\eit$.    Let $\eet>0$  be the constant that appears in  Theorem \ref{thm:fpp1} and in Theorem \ref{thm:mu-b}, also  equal to the constant   $\delta$   in  \eqref{kesten5} for the shifted weights $\w^{(-\eit)}$.

\begin{theorem} \label{v-thm0} Assume  $\eit\ge0$ and \eqref{pc-ass}, in other words, that weights are nonnegative   and  the infimum is subcritical.   Furthermore, assume that one of these two cases holds:   
 \begin{enumerate} [label=\rm(\alph{*}), ref=\rm(\alph{*})]  \itemsep=3pt
 \item\label{v-a} Unbounded case: the weights  satisfy  \eqref{16.ass3}  and  the  moment bound \eqref{lin-ass5} with $p=1$. 
 \item\label{v-b} Bounded case: the weights  satisfy  \eqref{ass78.1}. 
 \end{enumerate} 
 Then there exists a finite positive constant $M$ and a function  $D(b)>0$ of $b>0$  such that the following bounds hold  for all 
$b\in(0, \eit+\eet)$ and all  
  $\abs{x}_1\ge M$: 
   \begin{enumerate} [label=\rm(\roman{*}), ref=\rm(\roman{*})]  \itemsep=3pt
 \item  In the unbounded case \ref{v-a},  
\be\label{v98} 
 \E[T^{(-b)}_{\zevec,x}] \le  \E[T_{\zevec,x}] -  b\, \E[ \,\overline L_{\tspb\zevec,x}]    - D(b)b \abs{x}_1.   \ee
 
 \item   In the bounded case \ref{v-b}, 
\be\label{v98.7} 
 \E[T^{(-b)}_{\zevec,x}] \le  \E[T_{\zevec,x}] -  b\, \E[ \,\underline L_{\tspb\zevec,x}]    - D(b)b \abs{x}_1.   \ee
\end{enumerate} 
  \end{theorem}

Condition \eqref{lin-ass5} with $p=1$ guarantees that the expectation $\E[T_{\zevec,x}]$ above is  
 finite (Lemma 2.3 in \cite{Auf-Dam-Han-17}).  This together with Lemma \ref{lm:A6}  
then implies that $\E[T^{(-b)}_{\zevec,x}]$ is finite for $b\in(0, \eit+\eet)$.  
Since  $\E[ \,\overline L_{\tspb\zevec,x}]\ge \E[ \,\underline L_{\tspb\zevec,x}]$, \eqref{v98} provides a better bound than \eqref{v98.7}.  
This is due to the fact that  the modification argument  gives sharper control of the geodesic under unbounded weights. 



\medskip 

Our modification proofs  force the geodesic to follow explicitly  constructed  paths.  These paths are parametrized by two integers $k$ and  $\ell$  whose choice is governed by the support of $t(e)$ through the lemma below.  

\begin{lemma}\label{v890-lm} 
Fix reals $0<r<s$ and $b>0$.  Then  there exist arbitrarily large positive integers $k,\ell$  such that 
\be\label{v892.00} 
k(s+\delta)  < (k+2\ell)(r-\delta) < (k+2\ell)(r+\delta) < k(s-\delta)  +(2\ell-1)b
\ee
holds for sufficiently  small  real $\delta>0$. 
\end{lemma}

\begin{proof}    It suffices to show the existence of arbitrarily large positive integers $k,\ell$   that satisfy the strict inequalities 
\be\label{v890.2} 
ks < (k+2\ell)r< ks +(2\ell-1)b 
\ee
and then choose $\delta>0$ small enough.  
Let $0<\e<b/r$ and choose an integer $m>2/\e$.  Then for each $k\in\N$ there exists $\ell\in\N$ such that 
\be\label{v890.21} 
k\Bigl(\,\frac{s}r-1\Bigr) < 2\ell   < k\Bigl(\,\frac{s}r-1\Bigr) +m\e, 
\ee
and $k$ and $\ell$ can be taken arbitrarily large.  Rearranging \eqref{v890.21} and remembering the choice of $\e$ gives 
\[  ks < (k+2\ell)r < ks + m\e r < ks + mb .   \]
To get \eqref{v890.2}, take $k$ and $\ell$ large enough to have $m<2\ell-1$.
\end{proof} 

%

\begin{proof}[Proof of Theorem \ref{v-thm0}]   The proof has three stages. The first and the last are common to bounded and unbounded weights.  The most technical middle stage has to be tailored separately to the two cases.  We present the stages  in their logical  order, with  separate cases   for the middle stage. 

\bigskip  

\noindent
{\bf Stage 1 for both bounded and unbounded weights.}   

\medskip 

  Let $\pi(x)$ be a  geodesic for $T_{\zevec,x}$.   When geodesics are  not unique,  $\pi(x)$ will be chosen in particular measurable ways that are made precise later in the proofs.   Assume that  $\abs{x}_1\ge n_1$ so  that Lemma \ref{pei-lm1} applies with 
  $n=\abs{x}_1$.  The event in \eqref{vw56} lies in the union 
 \[  \bigcup_{j=1}^K \{\text{$\pi(x)$ crosses   at least $\fl{\abs{x}_1\delta_1}$   black boxes from   $\cB_j$}\}. \]
   By \eqref{vw56}, there is a {\it nonrandom}  index $j(x)\in[K]$ such that 
 \be\label{vw61}
 \P\bigl\{\text{$\pi(x)$ crosses   at least $\fl{\abs{x}_1\delta_1}$   black boxes from   $\cB_{j(x)}$}\bigr\} \ge \frac{ 1-e^{-D_1\abs{x}_1}}K. 
 \ee 
Define  the event  
\be\label{vw-La} \Lambda_{B,v,w,x} 
=\{ \text{$B$ is black and $\pi(x)$ has a $(B,v,w)$-crossing}\} . \ee
 Consequently 
 \be\label{vw58}  \begin{aligned} 
 \P\{ \text{$\Lambda_{B,v,w,x}$ occurs for   at least $\fl{\tsp\abs{x}_1\delta_1}$  elements   $(B,v,w)\in\cB_{j(x)}$}\}\ge \frac{1-e^{-D_1\abs{x}_1}}K  . 
 \end{aligned}  \ee
Turn this into a lower bound on the expected number of events, with a new constant $D_1>0$: 
\be\label{vw60}  \begin{aligned} 
 &\sum_{(B,v,w)\in\cB_{j(x)}}  \P(\Lambda_{B,v,w,x})=\E\bigl[\, \#\{  (B,v,w)\in\cB_{j(x)}:    \text{$\Lambda_{B,v,w,x}$ occurs} \} \bigr]  
 \ge D_1\abs{x}_1.   
 \end{aligned}  \ee
 
 The next Stage 2 of the proof shows that, after a modification of the environment on a black box,  the geodesic encounters a {\it $k+2\ell$ detour} whose weights are determined by the modification.  By this we mean that the geodesic runs through a straight-line $k$-step path segment of the form $\pi^+=(\pi^+_0+i\uvec)_{0\le i\le k}$ parallel to an integer  unit vector $\uvec\in\{\pm\evec_i\}_{i=1}^d$, with some initial vertex $\pi^+_0$.  A $k+2\ell$ detour associated to $\pi^+$ is a path $\pi^{++}=(\pi^{++}_i)_{0\le i\le k+2\ell}$ that shares both endpoints with $\pi^+$ and  translates the $k$-segment   by $\ell$ steps in a direction perpendicular to $\uvec$:    so for some integer unit vector $\uvec'\perp\uvec$, 
 \be\label{pi++}   \pi^{++}_i=  \begin{cases} \pi^{+}_0+i\uvec', &0\le i\le \ell \\  \pi^{+}_0+\ell\uvec'+(i-\ell)\uvec, &\ell+1\le i\le k+\ell\\
  \pi^{+}_0 +\ell\uvec'+k\uvec -(i-k-\ell) \uvec', &k+\ell+1\le i\le k+2\ell.    
\end{cases}  \ee
In particular, $\pi^+$ and $\pi^{++}$ are edge-disjoint while they share their endpoints.  
 
The $k\times\ell$  rectangle 
$G=[ \pi^+_0, \pi^+_0+k\uvec\tsp]\times[ \pi^+_0, \pi^+_0+\ell\uvec'\tspa]$ 
enclosed by $\pi^+$ and $\pi^{++}$ will be called a {\it detour rectangle}.  Its relative boundary on the plane spanned by $\{\uvec, \uvec'\}$
 is $\partial G=\pi^+\cup\pi^{++}$.   Throughout we use superscripts $+$ and $++$ to indicate objects associated with the two portions of the boundaries of detour rectangles $G$.  
 Figure \ref{fig:k+2ell:1} illustrates.

\begin{figure}[t]
	\begin{center}
		\begin{tikzpicture}[>=latex,  font=\footnotesize,scale=0.6]
		\draw[line width=0.5pt](0,0)--(0,4)--(8,4)--(8,0)--(4,0)node[above]{$\pi^+$}--(0,0);
		\draw[fill](0,0)node[below]{$\pi_0^+$}circle(2mm);
		\draw[fill](0,4)circle(2mm);
		\draw(0,4.1)node[above]{$\pi_\ell^{++}$};
		\draw[fill](8,4)circle(2mm);
		\draw(8,4.1)node[above]{$\pi_{k+\ell}^{++}$};
		\draw[fill](8,0)circle(2mm)node[below]{$\pi_k^+=\pi_{k+2\ell}^{++}$};
		\draw[<->,line width=0.5pt](0,-1.4)--(8,-1.4);
		\draw(4,-1.4)node[below]{$k$};
		\draw[<->,line width=0.5pt](-1,0)--(-1,4);
		\draw(-1,2)node[left]{$\ell$};
		\draw(4,2)node{$G$};
		\draw[<->,line width=1pt](-4,2.3)node[left]{$\uvec'$}--(-4,1)--(-2.7,1)node[below]{$\uvec$};
		\end{tikzpicture}
	\end{center}
	\caption{\small  Illustration of \eqref{pi++}: $\uvec$ and $\uvec'$ are two perpendicular unit vectors in $\Z^d$, $\pi^+$ is a path that takes $k$ $\uvec$-steps, while the detour $\pi^{++}$ first takes $\ell$ $\uvec'$-steps, followed by $k$ $\uvec$-steps, and last $\ell$ $(-\uvec')$-steps. The detour rectangle $G$ is bounded by these paths.}
	\label{fig:k+2ell:1}
	\medskip 
\end{figure}
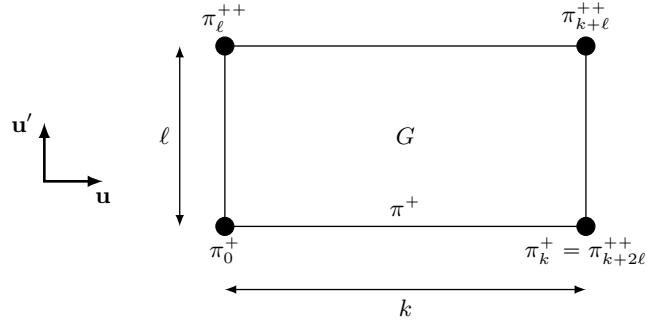
 
\medskip

Stage 2 is undertaken separately for bounded and unbounded weights. 

\bigskip  

\noindent
{\bf Stage 2 for bounded weights.}   

\medskip

\begin{lemma}\label{lm:sdelta1}  Assume \eqref{ass78.1}. 
For $ i\in\{0,1,2\}$   there exist   nondecreasing sequences $\{\param_i(q)\}_{q\in\N}$  with the following properties: 
 \be\label{nu3.31}  \eit+\delta_0 <  \param_0(q) \leq \param_1(q) \leq \param_2(q)   =\est,  \ee
   \be\label{nu3.5}   \lim_{q\to\infty}  \param_0(q)= \est  
   \quad\text{and}\quad 
    \lim_{q\to\infty}  \P\{ t(e)\le \param_0(q)\}=1,
    \ee
  \be\label{nu3.8}   \text{ for }  \e>0\text{ and } q\in\N, \ \  
   \P\{ \param_0(q) - \e \leq t(e) \leq \param_0(q) \}> 0,  
   \ee
\be  \label{nu5}
  \text{and   for }   i\in\{0,1\}  \text{ and } q\in\N, \ \   \P\bigl\{\param_i(q)\le  t(e)\le \param_{i+1}(q)\bigr\} > 0  .
\ee   
\end{lemma}

\begin{proof}   If $\P\{t(e)=\est\}>0$ then let $\param_i(q)=\est$ for all $i$ and $q$.  
So suppose $\P\{t(e)=\est\}=0$. 

Let $\param_0(0) = \eit+\delta_0$.   For $q\ge1$ define inductively  $\param_0(q)$   in the interval $(\param_0(q-1)\vee(\est-q^{-1}), \est)$ so that $\P\{\param_0(q)-\e\le t(e)\le \param_0(q) \} > 0$  for all $\e>0$.  This can be done as follows.
Let $\param_0(q)$ be  an atom of $t(e)$  in $(\param_0(q-1)\vee(\est-q^{-1}), \est)$ if one exists. If not, the c.d.f.\ of $t(e)$ is continuous in this  interval and we take $\param_0(q)$ to be  a point of strict increase which must exist.  
 
Then $\param_0(q)\to \est$ and thus $\P\{ t(e)\le \param_0(q)\}\to 1$. Furthermore, $\P\{ t(e)>\param_0(q)\} > 0$ for all $q$ because  $\param_0(q)<\est$.  Pick $\param'(q)\in[\param_{0}(q), \est)$ so that  $\P\{\param_0(q)\le  t(e)\le \param'(q)\}>0$.   Define a nondecreasing sequence by $\param_1(q)=\max_{j\le q}\param'(j)$.  Since
 $\param_1(q)<\est$  we have    $\P\{  t(e)> \param_1(q)\}>0$. 
\end{proof} 

We fix various   parameters for this stage of the proof.   Fix   $b\in(0,r_1)$ and determine $k, \lell, \delta'$ by applying Lemma \ref{v890-lm} to $0<b<r_1<\est$  to   have 
\be\label{v892.05} 
k(\est+\delta')  < (k+2\lell)(r_1-\delta') < (k+2\lell)(r_1+\delta') < k(\est-\delta')  +(2\lell-1)b. 
\ee
Since $\param_0(q)\to\est$ from below,  we can fix $q$ large enough and $\delta\in(0,\delta')$ small enough so that 
\be\label{v892.08} 
k(\param_0+\delta)  < (k+2\lell)(r_1-\delta) < (k+2\lell)(r_1+\delta) < k(\param_0-\delta)  +(2\lell-1)b
\ee
holds for    $\param_0=\param_0(q)$.  Note that this continues to hold if we increase $q$ to take $\param_0$ closer to $\est$ or decrease $\delta$.  

   Take  $N$ large enough,    $\delta_0>0$ small enough, and  $q$  large enough  so that the crossing bound \eqref{vw56}  of Lemma \ref{pei-lm1} is satisfied for the choice $\param_0=\param_0(q)$.  Drop $q$ from the notation and henceforth write $\param_i=\param_i(q)$.  
 
Shrink $\delta>0$ further so that 
   \begin{align}
   \label{44.51}  
   r_1+\delta < r_0+\delta_0
    \end{align}
    and 
   \begin{align}
   \label{44.52}  
   (\ell+1) \param_0  > (\ell+1)(r_1+\delta)+k\delta.
    \end{align}    
  
The construction to come will attach $k+2\ell$ detours  to edges of cubes.  The number of such attachments per edge  is given by the parameter 
	\begin{align*}
	\dtnr=\Bigl\lceil\frac{30d\est}{\eit+\delta_0-r_1-\delta}\Bigr\rceil+2.
	\end{align*}
     Let $m_1$  be an even positive  integer and define two constants 
                \be\label{c1} 
           c_1 = 2k\param_0 + 2m_1 (r_1 + \delta) 
       \ee
       and 
\be	\label{c2} c_2=  \eit+\delta_0 -\Bigl((r_1+\delta)\frac{m_1}{m_1 + k}+\param_0 \frac{k}{m_1 + k}\Bigr). 
\ee	
We have the  lower bound 
\[  c_2\ge c_2'=  \eit+\delta_0 -\Bigl((r_1+\delta)\frac{m_1}{m_1 + k}+\est \frac{k}{m_1 + k}\Bigr). \]  

Fix  $m_1$ 
 large enough so that 
   	 \begin{align} 
                \label{m1delta2}  &m_1\ge \frac{16 \lell \est}{\eit+\delta_0 - r_1 - \delta} 
                \, , \\[5pt] 
                \label{m1delta3}
 & m_1(r_1-\delta) > (k+2\ell)(r_1+\delta), \\[4pt] 
 \label{c2.1}	&c_2'>0 \quad
	\text{and}\quad 
	\frac{c_2'\bigl(\dtnr(m_1+k)-2\lell\bigr)}{6d\est} \ge  4 m_1+3 (k+1)(\lell+1).
	\end{align}
	Note that after fixing $m_1$,   \eqref{m1delta2} and   \eqref{m1delta3} remain true as we shrink $\delta$  and \eqref{c2.1} remains true with $c_2$ in place of $c_2'$ as we increase $\param_0$ towards $\est$. 

Set three size-determining integer parameters as 
\begin{equation}
    \lell_1=\dtnr(m_1+ k), \quad \lell_2' =\lell_1 - 2\lell, \quad\text{and}\quad \lell_2''=3\lell_2'. 
    \label{ell-defs}
\end{equation}
Set
     \be\label{m2m1} 
	m_2= \biggl\lfloor \frac{c_2\lell_2'}{6d\est}\biggr\rfloor = 
	 \biggl\lfloor\frac{c_2\bigl(\dtnr(m_1+k)-2\lell\bigr)}{6d\est} \biggr\rfloor\ge  4 m_1+3 (k+1)(\lell+1)   
  	\ee 
where we appealed to \eqref{c2.1}. 

As the last step  fix  $N$  so that 
   $N-2\lell_2'$ is a multiple of $\lell_1$ and large enough so that 
   	\be\label{Q7} \parQ= c_2N  - 4d(\lell_2''+\lell_1) \est - c_1\ge c_2N/2. \ee
  Increasing $N$ may force us to take $\param_0$ closer to $\est$ to maintain the crossing bound \eqref{vw56}.    As observed above, this can be done while maintaining all the inequalities above.  

We perform a  construction within each $N$-box $B$. 
 Let $V$ be a box inside $B$ that is    tiled with   cubes $V_i$ of the form $\prod_{j=1}^d [u_j,  u_j+\lell_1] $ where $(u_1,\dotsc, u_d)\in\Z^d$ is the lower left corner of the  cube and  the side-length $\lell_1$  comes from \eqref{ell-defs}.       The cubes $V_i$  are nonoverlapping but  neighboring  cubes share a  $(d-1)$-dimensional face. Then, $V = \bigcup_{i=1}^{\alpha} V_i$ where 
$\alpha=3^{d-1}\lell_1^{-d}({N-2\lell_2'})^d$ is the number of cubes required to tile $V$.     Inside box $B$, $V$ is surrounded   by an annular region $B\setminus V$ whose  thickness (perpendicular distance from a face of $V$ to $B^c$) is $\lell_2'$ in the direction where $B$ has width $N$ and $\lell_2''$ in the other directions. 
  
  A {\it boundary edge} of a cube $V_j$  is one of the $2^{d-1}d$ line segments (one-dimensional faces)  of length $\lell_1$ that lie on  the boundary $\partial V_j$.  
  
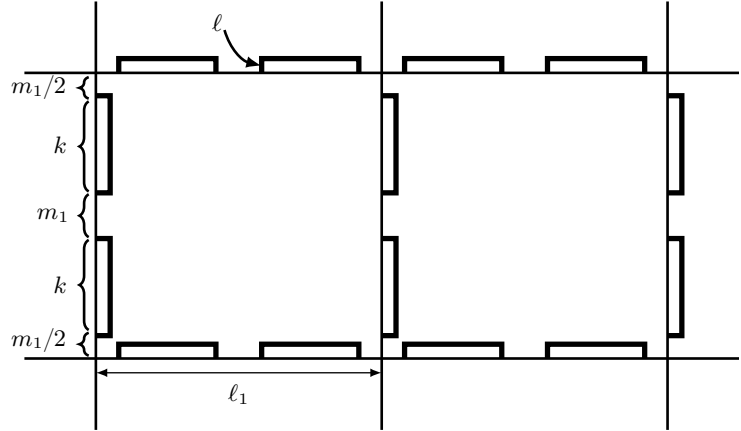
\begin{figure}[t]
	\begin{center}
		\begin{tikzpicture}[>=latex,  font=\footnotesize,scale=1.9]
		        \draw[line width=1pt](-.5,0)--(4.5,0);
		        \draw[line width=1pt](-.5,2)--(4.5,2);
		        \draw[line width=1pt](0,-.5)--(0,2.5);
		        \draw[line width=1pt](2,-.5)--(2,2.5);
		        \draw[line width=1pt](4,-.5)--(4,2.5);
		        \draw[line width=2pt](0,0.16)--(0.1,0.16)--(0.1,0.84)--(0,0.84);
		        \draw[line width=2pt](0,1.16)--(0.1,1.16)--(0.1,1.84)--(0,1.84);
		        \draw[line width=2pt](2,0.16)--(2.1,0.16)--(2.1,0.84)--(2,0.84);
		        \draw[line width=2pt](2,1.16)--(2.1,1.16)--(2.1,1.84)--(2,1.84);
		        \draw[line width=2pt](4,0.16)--(4.1,0.16)--(4.1,0.84)--(4,0.84);
		        \draw[line width=2pt](4,1.16)--(4.1,1.16)--(4.1,1.84)--(4,1.84);
		        \draw[line width=2pt](0.16,0)--(0.16,0.1)--(0.84,0.1)--(0.84,0);
		        \draw[line width=2pt](1.16,0)--(1.16,0.1)--(1.84,0.1)--(1.84,0);
		        \draw[line width=2pt](0.16,2)--(0.16,2.1)--(0.84,2.1)--(0.84,2);
		        \draw[line width=2pt](1.16,2)--(1.16,2.1)--(1.84,2.1)--(1.84,2);
		        \draw[line width=2pt](2.16,0)--(2.16,0.1)--(2.84,0.1)--(2.84,0);
		        \draw[line width=2pt](3.16,0)--(3.16,0.1)--(3.84,0.1)--(3.84,0);
		        \draw[line width=2pt](2.16,2)--(2.16,2.1)--(2.84,2.1)--(2.84,2);
		        \draw[line width=2pt](3.16,2)--(3.16,2.1)--(3.84,2.1)--(3.84,2);
		        \draw[decorate, decoration={brace, amplitude=3.5pt},line width=1pt] (-.05,0.02) -- (-0.05,0.18);
			\draw(-.4,0.13)node{$m_1/2$};
		        \draw[decorate, decoration={brace, amplitude=3.5pt},line width=1pt] (-.05,0.2) -- (-0.05,0.83);
			\draw(-.25,0.515)node{$k$};
		        \draw[decorate, decoration={brace, amplitude=3.5pt},line width=1pt] (-.05,0.85) -- (-0.05,1.15);
			\draw(-.3,1)node{$m_1$};
		        \draw[decorate, decoration={brace, amplitude=3.5pt},line width=1pt] (-.05,1.82) -- (-0.05,1.97);
			\draw(-.4,1.9)node{$m_1/2$};
		        \draw[decorate, decoration={brace, amplitude=3.5pt},line width=1pt] (-.05,1.17) -- (-0.05,1.8);
			\draw(-.25,1.5)node{$k$};
			\path[->,line width=1pt] (0.9,2.3) edge[bend right] (1.16,2.05);
			\draw(0.85,2.35)node{$\lell$};
			\draw[line width=0.5pt,<->](0,-0.1)--(2,-0.1);
			\draw(1,-0.25)node{$\lell_1$};
		\end{tikzpicture}
	\end{center}
	\caption{\small $k+2\ell$-detours attached to the south and west boundaries of  $\lell_1\times\lell_1$ $2$-faces. In this illustration each  edge has $\dtnr=2$ detours attached to it, spaced $m_1$ apart.}
	\label{fig:detours}
	\bigskip 
\end{figure}  

Attach $(k + 2\lell)$-detours along each of the boundary edges of the tiling so that the $k$-path $\pi^+$  is on the boundary edge and the detour $\pi^{++}$ 
is in the interior of one of the two-dimensional faces adjacent to this boundary edge.
Adopt the convention that if the boundary edge is $[v,v+\lell_1\evec_i]$ then the detour lies on the 2-dimensional face $[v,v+\lell_1\evec_i]\times[v,v+\lell_1\evec_j]$ for some $j\ne i$ (in other words, the detour points into a positive coordinate direction).  See Figure \ref{fig:detours}.

Place $\dtnr$ detours on each boundary edge of the tiling so that the detours are exactly distance $m_1$ apart from each other and a detour that is right next to a corner vertex of the tiling is exactly distance $m_1/2$ from that vertex.
This is consistent with the definition of $\lell_1$ in \eqref{ell-defs}.  

Since $m_1/2>\lell$ by \eqref{ass78.1} and  \eqref{m1delta2},  distinct  detour rectangles that happen to lie on the same   two-dimensional face do not intersect and the points on a detour are closer to the boundary edge of the detour than to any  other boundary edge. 
  
  Inside a particular $N$-box $B$, for $j\in\{0, 1,2\}$ let $W_j$
denote the union of the $j$-dimensional faces of the cubes $\{V_i\}$ tiling $V$.   
  Let   $W_1'$ be the union of $W_1$ (the boundary edges)  and the detours $\pi^{++}$ attached to the boundary edges.  
 
We describe   in  more detail the structure of the detours on the two-dimensional faces inside a particular $B$. Let $H\subset W_2$ be a two-dimensional $\lell_1\times\lell_1$  face.  For simplicity of notation suppose  $H=[0,\lell_1 \evec_1]\times[0,\lell_1\evec_2]$. Assume without loss of generality that the boundary edge $[0,\lell _1\evec_1]$ has its detours contained in $H$.  For $i\in[\dtnr]$ define the $i$th detour rectangle: 
\[  G_{i,S}=\bigl[\bigl(m_1/2+(i-1)(k+m_1)\bigr)\evec_1,\bigl(m_1/2+(i-1)(k+m_1) + k\bigr)\evec_1]\times[0,\lell \evec_2].
\] 
The subscript $S$ identifies these detour rectangles as attached to the southern boundary of $H$.  Similarly,  
if  the  detour rectangles attached to the western boundary of $H$ lie in $H$, we denote these by $\{ G_{i,W}:1\le i\le \dtnr\}$.

For a label $U \in \{S,W\}$, let $\pi^{++}_{i,U}=\partial G_{i,U}\setminus\partial H$ be  the portion of the boundary of $G_{i,U}$ in the interior of $H$.  $\pi^{++}_{i,U}$ is the detour  path of $k+2\lell$ edges.  Let $\pi^+_{i,U}=\partial G_{i,U}\cap\partial H$ be  the portion of the boundary of $G_{i,U}$ that lies on the boundary of $H$.   $\pi^+_{i,U}$ is a straight path of $k$ edges, the path bypassed by the detour.   
Let 
\be\label{HH1} 
	H'= \partial H \cup\!\!\!\!\!\bigcup_{\stackrel{1 \leq i \le \dtnr}{U\in\{S,W\}}}\!\!\!\!  \partial G_{i,U},\quad
    \overline{H}= \partial H \cup\!\!\!\!\! \bigcup_{\stackrel{1 \leq i \le \dtnr}{U\in\{S,W\}}}\!\!\!\!  G_{i,U},\quad\text{and}\quad
    H^+= \!\!\!\!\! \bigcup_{\stackrel{1 \leq i \le \dtnr}{U\in\{S,W\}}}\!\!\!\!\!  \pi^+_{i,U}.
  \ee
 See Figure \ref{fig:H}. 
Let $\overline{W}_{\!1}$ (resp.\ $W^+_1$) be the union of all $\overline{H}$ (resp.\ $H^+$) as $H$ ranges over all the two-dimensional faces that lie in $W_2$. The union of all $H'$ equals $W'_1$ as already defined above. 

\begin{figure}[t]
	\begin{center}
		\begin{tikzpicture}[>=latex,  font=\footnotesize,scale=1.3]
			\draw[pattern=north west lines](0,0)--(2,0)--(2,2)--(0,2)--(0,0);
			\draw(1,-.3)node{$H$};
			\begin{scope}[shift={(3,0)}]
			\draw(0,0)--(2,0)--(2,2)--(0,2)--(0,0);
		        \draw(0,0.16)--(0.1,0.16)--(0.1,0.84)--(0,0.84);
		        \draw(0,1.16)--(0.1,1.16)--(0.1,1.84)--(0,1.84);
		        \draw(0.16,0)--(0.16,0.1)--(0.84,0.1)--(0.84,0);
		        \draw(1.16,0)--(1.16,0.1)--(1.84,0.1)--(1.84,0);
			\draw(1,-.3)node{$H'$};
		        \end{scope}
			\begin{scope}[shift={(6,0)}]
			\draw(0,0)--(2,0)--(2,2)--(0,2)--(0,0);
		        \draw[pattern=north west lines](0,0.16)--(0.1,0.16)--(0.1,0.84)--(0,0.84);
		        \draw[pattern=north west lines](0,1.16)--(0.1,1.16)--(0.1,1.84)--(0,1.84);
		        \draw[pattern=north west lines](0.16,0)--(0.16,0.1)--(0.84,0.1)--(0.84,0);
		        \draw[pattern=north west lines](1.16,0)--(1.16,0.1)--(1.84,0.1)--(1.84,0);
			\draw(1,-.3)node{$\overline H$};
		        \end{scope}
			\begin{scope}[shift={(9,0)}]
		        \draw(0,0.16)--(0,0.84);
		        \draw(0,1.16)--(0,1.84);
		        \draw(0.16,0)--(0.84,0);
		        \draw(1.16,0)--(1.84,0);
			\draw(1,-.3)node{$H^+$};
		        \end{scope}
		\end{tikzpicture}
	\end{center}
	\caption{\small From left to right: a two-dimensional face $H$ (shaded),  $H'$ that  consists of the boundary $\partial H$  of $H$ and the boundaries of the detour rectangles in $H$, $\overline H$ that  consists of $\partial H$ and the full (shaded) detour rectangles in $H$, and finally $H^+$ that consists of the $\pi^+$-parts of the boundaries  of the detour rectangles in $H$.}
	\label{fig:H}
\end{figure}
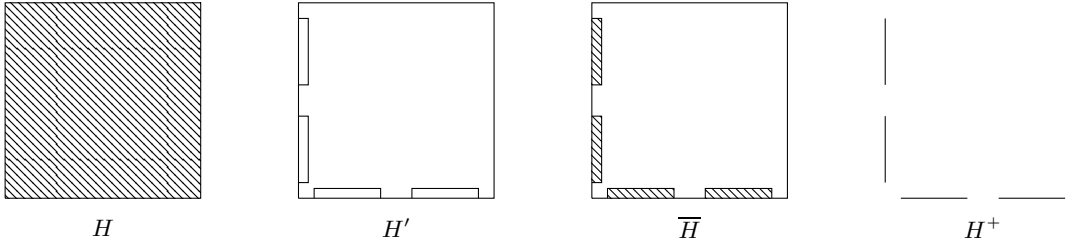

Since multiple geodesics are possible, we have to make a particular measurable choice of a geodesic to work on and one that relates suitably to the structure defined above.    For this purpose 
   order the admissible steps   for example as in 
\be\label{rng-ord} 
\varnothing\prec\evec_1\prec-\evec_1\prec\evec_2\prec-\evec_2\prec\dotsm
\prec\evec_d\prec-\evec_d  
\ee   
and then order the paths lexicographically.  
Here $\varnothing$ stands for a missing step. So if $\pi'$  extends $\pi$ with one or more steps,  then $\pi\prec\pi'$ in lexicographic ordering. Recall the choice of index $j(x)$ in \eqref{vw61}. 



\begin{lemma} \label{lm:pi*} Fix $x\in\Z^d\setminus\{\zevec\}$.  There exists a unique geodesic $\pi$  for $T_{\zevec,x}$ that satisfies the following two conditions. 
 \begin{enumerate} [label=\rm(\roman{*}), ref=\rm(\roman{*})]  \itemsep=3pt
 \item\label{lm:pi*.i}   For every $N$-box $B\in\cB_{j(x)}$ and points $u,v\in\pi_B=\pi\cap B$ the following holds: if both $u,v\in\overline{W}_{\!1}$ or $u\in\overline{W}_{\!1}$ and $v\in\partial B$ {\rm(}or vice versa{\rm)}, and if every edge of $\pi_{u,v}$ lies in $B$ but not in $\overline{W}_{\!1}$,
  then there is no geodesic between $u$ and $v$ that remains in $B$, uses only edges with strictly positive weights, and uses at least one edge in $\overline{W}_{\!1}$.  
\item\label{lm:pi*.ii}    $\pi$ is lexicographically first among all geodesics of $T_{\zevec,x}$ that satisfy point {\rm(i)}.  
\end{enumerate} 
\end{lemma} 

\begin{proof}
It suffices to  show the existence of  a geodesic that satisfies point \ref{lm:pi*.i}.  Point \ref{lm:pi*.ii}  then picks a unique one.   


Start with any $T_{\zevec,x}$-geodesic $\pi$  of maximal Euclidean length. 
 For the purpose of this proof consider $\pi$ as an ordered sequence of vertices and the edges connecting them.


 
 Consider in order each segment $\pi_{u,v}$  that violates point \ref{lm:pi*.i}.  When this violation happens,  there is a particular $N$-box $B\in\cB_{j(x)}$ such that  $\pi_{u,v}\subset B\setminus\overline{W}_{\!1}$ and  there is an alternative geodesic $\pi'_{u,v}\subset B$ that uses only edges with strictly positive weights and uses at least one edge in $\overline{W}_{\!1}$.  Replace the original segment $\pi_{u,v}$ with $\pi'_{u,v}$. 
 
 Since we replaced one geodesic segment with another, $\tpath(\pi'_{u,v})=\tpath(\pi_{u,v})$.   Suppose that after the replacement, the full path is no longer self-avoiding.  Then a portion of it can be removed and this portion contains part of $\pi'_{u,v}$.   
Since $\pi'_{u,v}$ uses only edges with strictly positive weights, this removal reduces the passage  time by a strictly positive amount, contradicting the assumption that the original passage time was optimal. 
 Consequently the new path is still a self-avoiding geodesic. 
 
  Since the original path was a geodesic of maximal Euclidean length, it   follows that $\abs{\pi'_{u,v}}\le \abs{\pi_{u,v}}$.   Since the replacement inserted into the geodesic at least one new edge from  $\overline{W}_{\!1}$,  $\pi'_{u,v}$ has strictly fewer edges in  $B\setminus\overline{W}_{\!1}$  than $\pi_{u,v}$.

   The new segment $\pi'_{u,v}$ may in  turn  contain smaller segments   $\pi'_{u_1,v_1},\dotsc, \pi'_{u_m,v_m}$  that   violate point \ref{lm:pi*.i}.    Replace each of these with alternative segments  $\pi''_{u_1,v_1},\dotsc, \pi''_{u_m,v_m}$. Continue like this until the entire path segment between $u$ and $v$ has been cleaned up, in the sense that no smaller segment  of it violates \ref{lm:pi*.i}.   
 This process must end because each replacement leaves strictly shorter segments that can potentially violate point \ref{lm:pi*.i}.  
 
 Observe that the   clean-up of the  segment  $\pi_{u,v}$  happens entirely  inside the particular $N$-box $B$, does not alter the endpoints  $u,v$ of the original segment, and does not alter the other portions $\pi_{\zevec,u}$ and $\pi_{v,x}$ of the geodesic because each replacement step produced a self-avoiding geodesic.   
  
   Proceed in this manner through all the path segments   that are in violation of point \ref{lm:pi*.i}.  
    There are only finitely many. 
  At the conclusion of this process we have a geodesic that satisfies point \ref{lm:pi*.i}.  
\end{proof}


Define the event
 \be\label{Ga-d3} 
    \begin{aligned}
    \Gamma_{B} 
        = \Big\{ \,\w: \;  &r_1-\delta< t(e)<r_1+\delta  
        \quad \forall e \in W_1'\, \setminus W^+_1, \\
        &\param_0-\delta<t(e) \le \param_0  \quad\forall e \in W^+_1,  \\[3pt]
        & \param_0 \le t(e) \leq \param_1 \quad\forall e \in \overline{W}_{\!1}\,\setminus W_1',\\[2pt] 
          &\param_1\le t(e)\le  \est  \quad\forall e \in B \setminus  \overline{W}_{\!1} 
        \,\Big\}.
    \end{aligned}
\ee
A key consequence of the definition of   the  event $\Gamma_{B}$ is that, by  \eqref{v892.08},   the boundary paths $\pi^+$ and $\pi^{++}$  of all detour rectangles $G$ in $W_1'$   satisfy 
 \be\label{v.1004.8} 
\tpath(\pi^+)<\tpath(\pi^{++})<\tpath(\pi^+)+(2\ell-1)b.   
\ee

Once the parameters have  been fixed,  then up to translations and rotations  there are only finitely many ways to choose the  constructions above.  
Thus   
\be\label{Ga84.1}  
\exists D_2>0 \text{ such that }  \P(\Gamma_{B})\ge D_2 \text{ for all   } B.
 \ee
   $D_2$ depends on $N$ and the probabilities of the events on  $t(e)$   that appear in $\Gamma_{B}$.    In particular, $D_2$ does not depend  on $x$.   
   
 Our point of view shifts now to the implications of the event $\Gamma_{B}$ for a particular $B\in\cB_{j(x)}$.

Let $\gamma$ be  a self-avoiding path  in  $W_1$.   
Then if $\w\in\Gamma_{B}$, 
       \be\label{675.74} 
           \tpath(\gamma) \leq \Norm{\gamma}{1} \left( \param_0 \frac{k}{m_1 + k} + (r_1 + \delta) \frac{m_1}{m_1 + k} \right) + c_1
       \ee
       where $c_1$ came from \eqref{c1}.      The main term on the right of \eqref{675.74} contains the weights of the $k$-paths of detours and $m_1$-gaps  completely covered by $\gamma$, and $c_1$ accounts for the partially covered pieces  at either end of $\gamma$. 
     

We say that a point $y\in\overline{W}_{\!1}$ is \textit{associated} with a boundary edge $I$ of a cube  $V_{i_0}$ if either $y \in I$ or $y$ lies  in one of the detour rectangles $G_{i,U}$ attached to the  edge $I$. We say that points  $y, z\in\overline{W}_{\!1}$ are $(\ell^1,W_1)$-\textit{related} if they are each associated to boundary edges $I \subset V_{i_0}$ and $J \subset V_{j_0}$ such that every point on $I$ can be connected to every point on $J$ by an $\ell^1$-path that remains entirely within $W_1$.  Recall that an $\ell^1$-path   $x_{\parng{m}{n}}$ satisfies $\abs{x_n-x_m}_1=n-m$. 

\begin{lemma}
   \label{lm:W1-1}
   Let $\w\in\Gamma_B$.
   Let $y, z\in\overline{W}_{\!1}$ be two $(\ell^1,W_1)$-related points. 
  Suppose a geodesic between $y$ and $z$ lies within $B$.  Then there  exists a geodesic between $y$ and $z$ that stays within $B$ and uses at least one edge in $\overline{W}_{\!1}$.
\end{lemma}

\begin{proof}
%
There are  two cases:
   \begin{enumerate}[label=(\Alph*)]
       \item \label{it:conn} $y,z$ are connected by an $\ell^1$-path  inside $\overline{W}_{\!1}$.
       \item \label{it:n-conn} $y,z$ cannot be connected by an $\ell^1$-path that remains entirely inside $\overline{W}_{\!1}$.
   \end{enumerate}

   In case \ref{it:conn}, any $\ell^1$-path   inside $\overline{W}_{\!1}$ takes weights that are at most $\param_1$ and any path  inside $B \setminus \overline{W}_{\!1}$ takes weights that are at least $\param_1$. Since we assume the existence of a geodesic between $y$ and $z$ that lies entirely inside $B$, we see that there must exist a geodesic that remains entirely within $\overline{W}_{\!1}$. 

In case \ref{it:n-conn}, suppose $\hat\pi\subset B$ is a self-avoiding path between $y$ and $z$ that lies 
outside  $\overline{W}_{\!1}$. 
 Construct a path $\pi'\subset\overline{W}_{\!1}$ from $y$ to $z$ by concatenating the following path segments:
 using at most $\lell$ steps, connect $y$ to the closest point $y'$ on the boundary edge $I$ that $y$ is associated with; using at most $\lell$ steps, connect $z$ to the closest point $z'$ on the boundary edge $J$ that $z$ is associated with;
  connect $y'$ to $z'$ with an $\ell^1$-path $\pi''$  in $W_1$.
We show that $\tpath(\pi')\le\tpath(\hat\pi)$, thus proving the lemma.
   
 
 We argue that 
 \be\label{it:304}  \text{ $\pi''$ uses  at least $m_1/2$ edges in $W_1'\setminus W^+_1$. }
 \ee
 Indeed, observe that $y$ and $z$ cannot both be on $W_1$ nor both in the same detour rectangle $G_{i,U}$, for otherwise we would be in case \ref{it:conn}. 
 On the other hand, if $y$ is in a detour rectangle and $z$ is on $W_1$, then $\pi''$ is an $\ell^1$-path that connects $y'$ to $z'=z$.  If in this case $\Norm{\pi''\cap(W_1'\setminus W^+_1)}{1}<m_1/2$, then it must be the case that $\pi''\subset I=J$. 
 But then in this case $\pi'$ is an $\ell^1$-path from $y$ to $z$ and we are again in case \ref{it:conn}.
 The symmetric case of $y\in W_1$ and $z$ in a detour rectangle is similar. 
 Lastly, if $y$ and $z$ belong to different detour rectangles, then the segment of $\pi''$ that connects the two rectangles must be of length at least $m_1$,  the distance between two neighboring detours. 
 
 We have verified \eqref{it:304}.  From \eqref{it:304}   and $m_1\ge8\lell$  comes  the lower bound 
 \[ \Norm{z-y}{1}\ge m_1/2-2\lell\ge m_1/4.\] 
 The $m_1/2$ edges in $\pi''\cap(W_1'\setminus W^+_1)$ all have weight at most $r_1+\delta$. 
 Furthermore, $\abs{\pi''}_1\le\Norm{z-y}{1}+ 2\lell$ and all the edges along $\pi''$ have weight no larger than $\param_0$.  This gives the bound 
 \[  \tpath(\pi'') \le m_1(r_1+\delta)/4 + ( \Norm{z-y}{1} -m_1/4+ 2\lell) \param_0. \]  
 
 Since $\hat\pi$ connects $y$ to $z$ and the weights along $\hat\pi$ are at least $\param_1$, 
 \[  \tpath(\hat\pi) \ge m_1\param_1/4 + ( \Norm{z-y}{1} -m_1/4)\param_1.  \] 

	Together these observations give the lower bound
   \[
       \tpath(\hat\pi) - \tpath(\pi'') \geq m_1 (\param_1 - (r_1 + \delta))/4 - 2\lell \param_0.
   \]
 From this, 
   \begin{align*}
       \tpath(\pi') \leq \tpath(\pi'') + 2\lell \param_1
        \leq \tpath(\hat\pi) - m_1 (\param_1 - (r_1 + \delta))/4 + 4\lell \param_1
       < \tpath(\hat\pi).  
   \end{align*}
The last inequality  used \eqref{m1delta2} and   $r_1+\delta< \eit+\delta_0<\param_1\le\est$.
\end{proof}

 \begin{figure}[t]
	\begin{center}
		\begin{tikzpicture}[>=latex,  font=\footnotesize,scale=0.6]
		\draw(0,-1)--(0,7);
		\draw(3,-1)--(3,7);
		\draw(6,-1)--(6,7);
		\draw(9,-1)--(9,7);
		\draw(-1,0)--(10,0);
		\draw(-1,3)--(10,3);
		\draw(-1,6)--(10,6);
		\draw[line width=1.2pt](3,3)--(3,6)--(6,6)--(6,3);
		\draw[line width=2pt](3,3)--(6,3);
		\draw[pattern=north west lines,line width =2pt](4,3)--(4,3.7)--(5,3.7)--(5,3);
		\draw[line width=1.2pt](3,4)--(3.7,4)--(3.7,5)--(3,5);
		\draw[<->](5.2,3)--(5.2,3.8);
		\draw(5.1,3.4)node[right]{$\ell$};
		\draw[<->](4,2.8)--(5,2.8);
		\draw(4.5,3)node[below]{$k$};
		\draw[<->](3,6.3)--(6,6.3);
		\draw(4.5,6.2)node[above]{$\ell_1$};
		\draw[<->](2.7,3)--(2.7,6);
		\draw(2.9,4.5)node[left]{$\ell_1$};
		\draw[dashed](1.6,3)--(1.6,1.6)--(7.4,1.6)--(7.4,5.2)--(1.6,5.2)--(1.6,3);
		\draw[<->](4.5,0)--(4.5,1.6);
		\draw(4.5,0.8)node[right]{$2\ell$};
		\draw[<->](7.4,2.7)--(9,2.7);
		\draw(8.2,2.7)node[below]{$2\ell$};
		\draw[<->](4.5,5.2)--(4.5,6);
		\draw(4.5,5.6)node[right]{$\ell$};
		\draw[->](2.4,2) to [out=0,in=-90](3.5,2.9);
		\draw(2.2,2)node{$I$};
		\draw[->](4.7,4.7) to [out=180,in=90](4.5,3.8);
		\draw(5.2,4.7)node{$G_S$};
		\end{tikzpicture}
	\end{center}
	\caption{\small The proof of Lemma \ref{lm:W1-2}. The light grid is $W_1$. The thicker square and its two detours are part of $\overline{W}_{\!1}$. The thickest edge of the square is denoted by $I$. The hashed box, denoted by $G_S$, is a detour rectangle attached to $I$. The points that are within distance $\ell_1-2\ell$ from a point on $I\cup G_S$ are all inside the dashed rectangle. All these points that are also on $W_1$ can be reached from any point on $I$ via an $\ell^1$ path that stays on $W_1$.}
	\label{fig:W1-2}
	\medskip 
\end{figure}
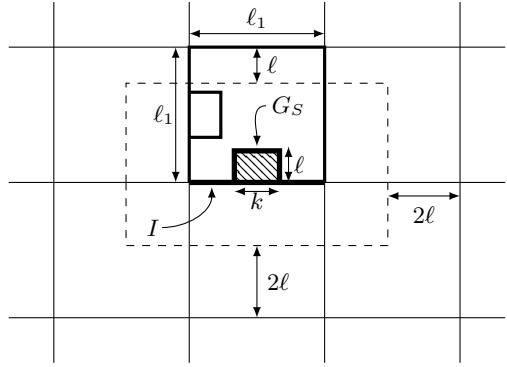

\begin{lemma}       \label{lm:W1-2}
    Let $\w\in\Gamma_B$. Suppose $y, z\in\overline{W}_{\!1}$ are not $(\ell^1,W_1)$-related and that they are connected by a path  $\hat\pi$ that remains entirely in $B \setminus \overline{W}_{\!1}$. Then
    \[
        \tpath(\hat\pi) \geq \param_1 (\lell_1 - 2\lell).
    \]
\end{lemma} 

\begin{proof}   Inspection of Figure \ref{fig:W1-2}  convinces that any two points $y, z\in\overline{W}_{\!1}$ such that $\abs{z-y}_1< \lell_1 - 2\lell$  must be $(\ell^1,W_1)$-related.  Thus 
$\abs{\hat\pi}_1\ge \lell_1 - 2\lell$ and by assumption it uses only weights  $\ge\param_1$. 
\end{proof}

\begin{lemma}    \label{lm:W1-3}
  Let $\w\in\Gamma_B$ and $y,z\in B$.  Assume that either both $y,z\in\overline{W}_{\!1}$ or that  $y\in\overline{W}_{\!1}$ and $z\in\partial B$.   Let $\pi$ be a geodesic between $y$ and $z$.    Assume that   the  edges  of $\pi$  lie entirely  outside  $\overline{W}_{\!1}$. Then either there is a geodesic between $y$ and $z$ inside $B$ that uses at least one edge in $\overline{W}_{\!1}$ or 
   \be\label{W1-8} 
       \tpath(\pi) \geq \min\bigl\{  \param_1(\lell_1 - 2\lell) , \,\param_1\lell_2' \bigr\}= \param_1\lell_2' . 
   \ee
\end{lemma}


\begin{proof}   
If $\pi$  reaches the boundary  $\partial B$ (in either case of $y,z$)   then $\pi$ must travel through $B\setminus V$ and consequently   $\tpath(\pi)\ge\param_1 \lell_2'$.   The other possibility is that  $\pi$ stays inside $B\setminus \overline{W}_{\!1}$.  If  $y$ and $z$ are $(\ell^1,W_1)$-related then Lemma \ref{lm:W1-1} gives a geodesic in $B$ that uses an edge in $\overline{W}_{\!1}$.  
If  $y$ and $z$ are  not  $(\ell^1,W_1)$-related,   Lemma   \ref{lm:W1-2} gives  $\tpath(\pi)\ge\param_1(\lell_1 - 2\lell)$.   The last  equality of \eqref{W1-8}  is from \eqref{ell-defs}. 
\end{proof}



 Henceforth we often work with two coupled environments $\w$ and $\w^*$.  Quantities calculated in the $\w^*$ environment will be marked with a star if $\w^*$ is not explicitly present. For example,   $T^*_{\zevec, x}=T_{\zevec,x}(\w^*)$ denotes the passage time between $\zevec$ and $x$  in the  environment $\w^*$. 

Recall  the event  $\Lambda_{B,v,w,x}$ defined in \eqref{vw-La}.

\begin{lemma} \label{lm:Psi1}  
  Let $\w$ and $\w^*$ be two environments that agree outside $B$ and satisfy 
  $\w\in\Lambda_{B,v,w,x}$ and $\w^*\in\Gamma_{B}$.
Then   there exists a  self-avoiding path $\wt\pi$ from $\zevec$ to $x$  such that 
\[
    \tpath^*(\wt\pi) \le  \tpath(\pi(x)) - \parQ.  
\]
\end{lemma} 

\begin{proof}
 Since   box $B$ is black on the event  $\Lambda_{B,v,w,x}$, 
 \begin{align*}   
   \tpath(\pi_{v,w}(x)) >  (\eit+\delta_0) (  \abs{w-v}_1 \vee N) . 
  \end{align*}   
 The bound above comes from \eqref{1.bl2}, on account of these observations:  regardless of  whether $\pi_{v,w}(x)$ exits $\overline B$, there is a segment inside $\overline B$ of length $\abs{w-v}_1$, and furthermore $\pi_{v,w}(x)$ contains a short crossing of $B$ that has length at least $N$. 
 
Define a path $\pi'$ from  $v$ to $w$ in $B$ as follows. Let $\lambda_1$ be an $\ell^1$-path from $v$ to some point $a\in W_1$. 
Similarly, let $\lambda_3$ be an $\ell^1$-path from   $w$ to some $b\in W_1$. These  paths satisfy  $\Norm{\lambda_1}{1}\vee\Norm{\lambda_3}{1}\le d \lell_2''+(d-2)\lell_1$.  
Let $\lambda_2$ be a shortest   path from $a$ to $b$ that remains in $W_1$.  Since  $\Norm{a - b}{1} \leq \Norm{v - w}{1} + 2d \lell_2''+2(d-2)\lell_1$, $\Norm{\lambda_2}{1}\le \Norm{v-w}{1} + 2d\lell_2'' + 2d\lell_1$. (To go from $a$ to $b$ along $W_1$ use $2\lell_1$ steps to go from $a$ and $b$ to the nearest vertices 
$a'$ and $b'$ in $W_0$, respectively, and an $\ell^1$-path along $W_1$ will take $\Norm{a'-b'}{1}\le\Norm{a-b}{1}+2\lell_1$ steps.)  

 Let $\pi'$  be the concatenation of   $\lambda_1$, $\lambda_2$ and $\lambda_3$.    
 Define $\wt\pi$ as the concatenation of $\pi_{\zevec,v}(x)$, $\pi'$, and $\pi_{w,x}(x)$.    
   The next calculation uses \eqref{675.74},  \eqref{Q7} and \eqref{c2},  and the facts that $\w=\w^*$ outside $B$,   $\w\in\Lambda_{B,v,w,x}$ and $\w^*\in\Gamma_{B}$.  
       \begin{align*}
   & \tpath(\pi(x))  -    \tpath^*(\wt\pi) =      \tpath(\pi_{v,w}(x)) - \tpath^*(\pi') \\[4pt] 
            &\quad 
             \geq (\eit + \delta_0) \max( \Norm{v-w}{1}, N) - 2(d \lell_2''+(d-2)\lell_1) \est \\ 
            &\quad \quad  - ( \Norm{v-w}{1} + 2d\lell_2'' + 2d\lell_1) \left( \param_0 \frac{k}{m_1 + k} + (r_1 + \delta) \frac{m_1}{m_1 + k} \right) - c_1  \\[4pt] 
            &\quad 
             \geq (\eit + \delta_0) \max( \Norm{v-w}{1}, N) - 2(d \lell_2''+(d-2)\lell_1) \est \\ 
            &\quad \quad  - \bigl( \max( \Norm{v-w}{1}, N) + 2d\lell_2'' + 2d\lell_1\bigr) \left( \param_0 \frac{k}{m_1 + k} + (r_1 + \delta) \frac{m_1}{m_1 + k} \right) - c_1  \\[4pt] 
            &\quad 
             = (\eit + \delta_0) \max( \Norm{v-w}{1}, N) - 2(d \lell_2''+(d-2)\lell_1) \est \\ 
            &\quad \quad  - \bigl( \max( \Norm{v-w}{1}, N)  + 2d\lell_2'' + 2d\lell_1\bigr) (\eit+\delta_0-c_2)- c_1  \\[4pt] 
            &\quad 
             =c_2\max( \Norm{v-w}{1}, N) - 2(d \lell_2''+(d-2)\lell_1) \est - ( 2d\lell_2'' + 2d\lell_1) (\eit+\delta_0)- c_1  \\[4pt] 
            &\quad 
             \ge c_2N  - 4d(\lell_2''+\lell_1) \est - c_1  = \parQ.
        \end{align*}
   In the first inequality, $2(d \lell_2''+(d-2)\lell_1) \est$ bounds  the time spent on $\lambda_1$ and $\lambda_3$ and the remaining  negative  terms bound the  passage time of $\lambda_2$. The lemma is proved.
\end{proof}

Henceforth we assume that $\w\in\Lambda_{B,v,w,x}$ and $\w^*\in\Gamma_B$.
Let $\pi^*(x)$ be the  geodesic  from $\zevec$ to $x$ in the $\w^*$-environment  specified in Lemma \ref{lm:pi*}.   
By Lemma \ref{lm:Psi1}, 
 \be\label{675.88}  
    \tpath^*(\pi^*(x)) \le \tpath^*(\wt\pi)  \le \tpath(\pi(x)) - \parQ \le  \tpath(\pi^*(x)) -\parQ.
    \ee
This implies that $\pi^*(x)$ must use edges in $\overline{W}_{\!1}$ because $\w$ and $\w^*$ agree outside $B$, while  $t(e)\le\param_0\le\param_1\le t^*(e)$ on edges in $B\setminus \overline{W}_{\!1}$.

 \begin{figure}[t]
	\begin{center}
		\begin{tikzpicture}[>=latex,  font=\footnotesize,scale=0.6]
		\draw[line width=0.5pt](0,0)--(0,6)--(15,6)--(15,0)--(0,0);
		\draw[line width=0.5pt](1,1)--(1,5)--(14,5)--(14,1)--(1,1);
		\draw[line width=0.5pt](2,2)--(2,4)--(13,4)--(13,2)--(2,2);
		\draw(0,4)node[left]{$B$};
		\draw(13,5)node[below]{$V$};
		\draw(12.5,4.1)node[below]{$\overline W_1$};
		\draw[line width=1.5pt](1,7)--(1,6.3)node[right]{$a_0$}--(1,6)--(1,5.5)--(1.5,5.5)--(1.5,4.5)--(2.5,4.5)--(2.5,4.3)node[right]{$a_1=x_1$}--(2.5,3.5)--(3,3.5)--(3,2.5)--(3.5,2.5)--(3.5,1.7)node[left]{$y_1$}--(3.5,0.5)--(4,0.5)--(4,-0.5)--(4.5,-0.5)--(4.5,1.5)--(5,1.5)--(5,1.7)node[right]{$z_1=x_2$}--(5,3)--(6,3)--(6,3.5)--(6.5,3.5)--(6.5,4)--(7,4)node[below]{$y_2$}--(7,4.5)--(7,5.5)--(8,5.5)--(8,4.5)--(8.5,4.5)--(8.5,4.3)node[right]{$z_2=x_3$}--(8.5,3)--(8,3)--(8,2.5)--(9,2.5)--(9,2.3)node[right]{$y_3$}--(9,1.5)--(10.5,1.5)--(10.5,2)--(11.5,2)node[below]{$z_3=x_4$}--(11.5,3)--(12.5,3)--(12.5,2.5)--(13.35,2.5)node[below]{$y_4$}--(13.5,2.5)--(13.5,3)--(15.5,3)--(15.5,3.5)--(14.5,3.5)--(14.5,4)--(13.5,4)--(13.5,4.5)--(14.5,4.5)--(14.5,5)--(16,5)node[below=-.1]{$z_4=b_1$};
		\draw[fill](1,6)circle(1.5mm);
		\draw[fill](2.5,4)circle(1.5mm);
		\draw[fill](3.5,2)circle(1.5mm);
		\draw[fill](5,2)circle(1.5mm);
		\draw[fill](7,4)circle(1.5mm);
		\draw[fill](8.5,4)circle(1.5mm);
		\draw[fill](9,2)circle(1.5mm);
		\draw[fill](10.5,2)circle(1.5mm);
		\draw[fill](13,2.5)circle(1.5mm);
		\draw[fill](15,5)circle(1.5mm);
	\end{tikzpicture}
	\end{center}
	\caption{\small The path segment $\pi^*_{a_0,b_1}(x)$ with four excursions $\pi^1,\dotsc,\pi^4$. The segment $\pi^{i,1}$  inside $\overline W_1$ goes from $x_i$ to $y_i$ and the segment $\pi^{i,2}$  outside $\overline W_1$ goes from $y_i$ to $z_i$. Note that $\overline W_1$ is {\it not} actually a box but is represented as one above for the purpose of illustration.}
	\label{fig:exc}
\medskip 
\end{figure}
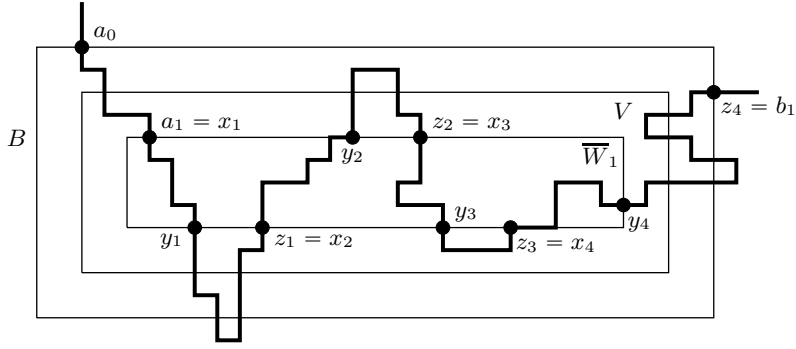

Let $a_0$ be the first vertex of $\pi^*(x)$ in $B$,  $a_1$ the first vertex of $\pi^*(x)$ in $\overline{W}_{\!1}$,  and $b_1$ the last vertex of $\pi^*(x)$ in $B$. Decompose the path segment $\pi^*_{a_1,b_1}(x)$ between $a_1$ and $b_1$ into excursions $\pi^1,\ldots,\pi^\exnr$ ($\exnr\in\N$)  as follows:  each excursion  $\pi^i$ begins with a nonempty segment $\pi^{i,1}$ of edges  inside   $\overline{W}_{\!1}$,   followed by  a nonempty segment $\pi^{i,2}$  of edges  outside $\overline{W}_{\!1}$.   
  The  excursions $\pi^1,\ldots,\pi^{\exnr-1}$ begin and end at a vertex in  $\overline{W}_{\!1}$, while  the last excursion $\pi^\exnr$ begins in  $\overline{W}_{\!1}$ and  ends at the vertex $b_1$ where $\pi^*(x)$ exits $B$.   Figure \ref{fig:exc} illustrates.



%

 By $\w^*\in\Gamma_B$, \eqref{v892.08} and \eqref{44.51},  $r_1-\delta>0$ and hence $t^*(e)>0$ for all edges  $e\in B$. Then 
condition \ref{lm:pi*.i} of Lemma \ref{lm:pi*}  ensures that  those portions of 
 the segments $\pi^{1,2}, \pi^{2,2} ,\ldots,\pi^{\exnr,2}$ that connect  $\overline{W}_{\!1}$ to itself or to $\partial B$ inside $B$ 
 cannot be replaced by segments that use edges in $\overline{W}_{\!1}$. Therefore these segments  
 obey bound \eqref{W1-8}. 
This  gives the last inequality below:  
\[  \tpath^*(\pi^*_{a_0,b_1}(x)) \ge \tpath^*(\pi^*_{a_1,b_1}(x)) \ge \sum_{i=1}^\exnr \tpath^*(\pi^{i,2})  
\ge \exnr \param_1 \lell_2'.  \]

 Since the maximal side length of $B$ is $3N$,  $a_0$ and $b_1$ can be connected  with a  path $\pi^o$ such that 
$
    \tpath^*(\pi^o) \leq 3 d N \est.
$
  Since $\pi^*(x)$ is optimal,   
$\exnr \param_1 \lell_2' \leq 3 d N \est$, 
and therefore
\be\label{si9} 
    \exnr \le \frac{3dN\est}{\param_1 \lell_2'}\,.
\ee
  Using \eqref{675.88}, and that $\w=\w^*$ outside $B$ while $\w\le\w^*$ on $B\setminus\overline{W}_{\!1}$, 
\begin{align*} 
 \parQ &\le   \tpath(\pi^*(x)) - \tpath^*(\pi^*(x)) \\[3pt] 
 &= 
 \tpath(\pi^*_{a_0,a_1}(x)) - \tpath^*(\pi^*_{a_0,a_1}(x))
 +  \tpath(\pi^*_{a_1,b_1}(x)) - \tpath^*(\pi^*_{a_1,b_1}(x))\\
& \le  \sum_{i=1}^{\exnr} \bigl[ \tpath(\pi^i) - \tpath^*(\pi^i)\bigr].
\end{align*} 
Then  some excursion $\bar\pi\in\{\pi^1,\ldots,\pi^\exnr\}$  must satisfy 
  \begin{align}
    \tpath(\bar\pi) - \tpath^*(\bar\pi) & \geq \frac{\parQ}{\exnr} 
    \ge \frac{c_2N\param_1 \lell_2' }{6dN\est}   = \frac{c_2\param_1 \lell_2'}{6d\est}.  
     \label{675.90} 
\end{align}
The second inequality comes from \eqref{Q7} and \eqref{si9}. 
The only positive contributions  to  $\tpath(\bar\pi) - \tpath^*(\bar\pi)$ can  come from $\bar\pi^1$, the segment  of $\bar\pi$   in $\overline{W}_{\!1}$.  
 Since $B$ is black,   $t(e) - t^*(e) \leq t(e)  \leq \param_0\leq \param_1$ for all edges $e \in B$. Therefore  the number of edges $ \abs{\bar\pi^1}_1$  satisfies  $\param_1 \abs{\bar\pi^1}_1  \ge  \tpath(\bar\pi) - \tpath^*(\bar\pi)$. From this and  \eqref{m2m1}
\begin{equation}
 \abs{\bar\pi^1}_1\ge   \frac{\tpath(\bar\pi) - \tpath^*(\bar\pi)}{\param_1}\ge\frac{c_2\lell_2' }{6d\est}\ge m_2
   \ge 4m_1+ 3 (k+1)(\lell+1).   
    \label{675.92}
\end{equation}

The next lemma ensures that the path segment $\bar\pi^1$ goes through the $k$-path of at least one $k+2\lell$-detour. 

\begin{lemma} \label{lm:W1-13}
  Let $\w$ and $\w^*$ be two environments that agree outside $B$ and satisfy 
  $\w\in\Lambda_{B,v,w,x}$ and $\w^*\in\Gamma_{B}$.
Let  $\pi^*(x)$ be the  geodesic for   $T_{\zevec,x}(\w^*)$ chosen in Lemma \ref{lm:pi*}. Then there exists a detour rectangle $G$ in $B$ with boundary paths $(\pi^+, \pi^{++})$ such that   $\pi^*(x)$  follows   $\pi^+$  and does not touch $\pi^{++}$, except at the endpoints shared by $\pi^+$  and   $\pi^{++}$. 
\end{lemma} 
\begin{proof}   By construction, the portion   $\bar\pi^1$ of $\pi^*(x)$  has a continuous   path segment   of length $m_2 \geq 4m_1 + 3 (k+1)(\lell+1)$ in $\overline{W}_{\!1}$.   This forces $\bar\pi^1$ to enter at least three $k\times\ell$ detour rectangles, because these rectangles are $m_1$ apart along  $\overline{W}_{\!1}$  and the path can use at most $(k+1)(\lell+1)$ edges in a gives detour rectangle.  Let   $G$ be  a middle rectangle along this path segment, in other words, one that is both entered and exited, and such that $\bar\pi^1$ covers two  $m_1$-segments on $W_1$  that connect $G$ to  some neighboring detour rectangles. 
(Recall Figure \ref{fig:detours}.)    
  We can assume without loss of generality  that $G$ lies in the $(\evec_1,  \evec_2)$-plane and that it is attached to  a boundary edge of some $V_j$ that lies along $\evec_1$.  

Let $\pi^+$ and $\pi^{++}$ 
be the boundary paths of $G$ and 
  $a$ and $b$    their common  endpoints.  
   Let  $\hat\pi=\bar\pi^1_{a,b}$ denote the segment of $\bar\pi^1$ between $a$ and $b$.  For ease of language assume that $\hat\pi$ visits $a$ first and then $b$.   We show that $\hat\pi=\pi^+$   by showing that all other cases are strictly worse. 
    
  Definition \eqref{Ga-d3} of $\Gamma_B$ and inequality  \eqref{v892.08} imply 
  $\tpath^*(\pi^{++}) \geq  (k + 2\lell)(r_1-\delta) > \param_0 k\ge \tpath^*(\pi^+)$ and  rule out the case $\hat\pi=\pi^{++}$.  
    If $\hat\pi$ coincides with neither $\pi^+$ nor $\pi^{++}$,  there are  points $a'$ and $b'$ on $\partial G$  such that  $\hat\pi$ visits $a,a',b',b$ in this order and   $\pi'=\hat\pi_{a',b'}$  lies in the interior  $G \setminus \partial G$. 
    
 If  $a'$ and $b'$  lie on the same or on adjacent sides of $\partial G$,    the $\ell^1$-path from $a'$ to $b'$ along $\partial G$  has smaller weight than $\pi'$. 
 
 Suppose   $a'$ and $b'$  lie  on opposite $\lell$-sides of $\pi^{++}$. Then 
 \[
        \tpath^*(\hat\pi) \geq \tpath^*(\hat\pi_{a,a'})  + \param_0 k + \tpath^*(\hat\pi_{b',b}) >\param_0 k\ge \tpath^*(\pi^+).  
    \]
   The   term $\param_0 k$ is a lower bound  on $\tpath^*(\pi')$. 
  The strict inequality comes from $r_1-\delta>0$ (from \eqref{v892.08})  and because  the segments $\hat\pi_{a,a'}$ and $\hat\pi_{b',b}$ are not degenerate paths.  This is the case because   no edge connects   the interior  $G \setminus \partial G$ to  either $a$ or $b$.  

    The remaining  option  is that  $a'$ and $b'$  lie  on opposite $k$-sides of $\partial G$.  Let  $a'$ be the first point at which $\hat\pi$ leaves $\partial G$, and let  $b'$ be the point of first return to $\partial G$.     
    
  \medskip

\begin{figure}[t]
	\begin{center}
		\begin{tikzpicture}[>=latex,  font=\footnotesize,scale=0.5]
		\draw[line width=0.5pt](0,0)--(0,4)--(8,4)--(8,0)--(0,0);
		\draw[fill](0,0)circle(2mm);
		\draw(0,-.1)node[below]{$a$};
		\draw[fill](8,0)circle(2mm)node[below]{$b$};
		\draw[line width=0.5pt](-2,0)--(10,0);
		\draw[<->,line width=0.5pt](0,-1.2)--(8,-1.2);
		\draw(4,-1.2)node[below]{$k$};
		\draw[<->,line width=0.5pt](-1,0)--(-1,4);
		\draw(-1,2)node[left]{$\ell$};
		\draw[line width=1.5 pt](5,0)--(5,1)--(4,1)--(4,1.5)--(3.25,1.5)--(3.25,2)--(3.5,2)--(3.5,2.25)--(4.25,2.25)node[right]{$\ \wh\pi_{a',b'}$}--(4.25,2.75)--(4.75,2.75)--(4.75,3.25)--(3.75,3.25)--(3.75,3.65)--(2.75,3.65)--(2.75,4);
		\draw[fill](5,0)circle(2mm)node[below]{$b'$};
		\draw[fill](2.75,4)circle(2mm);
		\draw(2.75,4.1)node[above]{$a'$};
		\end{tikzpicture}
	\end{center}
	\caption{\small  Case 1: $a'$ lies  on the $k$-side of $\pi^{++}$  and  $b'\in \pi^+$.}	
	\label{fig:case1}
\end{figure}
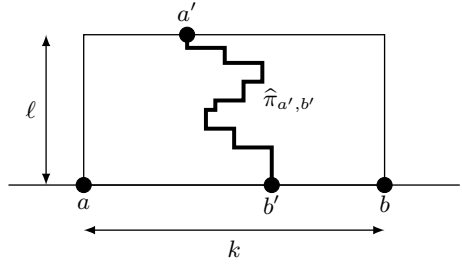

   {\it Case 1.}  
      Suppose  $a'$ lies  on the $k$-side of $\pi^{++}$  and  $b'\in \pi^+$ (Figure \ref{fig:case1}).   
 Fix coordinates as follows:   $a$ is at the origin, $a' = a_1'\evec_1 + \lell \evec_2$, and $b' = b_1' \evec_1$.   Then,
    \begin{align*}
       \tpath^*(\hat\pi_{a,b'}) & =  \tpath^*(\hat\pi_{a,a'}) + \tpath^*(\hat\pi_{a',b'})\\
       & \geq (r_1-\delta) |a - a'|_1 + |a' - b'|_1 \param_0\\
       & = (\lell+ a_1')(r_1-\delta) + (\lell  + |b_1' - a_1'| )\param_0\\
       & \geq 2 \lell r_1-\lell\delta + \bigl(a_1' (r_1-\delta) + |b_1' - a_1'| \param_0\bigr).
    \end{align*}
  Combine the above with  $\tpath^*(\pi^+_{a,b'}) \leq \param_0 b_1'$ and develop further: 
    \begin{align*}
       \tpath^*(\hat\pi_{a,b'}) - \tpath^*(\pi^+_{a,b'}) 
       & \geq 
       2\lell r_1-\lell\delta + a_1'(r_1-\delta)  + |b_1' - a_1'| \param_0 - \param_0 b_1'  \\
       &\ge   2\lell r_1-\lell\delta - a_1'(\param_0-r_1+\delta) \\
       &>  2 \lell(r_1-\delta)  -k(\param_0-r_1+\delta)>0.  
    \end{align*}
 The last  inequality is from  \eqref{v892.08}.       Thus   $\hat\pi$  cannot cross the interior of $G$ from $\pi^{++}$ to $\pi^+$.  
     
  \medskip   
    
   {\it Case 2.}  
      Suppose  $a'\in \pi^+$  and  $b'$ lies  on the $k$-side of $\pi^{++}$, so that $a' = a_1'\evec_1 $, and $b' = b_1' \evec_1+ \lell \evec_2$.   Then 
     \begin{align*}
       \tpath^*(\hat\pi_{a,b'}) & =  \tpath^*(\hat\pi_{a,a'}) + \tpath^*(\hat\pi_{a',b'})\\
       & \geq  a'_1(\param_0-\delta)    + (\lell  + |b_1' - a_1'| )\param_0\\
       & \geq   (\lell  + b_1')\param_0-a'_1\delta \\
       &>  (\lell  + b_1')(r_1+\delta) +(\lell  + 1) (\param_0-r_1-\delta)-k\delta\\
       &>  (\lell  + b_1')(r_1+\delta) \ge \tpath^*(\pi^{++}_{a,b'}).
    \end{align*}
The last strict inequality is from \eqref{44.52}.      Thus it is strictly better to take $\pi^{++}$ from  $a$ to $b'$.

\medskip 

In conclusion,   $\hat\pi$ does not coincide with $\pi^{++}$, nor does $\hat\pi$ visit the interior of the detour rectangle.  The only possibility is that    $\hat\pi=\pi^+$.  

\medskip 

It remains to argue that $\pi^*(x)$ does not touch $\pi^{++}$ except at the endpoints   $a$ and $b$ when it goes  through $\pi^+$.  Suppose on the contrary that $\pi^*(x)$ visits a vertex $\wh z$ on $\pi^{++}$.  This has to happen either before vertex $a$ or after vertex $b$.  The two cases are similar so suppose $\wh z$ is visited before $a$.  Then, by the choice of the detour rectangle $G$, the segment $\pi^*_{\wh z, a}(x)$ contains an $m_1$-segment  on $W_1$  that ends at $a$. Hence by the definition of $\Gamma_B$ and \eqref{m1delta3}, 
\[   \tpath^*(\pi^*_{\wh z, a}(x)) \ge m_1(r_1-\delta) > (k+2\ell)(r_1+\delta) .
\]
However, $(k+2\ell)(r_1+\delta)$ is an upper bound on the passage time of the path from $a$ to $\wh z$ along $\pi^{++}$, which is then strictly faster than $\pi^*_{\wh z, a}(x)$. Since $\pi^*_{\wh z, a}(x)$ must be a geodesic, the supposed visit to $\wh z$ cannot happen. 
\end{proof}

\medskip  

\noindent
{\bf Stage 2 for  unbounded weights.}   

\medskip

  In the unbounded weights case we   construct first  the $k+2\ell$ detour 
for a given triple $(B,v,w)$ and then the  good event $\Gamma_{B,v,w}$.    Given any $k,\ell\in\N$, the construction below can be carried out for all large enough $N$.  We label  the construction below so that we can refer to it again.    Figure \ref{fig:k+2ell} gives an illustration.


 \begin{figure}[t]
	\begin{center}
		\begin{tikzpicture}[>=latex,  font=\footnotesize,scale=0.5]
		\draw[line width=0.5pt](0,0)--(0,4)--(8,4)--(8,0)--(0,0);
		\draw[fill](0,0)circle(2mm);
		\draw[fill](8,0)circle(2mm);
		\draw[line width=0.5pt](-8,0)--(0,0);
		\draw[line width=0.5pt, dashed](-8,2)--(-6,2)--(-6,5.5)--(10,5.5)--(10,-2.1)--(-8,-2.1);
		\draw(8,5.5)node[above]{$\sigma'$};
		\draw(-2,-2.1)node[below]{$\sigma$};
		\draw[fill](10,0)circle(2mm)node[right]{$v+(k+\ell+2)\uvec$};
		\draw[fill](10,3)circle(2mm);
		\draw(10,3)node[left]{$x_0$};
		\draw[fill](-2,0)circle(2mm);
		\draw[<->,line width=0.5pt](0,-0.75)--(8,-0.75);
		\draw(4,-0.6)node[below]{$k$};
		\draw(4,0)node[above]{$\pi^+$};
		\draw(4,4)node[below]{$\pi^{++}$};
		\draw[<->,line width=0.5pt](-0.75,0)--(-0.75,4);
		\draw(-0.75,2)node[left]{$\ell$};
		\draw(10,3)--(12,3)--(12,4)--(13,4)--(13,6)--(15,6);
		\draw[fill](15,6)circle(2mm)node[right]{$w$};
		\draw(15,8)--(15,3);
		\draw(-8,8)--(-8,-6);
		\draw[fill](-8,2)circle(2mm)node[left]{$v+\uvec'$};
		\draw[fill](-8,0)circle(2mm)node[left]{$v$};
		\draw[fill](-8,-2.1)circle(2mm)node[left]{$v-\uvec'$};
		\draw[->](-4.8,-3.1)--(-2.2,-.2);
		\draw(-5,-3)node[below]{$v+\ell\uvec$};
		\draw[->](10,-3.1)--(8.2,-.2);
		\draw(10,-3)node[below]{$v+(k+\ell+1)\uvec$};
		\draw[->](1,-3)node[below]{$v+(\ell+1)\uvec$} to [out=90,in=225](-.2,-.2);
		\draw[<->,line width=1pt](13.5,-4.3)node[right]{$\uvec'$}--(13.5,-5.8)--(15,-5.8)node[below]{$\uvec$};
		\end{tikzpicture}
	\end{center}
	\caption{\small Illustration of Construction \ref{detour1}. The set $A$ of \eqref{A100} consists of the straight line from $v$ to $v+(\ell+1)\uvec$ and the $k\times\ell$ detour rectangle bounded by the union of the $k$-path $\pi^+$ and the $k+2\ell$-detour $\pi^{++}$. The figure shows case (ii) of the $A$-avoiding self-avoiding path from $v+(k+\ell+2)\uvec$ to $w$ via the point $x_0$ on $\sigma\cup\sigma'$.  }
	\label{fig:k+2ell}
	\medskip 
\end{figure}
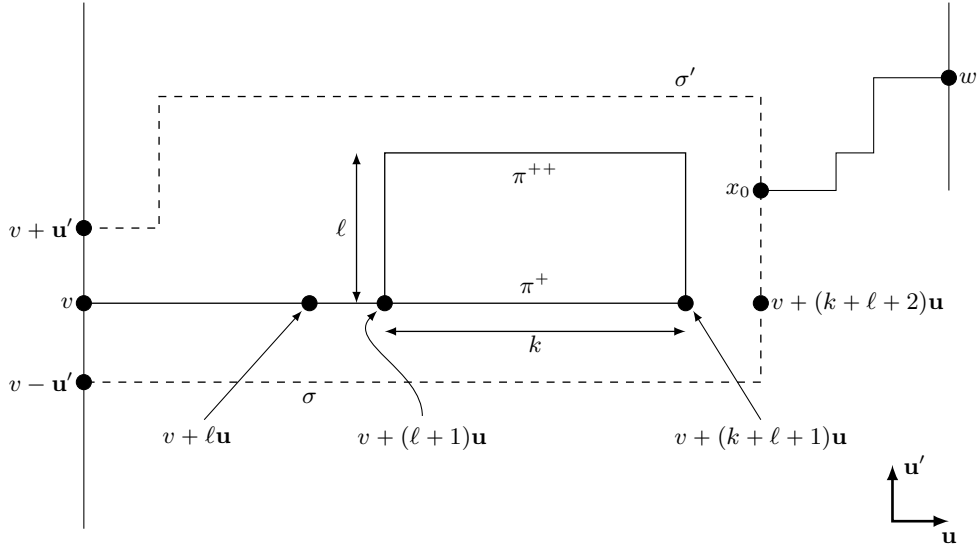


 \begin{construction}[{\bf The $k+2\ell$ detour for the unbounded weights case}]   \label{detour1}
Fix two unit vectors $\uvec$ and $\uvec'$ among $\{\pm\evec_i\}_{i=1}^d$  perpendicular to each other   so that the point   $v+(k+\ell+2)\uvec+\ell\uvec'$ lies in  $B$.  Hence also   the  rectangle  of size $(k+\ell+2)\times\ell$ with  corners $v$ and $v+(k+\ell+2)\uvec+\ell\uvec'$ lies in $B$.    Switch the labels $\uvec$ and $\uvec'$ if necessary to guarantee that  $w$ does not  lie in the set  
\be\label{A100}    A=\{v+h\uvec: 0\le h\le\ell\} \cup\{  v+i\uvec+j\uvec' : \ell+1\le i\le k+\ell+1, \,0\le j\le\ell\}. \ee 
  The two versions of $A$   obtained by interchanging $\uvec$ and $\uvec'$ have only $v$ in common,   so at least one of them does not contain $w$. 

From  $v+(k+\ell+2)\uvec$    there is a  self-avoiding  path to $w$ that stays inside $B$ and  does not intersect $A$.  
The existence of such a path and an  upper bound on the minimal  length of such a path can be seen as follows.  
\begin{enumerate} 
\item[(i)]   If   $w$ does not  lie on the plane through $v$  spanned by $\{\uvec, \uvec'\}$,  take a minimal length path from $v+(k+\ell+2)\uvec$    to $w$ that begins with a step $\zvec$  perpendicular to this plane.   Unit vector $\zvec$ is chosen so that $(w-v)\cdot\zvec>0$.   This path will not return to the $\{\uvec, \uvec'\}$ plane and hence avoids $A$.  The length of this path is at most $k+\ell+2+ \abs{w-v}_1$.  This is because a possible $A$-avoiding  route to $w$ takes first  the  $\zvec$-step  from $v+(k+\ell+2)\uvec$, then  $k+\ell+2$  $(-\uvec)$-steps   to $v+\zvec$,   and from $v+\zvec$ a minimal length path to $w$.  A   path from $v$ to $w$ includes a $\zvec$-step, hence the distance from $v+\zvec$   to $w$ is $\abs{w-v}_1-1$.

\item[(ii)]  Suppose    $w$  lies on the plane through $v$  spanned by $\{\uvec, \uvec'\}$.  Then we   move on this plane from $v+(k+\ell+2)\uvec$ to $w$ and  take care to avoid $A$.   
  First define  the minimal  $A$-avoiding    path  $\sigma$ from $v+(k+\ell+2)\uvec$    to $v-\uvec'$ in $k+\ell+3$ steps, and a minimal  $A$-avoiding    path  $\sigma'$ from $v+(k+\ell+2)\uvec$  to $v+\uvec'$ in $k+3\ell+3$ steps.     (We may be forced  to pick  between $v\pm\uvec'$ depending on which side of $A$  the point $w$ lies.)     Let $x_0$ be a closest point to $w$ on $\sigma\cup\sigma'$ (possibly $x_0=w$).   The $A$-avoiding  self-avoiding  path   from $v+(k+\ell+2)\uvec$ to $w$   then goes first  to $x_0$ along $\sigma$ or $\sigma'$ and from there takes a minimal length path   to $w$.    The length of this path is at most $k+3\ell+4+ \abs{w-v}_1$. 
 \end{enumerate}  
   
 Using the construction above,   fix  a self-avoiding path $\pi'$ in $B$  from $v$ to $w$  that begins with  $k+\ell+2$ $\uvec$-steps from $v$ to  $v+(k+\ell+2)\uvec$, avoids $A$ after that,   and has    
 \be\label{v894.0} \abs{\pi'}\le \abs{w-v}_1+2k+4\ell+6. \ee

Let  $\pi^+\subset\pi'$ be the $\uvec$-directed  straight line segment  of length $k$ from $\pi^+_0=\pi'_{\ell+1}=v+(\ell+1)\uvec$ to  $\pi^+_k=\pi'_{k+\ell+1}=v+(k+\ell+1)\uvec$.    Let $\pi^{++}\subset A$ be  the  detour   of length $k+2\ell$ between the endpoints  $\pi^{++}_0=\pi^+_0$ and  $\pi^{++}_{k+2\ell}=\pi^+_{k}$ defined as in \eqref{pi++}.   
The two endpoints of   $\pi^{++}$ lie on $\pi'$ but  $\pi^{++}$ is edge-disjoint from  $\pi'$.     
 This completes the construction of the $k+2\ell$ detour. 
 \hfill$\triangle$\end{construction}  
 



Let $b>0$ be given.     By assumption \eqref{16.ass3} we can choose $r<s$ in the support  of $t(e)$ so that   $b<r<s$.  Choose $k,\ell,\delta$ to satisfy \eqref{v892.00}.

Fix an element  $(B,v,w)$ for a while. 
 Define the following event  $\Gamma_{B,v,w}$ that depends only on the  weights $t(e)$   in $B$. Constants $s_0$ and $\delta_0$ are from definition \eqref{1.bl1}--\eqref{1.bl2}  of a black  $N$-box $B$. 
\be\label{vw76} \begin{aligned}
\Gamma_{B,v,w}=\bigl\{ \,& t(e)\in [\eit, \eit+\delta_0/2)  \,\text{ for }\,  e\in\pi'\setminus\pi^+, \\
& t(e)\in (s-\delta, s+\delta)  \,\text{ for }\,  e\in\pi^+, \\
    &t(e)\in(r-\delta, r+\delta)   \,\text{ for }\,  e\in\pi^{++}, \ \ \text{and}  \\
 & t(e)>s_0  \,\text{ for }\,   e\in B\setminus (\pi'\cup\pi^{++}) \,\bigr\}. 
\end{aligned}\ee
By \eqref{v892.00}, on the  event $\w\in\Gamma_{B,v,w}$, 
 \be\label{v.1004.1} 
\tpath(\pi^+)<\tpath(\pi^{++})<\tpath(\pi^+)+(2\ell-1)b.   
\ee

 Once $N$ has been fixed,  then up to translations and rotations  there are only finitely many ways to choose the points $v$ and $w$ on the boundary of $B$ and the paths $\pi', \pi^+, \pi^{++}$ constructed above.  
Thus   
\be\label{Ga84.2}  
\exists D_2>0 \text{ such that }  \P(\Gamma_{B,v,w})\ge D_2 \text{ for all triples } (B,v,w). \ee
   $D_2$ depends on $N$ and the probabilities of the events on  $t(e)$   that appear in $\Gamma_{B,v,w}$.    In particular, $D_2$ does not depend  on $x$.

On the event  $\Lambda_{B,v,w,x}$ of \eqref{vw-La},  $\pi(x)$ crosses $B$,   $v$ is the point of first entry into $B$ and $w$ the point of last exit from $B$.  Hence on this event  we can define   $\wb\pi$ as the self-avoiding path   from $\zevec$ to $x$  obtained by concatenating the segments $\pi_{\zevec,v}(x)$, $\pi'$, and $\pi_{w,\,x}(x)$.    For future reference at \eqref{PsiB2}, note that $\wb\pi$ is edge-disjoint from $\pi^{++}$. 

 \begin{lemma}  \label{lm77}  
   Let $\w$ and $\w^*$ be two environments that agree outside $B$ and satisfy 
  $\w\in\Lambda_{B,v,w,x}$ and $\w^*\in\Gamma_{B,v,w}$.  
 Then   $\wb\pi$ is a geodesic for $T_{\zevec,x}(\w^*)$.    Furthermore, if $\pi(x)$ was chosen to be a geodesic  of maximal Euclidean  length for $T_{\zevec,x}(\w)$, then $\wb\pi$ is   a geodesic of maximal length for $T_{\zevec,x}(\w^*)$.   The same works for minimal length.  \end{lemma}  
 
 \begin{proof}   Since   box $B$ is black on the event  $\Lambda_{B,v,w,x}$, 
 \be\label{vw60.4}   
   \tpath(\pi_{v,w}(x)) >  (\eit+\delta_0) (  \abs{w-v}_1 \vee N) . 
  \ee    
 The bound above comes from \eqref{1.bl2}, on account of these observations:  regardless of  whether $\pi_{v,w}(x)$ exits $\overline B$, there is a segment inside $\overline B$ of length $\abs{w-v}_1$, and furthermore $\pi_{v,w}(x)$ contains a short crossing of $B$ that has length at least $N$. 
 
From  $\w^*\in\Gamma_{B,v,w}$,  
 \begin{align*}    \tpath^*(\wb\pi_{v,w}) =\tpath^*(\pi') &<  k(s+\delta)+(  \abs{w-v}_1+ k+4\ell+6)(\eit+\tfrac12\delta_0) \\
 & \le \abs{w-v}_1(\eit+\tfrac12\delta_0) +  k(s+\eit+ \delta+\tfrac12\delta_0) + (4\ell+6)(\eit+\tfrac12\delta_0)\\
 &\le   \tpath(\pi_{v,w}(x)) - \tfrac12 (  \abs{w-v}_1 \vee N)\delta_0 + C_1\delta_0+C_2 \\
 &< \tpath(\pi_{v,w}(x)).  
  \end{align*}   
 Before the  last inequality above,  $C_i=C_i(k,\ell, \delta, s, r_0)$ are  constants determined by the quantities in parentheses.  The last inequality is then   guaranteed by fixing $N$ large enough relative to  $\delta_0$ and   these   other constants.    Observation \eqref{cross23} is used here. 
 
 Outside $B$ the weights $\w^*$ and $\w$ agree, and the segments   $\wb\pi_{\zevec,v}=\pi_{\zevec,v}(x)$  and   $\wb\pi_{w,x}=\pi_{w,x}(x)$  agree  and lie outside $B$. Hence 
the inequality above gives $\tpath^*(\wb\pi)<\tpath(\pi(x))$ and 
thereby, for any geodesic $\pi^*(x)$ from $\zevec$ to $x$ in environment $\w^*$,  
\be\label{vw78.7} \tpath^*(\pi^*(x))\le \tpath^*(\wb\pi)<\tpath(\pi(x)) . \ee
 This implies that every geodesic $\pi^*(x)$ must enter $B$ since otherwise \[  \tpath^*(\pi^*(x))=\tpath(\pi^*(x))\ge \tpath(\pi(x))>\tpath^*(\wb\pi), \] contradicting the optimality of $\pi^*(x)$ under $\w^*$. 

If   $\pi^*_B(x)\not\subset\pi'\cup\pi^{++}$, then $\pi^*(x)$ must use an edge $e$ in $B$ with weight $> s_0$.   Then by property \eqref{1.bl1} of a black box $B$,   $\tpath(\pi^*_B(x))\le s_0< \tpath^*(\pi^*_B(x))$.  Since $\w$ and $\w^*$ agree on $B^c$,    we get 
\begin{align*} 
 \tpath(\pi(x))\le \tpath(\pi^*(x))&=\tpath(\pi^*_{B^c}(x))+\tpath(\pi^*_B(x)) \\
 &
 <  \tpath^*(\pi^*_{B^c}(x))+\tpath^*(\pi^*_B(x))= \tpath^*(\pi^*(x)), \end{align*}
contradicting \eqref{vw78.7}.   Consequently  $\pi^*_B(x)\subset\pi'\cup\pi^{++}$.   Part of event $\Lambda_{B,v,w,x}$ is that  $\{\zevec,x\}\cap B=\varnothing$. Thus   $\pi^*(x)$ must  both enter and exit $B$.  As a geodesic $\pi^*(x)$ does not backtrack on itself. Hence it must traverse  the route between $v$ and  $w$.   By \eqref{v.1004.1} $\pi^+$ is better under $\w^*$ than $\pi^{++}$, and hence  $\pi^*_B(x)=\pi'=\wb\pi_{B}$.  

Outside $B$, under  both $\w$ and $\w^*$ since they agree on $B^c$,  $\wb\pi_{B^c}$ is an optimal union of two paths that connect the origin to one of  $v$ and  $w$,  and  the other one of  $v$ and  $w$ to $x$.  This concludes the proof that $\wb\pi$ is a geodesic for $T_{\zevec,x}(\w^*)$. 

Suppose $\pi(x)$ is a geodesic of maximal Euclidean  length under $\w$ but  under $\w^*$ there is a geodesic $\pi^*$ strictly  longer than  $\wb\pi$.   The argument  above showed $\pi^*_B=\wb\pi_{B}$.   Hence outside $B$,   $\pi^*_{B^c}$ must provide an $\w^*$-geodesic   from $\zevec$ or $x$ to one of $v$ or $w$ that is longer than that given by  $\wb\pi_{B^c}=\pi_{B^c}(x)$.  This contradicts the choice of $\pi(x)$ as a maximal length geodesic,  again because $\w$ and $\w^*$ agree on $B^c$.  Same works for minimal.  This completes the proof of Lemma \ref{lm77}.  
\end{proof}

\medskip  

\noindent
{\bf Stage 3 for both bounded and unbounded weights.}   

\medskip


We choose a particular geodesic $\pi(x)$ for $T_{\zevec,x}$.  In the bounded weights case, 
let $\pi(x)$ be the  geodesic  specified in Lemma \ref{lm:pi*}.  
In the unbounded weights case,  
let $\pi(x)$   be the unique lexicographically first geodesic among the geodesics  of maximal Euclidean length. 
Let $b>0$.   For $N$-boxes $B\in\cB_{j(x)}$ define the event 
\be\label{PsiB2} \begin{aligned}
\Psi_{B,x}=\bigl\{\, &\text{inside $B$ $\exists$ edge-disjoint    path segments $\pi^+$ and   $\pi^{++}$ that share both endpoints } \\
& \text{and satisfy  $\pi^+\subset \pi(x)$,  $(\pi(x)\setminus\pi^+)\cup\pi^{++}$ is a self-avoiding path, } \\ 
 & \text{$\abs{\pi^{++}}=  \abs{\pi^+}+2\ell$, and  $\tpath(\pi^+)<\tpath(\pi^{++})<\tpath(\pi^+)+(2\ell-1)b$ }
\bigr\}. 
 \end{aligned}\ee  
 
 Couple two i.i.d.\ edge weight configurations $\w=\{t(e)\}_{e\,\in\,\cE_d}$ and  $\w^*=\{t^*(e)\}_{e\,\in\,\cE_d}$ so that $t^*(e)=t(e)$ for $e\notin B$ (that is, at least one endpoint of $e$ lies outside $B$)  and so that the weights    $\{t(e)\}_{e\,\in\, \cE_d}$  and  $\{t^*(e)\}_{e\,\in\, B}$ are  independent.  

Lemma \ref{lm:W1-13} for bounded weights (with $\Gamma_{B,v,w}=\Gamma_B$) and    Lemma \ref{lm77}  for unbounded weights  
    imply that 
\[    \{\w\in\Lambda_{B,v,w,x}\}\cap \{\w^*\in\Gamma_{B,v,w}\} \,\subset\, \{ \w^*\in\Psi_{B,x}\} .\] 
In particular, by inequalities \eqref{v.1004.8} and \eqref{v.1004.1},  $\w^*\in\Gamma_{B,v,w}$ implies $\tpath^*(\pi^+)<\tpath^*(\pi^{++})<\tpath^*(\pi^+)+(2\ell-1)b$  
required for $\w^*\in\Psi_{B,x}$, 
where $\tpath^*$ denotes  passage time  in the environment $\w^*$.

By the independence of  $\{\w\in\Lambda_{B,v,w,x}\}$ and $ \{\w^*\in\Gamma_{B,v,w}\}$, and then by \eqref{Ga84.1} for bounded weights and by \eqref{Ga84.2} for unbounded weights, 
\be\label{v.vw79} \begin{aligned}
\P(\Psi_{B,x}) =\P\{ \w^*\in\Psi_{B,x} \} &\ge 
\P\{\w\in\Lambda_{B,v,w,x}\} \P\{\w^*\in\Gamma_{B,v,w}\} \,\ge\, D_2 \P(\Lambda_{B,v,w,x}).  
\end{aligned}\ee

Let $Y$ be  the number of  $(B,v,w)\in\cB_{j(x)}$ for which $\Psi_{B,x}$ occurs.  
   By  the above and  \eqref{vw60}, 
 \be\label{v.vw80}  \begin{aligned} 
 \E[Y]   
 &\ge \sum_{(B,v,w)\,\in\,\cB_{j(x)}} \P(\Psi_{B,x}) \\
&\ge \sum_{(B,v,w)\,\in\,\cB_{j(x)}} D_2 \P(\Lambda_{B,v,w,x})  \ge D_2D_1\abs{x}_1 \equiv D_3\abs{x}_1
 \end{aligned}  \ee
  for a new constant $D_3$. 

Since we have arranged the boxes in the elements  $(B,v,w)\in\cB_{j(x)}$ separated,   we can define a   self-avoiding  
path $\wh\pi$ from $\zevec$ to $x$ by replacing each $\pi^+$ segment with the  $\pi^{++}$ segment in  each box $B\in\cB_{j(x)}$ for which event $\Psi_{B,x}$ happens.  

Reduce the weights on each edge $e$  from $t(e)$  to $t^{(-b)}(e)=t(e)-b$.    By the definition of $\Psi_{B,x}$, the  $t^{(-b)}$-passage times of the segments   $\pi^+$  and   $\pi^{++}$  obey this inequality: 
\begin{align*}
 \tpath^{(-b)}(\pi^{++})= \tpath(\pi^{++}) -b\abs{\pi^{++}} < \tpath(\pi^+)+(2\ell-1)b -b\abs{\pi^{++}} = \tpath^{(-b)}(\pi^+)-b. 
\end{align*} 
Consequently, along the entire path  $\pi(x)$,  the replacements of  $\pi^+$   with   $\pi^{++}$  reduce  the $t^{(-b)}$-passage time by at least  $bY$. 
We get the following bound:   
\be\label{v.vw80.11}\begin{aligned}
T^{(-b)}_{\zevec,x} &\le \tpath^{(-b)}(\wh\pi)  
<  \tpath^{(-b)}(\pi(x))  -bY  = \tpath(\pi(x))-  b \,\abs{\pi(x)}   - b Y\\
&\begin{cases} 
 \le T_{\zevec,x} - b \underline L_{\tspb\zevec,x}  - b Y &\text{in the bounded weights case,} \\[2pt] 
= T_{\zevec,x} - b \overline L_{\tspb\zevec,x}  - b Y &\text{in the unbounded weights case.} 
\end{cases} 
\end{aligned}\ee 
 The case  distinction  above comes because in the unbounded case   $\abs{\pi(x)} =  \overline L_{\tspb\zevec,x}$ by our choice of $\pi(x)$,  while in the bounded case  our choice is different, but   any geodesic satisfies  $\abs{\pi(x)} \ge  \underline L_{\tspb\zevec,x}$.  
Note that the inequality above does not require  that $\wh\pi$ be a geodesic for $T^{(-b)}_{\zevec,x}$, as long as   $\wh\pi$  is self-avoiding. 

In order to take expectations in \eqref{v.vw80.11} we restrict to $b\in(0,\eit+\eet)$ which guarantees that  $\E[T^{(-b)}_{\zevec,x}]$ is finite, even if $-b<-\eit$ so that weights $\w^{(-b)}$ can be negative (Theorem \ref{thm:mu-b} in Appendix \ref{a:fpp<0}).    By Lemma 2.3 in \cite{Auf-Dam-Han-17}, moment bound \eqref{lin-ass5} with $p=1$  is equivalent to the finite expectation $\E[T_{\zevec,x}]<\infty$ for all $x$.    The inequalities  above then force $\underline L_{\tspb\zevec,x}$ and  $\overline L_{\tspb\zevec,x}$ to have finite expectations.  Apply \eqref{v.vw80}: 
in the bounded weights case 
\begin{align*}
\E[T^{(-b)}_{\zevec,x}] &\le  \E[T_{\zevec,x}] -  b\, \E( \underline L_{\tspb\zevec,x})   - b \,\E( Y) \le  \E[T_{\zevec,x}] -  b \E( \underline L_{\tspb\zevec,x})   - D_3b \abs{x}_1, 
\end{align*} 
while in the unbounded weights case $\E( \underline L_{\tspb\zevec,x})$ is replaced by $\E( \overline L_{\tspb\zevec,x})$. 
This completes the proof of Theorem \ref{v-thm0}. 
\end{proof} 
  
\medskip

 
\section{Modification proofs for nondifferentiability}  \label{sec:ndiff}

In this section we consider three scenarios under which we prove that, with probability bounded away from zero, there are  geodesics between two points whose lengths  differ on the scale of the distance between the  endpoints.  The setting and modification proofs in this section borrow heavily from Section \ref{sec:conc}. 

\begin{assumption}\label{ass:nd} 
We assume one of  these three situations for nonnegative weights. 
 \begin{enumerate} [label=\rm(\roman{*}), ref=\rm(\roman{*})]  \itemsep=3pt
 \item  \label{ass:nd.i}   Zero is an atom:  $\eit=0$ and  $0<\P\{t(e)=0\}<p_c$.  
\item  \label{ass:nd.ii}   The weights are unbounded $(\est=\infty)$ and there exist  strictly positive integers  $k$ and  $\ell$ and atoms $r'_1,\dotsc,r'_{k+2\ell}, s'_1,\dotsc,s'_k$ {\rm(}not necessarily all distinct{\rm)} such that   
\be\label{lin-ass7.4}\ 
 \sum_{i=1}^{k+2\ell} r'_i=\sum_{j=1}^{k} s'_j. 
\ee

\item  \label{ass:nd.iii}  The weights are bounded $(\est<\infty)$ and there exist  strictly positive integers  $k$ and $\ell$ and atoms $r<s$ such that $(k+2\ell)r=ks$.
 
\end{enumerate} 
\end{assumption}

\begin{theorem} \label{thm:ndi} Assume $\eit\ge0$, \eqref{pc-ass}, 
and the moment bound \eqref{lin-ass5} with $p>1$.   
Furthermore, assume one of the three scenarios \ref{ass:nd.i}--\ref{ass:nd.iii} of Assumption \ref{ass:nd}.
 Then there exist constants $0<D,\delta,M<\infty$    such that 
\be\label{z990}   \P\bigl(\, \overline L_{\tspb\zevec,x}-\underline  L_{\tspb\zevec,x} \ge D\abs{x}_1\bigr) \ge \delta  
\qquad\text{for   $\abs{x}_1\ge M$. }  \ee

\end{theorem}

Before the proof some observations about the assumptions are in order.  

\begin{remark} 
 Condition \eqref{lin-ass7.4} of case \ref{ass:nd.ii} is trivially true  if  zero  is an atom for $t(e)$.  Since this situation is taken care of by case \ref{ass:nd.i} of Assumption \ref{ass:nd}, let us suppose zero  is not an atom.  Then a necessary condition for \eqref{lin-ass7.4} is that  $t(e)$ has at least two strictly positive atoms. 
 
  A sufficient condition for \eqref{lin-ass7.4} is the existence of two atoms $r<s$ in $(0,\infty)$ such that  $s/r$ is rational.    This is exactly the assumption on the atoms in case \ref{ass:nd.iii}  of Assumption \ref{ass:nd}.   If $t(e)$ has  exactly two atoms $r<s$ in $(0,\infty)$ and no others, then \eqref{lin-ass7.4} holds if and only if $s/r$ is rational. 
  
   With more than two atoms, rational ratios are not necessary for \eqref{lin-ass7.4}.  For example, if $\theta>0$ is irrational and  $\{1,\theta, 1+2\theta\}$ are atoms, then  \eqref{lin-ass7.4} is satisfied and the ratios $\theta, 1+2\theta, \theta^{-1}+2$ are  irrational. 

We can prove a more general  result for unbounded weights because arbitrarily large weights can be used to force the geodesic to follow a specific path. With bounded weights the control of the geodesic is less precise. Hence the assumption in case \ref{ass:nd.iii}  is more restrictive on the atoms.  
\qedex\end{remark}

\begin{proof}[Proof of Theorem \ref{thm:ndi}] 
We prove the theorem by considering each case of Assumption \ref{ass:nd} in turn. 

\medskip 

\noindent 
{\bf Proof of Theorem \ref{thm:ndi} in Case \ref{ass:nd.i} of Assumption \ref{ass:nd}.}   

\medskip 

We assume that zero is an atom. 
 In this case conditions \eqref{1.bl0} or \eqref{1.bl1}  are not needed for a black box,  so  
 color a box $B$ black if \eqref{1.bl2} holds.  Fix  $N$ large enough and $\delta_0$ small enough.   Consider points $x$ with $\abs{x}_1$ large enough so that the Peierls estimate \eqref{vw56} is valid for $n=\abs{x}_1$. 

 Let $\pi(x)$ be the unique  geodesic for $T_{\zevec,x}$  that is  lexicographically first among the geodesics of minimal Euclidean length.  For this purpose 
   order $\range=\{\pm \evec_1,\dotsc, \pm \evec_d\}$ in some way,  for example as in \eqref{rng-ord}. 
      The index $j(x)$ and the event  $ \Lambda_{B,v,w,x}$ are defined as before in \eqref{vw61} and \eqref{vw-La}, and estimate \eqref{vw60} holds. 
 Let $\Gamma_{B}=\{\w: t(e)=0\; \forall e\in B\}$ be the event that all edge weights in $B$ are zero and 
  $  D_2=\P(\Gamma_{B})>0. $
  
 Given an $N$-box $B$, define edge weight configuration  $\w^*=\{t^*(e)\}_{e\,\in\,\cE_d}$ by setting $t^*(e)=t(e)$ for $e\notin B$ (that is, at least one endpoint of $e$ lies outside $B$)  and by resampling  $\{t^*(e)\}_{e\,\in\, B}$ independently.    Then $\w^*$ has the same i.i.d.\ distribution as the original weights  $\w=\{t(e)\}_{e\,\in\,\cE_d}$.  


 \begin{lemma}  \label{zlm77}  On the event $\{\w\in\Lambda_{B,v,w,x}\}\cap \{\w^*\in\Gamma_{B}\}$,   every geodesic from $\zevec$ to $x$ in the $\w^*$ environment    uses at least one edge in $B$.  
  \end{lemma}  
  
 \begin{proof}    On the event  $\{\w\in\Lambda_{B,v,w,x}\}$,  $\pi(x)$ goes through $v$ and $w$.  Let $\pi'$ be an arbitrary path from $v$ to $w$ that remains inside $B$ and  define   $\wb\pi$ as the path   from $\zevec$ to $x$  obtained by concatenating the segments $\pi_{\zevec,v}(x)$, $\pi'$, and $\pi_{w,\,x}(x)$.   
Then on the event $\{\w\in\Lambda_{B,v,w,x}\}\cap \{\w^*\in\Gamma_{B}\}$, 
 \begin{align*}    \tpath^*(\wb\pi_{v,w}) =\tpath^*(\pi')=0<\delta_0 (  \abs{w-v}_1 \vee N)< \tpath(\pi_{v,w}(x)).  
  \end{align*}   
 The justification for the last inequality was given below \eqref{vw60.4}.  
  
 Outside $B$  weights $\w^*$ and $\w$ agree, and the segments   $\wb\pi_{\zevec,v}=\pi_{\zevec,v}(x)$  and   $\wb\pi_{w,x}=\pi_{w,x}(x)$  agree  and lie outside $B$. Hence 
the inequality above gives $\tpath^*(\wb\pi)<\tpath(\pi(x))$ and 
thereby, for any geodesic $\pi^*(x)$ from $\zevec$ to $x$ in environment $\w^*$,  
\be\label{z-vw78} \tpath^*(\pi^*(x))\le \tpath^*(\wb\pi)<\tpath(\pi(x)) . \ee
 This implies that every geodesic $\pi^*(x)$ must use at least one edge in  $B$.
 For otherwise \[  \tpath^*(\pi^*(x))=\tpath(\pi^*(x))\ge\tpath(\pi(x))>\tpath^*(\wb\pi), \] contradicting the optimality of $\pi^*(x)$ for $\w^*$. 
 \end{proof}  

For $N$-boxes $B$ such that   $\zevec, x\notin B$  define the event 
\be\label{PsiB3} \begin{aligned}
\Psi_{B,x}=\{ &\text{ inside $B$ $\exists$    path segments $\pi^+$ and   $\pi^{++}$ that share both endpoints } \\
& \text{and satisfy  $\pi^+\subset \pi(x)$,   $(\pi(x)\setminus\pi^+)\cup\pi^{++}$ is a self-avoiding path, 
}\\
& \text{$\abs{\pi^{++}}\ge   \abs{\pi^+}+2$, and  $\tpath(\pi^+)=\tpath(\pi^{++})$  }
 \}. 
 \end{aligned}\ee  
 In particular, on the event $\Psi_{B,x}$, replacing $\pi^+$ with $\pi^{++}$ creates an alternative geodesic. 
 
By Lemma \ref{zlm77}, $\w^*\in\Psi_{B,x}$ holds on the  event  $\{\w\in\Lambda_{B,v,w,x}\}\cap \{\w^*\in\Gamma_{B}\}$.   This is seen as follows.
  Let $\pi^*(x)$ be the lexicographically first  geodesic of minimal Euclidean length in environment $\w^*$.  By Lemma \ref{zlm77}, $\pi^*(x)$ uses at least one edge in $B$.  Let $u_1$ be the first and $u_2$ the last point of $\pi^*(x)$ in $B$.    Since $\{\w^*\in\Gamma_{B}\}$ ensures that all edges in $B$ have zero weight and $\pi^*(x)$ is    a minimal length geodesic, the segment   $\pi^*_{u_1,u_2}(x)$ must be a path of length $\abs{u_2-u_1}_1$ from $u_1$ to $u_2$ inside $B$.    Now take $\pi^+=\pi^*_{u_1,u_2}(x)$ and let $\pi^{++}$ be any other path inside $B$ from $u_1$ to $u_2$ that takes more than the minimal number $\abs{u_2-u_1}_1$  of steps.  By the choice of $u_1$ and $u_2$, the other portions $\pi^*_{\zevec,u_1}(x)$ and $\pi_{u_2, x}^*(x)$ of the geodesic lie outside $B$, and consequently $\pi^{++}$ does not touch  these paths except at the points  $u_1$  and $u_2$. 

 By the independence of  $\{\w\in\Lambda_{B,v,w,x}\}$ and $ \{\w^*\in\Gamma_{B}\}$, 
\be\label{z-vw79} \begin{aligned}
\P(\Psi_{B,x}) =\P\{ \text{$\Psi_{B,x}$ occurs for $\w^*$} \} &\ge \P(\{\w\in\Lambda_{B,v,w,x}\}\cap \{\w^*\in\Gamma_{B}\}) \\
&= \P\{\w\in\Lambda_{B,v,w,x}\} \P\{\w^*\in\Gamma_{B}\} \,\ge\, D_2 \P(\Lambda_{B,v,w,x}).  
\end{aligned}\ee
Let $Y$ be  the number of  $(B,v,w)\in\cB_{j(x)}$ for which $\Psi_{B,x}$ occurs.  
   By  \eqref{vw60},  for another constant $D_3>0$,  
 \be\label{z-vw80}  \begin{aligned} 
 &\E[Y]   
\ge \sum_{(B,v,w)\in\cB_{j(x)}} D_2 \P(\Lambda_{B,v,w,x})  \ge D_3\abs{x}_1.   
 \end{aligned}  \ee

 
By Proposition  4.7(1)  of \cite{Auf-Dam-Han-17},   under  the assumption $\P\{t(e)=0\}<p_c$,  for any $p>0$ there exists a finite constant $C_p$ such that  for all $x\in\Z^d$, 
\be\label{fpp-78}  \E[\,(\overline L_{\tspb\zevec,x})^p\,]\le C_p\tspa \E[ \,(T_{\zevec,x})^p\,] . \ee
    By Lemma 2.3 in \cite{Auf-Dam-Han-17},  under   assumption \eqref{lin-ass5}  there exists a finite constant $C'$ such that  for all $x\in\Z^d$ 
\be\label{fpp-79}  \E[ \, (T_{\zevec,x})^p\,]\le C' \abs{x}_1^p.  \ee

An obvious upper bound on  $Y$ is   the number of edges on the  geodesic  $\pi(x)$.   Let $p>1$ be the power  for which  \eqref{lin-ass5}  is assumed to hold  and $q=\frac{p}{p-1}$ its conjugate exponent.   Then, by a combination of \eqref{z-vw80}, \eqref{fpp-78} and  \eqref{fpp-79},  
 \begin{align*}
 D_3\abs{x}_1\le \E(Y)&=\E(Y, \,Y< D_3\abs{x}_1/2) + \E(Y, \,Y\ge D_3\abs{x}_1/2) \\[3pt]
 &\le D_3\abs{x}_1/2  + \E( \abs{\pi(x)}, Y\ge D_3\abs{x}_1/2)  \\
 & \le D_3\abs{x}_1/2   +  \bigl(  \E[\, \abs{\pi(x)}^p\,]\bigr)^{\frac1p} \,  \P(Y\ge D_3\abs{x}_1/2)^{\frac1q} \\
 &  \le D_3\abs{x}_1/2   +    C\abs{x}_1   \P(Y\ge D_3\abs{x}_1/2)^{\frac1q} .
 \end{align*}
 From this  we get the bound 
 \[  \P\bigl(Y\ge \tfrac12 D_3\abs{x}_1\bigr)  \ge \delta_3>0  \quad\text{for large enough $\abs{x}_1$.}  \]

Since we have arranged the boxes $B$ in the elements  $(B,v,w)\in\cB_{j(x)}$ separated,   we can define a self-avoiding  path $\wh\pi(x)$ from $\zevec$ to $x$ by replacing each $\pi^+$ segment of $\pi(x)$  with the  $\pi^{++}$ segment in  each box $B$ for which event $\Psi_{B,x}$ happens.  This path $\wh\pi(x)$ has  the same passage time $\tpath(\wh\pi(x))=\tpath(\pi(x))$ and hence both  $\pi(x)$ and $\wh\pi(x)$ are geodesics.   By the construction, the numbers of edges on these paths satisfy    $\abs{\wh\pi(x)}\ge \abs{\pi(x)} +2Y$.     Thus we get these inequalities between the maximal and minimal geodesic length: 
 \begin{align*}
 \overline L_{\tspb\zevec,x} \ge \abs{\wh\pi(x)}  \ge  \abs{\pi(x)} +2Y\ge \underline L_{\tspb\zevec,x}  +2Y
 \end{align*}
and then 
 \[    \P\bigl(\, \overline L_{\tspb\zevec,x}-\underline  L_{\tspb\zevec,x} \ge D_3 \abs{x}_1\bigr) \ge
  \P\bigl(Y\ge \tfrac12 D_3\abs{x}_1\bigr)  \ge \delta_3. \]
  \eqref{z990}  has been proved.

%
%

\bigskip 

\noindent 
{\bf Proof of Theorem \ref{thm:ndi} in Case \ref{ass:nd.ii} of Assumption \ref{ass:nd}.}   

\medskip

By assumption \eqref{lin-ass7.4}  we can fix 
 $s_1<\infty$ large enough so that, for i.i.d.\ copies $t_i, t_j'$ of the edge weight $t(e)$, 
\be\label{lin-ass8}
\P\biggl\{  \,t_i\le s_1\;\forall i\in[k+2\ell], \;  t_j'\le s_1\;\forall j\in[k], \; \text{ and } \;\sum_{i=1}^{k+2\ell} t_i=\sum_{j=1}^{k} t_j'\biggr\} >0.  \ee

Apply Construction \ref{detour1} of the $k+2\ell$ detour in an $N$-box $B$ with given boundary points $v$ and $w$, to define paths $\pi'$, $\pi^+$ and $\pi^{++}$ in $B$ with $\abs{\pi^+}=k$ and $\abs{\pi^{++}}=k+2\ell$.  
 Define the  event  $\Gamma_{B,v,w}$ that depends only on the  weights $t(e)$   in $B$:  
\be\label{lin-vw76} \begin{aligned}
\Gamma_{B,v,w}=\Bigl\{ \,& t(e)\in [\eit, \eit+\delta_0/2)  \,\text{ for }\,  e\in\pi'\setminus\pi^+, \\[3pt]
&t(e)\le s_1  \,\text{ for }\,   e\in  \pi^+\cup\pi^{++} , \\[3pt] 
& \sum_{e\,\in\,\pi^{++}} t(e) = \sum_{e'\in\pi^{+}} t(e')  \quad  \ \ \text{and}  \\[3pt] 
 & t(e)>s_0  \,\text{ for }\,   e\in B\setminus (\pi'\cup\pi^{++}) \,\Bigr\}. 
\end{aligned}\ee
By \eqref{lin-ass8}, unbounded weights,  and the detour construction,  there exists a constant $D_2$ such that 
 $\P(\Gamma_{B,v,w})\ge D_2>0$ for all triples $(B,v,w)$.    
 
The steps follow those of the proof of Theorem \ref{v-thm0} and the proof of  Case \ref{ass:nd.i} of Theorem \ref{thm:ndi}.  First sample $\w$, and then  define   $\w^*=\{t^*(e)\}_{e\,\in\,\cE_d}$ by setting $t^*(e)=t(e)$ for $e\notin B$  and by resampling  $\{t^*(e)\}_{e\,\in\, B}$ independently.  Let  $\pi(x)$ be a self-avoiding geodesic of minimal Euclidean length  for $T_{\zevec,x}(\w)$.  On  the event  $\{\w\in\Lambda_{B,v,w,x}\}\cap \{\w^*\in\Gamma_{B,v,w}\}$ define the path   $\wb\pi$  from $\zevec$ to $x$   by concatenating the segments $\pi_{\zevec,v}(x)$, $\pi'$, and $\pi_{w,x}(x)$.

 \begin{lemma}  \label{lin-lm77}  When $N$ is fixed large enough,  on the event $\{\w\in\Lambda_{B,v,w,x}\}\cap \{\w^*\in\Gamma_{B,v,w}\}$ the path   $\wb\pi$ is a self-avoiding geodesic of minimal Euclidean length  for $T_{\zevec,x}(\w^*)$.  \end{lemma}  
 
 \begin{proof}   As before, since   box $B$ is black on the event  $\Lambda_{B,v,w,x}$, 
 \begin{align*}   
   \tpath(\pi_{v,w}(x)) >  (\eit+\delta_0) (  \abs{w-v}_1 \vee N) . 
  \end{align*}   

Then by  $\w^*\in\Gamma_{B,v,w}$,  
 \begin{align*}    \tpath^*(\wb\pi_{v,w}) =\tpath^*(\pi') &<  k s_1  +(  \abs{w-v}_1+ k+4\ell+6)(\eit+\tfrac12\delta_0) \\
 & \le \bigl(  \abs{w-v}_1 \vee N\bigr)(\eit+\tfrac12\delta_0) +  k(s_1+\eit +\tfrac12\delta_0) + (4\ell+6)(\eit+\tfrac12\delta_0)\\
 &\le   \tpath(\pi_{v,w}(x)) - \tfrac12 (  \abs{w-v}_1 \vee N)\delta_0 + C \\
 &< \tpath(\pi_{v,w}(x)).  
  \end{align*}   
 Before the  last inequality above,  $C=C(k,\ell,   s_1, r_0, \delta_0)$ is a constant determined by the quantities fixed thus far in the proof.  The last inequality is then   guaranteed by fixing $N$ large enough relative to   these   other constants.   
 Outside $B$  weights $\w^*$ and $\w$ agree, and the segments   $\wb\pi_{\zevec,v}=\pi_{\zevec,v}(x)$  and   $\wb\pi_{w,x}=\pi_{w,x}(x)$  agree  and lie outside $B$. Hence 
the inequality above gives $\tpath^*(\wb\pi)<\tpath(\pi(x))$ and 
thereby, for any geodesic $\pi^*(x)$ from $\zevec$ to $x$ in environment $\w^*$,  
\be\label{lin-vw78} \tpath^*(\pi^*(x))\le \tpath^*(\wb\pi)<\tpath(\pi(x)) . \ee
 As explained below \eqref{z-vw78}, this implies that every $\w^*$ geodesic $\pi^*(x)$ must enter $B$. 

If   $\pi^*_B(x)\not\subset\pi'\cup\pi^{++}$, then $\pi^*(x)$ must use an edge $e$ in $B$ with weight $> s_0$.   Then by property \eqref{1.bl1} of a black box $B$,   $\tpath(\pi^*_B(x))\le s_0< \tpath^*(\pi^*_B(x))$.  Since $t$ and $t^*$ agree on $B^c$,    we get 
\begin{align*} 
 \tpath(\pi(x))\le \tpath(\pi^*(x))&=\tpath(\pi^*_{B^c}(x))+\tpath(\pi^*_B(x)) \\
 &=\tpath^*(\pi^*_{B^c}(x))+\tpath(\pi^*_B(x)) <  \tpath^*(\pi^*_{B^c}(x))+\tpath^*(\pi^*_B(x))= \tpath^*(\pi^*(x)), \end{align*}
contradicting \eqref{lin-vw78}.   Consequently  $\pi^*_B(x)\subset\pi'\cup\pi^{++}$.   As a geodesic $\pi^*(x)$ does not backtrack on itself.  Hence it must traverse  the route between $v$ to $w$, either via $\pi'$ or via $\pi'$ with $\pi^+$ replaced by $\pi^{++}$.   By \eqref{lin-vw76} $\tpath^*(\pi^+)=\tpath^*(\pi^{++})$ so there is no travel time distinction between the two routes between $v$ and $w$. 

 Since $\w$ and $\w^*$ agree on $B^c$,  $\wb\pi_{B^c}$ is an optimal union of two paths that connect $\zevec$ to one of  $v$ and  $w$,  and $x$ to the other one of  $v$ and  $w$.  Thus  $\wb\pi$ is a geodesic for $T_{\zevec,x}(\w^*)$. 

The argument above  showed that every geodesic of $T_{\zevec, x}(\w^*)$  goes from $v$ to $w$ utilizing edges in  $\pi'\cup\pi^{++}$ and otherwise remains outside $B$.  If there were a  geodesic $\pi^o$ strictly  shorter than $\wb\pi$, $\pi^o$ would have to use an alternative shorter geodesic path between $\zevec$ and $v$ or between $w$ and $x$. This contradicts the choice of $\pi(x)$ as the shortest geodesic. 
\end{proof}  

Define $\Psi_{B,x}$ as in \eqref{PsiB3} above.  By Lemma \ref{lin-lm77}, $\w^*\in\Psi_{B,x}$ holds on the  event  $\{\w\in\Lambda_{B,v,w,x}\}\cap \{\w^*\in\Gamma_{B,v,w}\}$.  
The proof of this case  is completed exactly as was done  in the previous case  from equation \eqref{z-vw79} onwards.

\bigskip 

\noindent 
{\bf Proof of Theorem \ref{thm:ndi} in Case \ref{ass:nd.iii} of Assumption \ref{ass:nd}.}   

\medskip 

The weights are now assumed bounded. 
We work under assumption \eqref{lin-ass7.4} until the last stage of the proof  where we have to invoke the more stringent assumption  of Case \ref{ass:nd.iii} under which \eqref{lin-ass7.4} is restricted to the case where all $r'_i=r$ and all $s_j'=s$.  
 Since the case of a zero atom has been taken care of,  we can assume
 that  these atoms $\{r'_i, s'_j\}$ are strictly positive and 
   that zero is not an atom.  
Since zero is not an atom, condition \eqref{ass78.1} holds.  


As in the cases above,  all that is needed for the conclusion is that the geodesic encounters $(\pi^+, \pi^{++})$-pairs whose passage times coincide. 
This proof follows closely the bounded weight case of Stage 2 of the proof of Theorem \ref{v-thm0}, which required condition \eqref{ass78.1}.  
Lemma \ref{lm:sdelta1}  can be enhanced to include the additional conclusion 
   \be\label{nu3.3}   \max_{i,j}\{r'_i, s'_j\} \le \param_0(q). \ee
 The only change  required in the proof of Lemma \ref{lm:sdelta1}   is that  induction begins with  $\param_0(0) = (\eit+\delta_0)\vee \max_{i,j}\{r'_i, s'_j\} $, after the case $\P\{t(e)=\est\}>0$ has been taken care of. 

The construction of  $W_1$, $W^+_1$,  $W'_1$, $\overline{W}_{\!1}$ and $W_2$ in each black box $B$   goes exactly as before around \eqref{HH1}.   
Let $\{ \pi^+_{B,\tsp j}\,, \pi^{++}_{B,\tsp j}\}_{1\le j\le j_1(B)}$  be the $\pi^+$ and $\pi^{++}$ boundary path segments of the   detour rectangles $\{ G_{B,\tsp j}\}_{1\le j\le j_1(B)}$ constructed in the box  $B$.  In particular, 
\[  W_1^+=\bigcup_j \pi^+_{B,\tsp j}\subset W_1
\quad\text{ and }\quad 
W'_1= \biggl( W_1\cup\bigcup_j \pi^{++}_{B,\tsp j}\biggr)  \subset  \biggl( W_1\cup\bigcup_j G_{B,\tsp j}\biggr) =\overline{W}_{\!1}.
\] 
Define the event
 \be\label{Ga-d3.1} 
    \begin{aligned}
    \Gamma_{B} 
        = \Big\{ \,\w: \;  &r_1-\delta< t(e)<r_1+\delta  
        \quad \forall e \in W_1\, \setminus W^+_1, \\
   & \sum_{e\,\in\,\pi^{++}_{B,j}} t(e) = \sum_{e'\,\in\,\pi^{+}_{B,j}} t(e')  \quad\forall j , \\[2pt]
        &0<t(e) \le \param_0  \quad\forall e \in W_1',  \\[3pt]
   & \param_0 \le t(e) \leq \param_1 \quad\forall e \in \overline{W}_{\!1}\,\setminus W_1',\\[2pt] 
      \text{and}   \quad        &\param_1\le t(e)\le  \est  \quad\forall e \in B \setminus  \overline{W}_{\!1} 
        \,\Big\}.
    \end{aligned}
\ee
The condition $t(e) \le \param_0 \;\forall e \in W_1'$ is implied by the conditions before it.  It is stated explicitly merely  for clarity.    The condition $t(e) >0 \;\forall e \in W_1'$ can be imposed because  (i) for $e \in W_1\, \setminus W^+_1$ it follows  from $t(e)>r_1-\delta$ (recall from \eqref{v892.08} that $r_1-\delta>0$), and (ii) for edges $e\in\bigcup_{1\le j\le j_1(B)} (\pi^+_{B,\tsp j}\cup \pi^{++}_{B,\tsp j})$ we can use the strictly positive atoms $\{r'_i, s'_j\}$. 
Again $ \P(\Gamma_{B})\ge D_2$ for a constant $D_2$.  

As before, given an $N$-box $B$ we work with two environments $\w$ and $\w^*$ that agree outside $B$.  
Let 
 $\pi^*(x)$ be  the  $T_{\zevec,x}(\w^*)$  geodesic    specified in Lemma \ref{lm:pi*}.   
Starting from inequality \eqref{675.74}, Stage 2 for bounded weights in  the proof of Theorem \ref{v-thm0} can be followed down to inequality \eqref{675.92}, to get  the existence of an excursion $\bar\pi$ in $\pi^*(x)$  whose  segment  $\bar\pi^1$  in $\overline{W}_{\!1}$ satisfies \eqref{675.92}.   
The previous Lemma \ref{lm:W1-13} is then replaced by the next lemma. 


\begin{lemma} \label{lm:W1-10}
 Assume $\w\in\Lambda_{B,v,w,x}$ and $\w^*\in\Gamma_B$.    Then there exist three path segments  $\hat\pi, \pi^+, \pi^{++}$ in $B$  with the same endpoints  and such that  the following holds: 
  \begin{enumerate} [label=\rm(\roman{*}), ref=\rm(\roman{*})]  \itemsep=3pt
  \item   the   pair  $(\pi^+, \pi^{++})$ forms the boundaries of a detour rectangle, 
  \item  $\hat\pi\subset\pi^*(x)$, and 
  \item replacing $\hat\pi$  in  $\pi^*(x)$ with  either $\pi^+$ or $\pi^{++}$ produces two self-avoiding geodesics for $T_{\zevec,x}(\w^*)$. 
 \end{enumerate} 
\end{lemma} 

\begin{proof}    As in the proof of Lemma \ref{lm:W1-13},  $\bar\pi^1$ has a segment  $\hat\pi=\bar\pi^1_{a,b}$ between the common endpoints  $a$ and  $b$ of the boundary paths  $\pi^+$ and $\pi^{++}$ of some detour rectangle $G$  in $B$.   We show that $\pi^*(x)$ can  be redirected to take either $\pi^+$ or $\pi^{++}$, by showing that (i) $\hat\pi$ cannot be  strictly better than  $\pi^+$ or $\pi^{++}$ and (ii) replacing $\hat\pi$ with   $\pi^+$ or $\pi^{++}$ does  not violate the requirement that a geodesic  be self-avoiding.

Suppose $\tpath^*(\hat\pi)<\tpath^*(\pi^+) =\tpath^*(\pi^{++}) $. 
 Then  there are points $a'$ and $b'$ on $\partial G$  such that  $\hat\pi$ visits $a,a',b',b$ in this order and  the edges of  $\pi'=\hat\pi_{a',b'}$  lie in the interior  $G \setminus \partial G$.   Recall that on the event $\Gamma_B$, the weights on $\partial G$ are at most $\param_0$ while the weights  in the interior  $G \setminus \partial G$ are at least $\param_0$. 
    
 The points    $a'$ and $b'$ cannot lie on the same or on adjacent sides of $\partial G$ since the $\ell^1$-path from $a'$ to $b'$ along $\partial G$  has no larger weight than $\pi'$. 
 
 Suppose   $a'$ and $b'$  lie  on opposite $\lell$-sides of $G$. Then 
    \[
        \tpath^*(\hat\pi) \geq \tpath^*(\pi') \ge    \param_0 k  \ge \tpath^*(\pi^+) = \tpath^*(\pi^{++}) .  
    \]
So we can do at least as well  by picking  $\pi^+$ or $\pi^{++}$.  

    The remaining  option  is that  $a'$ and $b'$  lie  on opposite $k$-sides of $G$. 
  Let us suppose that $a'$ is the first point at which $\hat\pi$ leaves $\partial G$ and  $b'$ the first return to $\partial G$.    
  
  For this  argument   we use  the most restrictive assumption that there are two 
   atoms $r<s$ such that $(k+2\ell)r=ks$,  with weights $t(e)=s$ on edges $e\in\pi^+$ and $t(e)=r$ on edges $e\in\pi^{++}$.   

\medskip 
    
 {\it Case 1.}    
     Suppose  $a'$ lies  on the $k$-segment  of $\pi^{++}$  and  $b'\in\pi^+$.    
     (See again Figure \ref{fig:case1}.) 
    We  can  assume that $a$ is at the origin, $a' = a_1'\evec_1 + \lell \evec_2$, and $b' = b_1' \evec_1$.  Then,
    \begin{align*}
       \tpath^*(\hat\pi_{a,b'}) & =  \tpath^*(\hat\pi_{a,a'}) +  \tpath^*(\hat\pi_{a',b'})\\
       & \geq   |a - a'|_1 r + |a' - b'|_1 \param_0\\
       & = (\lell+ a_1')r + (\lell  + |b_1' - a_1'| )\param_0. 
    \end{align*}
    
  From $a_1'\le k-1$ and  the assumptions $\param_0\ge s>r$ and $ks=(k+2\ell)r$  we deduce:
  \begin{align*}
& \ell(r+s)\ge 2\ell r = k(s-r) >  a_1'(s-r) \ge \bigl( b_1'- \abs{b_1' - a_1'}  \bigr) s -a_1'r  \\ 
 & \quad \implies  \ \ (\ell+a_1') r + (\lell  + |b_1' - a_1'| )\param_0 >    b_1'  s \\
 & \quad \implies \ \   \tpath^*(\hat\pi_{a,b'}) >  \tpath^*(\pi^+_{a,b'}). 
  \end{align*}   
 In other words, we can do better by  taking $\pi^+$ from  $a$ to $b'$.  
  
\medskip 
    
 {\it Case 2.}    
     Suppose   $a'\in\pi^+$ and  $b'$ lies  on the $k$-segment  of $\pi^{++}$ so that   $a' = a_1'\evec_1$ and $b' = b_1' \evec_1+ \lell \evec_2$.  Then,
    \begin{align*}
       \tpath^*(\hat\pi_{a,b'}) &  \geq  a'_1 s + (\ell + |b_1' - a_1'| )\param_0
       > (\ell+b_1')r = \tpath^*(\pi^{++}_{a,b'}) . 
    \end{align*}
This time it is better to take $\pi^{++}$ from  $a$ to $b'$. 

\medskip 

We have shown  that the passage time is not made worse by forcing $\hat\pi$ to take    $\pi^+$ or $\pi^{++}$.  Suppose doing so violates self-avoidance of the overall path from $\zevec$ to $x$.  Then we can cut out part of the path, and the removed piece includes at least one edge of either $\pi^+$ or $\pi^{++}$.   The assumption $\w^*\in\Gamma_B$ implies that $t^*(e)>0$ for these edges. Consequently the original passage time could not have been optimal. 
      \end{proof}

The event $\Psi_{B,x}$ earlier defined  in \eqref{PsiB3} has to be reworded slightly for the present case.  Let $\pi(x)$ be the $T_{\zevec,x}(\w)$  geodesic  chosen in Lemma \ref{lm:pi*}.   
\be\label{PsiB3.7} \begin{aligned}
\Psi_{B,x}=\{ &\text{ inside $B$ $\exists$    path segments $\hat\pi$, $\pi^+$ and   $\pi^{++}$ that share both endpoints } \\
& \text{and satisfy  $\hat\pi\subset \pi(x)$,  both   $(\pi(x)\setminus\hat\pi)\cup\pi^{+}$  and   $(\pi(x)\setminus\hat\pi)\cup\pi^{++}$ } \\
& \text{are  self-avoiding paths from $\zevec$ to $x$, 
}\\
& \text{$\abs{\pi^{++}}\ge   \abs{\pi^+}+2$, and  $\tpath(\hat\pi)=\tpath(\pi^+)=\tpath(\pi^{++})$  }
 \}. 
 \end{aligned}\ee  
 It is of course possible that $\hat\pi$ agrees with either $\pi^+$ or $\pi^{++}$. 
 By Lemma \ref{lm:W1-10}, $\w^*\in\Psi_{B,x}$ holds on the  event  $\{\w\in\Lambda_{B,v,w,x}\}\cap \{\w^*\in\Gamma_{B,v,w}\}$.  

Now follow the proof of  the previous case  from equation \eqref{z-vw79} onwards. 
 Again, since   the boxes $B$ in the elements  $(B,v,w)\in\cB_{j(x)}$ are separated,   we can define two  self-avoiding  paths $\pi^+(x)$ and $\pi^{++}(x)$  from $\zevec$ to $x$ by replacing each $\hat\pi$ segment of $\pi(x)$  with the  $\pi^+$ (respectively,  $\pi^{++}$)  segment in  each box $B$ that appears among $(B,v,w)\in\cB_{j(x)}$ and  for which event $\Psi_{B,x}$ happens.   Then both  $\pi^+(x)$ and $\pi^{++}(x)$ are self-avoiding geodesics for $T_{\zevec,x}(\w)$.  
 
 
  By the construction, the Euclidean lengths of these paths satisfy    $\abs{\pi^{++}(x)}\ge \abs{\pi^+(x)} +2Y$ where $Y$ is again   the number of  $(B,v,w)\in\cB_{j(x)}$ for which $\Psi_{B,x}$ occurs. Hence 
  \begin{align*}
 \overline L_{\zevec,x} \ge \abs{\pi^{++}(x)}  \ge  \abs{\pi^+(x)} +2Y\ge \underline L_{\tspb\zevec,x}  +2Y.  
 \end{align*}
 
 This completes the proof of the third case and thereby the proof of Theorem \ref{thm:ndi}. 
 \end{proof}

\bigskip 

\section{Proofs of the main theorems} \label{sec:m-pf}

This section proves the remaining claims   of Section  \ref{sec:rfpp}  by appeal to the preparatory work of Section \ref{sec:tech1} and the modification results of Sections \ref{sec:conc} and \ref{sec:ndiff}.  

\subsection{Strict concavity, derivatives, and geodesic length}  \label{sec:pf-cdg}

  The next theorem gives part \ref{thm:fppb2.ii} of Theorem \ref{thm:fppb2} and thereby completes the proof of Theorem \ref{thm:fppb2}.    Recall that $\eit=\essinf t(e)$ and $\e_0>0$ is the constant  specified in Theorems \ref{thm:fpp1} and \ref{thm:mu-b}. 
 
\begin{theorem} \label{nu-thm-1}   Assume $\eit\ge0$,  \eqref{pc-ass}, and  moment bound  \eqref{lin-ass5} with $p=d$.    Then there exist strictly positive constants  $D(a,h)$     
such that the following holds for all $\xi\in\R^d\tspa\setminus\{\zevec\}$: whenever $a\ge-\eit$ and  $-\eit-\eet< a-h<a$, 
\be\label{nu-45} \begin{aligned}
\fppb_\xi(a-h) 
&\le \fppb_\xi(a) -  h\tspa \fppb_\xi'(a+)    - D(a,h)\tsp h\tspa  \abs{\xi}_1.
\end{aligned}\ee
As a consequence,  
$\fppb_\xi'(a_0+)>\fppb_\xi'(a_1-)$  whenever  $-\eit\le a_0<a_1<\infty$  and $\fppb_\xi'(b\pm)> \fppb_\xi'((-\eit)+)$ for all $b\in(-\eit-\eet, -\eit)$. 
\end{theorem}

Note that the theorem does not rule out a linear segment of $\fppb_\xi$  immediately to the left of $-\eit$ which happens if $\fppb_\xi'(b+)=\fppb_\xi'((-\eit)-)$ for some $b\in(-\eit-\eet, -\eit)$.   But this does force $\fppb_\xi'((-\eit)-)>\fppb_\xi'((-\eit)+)$ and thereby a singularity at $-\eit$.  

\begin{proof}
We start by deriving the  last statement of strict concavity from \eqref{nu-45}. 
Suppose that $\fppb_\xi'(a_0+)=\fppb_\xi'(a_1-)=\tau$ for some $-\eit\le a_0<a_1<\infty$.  Then by concavity  $\fppb_\xi$ must be affine on the open interval $(a_0, a_1)$:    $\fppb_\xi(a)=\fppb_\xi(a_0)+\tau(a-a_0)$ and  $\fppb_\xi'(a)=\tau$ for $a\in(a_0, a_1)$.  This violates  \eqref{nu-45}.    The second claim of the  last statement  follows similarly.  

\medskip 

For this and a later proof, we  check here the validity of the  middle portion of \eqref{fppb5.24}.
Let $b>-\eit-\eet$,     $\xi\in\R^d\tspa\setminus\{\zevec\}$,   $\w\in\Omega_0=$ the full measure event specified in Theorem \ref{thm:mu-b}, and   $x_n/n\to\xi$.    Take limits \eqref{fpp-lim} in the extremes of \eqref{fppb56},  limits   $\varliminf n^{-1}\underline L^{(b)}_{\tspb\zevec,\tspb x_n}(\w)$  and  $\varlimsup n^{-1} \overline L^{\tspa(b)}_{\tspb\zevec,\tspb x_n}(\w)$ in the middle of \eqref{fppb56}, and then let $\delta, \eta\searrow0$.  This gives  
 \be\label{98.13}    
  \fppb_\xi'(b+)  \le 
 \varliminf_{n\to\infty} \frac{\underline L^{(b)}_{\tspb\zevec,\tspb x_n}(\w)}{n}   \le  \varlimsup_{n\to\infty} \frac{\overline L^{\tspa(b)}_{\tspb\zevec,\tspb x_n}(\w)}{n}
 \le   \fppb_\xi'(b-). \ee

\medskip 

To prove \eqref{nu-45}, consider first the case where   $a> -\eit$ or  $a=-\eit$ but $\P\{t(e)=\eit\}=0$.   The hypotheses of  Theorem \ref{v-thm0} are satisfied for the shifted weights $\w^{(a)}$. In particular, the extra assumption \eqref{ass78.1} of the bounded weights case that requires the existence of a positive support point $r_1$ close enough to the lower bound is valid because either  $\essinf t^{(a)}(e)>0$ or $\essinf t^{(a)}(e)=0$ but $0$ is not an atom.  

From  Theorem \ref{v-thm0} applied to the shifted weights $\w^{(a)}$  
  we take the conclusion  \eqref{v98.7} which is valid in both cases of the theorem:
    \be\label{v98.7.6} 
 \E[T^{(a-h)}_{\zevec,x}] \le  \E[T^{(a)}_{\zevec,x}] -  h\, \E[ \,\underline L^{(a)}_{\tspb\zevec,x}]    - D(a,h)\tsp h\tsp \abs{x}_1.   \ee
  The constant $D(a,h)$    given by the theorem  depends now also on $a$. 
  
  In    \eqref{v98.7.6} take $x=x_n$, divide through by $n$, and let $n\to\infty$ along a suitable subsequence. The expectations of normalized passage times converge by Theorem \ref{thm:mu-b}.   We obtain 
 \be\label{v98.60} 
\fppb_\xi(a-h) \le \fppb_\xi(a) -  h \varlimsup_{n\to\infty}  n^{-1} \E[\,\underline L^{(a)}_{\tspb\zevec,\tspb x_n}\,]     - D(a,h) \tsp h\tsp \abs{\xi}_1.   \ee
   By Fatou's lemma and \eqref{98.13}, 
\be\label{v890.001}   \begin{aligned}
  \varlimsup_{n\to\infty}  n^{-1} \E[\,\underline L^{(a)}_{\tspb\zevec,\tspb x_n}\,] 
 \ge  \varliminf_{n\to\infty}  n^{-1} \E[\,\underline L^{(a)}_{\tspb\zevec,\tspb x_n}\,]
 \ge  \E\bigl[\, \varliminf_{n\to\infty}    n^{-1} \underline L^{(a)}_{\tspb\zevec,\tspb x_n}\,\bigr] \ge \fppb_\xi'(a+). 
\end{aligned}\ee
This substituted into \eqref{v98.60} gives  
   \eqref{nu-45}.    

\medskip 

Last we take up the case   $a=-\eit$ and  $0<\P\{t(e)=\eit\}<p_c$.   The shifted weights $\w^{(-\eit)}$ satisfy $0<\P\{t(e)=0\}<p_c$.   This puts us in case (i) of Theorem \ref{thm:ndi}. Its conclusion  \eqref{z990} 
   implies the existence of a constant $D>0$ such that 
\[  \P\bigl( \,\overline L^{\tspa(-\eit)}_{\tspb\zevec,x_{n}}-\underline  L^{(-\eit)}_{\tspb\zevec,x_{n}} \ge D\abs{x_{n}}_1 \text{  for infinitely many $n$}\tspb \bigr) \ge \delta. \]  
 Hence  
 \eqref{98.13} implies $\fppb_\xi'((-\eit)-)-  \fppb_\xi'((-\eit)+)\ge D\abs{\xi}_1$. Note that $D$ does not depend on the sequence $\{x_n\}$ or $\xi$. 
 \eqref{nu-45} comes from  concavity: 
\[  
\fppb_\xi(-\eit-h) \le \fppb_\xi(-\eit) - \fppb_\xi'((-\eit)-)h 
\le  \fppb_\xi(-\eit) - \fppb_\xi'((-\eit)+)h  - D\tsp h\tsp \abs{\xi}_1 . 
\qedhere   \]  
 \end{proof}

%
%
%

\smallskip 

\begin{corollary}\label{cor:lain2}  Assume $\eit\ge0$,  \eqref{pc-ass}, and  moment bound  \eqref{lin-ass5} with $p=d$.   There exists a constant $D>0$ such that 
$\lain (\xi)\ge (1+D)\abs{\xi}_1$ for all $\xi\in\R^d\setminus\{\zevec\}$.  
\end{corollary} 

\begin{proof} 
Fix  $0< a-b<a$ and let $D=D(a,b)$ from \eqref{nu-45}.  Then, for $\xi\ne\zevec$,  
\be\label{v98.76}   
\lain (\xi) = \fppb_\xi'(0+) \ge   \frac{\fppb_\xi(a)-\fppb_\xi(a-b)}b \ge \fppb_\xi'(a+)+D(a,b)\abs{\xi}_1\ge  (1+D)\abs{\xi}_1.  
\ee
The first equality is from the characterization of the superdifferential in \eqref{supd-1} if $\eit>0$ and in \eqref{supd-3} if $\eit=0$. The first inequality is
concavity and the second one is  \eqref{nu-45}.  The last inequality is the easy bound  from  \eqref{fppb62}.  
\end{proof}

\begin{proof}[Proof of Theorem \ref{thm:fppb4}]  
We prove the first inequality of  \eqref{fppb5.24}.  
For  $b\ge-\eit$  the characterizations of the superdifferentials in \eqref{supd-1} and \eqref{supd-3} give $\fppb_\xi'(b+)=\aalain{b}(\xi)$.  Corollary  \ref{cor:lain2} gives  constants  $D(b)>0$ such that $\aalain{b}(\xi)\ge(1+D(b))\abs{\xi}_1$. By the monotonicity of the derivatives, $D(b)=D(-\eit)$ works for $b<-\eit$.  To produce a nonincreasing function, replace $D(b)$ with $\inf_{-\eit\le a\le b}D(a)$.     

  The three middle inequalities of  \eqref{fppb5.24} are in \eqref{98.13} above.

  To prove  the rightmost bound of  \eqref{fppb5.24}, consider first  $b\in( -\eit-\eet, -\eit]$.  Take $a= (b-\eit-\eet)/2 \in( -\eit-\eet, b)$.   Let $\w\in\Omega_0$ and $x_n/n\to\xi$.   Concavity, \eqref{98.13},  and 
  \eqref{L-b88}   give   
 \[   \fppb_\xi'(b-) \le \fppb_\xi'(a+) \le    \varliminf_{n\to\infty} \frac{\underline L^{(a)}_{\tspb\zevec,\tspb x_n}(\w)}{n}   \le  \frac{c}{(a+\eit)\wedge 0+\eet}\,\abs{\xi}_1 = \frac{2c}{(b+\eit)\wedge 0+\eet}\,\abs{\xi}_1 . \]  
   The rightmost bound of  \eqref{fppb5.24}  extends to all  $b\ge-\eit$  because  $\fppb_\xi'(b-)$ is nonincreasing in $b$. 
 \end{proof}

\begin{proof}[Proof of Theorem \ref{thm:geod4}]   Using Proposition \ref{pr:g2}\ref{pr:g2-1}, the continuity of the shape functions $\gly$ and $\zgpp$  on $\inter\Uset$, and   $\lain (\xi)\ge (1+D)\abs{\xi}_1$ from Corollary \ref{cor:lain2} above,   choose constants $\eta, \delta>0$ small enough so that  for any $\abs{\xi}_1=1$, 
\be\label{sh7}  \abs{\,\gly(\xi) -\tau\zgpp(\xi/\tau) \,} \le \eta \ \ \implies \  \ \tau\ge 1+\delta.  \ee  
  From \eqref{L105} or \eqref{L-b88}  pick finite deterministic $\kappa$ and random $K$ such that 
\be\label{L105a} \overline L_{\tspb\zevec,x}\le \kappa\abs{x}_1 \quad
\text{ for all } \quad  \abs{x}_1\ge K.  
\ee 
Let $\alpha=\delta/4$. Increase $\kappa$ if necessary so that $\kappa>2+\alpha$.  Let   $0<\e<\eta/(1+\kappa)$. 
Increase $K$ if necessary so that  (i)  $K\ge 4/\delta$,  (ii)  $K$ works in \eqref{wgpp-sh}  for $\alpha, \e$, and (iii) $K$ satisfies  the FPP shape theorem  (\cite[p.~11]{Auf-Dam-Han-17}, also \eqref{shape-mu-b})
\be\label{fpp-sh}   \abs{\, T_{\zevec,x} -\gly(x)\,} \le \e\abs{x}_1   
\qquad
\text{for }   \abs{x}_1\ge K. 
\ee

Let $\abs{x}_1\ge K$ and let $\pi$ be a  geodesic   for $T_{\zevec, x}$.   Let $k=\abs{\pi}\vee\ce{ \tsp(1+\alpha)\abs{x}_1\tsp}$.   Then 
\[  T_{\zevec,x}=\zGpp_{\zevec,(\abs\pi),x}=\zGpp_{\zevec,(k),x}. \] 
A combination of \eqref{wgpp-sh} and \eqref{fpp-sh}, the homogeneity of $\gly$, and $k\le\kappa\abs{x}_1$   give 
\begin{align*}
&\abs{ \,\gly(x)\,-\, k\tspa\zgpp(x/k) \,} \le \e\abs{x}_1+\e k \le \e(1+\kappa)\abs x_1 \\[3pt] 
&\implies\quad \biggl\lvert  \,\gly\biggl(\frac{x}{\abs{x}_1}\biggr)\,-\, \frac{k}{\abs{x}_1}\tspa\zgpp\biggl(\frac{x/\abs{x}_1}{k/\abs{x}_1}\biggr) \,   \biggr\rvert \,\le\,   \e(1+\kappa)  < \eta. 
\end{align*}
Now \eqref{sh7} implies $k\ge (1+\delta)\abs{x}_1$.   On the other hand, $\abs{x}_1\ge K>4/\delta$ implies that 
\[  k = \abs{\pi}\vee\ce{ \tsp(1+\alpha)\abs{x}_1\tsp} \le  \abs{\pi}\vee (1+\delta/2)\abs{x}_1. \] 
Together these force $\abs\pi\ge (1+\delta)\abs{x}_1$.    
\end{proof}

\begin{proof}[Proof of Theorem \ref{thm:fpp10}]
Part \ref{thm:fpp10.i}. The statements about $\lain(\xi)$ come from Lemma \ref{lm:mu=g}.  The statements about $\lasu(\xi)$ come from the definition \eqref{lasu1} and Proposition \ref{pr:g2}\ref{pr:g2-2}.  The semicontinuity claims are in Lemma \ref{lm:sc}.  The finite-infinite dichotomy of $\lasu(\xi)$ is in \eqref{mgly53} and \eqref{mgly55}.

Part \ref{thm:fpp10.ii}.  To derive  \eqref{lainsu1}, combine  Corollary \ref{cor:lain2},  \eqref{mgly53}, \eqref{mgly55},    \eqref{98.13}, and the characterizations of the derivatives $\gly_\xi'(0\pm)$ from \eqref{supd-1} when $\eit>0$ and from \eqref{supd-3} when $\eit=0$.
\end{proof}

 \begin{proof}[Proof of Theorem \ref{thm:gdiff}]   
   Part \ref{thm:gdiff.i} was proved in Lemma \ref{lm:gpp9}.      Part \ref{thm:gdiff.ii} comes from  Proposition \ref{pr:g2}.



\medskip    
 
   Part \ref{thm:gdiff.iii}.  
 Begin by noting that differentiability of  $t\mapsto\wgpp(t\xi)$ is equivalent to differentiability of  $\tau\mapsto\tau\wgpp(\xi/\tau)$ and on an open interval a differentiable convex function is continuously differentiable. 
   
   Since $\lain(\xi)>\abs{\xi}_1$ and by the limit \eqref{fppb8}, the union of the superdifferentials on the right-hand sides of \eqref{supd-1} and \eqref{supd-2} is equal to the interval  $(\abs{\xi}_1,\infty)$.  
 General convex analysis gives the equivalence   
 \[    -b\in\partial_\tau[ \tau\tsp\gpp(\xi/\tau)] \ \iff \ \tau\in\partial\fppb_\xi(b). 
 \]
 By the strict concavity of $\fppb_\xi$, a given $\tau$  lies in  $\partial\fppb_\xi(b)$ for a unique $b$, and hence  the subdifferential  $\partial_\tau[ \tau\tsp\gpp(\xi/\tau)]$ consists of  a unique value $-b\in(-\infty, \eit]$. This implies that $\tau\mapsto\tau\tsp\gpp(\xi/\tau)$ is differentiable at $\tau\in(\abs{\xi}_1,\infty)$.  
 
 
 
 Continuous differentiability of $\tau\mapsto\tau\zgpp(\xi/\tau)$ for $\tau>\abs{\xi}_1$ now follows from Proposition \ref{pr:g2}. Namely,  $\tau\zgpp(\xi/\tau)=\tau\gpp(\xi/\tau)$ for $\tau\in[\,\abs{\xi}_1, \lain(\xi)\tspa]$, which we now know to be a nondegenerate interval, and   their common left $\tau$-derivative vanishes at the minimum $\tau=\lain(\xi)$.  On $[\tspa\lain(\xi), \infty)$, $\tau\zgpp(\xi/\tau)=\gly(\xi)$ is constant and hence connects in a $C^1$ fashion to the part on $[\tspb\abs{\xi}_1, \lain(\xi)\tspa]$.  
 
 If  $\wgpp(\xi/\abs{\xi}_1)=\infty$  then necessarily $\lim_{t\nearrow\abs{\xi}_1^{-1}} (\wgpp)'(t\xi)=+\infty$.  
 
 The remaining claims follow if we assume $\wgpp(\xi/\abs{\xi}_1)<\infty$  and show   that 
 \be\label{gpp84}
 \lim_{t\nearrow\abs{\xi}_1^{-1}} \frac{\wgpp(\abs{\xi}_1^{-1}\xi) - \wgpp(t\xi)}{\abs{\xi}_1^{-1}-t}=+\infty. 
 \ee
 It suffices to treat $\gpp$ since $\zgpp=\gpp$  close enough to the boundary of $\Uset$ by part \ref{thm:gdiff.ii}.   
 
 Take $\alpha=1/t>\abs{\xi}_1$ and rewrite the ratio above as 
 \begin{align*}
 \abs{\xi}_1 \gpp(\abs{\xi}_1^{-1}\xi) +  \abs{\xi}_1 \,\frac{\abs{\xi}_1 \gpp(\xi/\abs{\xi}_1) - \alpha\tsp\gpp(\xi/\alpha)}{\alpha-\abs{\xi}_1}. 
 \end{align*} 
Thus by the duality in Theorem \ref{thm:fppb3}, \eqref{gpp84}  is equivalent  to 
 \be\label{gpp87}  \lim_{\alpha\searrow\abs{\xi}_1} \frac{\fppc_\xi^*(\alpha) - \fppc_\xi^*(\abs{\xi}_1)}{\alpha-\abs{\xi}_1}
 =\infty.  \ee
By concavity, the ratio  in \eqref{gpp87} above is a nonincreasing function of $\alpha>\abs{\xi}_1$.    Hence if \eqref{gpp87} fails, there exists $b_0<\infty$ such that, $\forall\alpha>\abs{\xi}_1$ and $\forall b\ge b_0$, 
\[  \abs{\xi}_1b- \fppc_\xi^*(\abs{\xi}_1) \le  \alpha b- \fppc_\xi^*(\alpha).  \]
It then follows from the duality (\eqref{nustar9} or \eqref{fppb4}) that 
\[   \fppc_\xi(b)=\abs{\xi}_1b- \fppc_\xi^*(\abs{\xi}_1) \qquad \text{for } b\ge b_0. \]
This contradicts the strict concavity of $\fppc_\xi$.  \eqref{gpp84} has been verified. 
  \end{proof}

\subsection{Nondifferentiability}  \label{sec:pf-nd} 

\begin{proof}[Proof of Theorem \ref{thm:ndi3}]
Bound \eqref{ndi3.i.1}  is contained in Theorem \ref{thm:ndi}.  
\eqref{ndi3.i.1}   implies that, along any subsequence $\{n_i\}$, 
\[  \P\bigl( \,\overline L_{\tspb\zevec,x_{n_i}}-\underline  L_{\tspb\zevec,x_{n_i}} \ge D\abs{x_{n_i}}_1 \text{  for infinitely many $i$}\tspb \bigr) \ge \delta. \]  
Now  \eqref{98.13} implies $\fppb_\xi'(0-)-  \fppb_\xi'(0+)\ge D\abs{\xi}_1$. 
\end{proof}


 \begin{proof}[Proof of Theorem \ref{thm:ndi9}]
 Let  $r<s$ be two   atoms of $t(e)$ in $[\eit,\infty)$.    Fix an arbitrary $\ell\in\N$ and then pick $k\in\N$ so that 
 \[   \frac{(k-1)(s-r)}{2\ell} \le r-\eit < \frac{k(s-r)}{2\ell}. \]
%
 For $m\in\Z_+$ let 
 \[   b_m 
 =\frac{(k+m)(s-r)}{2\ell}-r  \;\in\; (-\eit,\infty).  \]   
 Then $b_m+r$ and $b_m+s$ are atoms of $t^{(b_m)}(e)$ such that 
 \[  (k+m)(s+b_m)= (k+m+2\ell)(r+b_m) \qquad \text{for all $m\in\Z_+$.} \]
 The other hypotheses of Theorem \ref{thm:ndi3} are inherited by $\w^{(b_m)}$ and so  the conclusions of Theorem \ref{thm:ndi3} hold for all $\w^{(b_m)}$.  In particular,  since $\fppb_\xi^{(b_m)}(a)=\fppb_\xi(a+b_m)$, $\fppb_\xi^{(b_m)}$ has a corner at $0$ if and only if $\fppb_\xi$ has a corner at $b_m$.  
 
    No point of $[-\eit,\infty)$ is farther than  $\frac{s-r}{2\ell}$ from the nearest $b_m$.    We get the dense set $B$ by   combining  the collections $\{b_m\}$ for all   $\ell\in\N$. 
\end{proof}

\appendix 


\section{First-passage percolation with slightly negative weights} \label{a:fpp<0} 

This appendix extends the shape theorem  of standard FPP to real-valued weights $\{t(e)\}$  under certain hypotheses. 
 The setting is the same as in Section \ref{sec:set}. 
As before,  
$\{t_i\}$ denotes  i.i.d.\ copies of the edge weight $t(e)$.  Assumption \eqref{lin-ass5} is reformulated for positive parts as 
\be\label{lin-ass5+} 
\E[ \, (\min\{t^+_1,\dotsc, t^+_{2d}\})^p\,] <\infty . 
\ee
  Passage times $T_{x,y}$ are defined as in \eqref{def-fpp.2} and now the restriction to self-avoiding paths is essential.

\begin{theorem}\label{thm:mu-b}   Assume   $\eit=\essinf t(e)\ge 0$,  \eqref{pc-ass}, and \eqref{lin-ass5+} {\rm(}equivalently, \eqref{lin-ass5}{\rm)} 
with $p=d$. 
Then there exist 
\begin{enumerate} [label=\rm(\alph{*}), ref=\rm(\alph{*})]
\item\label{thm:mu-b:a}   a constant $\eet>0$ determined by the distribution of the shifted weights $\w^{(-\eit)}$, 
\item\label{thm:mu-b:b}  for each real $b>-\eit-\eet$,  a positively homogeneous continuous convex function $\gly^{(b)}:\R^d\to\R_+$, and 
\item\label{thm:mu-b:c}  an event $\Omega_0$ of full probability, 
\end{enumerate} 
such that    the properties listed below  in points {\rm(i)--(iii)} are satisfied.  
 \begin{enumerate} [label=\rm(\roman{*}), ref=\rm(\roman{*})] \itemsep=3pt
\item\label{thm:mu-b:i} For each $\w\in\Omega_0$ and $b>-\eit-\eet$ the following pointwise statements hold.  For each $x\in\Z^d$,   $T^{(b)}_{\zevec,x}$ is finite and has a geodesic, that is, a self-avoiding path $\pi$ from $\zevec$ to $x$ such that $T^{(b)}_{\zevec,x}=\tpath^{(b)}(\pi)$.   	   There exist a deterministic finite constant $c$ and an $\w$-dependent finite  constant $K=K(\w)$ such that 
\be\label{L-b88}   
  \overline L^{(b)}_{\zevec, x} \le \frac{c}{\eet+(\eit+b)\wedge 0} \,\abs{x}_1 
\qquad  \text{whenever } \  \  
\abs{x}_1\ge K.  
\ee
 The shape theorem holds, locally uniformly in the shift $b$:  for any $a_0<a_1$ in $(-\eit-\eet, \infty)$, 
\begin{align}\label{shape-mu-b}
	\lim_{n\to\infty}\;\sup_{\abs{x}_1\ge n} \,\sup_{b\tspb\in\tspb[a_0, a_1]} \,\frac{\abs{\tspb T^{(b)}_{\zevec,x}-\gly^{(b)}(x)\tspb}}{\abs{x}_1}=0.
	\end{align}

	\item\label{thm:mu-b:ii}    For each   $b>-\eit-\eet$  the following  statements hold.   $T^{(b)}_{\zevec,x}\in L^1(\P)$  for all $x\in\Z^d$.  For any sequence $x_n\in\Z^d$ with $x_n/n\to\xi\in\R^d$,  the convergence 
 $n^{-1}\tpath^{(b)}_{\zevec,x_n}\to\gly^{(b)}(\xi)$ holds  almost surely and in $L^1(\P)$.

\item\label{thm:mu-b:iii}  The shape function satisfies these  Lipschitz bounds   for shifts   
$b_2>b_1>-\eit-\eet$ and all $\xi\in\R^d$:  
\be\label{L-b89}   
  \gly^{(b_1)}(\xi)  \le   \gly^{(b_2)}(\xi)  \le \gly^{(b_1)}(\xi) + \frac{c\tsp\abs{\xi}_1}{\eet+(\eit+b_1)\wedge 0} \,(b_2-b_1).  
\ee
For $b>-\eit-\eet$,   $\gly^{(b)}(\zevec)=0$ and  $\gly^{(b)}(\xi)>0$ for all $\xi\ne0$.  

\end{enumerate} 
\end{theorem}

We prove this theorem at the end of the section after proving a more general shape result in Theorem \ref{thm:mu1} below.

\begin{lemma}\label{A/n}  Let $\P$ be a probability measure  invariant under a group $\{\theta_x\}_{x\tspa\in\tspa\Z^d}$ of measurable  bijections.  
Let $A$ be a nonnegative random variable such that  $\E[A^d]<\infty$. Then
	\begin{align}\label{eqn:A/m}
	\lim_{m\to\infty}m^{-1}\max_{\abs{x}_1\le m} A\circ \theta_x=0
	\quad\text{ with probability one.} 
	\end{align}
\end{lemma}

\begin{proof}  The conclusion is equivalent to $\abs{x}_1^{-1} A\circ\theta_x\to 0$  as $\abs{x}_1\to\infty$.    
Apply Borel-Cantelli  with the estimate below for  $\e>0$: 
\begin{align*}
\sum_x \P\{ A\circ\theta_x \ge \e\abs{x}_1\} &=\sum_{k=0}^\infty \sum_{\abs{x}_1=k} \P\{ A\circ\theta_x \ge k\e\}  \le  1+  C(d) \sum_{k=1}^\infty k^{d-1}   \P\{ A \ge k\e\} \\
&\le 1+ C(d, \e)  \, \E[A^d]<\infty.  
\qedhere\end{align*}
 \end{proof}
Because the inequalities in the proof can be reversed with different constants, an i.i.d.\ example shows that $p<d$ moments does not suffice for the conclusion.   

Let  $x^-=(-x)\vee 0$ denote the negative part of a real number.  Following \cite{Smy-Wie-78}, 
define the random variable 
 	\begin{align}\label{condA}
	A=2 \sup_{x\tspa\in\tspa\Z^d}\tpath_{\zevec,x}^{\tspa-}   
	\end{align}
	
We first prove a moment bound for the shifts of $A$ that was used in the concavity result of Section \ref{sec:concave}.  
	
\begin{lemma}\label{lm:A6} 
Assume $\eit\ge0$ and  the subcriticality assumption \eqref{pc-ass}. 
Let   $\delta>0$ be the constant  in the bound  \eqref{kesten5} for the shifted weights $\w^{(-\eit)}$.   
Then there exists $s>0$ such that $\E[e^{s A^{(b)}}]<\infty$ for all shifts $A^{(b)}=2\sup_{x\tspa\in\tspa\Z^d}\bigl(\tpath_{\zevec,x}^{(b)}\bigr)^-$ such that  $b\ge-\eit-\delta$.  
\end{lemma} 	

\begin{proof} 
 By monotonicity it is enough to consider the case $b=-\eit-\delta$.
The proof is the same as that of the corollary of Theorem 3 in \cite{Kes-80}.    
 $A^{(-\eit-\delta)}\ge a>0$ implies the existence of 
 a self-avoiding path $\gamma$ from $\zevec$    such that 
$\tpath^{(-\eit-\delta)}(\gamma)<-a/4$. Turn this into 
	\[
	-\delta\abs\gamma\le \tpath^{(-\eit)}(\gamma)-\delta\abs\gamma
	=\tpath^{(-\eit-\delta)}(\gamma)<-a/4<0.
	\]
Then  $\abs\gamma>a/(4\delta)$ and   \eqref{kesten5} gives the bound 
	\begin{align*}
	\P\{A^{(-\eit-\delta)}\ge a\}
	&\le\P\bigl\{\text{$\exists$ self-avoiding path $\gamma$ from the origin} \\
 &\qquad \qquad 
 \text{such that  $\abs\gamma\ge {a}/{(4\delta)}$ and $\tpath^{(-\eit)}(\gamma)\le \delta\abs\gamma$}\bigr\} \le Ce^{-c_1a/(4\delta)}. 
 \qedhere 
	\end{align*}
\end{proof} 

\medskip 

The next item is a shape theorem whose hypotheses are stated in terms of the random variable $A$ of \eqref{condA}. 

\begin{theorem}\label{thm:mu1}   Let $\w=(t(e): e\in\cE_d)$ be i.i.d.\ real-valued weights. 


{\rm (i)}  Assume \eqref{lin-ass5+} with $p=1$ and that the random variable from \eqref{condA} satisfies $A\in L^1$. 
Then   $T_{x,y}$ is a  finite integrable random variable  for all $x,y\in\Z^d$. 
There exists 
a  non-random positively homogeneous continuous convex function $\gly:\R^d\to\R_+$ such that 
for any sequence $\{x_n\}\subset\Z^d$ with $x_n/n\to\xi\in\R^d$,  
\begin{align}\label{shape-xi-L1}
\lim_{n\to\infty}\E[\tspb\abs{n^{-1}\tspa\tpath_{\zevec,x_n}-\gly(\xi)}\tspb]=0.
\end{align}

{\rm(ii)}    Assume furthermore   \eqref{lin-ass5+} with $p=d$ and 
$\E[A^{d}]<\infty$. 
	 Then the following hold with probability one:
	\begin{align}\label{shape-mu}
	\lim_{n\to\infty}\sup_{\abs{x}_1\ge n}\frac{\abs{\tpath_{\zevec,x}-\gly(x)}}{\abs{x}_1}=0
	\end{align}
and for all $\xi\in\R^d$ and any sequence $x_n\in\Z^d$ such that  $x_n/n\to\xi$ 
\begin{align}\label{shape-xi}
\gly(\xi)=\lim_{n\to\infty}\frac{\tpath_{\zevec,x_n}}{n}\,.
\end{align}
\end{theorem}

\begin{proof}
Let $A_x=A\circ\theta_x$. Consider two paths  $\pi_{x,y}\in\saPaths_{x,y}$ and $\pi_{y,z}\in\saPaths_{y,z}$.
Their   concatenation   may fail to  be self-avoiding. Choose  a point $u$ belonging to both paths 
such that erasing the portion of $\pi_{x,y}$ from $u$ to $y$ (denoted by $\pi'_{u,y}$) and erasing the portion of $\pi_{y,z}$ from $y$ to $u$ (denoted by $\pi''_{y,u}$) leaves a self-avoiding path $\pi_{x,z}$  from $x$ to $z$.
(If the concatenation was self-avoiding to begin with, then $u=y$.)
Note that $\pi'_{u,y}$ and $\pi''_{y,u}$ are self-avoiding paths. This implies that
	\begin{align*}
	\tpath(\pi_{x,y})+\tpath(\pi_{y,z})&= \tpath(\pi_{x,z})+\tpath(\pi'_{u,y})+\tpath(\pi''_{y,u})\ge \tpath_{x,z}+\tpath_{u,y}+\tpath_{y,u}\\
	&\ge \tpath_{x,z}-T^-_{u,y}-T^-_{y,u}\ge \tpath_{x,z}-A_y.
	\end{align*}  
Taking infimum over $\pi_{x,y}$ and $\pi_{y,z}$ gives
	$\tpath_{x,y}+\tpath_{y,z}\ge \tpath_{x,z}-A_y.$
Rearranging, we get
	\begin{align}\label{A-subadd}
	0\le\tpath_{x,z}+A_z\le \tpath_{x,y}+A_y+\tpath_{y,z}+A_z.
	\end{align}
To apply the subadditive ergodic theorem, we derive a moment bound. 	

Let 
$\w^+=(t(e)^+:e\in\cE_d)$. Take any $\ell^1$-path $x_{\parng{0}{k}}$ from $\zevec$ to $x$ (where $k=\abs{x}_1$) and use the subadditivity of the passage times in weights $\w^+$ to write
	\[\tpath_{\zevec,x}(\w^+)\le\sum_{i=0}^{k-1}\tpath_{x_i,x_{i+1}}(\w^+).\]
Since $\E[\tpath_{\zevec,\pm\evec_i}(\w^+)]$ are all identical, 
	\begin{align}\label{T<Tabs}
	\E[\tpath_{\zevec,x}(\w)]\le \E[\tpath_{\zevec,x}(\w^+)]\le \E[\tpath_{\zevec,\evec_1}(\w^+)]\,\abs{x}_1.
	\end{align}
Assumption \eqref{lin-ass5+} with $p=1$  implies that 
$\E[\tpath_{\zevec,\evec_1}(\w^+)]<\infty$ (Lemma 2.3 in \cite{Auf-Dam-Han-17}). By the assumption $A\in L^1$, 
	\[\E[\tpath_{\zevec,x}+A_x]
	\le C\abs{x}_1+\E[A] <\infty.\] 
Standard  subadditivity arguments 
give the existence of a positively homogeneous convex function $\overline\gly:\Q^d\to\R_+$ such that for all
$\zeta\in\Q^d$ and $\ell\in\N$ with $\ell\zeta\in\Z^d$, 
almost surely and in $L^1$, 
\begin{align}\label{mubar-rat}
\overline\gly(\zeta)
=\lim_{n\to\infty}\frac{\tpath_{\zevec,n\ell\zeta}+A_{n\ell\zeta}}{n\ell}
=\lim_{n\to\infty}\frac{\tpath_{\zevec,n\ell\zeta}}{n\ell}\in[0,C\abs{\zeta}_1],
\end{align}
and $\overline\gly(\zeta)$ does not depend on the choice of $\ell$.  The assumption $A\in L^1$ allows us to drop the term $A_{n\ell\zeta}$ from above.  The first inequality of \eqref{A-subadd} gives $\overline\gly(\zeta)\ge0$. 



Fix $x,y\in\Z^d$. Use subadditivity \eqref{A-subadd} to write
	\begin{align*}
	\tpath_{\zevec,x}-\tpath_{\zevec,y}
	&=\tpath_{\zevec,x}+A_x+\tpath_{x,y}+A_y-\tpath_{\zevec,y}-A_y-\tpath_{x,y}-A_x\\
	&\ge-\tpath_{x,y}-A_x
	\ge-\tpath_{x,y}(\w^+)-A_x.
	\end{align*}
Switching $x$ and $y$ gives a complementary  bound and so 
	\begin{align}\label{A-subadd2}
	\abs{\tpath_{\zevec,x}-\tpath_{\zevec,y}}\le\tpath_{x,y}(\w^+)+A_x+A_y.
	\end{align}
By  \eqref{T<Tabs} 
	\begin{align}\label{E[T-diff]}
	\E[\tspb\abs{\tpath_{\zevec,x}-\tpath_{\zevec,y}}\tspb]\le
	C\abs{x-y}_1+2\E[A].
	\end{align}

Now take $\zeta,\eta\in\Q^d$ and $\ell\in\N$ such that $\ell\zeta$ and $\ell\eta$ are both in $\Z^d$ 
and apply the above to get
	\[\abs{\E[\tpath_{\zevec,n\ell\zeta}]-\E[\tpath_{\zevec,n\ell\eta}]}\le Cn\ell\abs{\zeta-\eta}_1+2\E[A].\]
Divide by $n\ell$ and take $n$ to $\infty$ to get
		\begin{align}\label{eq:Lip}
		\abs{\overline\gly(\zeta)-\overline\gly(\eta)}\le C\abs{\zeta-\eta}_1.
		\end{align}
As a Lipschitz function $\overline\gly$   extends  uniquely to a continuous positively homogenous convex function $\gly:\R^d\to\R$. 


Fix $\xi\in\R^d\setminus\{\zevec\}$ and a sequence $x_n$ in $\Z^d$ such that $x_n/n\to\xi$. 
Fix $\zeta\in\Q^d$. Take $\ell\in\N$ such that $\ell\zeta\in\Z^d$. For  $n\in\N$ let $m_n=\fl{n/\ell}$.  
By \eqref{E[T-diff]}, 
	\begin{align*}
	\E[\abs{n^{-1}\tpath_{\zevec,x_n}-\gly(\xi)}]
	&\le n^{-1}\E[\tspb\abs{\tpath_{\zevec,x_n}-\tpath_{\zevec,m_n\ell\zeta}}\tspb]
	+ \E[\abs{n^{-1}\tpath_{\zevec,m_n\ell\zeta}-\overline\gly(\zeta)}]
	+\abs{\overline\gly(\zeta)-\gly(\xi)}\\
   	&\le n^{-1}C\abs{x_n-m_n\ell\zeta}_1+2n^{-1}\E[A]
	+\E[\abs{n^{-1}\tpath_{\zevec,m_n\ell\zeta}-\overline\gly(\zeta)}]
	+\abs{\overline\gly(\zeta)-\gly(\xi)}.
	\end{align*}
Take $n\to\infty$ to get 
	\[\varlimsup_{n\to\infty}n^{-1}\E[\abs{\tpath_{\zevec,x_n}-\gly(\xi)}]\le C\abs{\xi-\zeta}_1+\abs{\overline\gly(\zeta)-\gly(\xi)}.\]
Let $\zeta\to\xi$ to get \eqref{shape-xi-L1}.   This completes the proof of part (i). 


\medskip 

Now strengthen the assumptions to  $\E[A^{d}]<\infty$ and \eqref{lin-ass5+} with $p=d$.  
For \eqref{shape-mu} we follow the proof of the  Cox-Durrett shape theorem presented in \cite[Section 2.3]{Auf-Dam-Han-17}. 

Let $\Omega_0$ be the full probability event on which \eqref{mubar-rat} holds for all $\zeta\in\Q^d$. 
By Lemma 2.22 and Claim 1 on p.~22 of  \cite{Auf-Dam-Han-17},   under \eqref{lin-ass5+} there exists a finite positive constant $\kappa$ and a full probability event $\Omega_1$ such that for any $\w\in\Omega_1$
and $y\in\Z^d\setminus\{\zevec\}$, there exists a strictly increasing random sequence $m(n)\in\N$ such that $m(n+1)/m(n)\to1$ as $n\to\infty$ and 
	\begin{align}\label{fifo}
	\tpath_{m(n) y,z}(1+\w^+)\le \kappa\tsp\abs{m(n) y-z}_1\quad\text{for all }z\in\Z^d\text{ and }n\in\N.
	\end{align}
Here, $\tpath_{x,y}(1+\w^+)$ is the first-passage time from $x$ to $y$ under the  weights $1+\w^+=(1+t(e)^+:e\in\cE_d)$.  The results from \cite{Auf-Dam-Han-17} apply because these weights are strictly positive and satisfy \eqref{lin-ass5+} with $p=d$. 

Let $\Omega_2$ be the full probability event on which \eqref{eqn:A/m} holds for the random variable $A$ of \eqref{condA}. 
We show that  \eqref{shape-mu} holds for  each  fixed $\w\in\Omega_0\cap\Omega_1\cap\Omega_2$. 
Let $x_k\in\Z^d$ be an $\w$-dependent  sequence such that $\abs{x_k}_1\to\infty$ and
	\begin{align}\label{temp-shape}
	\lim_{k\to\infty}\frac{\abs{\tpath_{\zevec,x_k}-\gly(x_k)}}{\abs{x_k}_1}=\varlimsup_{n\to\infty}\sup_{\abs{x}_1\ge n}\frac{\abs{\tpath_{\zevec,x}-\gly(x)}}{\abs{x}_1}\,.
	\end{align}
By passing to a subsequence we can assume   $x_k/\abs{x_k}_1\to \xi\in\R^d$ with $\abs{\xi}_1=1$.   Let  $\zeta\in\Q^d$ satisfy   $\abs{\zeta}_1=1$ and pick $\ell\in\N$ such that $\ell\zeta\in\Z^d$. 
Choose $m(n)$ in \eqref{fifo} for  $y=\ell\zeta$.  
For each $k\in\N$ take $n_k\in\N$ such that 
	\begin{align}\label{m-x}
	m(n_k)\ell\le\abs{x_k}_1\le m(n_k+1)\ell.
	\end{align}
Abbreviate $m_k=m(n_k)$. There exists $\w$-dependent  $k_0$ such that for all $k\ge k_0$, $m(n_k+1)\le2m_k$. 
Triangle inequality: 
	\begin{align*}
	\Bigl|\frac{\tpath_{\zevec,x_k}}{\abs{x_k}_1}-\gly(\xi)\Bigr|
	&\le	
	\frac{\abs{\tpath_{\zevec,x_k}-\tpath_{\zevec,m_k\ell\zeta}}}{\abs{x_k}_1}
	+\frac{m_k\ell}{\abs{x_k}_1}\cdot\Bigl|\frac{\tpath_{\zevec,m_k\ell\zeta}}{m_k\ell}-\gly(\zeta)\Bigr|\\
	&\qquad+\Bigl|\frac{m_k\ell}{\abs{x_k}_1}-1\Bigr|\cdot\abs{\gly(\zeta)}
	+\abs{\gly(\zeta)-\gly(\xi)}.
	\end{align*}
Use \eqref{A-subadd2}, \eqref{fifo} applied to $y=\ell\zeta$, 
and take $k\ge k_0$: 
	\begin{align*}
	\Bigl|\frac{\tpath_{\zevec,x_k}}{\abs{x_k}_1}-\gly(\xi)\Bigr|
	&\le	
	\frac{\tpath_{m_k\ell\zeta,x_k}(1+\w^+)+A_{m_k\ell\zeta}+A_{x_k}}{\abs{x_k}_1}
	+\frac{m_k\ell}{\abs{x_k}_1}\cdot\Bigl|\frac{\tpath_{\zevec,m_k\ell\zeta}}{m_k\ell}-\gly(\zeta)\Bigr|\\
	&\qquad+\Bigl|\frac{m_k\ell}{\abs{x_k}_1}-1\Bigr|\cdot\abs{\gly(\zeta)}
	+\abs{\gly(\zeta)-\gly(\xi)}\\[5pt]
	&\le	
	\frac{\kappa\abs{m_k\ell\zeta-x_k}_1}{\abs{x_k}_1}
	+\frac{2\max_{\abs{x}_1\le2m_k\ell}A_x}{m_k\ell}
	+\frac{m_k\ell}{\abs{x_k}_1}\cdot\Bigl|\frac{\tpath_{\zevec,m_k\ell\zeta}}{m_k\ell}-\gly(\zeta)\Bigr|\\
	&\qquad+\Bigl|\frac{m_k\ell}{\abs{x_k}_1}-1\Bigr|\cdot\abs{\gly(\zeta)}
	+\abs{\gly(\zeta)-\gly(\xi)}.
	\end{align*}
As $k\to\infty$ the   right-hand side converges to  $\kappa\abs{\zeta-\xi}+\abs{\gly(\zeta)-\gly(\xi)}$.   Letting $\zeta\to\xi$ then proves that $\tpath_{\zevec,x_k}/\abs{x_k}_1\to\gly(\xi)$ as $k\to\infty$. 
Since $\gly$ is continuous and homogeneous, we also have $\gly(x_k)/\abs{x_k}_1=\gly(x_k/\abs{x_k}_1)\to\gly(\xi)$.  Now \eqref{shape-mu} follows from
\eqref{temp-shape}.

\eqref{shape-xi} follows from \eqref{shape-mu} and the continuity and homogeneity of $\gly$.
\end{proof}

\begin{remark}
In the last inequality of the proof above    $(m_k\ell)^{-1}\max_{\abs{x}_1\le2m_k\ell}A_x$  can be replaced  by a smaller term as follows.
First, fix a rational $\e>0$. Take $k_0$ to be large enough so that for all $k\ge k_0$, $m(n_k+1)\le2m_k$, as before, but also
	\[\Bigl|\frac{m_k\ell}{\abs{x_k}_1}-1\Bigr|\le\frac{\e}{3\ell}\quad\text{and}\quad\Bigl|\frac{x_k}{\abs{x_k}}-\xi\Bigr|_1\le\frac{\e}{3\ell}\,.\]
Take $\zeta\in\Q^d$ (still with $\abs{\zeta}_1=1$) so that $\abs{\zeta-\xi}_1\le\e/(3\ell)$. Now we have
	\[\frac{\abs{x_k-m_k\ell\zeta}_1}{\abs{x_k}_1}\le\Bigl|\frac{m_k\ell}{\abs{x_k}_1}-1\Bigr|+\abs{\zeta-\xi}_1+\Bigl|\frac{x_k}{\abs{x_k}}-\xi\Bigr|_1\le\frac{\e}\ell\,.\]
Consequently, 
	\[\abs{x_k-m_k\ell\zeta}_1\le\e\ell^{-1}\abs{x_k}_1\le\e m(n_k+1)\le2\e m_k.\]
Thus, instead of \eqref{eqn:A/m} one now needs the hypothesis  that for any fixed $z\in\Z^d$, 
	\[\lim_{\e\searrow0}\;\varlimsup_{m\to\infty}\;m^{-1}\max_{x:\,\abs{x-mz}_1\le\e m}A\circ\theta_x=0\quad\text{almost surely}.\]
In our application to the proof of Theorem \ref{thm:mu-b} the random variable $A$ has all moments, so we do not  pursue  sharper  assumptions on $A$ than those stated in Theorem \ref{thm:mu1}.
\qedex\end{remark}


\begin{proof}[Proof of Theorem \ref{thm:mu-b}]  

The constant  in point  \ref{thm:mu-b:a}   is taken to be  $\eet=\delta=$ the constant  in  the bound \eqref{kesten5} for the shifted weights $\w^{(-\eit)}$. 
Then by Lemma \ref{lm:A6},   $A^{(b)}$ 
has all moments  for all $b>-\eit-\eet$.  Hence we can apply   Theorem \ref{thm:mu1} to the shifted weights $\w^{(b)}$ to define  the shape functions $\gly^{(b)}$ whose existence is asserted in   point  \ref{thm:mu-b:b}.
We specify the full-probability event $\Omega_0$ for point \ref{thm:mu-b:c} in the course  of the proof. 
 
\medskip

Proof of part \ref{thm:mu-b:i}. Start with the obvious point that $T^{(b)}_{\zevec, x}\le  \tpath^{(b)}(\wt\gamma)<\infty$ for any particular self-avoiding  path $\wt\gamma$ from $\zevec$ to $x$.  
By bound \eqref{kesten5} there exists a full probability  event $\Omega_0$  and a finite  random variable $K(\w)$ such that on the event $\Omega_0$,   every  self-avoiding path $\gamma$ from the origin such that $\abs\gamma\ge K$ satisfies the bound  $\tpath^{(-\eit)}(\gamma)> \eet\abs\gamma$.  
 Then for any shift $b$ these paths satisfy 
\be\label{a:770}   \tpath^{(b)}(\gamma) =  \tpath^{(-\eit)}(\gamma) + (\eit+b)\abs\gamma >    (\eet+\eit+b)\abs\gamma. \ee
From this we conclude that, for any $x\in\Z^d$, $b>-\eit-\eet$, and any path $\gamma$,  
\be\label{a:780}   \abs\gamma\;\ge\; K \vee \frac{T^{(b)}_{\zevec, x}}{\eet+\eit+b}  
\qquad\text{implies}\qquad   \tpath^{(b)}(\gamma)>   T^{(b)}_{\zevec, x}. 
\ee 
Thus   the infimum that defines $T^{(b)}_{\zevec, x}$  in \eqref{def-fpp.2}  cannot be taken outside a certain $\w$-dependent  finite set of paths.   Consequently on the event $\Omega_0$  a  minimizing path exists and both $T^{(b)}_{\zevec, x}$ and  $\overline L^{(b)}_{\zevec, x}$ are finite for all $x\in\Z^d$ and $b>-\eit-\eet$. 

Next, shrink the event $\Omega_0$ (if needed) so that  for $\w\in\Omega_0$ the shape theorem  \eqref{shape-mu} is valid for the weights $\w^{(-\eit)}$. 
Then  we can  increase $K$ and pick a deterministic positive constant $c$  so that 
$T^{(-\eit)}_{\zevec, x} \le c\abs{x}_1$ whenever $\abs{x}_1\ge K$.  By monotonicity  $T^{(b)}_{\zevec, x} \le c\abs{x}_1$ for all $b\le -\eit$ whenever $\abs{x}_1\ge K$.  If necessary increase $c$ so that $c\ge\eet$.    Then by \eqref{a:780},   when $\abs{x}_1\ge K$ and $b\in(-\eit-\eet, -\eit]$,  a self-avoiding path $\gamma$ between $\zevec$ and $x$ that satisfies 
\[  \abs\gamma\ge  \frac{c\abs{x}_1}{\eet+\eit+b} 
\]
  cannot be a  geodesic  for $T^{(b)}_{\zevec, x}$.  We conclude that   for $\w\in\Omega_0$, 
\[   \overline L^{(b)}_{\zevec, x} \le \frac{c}{\eet+\eit+b} \abs{x}_1 
\qquad  \text{whenever } \  b\in(-\eit-\eet, -\eit] \ \text{ and } \ \abs{x}_1\ge K.  
\]
Since $\overline L^{(b)}_{\zevec, x} $ is nonincreasing in $b$ (Remark \ref{rm:geod3}(iii)),  we can extend the  bound above to all $b\ge-\eit$ in the form  \eqref{L-b88}.  

By taking advantage of \eqref{L-b88} now proved,  we get these Lipschitz bounds:  for all $\w\in\Omega_0$,  $b_2>b_1>-\eit-\eet$ and $\abs{x}_1\ge K$,  and with $\pi^{(b)}_{\zevec,x}$ denoting a geodesic of $T^{(b)}_{\zevec, x}$, 
\be\label{a:820} \begin{aligned}
T^{(b_1)}_{\zevec, x}&\le T^{(b_2)}_{\zevec, x} \le T^{(b_2)}(\pi^{(b_1)}_{\zevec,x}) 
=   T^{(b_1)}(\pi^{(b_1)}_{\zevec,x}) + (b_2-b_1) \abs{\pi^{(b_1)}_{\zevec,x}}  \\[3pt]  &
\le  T^{(b_1)}_{\zevec, x}  +   \frac{c\tspb (b_2-b_1)\abs{x}_1}{\eet+(\eit+b_1)\wedge 0}
\equiv  T^{(b_1)}_{\zevec, x}  +   \kappa(b_1) (b_2-b_1)\abs{x}_1 . 
\end{aligned}\ee
The last equality defines the constant $\kappa(b)$ which is nonincreasing in $b$. 

We establish the locally uniform shape theorem \eqref{shape-mu-b}.   Let $B$ be a countable dense subset of $(-\eit-\eet, \infty)$.  
Shrink the event $\Omega_0$ further  so that  for $\w\in\Omega_0$ the shape theorem \eqref{shape-mu}   holds for the shifted weights $\w^{(b)}$ for all $b\in B$.  By passing to the limit, \eqref{a:820} gives  the macroscopic Lipschitz bounds  \eqref{L-b89} for shifts $b_1<b_2$ in this countable dense set $B$.  

Let $\w\in\Omega_0$, $\e>0$, and $a_0<a_1$ in $B$.  
Pick a partition $a_0=b_0<b_1<\dotsm<b_m=a_1$ so that each $b_i\in B$ and $\kappa(a_0)(b_i-b_{i-1})<\e/2$. 
Fix a constant $K_0=K_0^{b_0,b_1,\dotsc,b_m}(\w)$ such that 
\[  \abs{ \,T^{(b_i)}_{\zevec, x} - \gly^{(b_i)}(x)\,} \le \e\abs x_1/2 
\qquad \text{for $i=0,1,\dotsc,m$ whenever }   \ \abs{x}_1\ge K_0.  \]
Now for $i\in[m]$, $b\in[b_{i-1}, b_i]$, and $\abs{x}_1\ge K_0$, utilizing the monotonicity in $b$ of $T^{(b)}_{\zevec, x}$ and $\gly^{(b)}(x)$ and the Lipschitz bounds  \eqref{a:820} and \eqref{L-b89}, 
\begin{align*}
\abs{ \,T^{(b)}_{\zevec, x} - \gly^{(b)}(x)} \le \abs{ \,T^{(b_i)}_{\zevec, x} - \gly^{(b_i)}(x)\tspb} + \kappa(a_0)(b_i-b_{i-1})\abs x_1 \le \e\abs x_1. 
\end{align*} 
The shape theorem   \eqref{shape-mu-b} has been proved. 

 \medskip

Proof of part \ref{thm:mu-b:ii}.  The integrability and $L^1$ convergence follow from Theorem \ref{thm:mu1}(i).   The almost sure convergence comes from the homogeneity and continuity of $\gly^{(b)}$ and the shape  theorem   \eqref{shape-mu-b}.

 \medskip

Proof of part \ref{thm:mu-b:iii}.   We already established \eqref{L-b89} for a dense set of shifts $b_1<b_2$.  Monotonicity of $b\mapsto \gly^{(b)}(\xi)$ extends \eqref{L-b89}  to all shifts $b$. 

That $\gly^{(b)}(\zevec)=0$ follows from homogeneity. 
The final claim   that $\gly^{(b)}(\xi)>0$ for $b>-\eit-\eet$ and $\xi\ne\zevec$ follows from  \eqref{a:770},  which  implies $T^{(b)}_{\zevec, x}\ge  (\eet+\eit+b)\abs x_1$ whenever $\abs x_1\ge K$. 
\end{proof}

\medskip 




\section{Restricted path length shape theorem} \label{a:gpp} 

This section proves the next shape theorem  in the interior of $\Uset$ for the restricted path length FPP processes defined in \eqref{wGdef}.   As throughout,  the  edge weights $\{t(e):e\in\cE_d\}$ are  independent and  identically distributed  (i.i.d.) real-valued random variables, $\eit=\essinf t(e)$,  the set $\wDset_\ell$ of points reachable by $\ell$-paths is defined by \eqref{wDset1}, and $\Uset=\{\xi\in\R^d:  \abs{\xi}_1\le 1\}$ is the $\ell^1$ unit ball.    We also write  $\{t_i\}$ for  i.i.d.\ copies of the edge weight $t(e)$.


\begin{theorem}\label{thm:shape-G2}
Assume  $\eit>-\infty$ 
and moment assumption  \eqref{lin-ass5+}   with $p=d$. 
Fix $\wild\in\{\eee,o\}$. There exists a deterministic,  continuous, convex shape function $\wgpp:\inter\Uset\to[\eit\wedge0, \infty)$ that satisfies the following:  for each $\alpha, \e>0$ there exists an almost-surely  finite random constant $K(\alpha, \e)$ 
such  that
\be\label{wgpp-sh} \abs{\, \wGpp_{\zevec,(k),x} -k\tspa\wgpp(x/k)\,} \le \e k  
\ee
whenever $ k\ge K(\alpha, \e)$,   $k\ge (1+\alpha)\abs{x}_1$, and  $x\in\wDset_k$.   
\end{theorem}


The shape theorem can be proved all the way to the boundary of $\Uset$.  This requires (i) stronger moment bounds that vary with the dimension of each boundary face and (ii) further technical constructions beyond what is done in the proof below,  because there are fewer paths to the boundary than to interior points.  We have no need for  the shape theorem on all  of $\Uset$ in the present paper. Our purposes are met by extending the shape function from the interior  to the boundary through radial limits (Theorem \ref{thm:gpp-ext} and Lemma \ref{lm:g1}).

\medskip

We begin with a basic   tail bound on   $\wGpp_{\zevec,(\ell),x}$.  
 
\begin{lemma}\label{lm:G<infty}  Assume the weights are arbitrary real-valued i.i.d.\ random variables. 
Let $\ell\in\N$,  $\wild\in\{\eee,o\}$ and $x\in\wDset_\ell$. 
Assume $\ell-\abs{x}_1\ge8$.    
Then for any real $s\ge0$, 
	\begin{align}\label{nu=d}
	\P\{\wGpp_{\zevec,(\ell),x}\ge s\}\le\ell^{\tsp 2d}\tspb\P\{\min(t_1,\dotsc,t_{2d})\ge s/\ell\}\,.
	\end{align}
\end{lemma}

\begin{proof}
It is enough to prove the lemma for $\wild=\eee$.  Then we can assume that $\ell-\abs{x}_1$ is an even integer  because otherwise $\Paths_{\zevec, (\ell),x}=\varnothing$.  
The reason that    the case  $\wild=o$ is also covered is that   $\zGpp_{\zevec,(\ell),x}\le\Gpp_{\zevec,(\ell-1),x}\wedge\Gpp_{\zevec,(\ell),x}$.

To prove \eqref{nu=d} we construct a total of $2d$ edge-disjoint paths in $\Paths_{\tsp\zevec,(\ell),x}$.  
Let $A=\{z\in\range:z\cdot x>0\}$ and $B=\{z\in\range:z\cdot x=0\}$ with cardinalities $\nu_1\ge0$ and $\nu_2=2d-2\nu_1$, respectively.
Enumerate these sets as $A=\{z_1,\dotsc,z_{\nu_1}\}$ and $B=\{z_{\nu_1+1},\dotsc,z_{\nu_1+\nu_2}\}$. 

Suppose $x\ne\zevec$, in which case $\nu_1\ge1$. For each $i\in[\nu_1]$ let $\pi_i'\in\Paths_{\tsp\zevec,(\abs{x}_1),x}$ be the $\ell^1$-path from $\zevec$ to $x$  that takes the necessary steps in the order $z_i, z_{i+1},\dotsc,z_{\nu_1},z_1, \dotsc,z_{i-1}$.
Then for each  $i\in[\nu_1]$ let $\pi_i\in\Paths_{\tsp\zevec,(\ell),x}$   be the path that starts with  $(\ell-\abs{x}_1)/2$ repetitions of the  $(z_i,-z_i)$ pair  and then follows   $\pi'_i$.
For $i\in[\nu_2]$ let $\pi_{\nu_1+i}\in\Paths_{\tsp\zevec,(\ell),x}$ be the path that starts with a $z_{\nu_1+i}$ step, then repeats the    $(z_1,-z_1)$ pair  $(\ell-\abs{x}_1-2)/2$ times, then follows the steps of $\pi'_1$, and finishes  with a $-z_{\nu_1+i}$ step. Thus far we have constructed $\nu_1+\nu_2=2d-\nu_1$ paths. For the remaining $\nu_1$ paths we distinguish two cases.

If $\nu_1=1$ we   need  only one more path  $\pi_{2d}\in\Paths_{\tsp\zevec,(\ell),x}$. Take this to be the path that starts with a $-z_1$ step, repeats the   $(z_1,-z_1)$ pair   $(\ell-\abs{x}_1-8)/2$ times,  
takes two $z_2$ steps,  one $z_1$ step, follows the steps of $\pi'_1$,  takes  one $z_1$ step, two $-z_2$ steps, and finishes with a $-z_1$ step.

If $\nu_1>1$, then for $i\in[\nu_1-1]$, let $\pi_{\nu_1+\nu_2+i}\in\Paths_{\tsp\zevec,(\ell),x}$ be the path that starts with a $-z_i$ step, repeats the   $(z_i,-z_i)$ pair   $(\ell-\abs{x}_1-4)/2$ times,  takes a $z_{i+1}$ step, follows 
the steps of $\pi'_{i+1}$, and ends with a $z_i$ step followed by a $-z_{i+1}$ step. 
For $i=\nu_1$ the path $\pi_{2d}$ is defined similarly, except that $z_{i+1}$ and $\pi'_{i+1}$ are replaced by $z_1$ and $\pi'_1$, respectively.  

One can check that the paths $\pi_i\in\Paths_{\tsp\zevec,(\ell),x}$, $i\in[2d]$,  are edge-disjoint.
From  
	\begin{align}\label{faa10101} 
	\Gpp_{\zevec,(\ell),x}\le\min_{i\in[2d]}\tpath(\pi_i)
	\end{align}
	follows 
	\be\label{useful0101}\begin{aligned}
	\P\{\Gpp_{\zevec,(\ell),x}\ge s\}&\le\prod_{i=1}^{2d}\P\{\tpath(\pi_i)\ge s\}\le\bigl(\ell\tspa \P\{t(e)\ge s/\ell\}\bigr)^{2d}\\
	&=\ell^{2d}\P\{\min(t_1,\dotsc,t_{2d})\ge s/\ell\}\,.
	\end{aligned}\ee

If $x=\zevec$ (and hence $\nu_1=0$ and $\nu_2=2d$) then redo the last computation with the edge-disjoint paths $\pi_i$, $i\in[2d]$, that just repeat the pair $(z_i,-z_i)$.
\end{proof}

Below we use the condition that a rational point $\zeta\in\Uset$ satisfies $\ell\zeta\in\wDset_\ell$ for a positive  integer $\ell$ such that $\ell\zeta\in\Z^d$. When zero steps are admissible  ($\wild=o$) this is of course trivial, and without zero steps ($\wild=\eee$) this can be achieved if 
 $\ell(1-\abs{\zeta}_1)$ is even.
Therefore, one can take for example $\ell=2\ell'$ for $\ell'\in\N$  such that $\ell'\zeta\in\Z^d$.
Properties of convex sets used below can be found in Chapter 18 of \cite{Roc-70}.

The following theorem comes by a standard application of the subadditive ergodic theorem. 

\begin{theorem}\label{th:G-zeta}
Assume $\eit>-\infty$. 
Fix $\zeta\in\Q^d\cap\Uset$ and $\wild\in\{\eee,o\}$. Let $\ell\in\N$ be such that $\ell\zeta\in\wDset_\ell$. 
Assume $\E[\wGpp_{0,(\ell),\ell\zeta}]<\infty$.
Then the limit 
	\begin{align}\label{gpp-lim}
	\wgpp(\zeta)=\inf_{n\tsp\in\tsp\N}\frac{\E[\wGpp_{\zevec,(n\ell),n\ell\zeta}]}{n\ell}=\lim_{n\to\infty}\frac{\wGpp_{\zevec,(n\ell),n\ell\zeta}}{n\ell}\in\bigl[\eit\wedge0,\,\ell^{-1}\E[\wGpp_{0,(\ell),\ell\zeta}]\bigr]
	\end{align}
exists almost surely and in $L^1$ and does not depend on the choice of $\ell$. As a function of $\zeta\in\Q^d\cap\Uset$, $\wgpp$ is convex. 
Precisely, if $\zeta,\eta\in\Q^d\cap\Uset$ are such that $\E[\wGpp_{0,(\ell),\ell\zeta}]<\infty$ and $\E[\wGpp_{0,(\ell),\ell\eta}]<\infty$ for some $\ell\in\N$, then
for any $t\in(0,1)\cap\Q$, $\E[\wGpp_{0,(\ell'),\ell'(t\zeta+(1-t)\eta)}]<\infty$ for some $\ell'\in\N$ and
	\begin{align}\label{g-conv}
	\wgpp\bigl(t\zeta+(1-t)\eta\bigr)\le t\wgpp(\zeta)+(1-t)\wgpp(\eta).
	\end{align}
\end{theorem}

\begin{remark}[Conditions for finiteness]\label{rk:finite}
By Lemma \ref{lm:G<infty},  assumption \eqref{lin-ass5+}   with $p=1$ implies that $\E[\wGpp_{0,(\ell),\ell\zeta}]<\infty$ 
for any $\zeta\in\Q^d\cap\Uset$ and any large enough $\ell\in\N$ that satisfies $\ell\zeta\in\wDset_\ell$. 
\end{remark}

Next, from convexity we deduce local boundedness and then a local Lipschitz property. 

\begin{lemma}\label{lm:g-bded}
Assume $\eit>-\infty$ and \eqref{lin-ass5+}  with $p=1$.
Fix $\zeta\in\Q^d\cap\inter\Uset$ and $\wild\in\{\eee,o\}$. 
There exist $\e>0$ and a finite constant $C$ such that 
\begin{align}\label{g-bded}
\wgpp(\eta)\le C\quad\text{for all }\eta\in\Q^d\cap\Uset\text{ such that }\abs{\eta-\zeta}_1\le\e.
\end{align}
\end{lemma}

\begin{proof}
Take $\e>0$  rational and small enough so that $\cA=\{\eta\in\Uset:\abs{\eta-\zeta}_1\le\e\}\subset\inter\Uset$.
Let $\{\eta_i:i\in[2d]\}\subset\Q^d\cap\inter\Uset$ be the extreme points of the convex set $\cA$. 
For $\eta\in\Q^d\cap\cA$ write $\eta=\sum_{i=1}^{2d}\alpha_i \eta_i$ with rational $\alpha_i\in[0,1]$ such that  $\sum_{i\in[2d]}\alpha_i=1$.  
By bound \eqref{gpp-lim} and Remark \ref{rk:finite}, $\wgpp(\eta_i)<\infty$ for $i\in[2d]$.  Convexity \eqref{g-conv} implies
	\[\wgpp(\eta)\le\sum_{i\in[2d]}\alpha_i\wgpp(\eta_i)\le\max_{i\in[2d]}\wgpp(\eta_i)\]
and Lemma \ref{lm:g-bded}  is proved.
\end{proof}

\begin{lemma}\label{lm:g-Lip}
Assume $\eit>-\infty$ and \eqref{lin-ass5+} with $p=1$.
Fix $\zeta\in\Q^d\cap\inter\Uset$.
There exist $\e>0$ and a finite positive constant $C=C(\zeta, \e,\eit)$ such that 
for both $\wild\in\{\eee,o\}$ 
	\[\abs{\wgpp(\eta)-\wgpp(\eta')}\le C\abs{\eta-\eta'}_1\quad\forall\eta,\eta'\in\Q^d\cap\inter\Uset\text{ with }\abs{\eta-\zeta}_1\le\e\text{ and }\abs{\eta'-\zeta}_1\le\e.\]
\end{lemma}

\begin{proof}
The assumptions of Lemma \ref{lm:g-bded} are satisfied and therefore 
there exists a rational $\e>0$ and a finite constant $C$ such that \eqref{g-bded} holds. 
By taking $\e>0$ smaller, if necessary, we can also guarantee
that for any $\eta\in\R^d$, $\abs{\eta-\zeta}_1\le\e$ implies $\eta\in\inter\Uset$.

Take $\eta\ne\eta'$ in $\inter\Uset$ with $\abs{\eta-\zeta}_1\le\e/2$ and $\abs{\eta'-\zeta}_1\le\e/2$. Abbreviate $\delta=2\e^{-1}\abs{\eta-\eta'}_1$ and write
	\[\eta=\frac1{1+\delta}\cdot\eta'+\frac{\delta}{1+\delta}\cdot\bigl(\eta+\delta^{-1}(\eta-\eta')\bigr).\]
Note that
	\[\abs{\eta+\delta^{-1}(\eta-\eta')-\zeta}_1\le\frac\e2+\delta^{-1}\abs{\eta-\eta'}_1=\e.\]
Therefore, $\eta+\delta^{-1}(\eta-\eta')\in\inter\Uset$. By convexity \eqref{g-conv} and boundedness \eqref{g-bded} we have
	\[\wgpp(\eta)\le\frac1{1+\delta}\cdot\wgpp(\eta')+\frac{\delta}{1+\delta}\cdot\wgpp\bigl(\eta+\delta^{-1}(\eta-\eta')\bigr)\le\frac1{1+\delta}\cdot\wgpp(\eta')+\frac{C\delta}{1+\delta}\,.\]
From $C\ge\wgpp(\eta')\ge \eit\wedge0$,  
\begin{align*}
\wgpp(\eta)-\wgpp(\eta')\le\frac{\delta}{1+\delta}(-\wgpp(\eta')+C) \le \delta(\abs{\eit\wedge0}+C)= 2\e^{-1}(\abs{\eit\wedge0}+C)\tspa\abs{\eta-\eta'}_1.
\end{align*} 
The other bound comes by switching around $\eta$ and $\eta'$.
\end{proof}

The next lemma is an immediate consequence of the local Lipschitz property proved in the previous lemma.

\begin{lemma}\label{lm:g-ext}
Assume $\eit>-\infty$ and \eqref{lin-ass5+} with $p=1$.
Then   $\gpp$ and $\zgpp$ extend to locally Lipschitz, continuous, convex functions on $\inter\Uset$.
\end{lemma}

Before we  prove the shape theorem we need two more auxiliary lemmas.   
 
\begin{lemma}\label{lm:aux0000}
Assume \eqref{lin-ass5+}   with $p=d$.
Then there exists a finite constant $\kappa$ such that 
\begin{align}\label{eqn:aux0000}
\begin{split}
&\P\Biggl\{\forall \tspa \text{pair } \epsilon<\rho\text{ in }(0,1)\ \exists\ell_0=\ell_0(\epsilon,\rho, \w) \text{ such that }\\
&\qquad\qquad\forall\ell\ge\ell_0,\,\forall\wild\in\{\eee,o\}:\sup_{\substack{y\tsp\in\tsp \wDset_{\ell}\\ \epsilon\ell\le\abs{y}_1\le\rho\ell}}\ell^{-1}\wGpp_{\zevec,(\ell),y}
\le\kappa\Biggr\}=1.
\end{split}
\end{align}
\end{lemma}

\begin{proof}
It is enough to work with $\wild=\eee$ since $\zGpp_{\zevec,(\ell),y}\le\Gpp_{\zevec,(\ell),y}$. 
It is also enough to work with fixed $\epsilon<\rho$ since the suprema in question increase as we increase $\rho$ and decrease $\epsilon$.

Fix an integer $r\ge5$ such that 
	\[\rho(1+8/r)<1.\]
The strategy of the proof will be to bound 	$\Gpp_{\zevec,(\ell),y}$ by constructing edge-disjoint paths on the  coarse-grained lattice $r\Z^d$  to a point $\underline y$ that approximates $y$.  An approach to finding such paths was developed in the proof of Lemma \ref{lm:G<infty}. 
	
Take $\ell$ large enough so that 
	\begin{align}\label{ell-cond}
	\ell\ge2(d+8)r\epsilon^{-1}\quad\text{and}\quad\rho(1+8/r)\ell+(d+8)(r+8)+dr+8\le\ell.
	\end{align}
For each $y\in\Dset_\ell$ with $\epsilon\ell\le\abs{y}_1\le\rho\ell$
pick $\underline y\in r\Z^d$ so that $\abs{y-\underline y}_1\le dr$. As in Lemma \ref{lm:G<infty}, let $\nu_1$ be the number of $z\in\range$ such that $\underline y\cdot z>0$ and let $\nu_2=2d-2\nu_1$.
Following the construction in the proof of  Lemma \ref{lm:G<infty} 
we can produce edge-disjoint nearest-neighbor paths $\pi'_i$, $i\in[2d]$, on the coarse-grained lattice  $r\Z^d$ from $\zevec$ to $\underline y$ such that, in terms of the number steps taken on $r\Z^d$,  $\pi'_i$ has length $\underline\ell_i=\underline\ell=\abs{\underline y}_1/r$ for $i\in[\nu_1]$, $\pi'_i$ has length $\underline\ell_i=\underline\ell+2$ for $i\in\{\nu_1+1,\dotsc,\nu_1+\nu_2\}$, and for $i>\nu_1+\nu_2$, 
$\pi'_i$ has length $\underline\ell_i=\underline\ell+8$ if $\nu_1=1$ and $\underline\ell_i=\underline\ell+4$ if $\nu_1>1$.

From  $\abs{y}_1\le\rho\ell$ and  $\abs{y-\underline y}_1\le dr$ follows  $\underline\ell=\abs{\underline y}_1/r\le(\rho\ell+dr)/r$, and then from \eqref{ell-cond} 
	\[(\underline\ell+8)(r+8)+\abs{y-\underline y}_1+8
	\le\bigl((\rho\ell+dr)/r+8\bigr)(r+8)+dr+8
	\le\ell.\]
Define 
	\[q=\Bigl\lfloor\frac12\Bigl(\frac{\ell-8(r+8)-\abs{y-\underline y}_1-8}{\underline\ell}-r-8\Bigr)\Bigr\rfloor.\]
Then  
	\[0\le q\le\frac\ell{2\underline\ell}\le\frac{r\ell}{2(\abs{y}_1-dr)}\le r\epsilon^{-1}.\]
Define 
	\be\label{def:m}  m=\Bigl\lfloor\frac{\ell-8(r+8)-\abs{y-\underline y}_1-8-\underline\ell(r+8+2q)}{2}\Bigr\rfloor.\ee
Then   
	\[0\le m\le\underline\ell.\]

Let $\pi'_{i,s}$ denote the position (on the original lattice)  of the path $\pi'_i$ after $s$ steps (of size $r$). Let 
	\be\label{aux128} \ell_i=(r+8)\underline\ell_i+2q\underline\ell+2m.\ee
For each $i\in[2d]$ we have this identity: 
\[(r+10+2q)m+(r+8+2q)(\underline\ell-m)+(r+8)(\underline\ell_i-\underline\ell)+\ell-\ell_i
=\ell.\]
This equation gives a way to decompose the $\ell$ steps from $\zevec$ to $y$ so that we first go through  the vertices  $\{\pi'_{i,s}\}_{0\le s\le \underline\ell_i}$  and then use  the last $\ell-\ell_i$ steps to go  from $\underline y$ to $y$.  We continue with this next bound:  
\be\label{aux132}	\begin{aligned}
	\Gpp_{\zevec,(\ell),y}
	&\le\min_{i\in[2d]}\biggl\{\;\sum_{s=0}^{m-1}\Gpp_{\pi'_{i,s},(r+10+2q),\pi'_{i,s+1}}+\sum_{s=m}^{\underline\ell-1}\Gpp_{\pi'_{i,s},(r+8+2q),\pi'_{i,s+1}}  \\
	&\qquad \qquad\qquad\qquad\qquad\qquad
	+\sum_{s=\underline\ell}^{\underline\ell_i-1}\Gpp_{\pi'_{i,s},(r+8),\pi'_{i,s+1}}+\Gpp_{\underline y,(\ell-\ell_i),y}\biggr\}.
	\end{aligned}\ee
Bound $m$ in \eqref{def:m}  by dropping $\fl{\ }$ to turn  \eqref{aux128}  into this inequality (note that terms $2q\underline\ell$ cancel): 
	\begin{align*}
	\ell_i
	&\le(r+8)(\underline\ell+8)+\ell-8(r+8)-\abs{y-\underline y}_1-8-(r+8)\underline\ell
	=\ell-\abs{y-\underline y}_1-8.
	\end{align*}
Similarly,
	\[\ell_i\ge\ell-\abs{y-\underline y}_1-10\ge\ell-dr-10.\]
Fix $\kappa>0$. For $j=d+1,\dotsc,2d$ let $\evec_j=-\evec_{j-d}$. Define the events
	\begin{align*}
	&\cE^1_\ell=\bigl\{\exists j\in[2d]: \Gpp_{\zevec,(r+8+2q),r\evec_j}\ge\kappa\ell/14\text{ or }\Gpp_{\zevec,(r+10+2q),r\evec_j}\ge\kappa\ell/14\bigr\},\\
	&\cE^2_\ell=\bigl\{\exists j\in[2d]:\Gpp_{-r\evec_j,(r+8),\zevec}\ge\kappa\ell/14\text{ or }\Gpp_{-r\evec_j,(r+8+2q),\zevec}\ge\kappa\ell/14\text{ or }\Gpp_{-r\evec_j,(r+10+2q),\zevec}\ge\kappa\ell/14\bigr\},\\
	&\cE^3_\ell=\bigl\{\exists k\in\N, z\in \Dset_k:\abs{z}_1\le dr,\ \abs{z}_1+8\le k\le dr+10,\ \Gpp_{z,(k),\zevec}\ge\kappa\ell/14\bigr\},
	\end{align*}
and the event $\cE^4_{\ell,y}$ on which for all $i\in[2d]$
	\begin{align}\label{foo10101}
	\sum_{s=1}^{m\wedge(\underline\ell-1)-1}\Gpp_{\pi'_{i,s},(r+10+2q),\pi'_{i,s+1}}+\sum_{s=m\vee1}^{\underline\ell-2}\Gpp_{\pi'_{i,s},(r+8+2q),\pi'_{i,s+1}}\ge\kappa\ell/14.
	\end{align}
Then for $\ell_0$ large enough to satisfy \eqref{ell-cond} and $\kappa>0$, 
	\be\label{foo10103}
	\begin{aligned} 
	\Biggl\{\sup_{\ell\ge\ell_0}\sup_{\substack{y\tsp\in\tsp \Dset_{\ell}\\ \epsilon\ell\le\abs{y}_1\le\rho\ell}}\ell^{-1}\Gpp_{\zevec,(\ell),y}>\kappa\Biggr\}
&\;\subset\; \biggl(\;\bigcup_{\ell\ge\ell_0}\cE^1_\ell\biggr) \cup\biggl(\;\bigcup_{\ell\ge\ell_0}\bigcup_{\substack{\underline y\tsp\in\tsp r\Z^d\\ \epsilon\ell/2\le\abs{\underline y}_1\le\ell}}\cE^2_\ell\circ\theta_{\underline y}\biggr) \\
&\qquad
\cup\biggl(\;\bigcup_{\ell\ge\ell_0}\bigcup_{\substack{y\tsp\in\tsp\Z^d\\ \epsilon\ell\le\abs{y}_1\le\ell}}\cE^3_\ell\circ\theta_y\biggr)  
 \cup\biggl(\; \bigcup_{\ell\ge\ell_0}\bigcup_{\substack{y\tsp\in\tsp\Dset_\ell\\ \epsilon\ell\le\abs{y}_1\le\ell}}\cE^4_{\ell,y}\biggr).
	\end{aligned}\ee
	
Here is the explanation for the inclusion above. 
\begin{enumerate}   [label=\rm(\roman{*}), ref=\rm(\roman{*})] \itemsep=3pt 
\item   Further down the proof we add auxiliary paths around the $r$-steps of the path $\pi'_i$. Because  the first $r$-steps share their initial point $\zevec$, their auxiliary paths would intersect and independence would be lost.  The same is true for  the last $r$-steps that share the endpoint $\underline y$. Hence these special steps are handled separately. 

  The event  $\cE^1_\ell$ takes care of the first step of the path $\pi'_i$ which is either in the first sum on the right in \eqref{aux132}, or in the second sum in case $m=0$ and the first sum is empty.  
   
  The event   $\cE^2_\ell$ takes care of the last step  to $\underline y$ which can come from any one of the three sums on the right in \eqref{aux132}.   We have to check that the possible endpoints fall within the range  $\epsilon\ell/2\le\abs{\pi_{i,s}'}_1\le\ell$ of the union of shifts of $\cE^2_\ell$:  
for  $i\in[2d]$ and $\underline\ell\le s\le\underline\ell_i$, 
\[  \abs{\pi_{i,s}'}_1\ge\abs{y}_1-(d+8)r\ge\epsilon\ell/2\]  
 and since $\pi_{i,s}'$ is on an admissible path of length $\ell$ from $\zevec$ to $y$, it must be that $\abs{\pi_{i,s}'}_1\le\ell$.

\item  The event  $\cE^3_\ell$ takes care of the path segment from $\underline y$ to $y$.
  
\item On the complement of the first three unions  on the right-hand side  of \eqref{foo10103} we have for each $i\in[2d]$, 
	\begin{align*}
	&\Gpp_{\pi'_{i,0},(r+10+2q),\pi'_{i,1}}\one\{m\ge1\}+\Gpp_{\pi'_{i,0},(r+8+2q),\pi'_{i,1}}\one\{m=0\}+\Gpp_{\pi'_{i,\underline\ell-1},(r+10+2q),\pi'_{i,\underline\ell}}\one\{m=\underline\ell\}\\
	&\qquad\qquad+\Gpp_{\pi'_{i,\underline\ell-1},(r+8+2q),\pi'_{i,\underline\ell}}\one\{m<\underline\ell\}+\sum_{s=\underline\ell}^{\underline\ell_i-1}\Gpp_{\pi'_{i,s},(r+8),\pi'_{i,s+1}}+\Gpp_{\underline y,(\ell-\ell_i),y}<13\kappa\ell/14.
	\end{align*}
Since $\underline\ell_i-\underline\ell\le8$, the left-hand side has at most 13  terms, which explains the bound on the right.   
Thus, if in addition  $\Gpp_{\zevec,(\ell),y}>\kappa\ell$, then event $\cE^4_{\ell,y}$ must occur.
  \end{enumerate} 

By bounding the probabilities of the unions on the right of \eqref{foo10103}, 
we show next that for some fixed $\kappa$  that does not depend on $0<\epsilon<\rho<1$, 
\be\label{aux150}
\lim_{\ell_0\to\infty} \P\Biggl\{\sup_{\ell\ge\ell_0}\sup_{\substack{y\tsp\in\tsp \Dset_{\ell}\\ \epsilon\ell\le\abs{y}_1\le\rho\ell}}\ell^{-1}\Gpp_{\zevec,(\ell),y}>\kappa\Biggr\} =0.
\ee
This will imply the conclusion \eqref{eqn:aux0000} as we point out at the end of the proof.  

By \eqref{nu=d}, $\P(\cE_\ell^1)$ is summable if \eqref{lin-ass5+} is satisfied with $p=1$. Then $\P\bigl(\,\bigcup_{\ell\ge\ell_0}\cE_\ell^1\,\bigr)\to0$ as $\ell_0\to\infty$. 
Next, observe that 
	\[\bigcup_{\ell\ge\ell_0}\bigcup_{\substack{y\tsp\in\tsp\Z^d\\ \epsilon\ell\le\abs{y}_1\le\ell}}\cE^3_\ell\circ\theta_y\subset\bigcup_{\ell\ge\epsilon\ell_0}\bigcup_{\substack{y\tsp\in\tsp\Z^d\\ \abs{y}_1=\ell}}\cE^3_\ell\circ\theta_y\]
and hence 
\[\P\Bigl(\;\bigcup_{\ell\ge\ell_0}\bigcup_{\substack{y\tsp\in\tsp\Z^d\\ \epsilon\ell\le\abs{y}_1\le\ell}}\cE^3_\ell\circ\theta_y\Bigr)\le\sum_{\ell\ge\epsilon\ell_0}\P\Bigl(\bigcup_{\substack{y\tsp\in\tsp\Z^d\\ \abs{y}_1=\ell}}\cE^3_\ell\circ\theta_y\Bigr),\]
which goes to $0$ when $\ell_0\to\infty$ if $\ell^{d-1}\P(\cE^3_\ell)$ is summable. This is the case if
\eqref{lin-ass5+} is satisfied with $p=d$. The union over $\cE^2_\ell\circ\theta_{\underline y}$ is controlled similarly.

\begin{figure}
	\begin{center}
		\begin{tikzpicture}[>=latex,  font=\footnotesize,scale=.1]
		
		    \foreach \i in {-2,...,9} {
        \draw [line width=0.1pt,gray!50,dashed] (5*\i,-10) -- (5*\i,25);
    }
    \foreach \i in {-2,...,5} {
        \draw [line width=0.1pt,gray!50,dashed] (-10,5*\i) -- (45,5*\i);
    }

		\draw[line width=0.5pt](0,0)--(0,15)--(35,15)--(35,0)--(0,0);
		\draw[line width=0.5pt](0,0)--(-5,0)--(-5,20)--(35,20)--(35,15);
		\draw[line width=0.5pt](0,0)--(0,-5)--(40,-5)--(40,15)--(35,15);
		\draw[line width=1pt](5,0)--(5,1)--(10,1)--(10,0);
		\draw[line width=1pt](5,0)--(5,-1)--(10,-1)--(10,0);
		\draw[line width=1pt](5,0)--(4,0)--(4,-2)--(11,-2)--(11,0)--(10,0);
		\draw[line width=1pt](0,-5)--(0,-4)--(5,-4)--(5,-5);
		\draw[line width=1pt](0,-5)--(0,-6)--(5,-6)--(5,-5);
		\draw[line width=1pt](0,-5)--(-1,-5)--(-1,-3)--(6,-3)--(6,-5)--(5,-5);
		\draw[line width=1pt](15,0)--(15,1)--(20,1)--(20,0);
		\draw[line width=1pt](15,0)--(15,-1)--(20,-1)--(20,0);
		\draw[line width=1pt](15,0)--(14,0)--(14,-2)--(21,-2)--(21,0)--(20,0);
		\draw(-1.5,2)node{$\zevec$};
		\draw(37,17)node{$\underline y$};
		\end{tikzpicture}
	\end{center}
	\caption{\small  The light dashed grid is the coarse-grained lattice $r\Z^d$.  The thin lines along this grid represent four  $\pi'_i$-paths   from $\zevec$ to $\underline y$.  Three $r$-steps on two $\pi'_i$-paths  are decorated with auxiliary paths  represented by thick lines. The auxiliary  paths   are edge-disjoint  as long as they associate (i) with  different $\pi'_i$-paths,  (ii)  with non-consecutive $r$-steps on the same path $\pi'_i$, or (iii)  with $r$-steps that are neither the  first nor the last one of a $\pi'_i$-path.  	}
	\label{fig:paths1}
\end{figure}
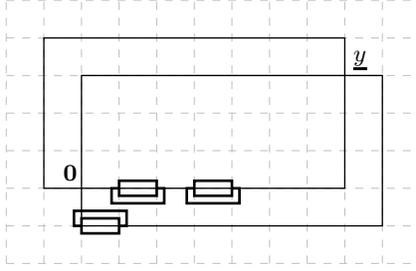

It remains to  control the probability of the  union of the events $\cE^4_{\ell,y}$ in \eqref{foo10103}. 
 For $i\in[2d]$ and  $s\in[\underline\ell-2]$,  for each segment  $[\pi'_{i,s}, \pi'_{i,s+1}]$,   bound both   passage times $\Gpp_{\pi'_{i,s},(r+10+2q),\pi'_{i,s+1}}$ and $\Gpp_{\pi'_{i,s},(r+8+2q),\pi'_{i,s+1}}$ as was done in \eqref{faa10101} by using $2d$ independent auxiliary paths of the appropriate lengths.  For each segment  $[\pi'_{i,s}, \pi'_{i,s+1}]$ add the two upper bounds and denote the result by $A_{\pi'_{i,s},\pi'_{i,s+1}}$.

 The terms for  $s=0$ and $s\ge\underline\ell-1$ were excluded from the events $\cE^4_{\ell,y}$  so that for distinct  indices $i\in[2d]$  the $2d$ auxiliary paths constructed around the segments  $\{[\pi'_{i,s}, \pi'_{i,s+1}]\}_{s\in[\tsp\underline\ell-2]}$ stay separated.     (We chose $r\ge5$ at the outset to guarantee this separation.)  
Replace the edge weights $t(e)$ with $t^+(e)=\max(t(e),0)$ to ensure that the upper bounds are nonnegative.  
After these steps,  the left-hand side of \eqref{foo10101} is bounded above by $\sum_{s=1}^{\underline\ell-2}A_{\pi'_{i,s},\pi'_{i,s+1}}$.

All   the   $A$-terms have the same distribution as $A_{\zevec,r\evec_1}$.  As explained above, over distinct  indices $i\in[2d]$  the random vectors $\{A_{\pi'_{i,s},\pi'_{i,s+1}}:s\in[\underline\ell-2]\}$  are independent. 
For any particular  $i\in[2d]$, $\{A_{\pi'_{i,s},\pi'_{i,s+1}}:s\in[\underline\ell-2] \text{ even} \}$ are i.i.d.\ and $\{A_{\pi'_{i,s},\pi'_{i,s+1}}:s\in[\underline\ell-2] \text{ odd} \}$ are i.i.d.\ because now we skip every other $r$-step.   See Figure \ref{fig:paths1}.  

We derive the concluding estimate.   Recall that
	\[\underline\ell\le(\rho\ell+dr)/r\le (\rho r^{-1}+1)\ell.\]
Let $c=\ce{(\rho r^{-1}+1)/2}$.  Let $S_n$ denote the sum of $n$ independent copies of $A_{\zevec,r\evec_1}$. 
Since the $A$-terms are nonnegative we have	
	\[\P\Bigl(\sum_{s\in[\underline\ell-2]\text{ even}}A_{\pi'_{i,s},\pi'_{i,s+1}}\ge\kappa\ell/28\Bigr)\le\P(S_{c\tsp\ell}\ge\kappa\ell/28).\]
 The same holds for the sum over odd $s$. Thus we have
	\[\P(\cE_{\ell,y}^4)\le 2^{2d}\tspa\P(S_{c\tsp\ell}\ge\kappa\ell/28)^{2d},\]
Take $\kappa>28\tsp c\tsp\E[A_{\zevec,r\evec_1}]$ and use the fact that there are no more than $(2\ell+1)^d$ points $y\in\Dset_\ell$ to get
	\begin{align*}
	\P\biggl(\;\bigcup_{\ell\ge\ell_0}\bigcup_{\substack{y\tsp\in\tsp\Dset_\ell\\ \epsilon\ell\le\abs{y}_1\le\rho\ell}}\cE^4_{\ell,y}\biggr)
	&\le\sum_{\ell\ge\ell_0}(2\ell+1)^d\P(\cE^4_{\ell,y})\le \sum_{\ell\ge\ell_0}(8\ell+4)^d\P(S_{c\tsp\ell}\ge\kappa\ell/28)^{2d}\\
	&\le \sum_{\ell\ge\ell_0}\frac{(8\ell+4)^d c^{2d}\Var(A_{\zevec,r\evec_1})^{2d}}{(\kappa/28-c\E[A_{\zevec,r\evec_1}])^{4d}\ell^{2d}}\,.
	\end{align*}
The bound   \eqref{useful0101} can be utilized to show that each $\Gpp_{\zevec,(\ell),x}$ and thereby $A_{\zevec,r\evec_1}$ is square-integrable   if \eqref{lin-ass5+} holds with $p=2$. The above then converges to $0$ as $\ell_0\to\infty$.
We have verified  \eqref{aux150}. The claim of the lemma follows:
\[\P\Biggl\{\forall\ell_0\ \exists\ell\ge\ell_0:\sup_{\substack{y\tsp\in\tsp \Dset_{\ell}\\ \epsilon\ell\le\abs{y}_1\le\rho\ell}}\ell^{-1}\Gpp_{\zevec,(\ell),y}>\kappa\Biggr\}
=\lim_{\ell_0\to\infty} \P\Biggl\{\exists\ell\ge\ell_0:\sup_{\substack{y\tsp\in\tsp \Dset_{\ell}\\ \epsilon\ell\le\abs{y}_1\le\rho\ell}}\ell^{-1}\Gpp_{\zevec,(\ell),y}>\kappa\Biggr\} =0.\qedhere\]
\end{proof}

\medskip 

\begin{lemma}\label{lm:aux0101}
Assume \eqref{lin-ass5+} with $p=d$.
 Then for any $0<\epsilon<\rho<1$ there exists a deterministic constant  $\kappa\in(0,\infty)$ such that, with probability one for each $x\in\Z^d$, there exists a strictly increasing random sequence $\{m(n)\}_{n\tsp\in\tsp\N}\subset\N$ such that $m(n+1)/m(n)\to1$ 
 and for $\wild\in\{\eee,o\}$ and $\ell\in\N$ 
 	\begin{align}\label{whatever}
	\begin{split}
	&\wGpp_{m(n)x,(\ell),z}\le\kappa\ell\quad\forall z\in m(n)x+\wDset_{\ell}
	\text{ such that } \epsilon\ell\le\abs{z-m(n)x}_1\le\rho\ell.
	\end{split}
	\end{align}
\end{lemma}

\begin{proof}
If $x=\zevec$ take   $m(n)=\ell_0+ n$ from  Lemma \ref{lm:aux0000}. Next suppose  $x\ne\zevec$.
Fix $\epsilon<\rho$ in $(0,1)$.
Apply Lemma \ref{lm:aux0000} to choose 
  a finite constant  $\kappa$ such that
\[ \P(\cE) \equiv \P\Biggl\{\forall\wild\in\{\eee,o\}:\ 
\sup_{\substack{\ell\in\N,y\tsp\in\tsp \wDset_{\ell}\\ \epsilon\ell\le\abs{y}_1\le\rho\ell}}\ell^{-1}\wGpp_{\zevec,(\ell),y}\le\kappa
\Biggr\}>0.\]
The ergodic theorem  implies that with probability one, for each $x\in\Z^d\setminus\{\zevec\}$ there exist infinitely many $m\in\N$ such that 
	\[\forall\wild\in\{\eee,o\}:\sup_{\substack{\ell\in\N,y\tsp\in\tsp mx+\wDset_{\ell}\\ \epsilon\ell\le\abs{y-mx}_1\le\rho\ell}}\ell^{-1}\wGpp_{mx,(\ell),y}\le\kappa.\]
Enumerate these $m$'s as a strictly increasing sequence $\{m(n):n\in\N\}$. Then for $\P$-almost every $\w$
	\[\lim_{n\to\infty}\frac{n}{m(n)}=\lim_{n\to\infty}\frac1{m(n)}\sum_{k=1}^{m(n)}\one\{\theta_{kx}\w\in\cE\}=\P(\cE)>0.\]
Consequently, $m(n+1)/m(n)$ converges to $1$. 
\end{proof}

We are ready for the shape theorem.

\begin{theorem}\label{thm:shape-G}
Assume $\eit>-\infty$  and   \eqref{lin-ass5+}  with $p=d$. Fix $\wild\in\{\eee,o\}$. Let $\cV$ be a closed subset of $\inter\Uset$. 
The following holds with probability one: 
	\begin{align}\label{shape-G}
	\lim_{\ell\to\infty} \;
	\max_{x\tsp\in\tsp\wDset_\ell: \,x/\ell\tsp\in\tsp\cV}\;
	\ell^{\tsp-1}\abs{\wGpp_{\zevec,(\ell),x}-\ell \wgpp(x/\ell)}=0. 
	\end{align}
\end{theorem}

\begin{proof}
The proof follows steps similar  to those of  \eqref{shape-mu-b}.
We treat the case $\wild=o$, the other case being a simpler version.  
Let $\Omega_0$ be  the full probability event that consists of intersecting the event on which \eqref{gpp-lim} holds for all $\zeta\in\Q^d\cap\inter\Uset$ with the event in \eqref{eqn:aux0000} 
and the events in Lemma \ref{lm:aux0101} for all rational $\epsilon<\rho$ in $(0,1)$. 
Fix  $\w\in\Omega_0$.    We show that for this $\w$
	\begin{align}
	&\varliminf_{\ell\to\infty}\;\min_{x\tsp\in\tsp\wDset_\ell: \,x/\ell\tsp\in\tsp\cV}\;
	\ell^{-1}\bigl(\zGpp_{\zevec,(\ell),x}-\ell\zgpp(x/\ell)\bigr)\ge0\quad\text{and}\label{shape1}\\
	&\varlimsup_{\ell\to\infty}\; \max_{x\tsp\in\tsp\wDset_\ell: \,x/\ell\tsp\in\tsp\cV}\;
	\ell^{-1}\bigl(\zGpp_{\zevec,(\ell),x}-\ell\zgpp(x/\ell)\bigr)\le0.\label{shape2}
	\end{align}

 {\bf Proof of \eqref{shape1}.}   
Fix  ($\w$-dependent) sequences $\ell_k\to\infty$ and $x_k\in\zDset_{\ell_k}$ that realize the $\varliminf$ on the left-hand side of  \eqref{shape1}. 
Since $x_k\in\zDset_{\ell_k}$ there are coefficients $a^\pm_{i,k}\in\Z_+$ such that 
	\begin{align}\label{x-cov}
	x_k=\sum_{i=1}^d (a_{i,k}^+ -a^-_{i,k})\evec_i
	\quad\text{and}\quad 
	\sum_{i=1}^d(a_{i,k}^++a_{i,k}^-)\le \ell_k.   
	\end{align}	
Pass to subsequences, still denoted by $\ell_k$ and $x_k$, such that  
	\begin{align}\label{alph-cov}
	a_{i,k}^\pm/\ell_k\mathop{\longrightarrow}_{k\to\infty}\alpha_i^\pm\in[0,1]\quad\text{with}\quad\sum_{i=1}^d(\alpha_i^++\alpha_i^-)\le 1.  
	\end{align}
 Let  $\xi=\sum_{i=1}^d(\alpha_i^+-\alpha_i^-)\evec_i=\lim_{k\to\infty} x_k/\ell_k\in\cV\subset\inter\Uset$.   We approximate $\xi$ with a rational point $\zeta$ to which we can apply \eqref{gpp-lim}.  
 Bound \eqref{shape1} comes by building a path from $x_k$ to a multiple of $\zeta$ and by the subadditivity of passage times.  
 Here are the details.

First, we dispose of the case where there are infinitely many $k$ for which $a_{i,k}^+=a_{i,k}^-=0$ for all $i\in[d]$. If this is the case, then going along a further subsequence we can assume that $x_k=\zevec$ for all $k$.
Applying \eqref{gpp-lim} with $\zeta=\zevec$ gives $\ell_k^{-1}\zGpp_{\zevec,(\ell_k),x_k}\to\zgpp(\zevec)$ and since $\zgpp(x_k/\ell_k)=\zgpp(\zevec)$ for all $k$ we see that the $\varliminf$ on the left-hand side of \eqref{shape1} is $0$.
We can therefore assume that for each $k$ there exists some $i\in[d]$ such that $a_{i,k}^+\ge1$ or $a_{i,k}^-\ge1$.
Consequently, if we let $\cI$ denote the set of indices $i\in[d]$ for which $a_{i,k}^+\ge1$ or $a_{i,k}^-\ge1$ for infinitely many $k$, then $\cI\neq\varnothing$.
 	
Let 
	\begin{align}\label{def:gamma}
	\gamma=\min\{\alpha_i^-:\alpha_i^->0,i\in[d]\}\wedge\min\{\alpha_i^+:\alpha_i^+>0,i\in[d]\}>0,
	\end{align}
with the convention that $\min\varnothing=\infty$, which takes care of the case $\alpha_i^\pm=0$ for all $i\in[d]$. 
Let  $\delta$ be a rational in $\bigl(0,(\gamma\wedge1)/(4d)\bigr)$.
For $i\in[d]\setminus\cI$ let $\beta_i^+=\beta_i^-=0$ and note that we also have $\alpha_i^+=\alpha_i^-=0$.
For $i\in\cI$ take  $\beta_i^\pm\in[\delta,1]\cap\Q$  such that $\abs{\alpha_i^\pm-\beta_i^\pm}\le2d\delta$,
	\[\sum_{i=1}^d(\beta_i^++\beta_i^-)\le1,\quad\text{and}\quad\forall j\in\cI:(1+5d\gamma^{-1})(\beta_j^+-\beta_j^-)\ne\alpha_j^+-\alpha_j^-.\]   
 Let $\zeta=\sum_{i=1}^d (\beta_i^+-\beta_i^-)\evec_i$ and take $\delta>0$ small enough so that $\zeta\in\inter\Uset$. We will eventually take $\delta\to0$, which sends $\zeta\to\xi$.

We have for all $i\in\cI$
		\begin{align}\label{beta-alpha}
		(1+5d\delta\gamma^{-1})\beta_i^+-\alpha_i^+\ge\delta\quad\text{and}\quad(1+5d\delta\gamma^{-1})\beta_i^--\alpha_i^-\ge\delta.
		\end{align}
To see this, note that when $\alpha_i^+>0$ we have
	\[(1+5d\delta\gamma^{-1})\beta_i^+-\alpha_i^+\ge(1+5d\delta\gamma^{-1})(\alpha_i^+-2d\delta)-\alpha_i^+\ge d\delta/2\ge\delta\]
and when $\alpha_i^+=0$ (but $i\in\cI$) we have
	\[(1+5d\delta\gamma^{-1})\beta_i^+-\alpha_i^+=(1+5d\delta\gamma^{-1})\beta_i^+\ge\beta ^+_i\ge\delta.\]
The same holds with superscript $-$. 

Let
	\[\zeta'=\frac{\sum_{i\in\cI}\Bigl((1+5d\delta\gamma^{-1})(\beta_i^+-\beta_i^-)-(\alpha_i^+-\alpha_i^-)\Bigr)\evec_i}{\sum_{i\in\cI}\Bigl((1+5d\delta\gamma^{-1})(\beta_i^++\beta_i^-)-(\alpha_i^++\alpha_i^-)\Bigr)}.\]
The choice of $\beta_i^\pm$ guarantees that $\zeta'\ne\zevec$. Furthermore, \eqref{beta-alpha} shows that $\zeta'$ is a convex combination of the vectors $\{\pm\evec_i:i\in\cI\}$ with all strictly positive coefficients. 
Consequently, $\zeta'\in\inter\Uset$.

Take rational $\epsilon<\rho$ in $(0,1)$ such that $\epsilon<\abs{\zeta'}_1<\rho$.
Let $\ell\in\N$ be such that $\ell\beta_i^+,\ell\beta_i^-\in\N$ for $i\in\cI$ and take $\bar n_k$ such that  
	\[m(\bar n_k-1)\le(1+5d\delta\gamma^{-1})\ell_k/\ell\le m(\bar n_k),\]
for the sequence $m(n)$ in Lemma \ref{lm:aux0101} corresponding to the above choice of $\epsilon$ and $\rho$ and to $x=\ell\zeta$. 
Abbreviate $\mbar_k=m(\bar n_k)$. 
Using \eqref{beta-alpha} we have for $i\in\cI$
		\begin{align}\label{limaux}
		\lim_{k\to\infty}\ell_k^{-1}(\mbar_k\ell\beta_i^+-a_{i,k}^+)=(1+5d\delta\gamma^{-1})\beta_i^+-\alpha_i^+\ge\delta.
		\end{align}
The same holds with superscript $-$. Thus,  for all $i\in\cI$ and for large $k$  	
	\begin{align}\label{m.zeta>x}
	 \mbar_k\ell\beta_i^\pm\ge a_{i,k}^\pm+\delta\ell_k/2.
	\end{align}
This implies  that when $k$ is large, $\overline m_k\ell\zeta$ (which belongs to $\Z^d$) is  accessible from $x_k$ by an $\range$-admissible  path of length
	\begin{align}\label{x-length}
	\overline j_k=\sum_{i=1}^d(\overline m_k\ell\beta_i^+-a_{i,k}^+)+\sum_{i=1}^d(\overline m_k\ell\beta_i^--a_{i,k}^-).
	\end{align}
Note that
	\begin{align}\label{j/ell}
	\lim_{k\to\infty}\overline j_k/\ell_k=\sum_{i=1}^d((1+5d\delta\gamma^{-1})\beta_i^+-\alpha_i^+)+\sum_{i=1}^d((1+5d\delta\gamma^{-1})\beta_i^--\alpha_i^-)\le(4d+5\gamma^{-1})d\delta.
	\end{align}
The first equality and \eqref{limaux} imply that 
	\[\lim_{k\to\infty}\frac{\mbar_k\ell\zeta-x_k}{\overline j_k}=\zeta'\]
and therefore $\epsilon\overline j_k\le\abs{\mbar_k\ell\zeta-x_k}_1\le\rho\overline j_k$ for $k$ large enough. This will allow us to apply \eqref{whatever}.

Since $x_k$ is accessible from $\zevec$ by an $\range$-admissible  path of length $\sum_{i=1}^d(a_{i,k}^++a_{i,k}^-)\le\ell_k$, 
concatenating this path and the one from $x_k$ to $\overline m_k\ell\zeta$ gives an $\range$-admissible  path from $\zevec$ to $\overline m_k\ell\zeta$ of length 
	\[\sum_{i=1}^d\overline m_k\ell(\beta_i^++\beta_i^-)\le\overline m_k\ell.\]  
Hence $\overline m_k\ell\zeta\in\zDset_{\overline m_k\ell}$.  Subadditivity now gives
	\begin{align*}
	\zGpp_{\zevec,(\mbar_k\ell),\mbar_k\ell\zeta}\le\zGpp_{\zevec,(\ell_k),x_k}+\zGpp_{x_k,(\overline j_k),\mbar_k\ell\zeta}.
	\end{align*}
Using this, \eqref{whatever}, and \eqref{j/ell}, we get
	\[(1+5d\delta\gamma^{-1})\zgpp(\zeta)=\lim_{k\to\infty}\zGpp_{\zevec,(\mbar_k\ell),\mbar_k\ell\zeta}\le\varliminf_{k\to\infty}\frac{\zGpp_{\zevec,(\ell_k),x_k}}{\ell_k}+\kappa(4d+5\gamma^{-1})d\delta.\]
Taking $\delta\to0$ and the continuity of $\zgpp$ on $\inter\Uset$ gives 
	\[\zgpp(\xi)\le\varliminf_{k\to\infty}\frac{\zGpp_{\zevec,(\ell_k),x_k}}{\ell_k}.\]
Since $x_k/\ell_k\in\cV\subset\inter\Uset$ and  $x_k/\ell_k\to\xi$, using again the continuity of $\zgpp$ on $\inter\Uset$ completes the proof of  \eqref{shape1}: 
	\begin{align*}
	\varliminf_{k\to\infty}\ell_k^{\tsp-1}\bigl(\zGpp_{\zevec,(\ell_k),x_k}-\ell_k\zgpp(x_k/\ell_k)\bigr)\ge 0.
	\end{align*}
 
\medskip 
  
 {\bf Proof of \eqref{shape2}.}   
Proceed similarly to the  proof of \eqref{shape1}, but with the sequences $\ell_k\to\infty$ and $x_k\in\zDset_{\ell_k}$ realizing the $\varlimsup$ on the left-hand side of  \eqref{shape2}.
Again, we have the representation \eqref{x-cov}, the limits \eqref{alph-cov}, and $\xi=\sum_{i\in[d]}(\alpha_i^+-\alpha_i^-)\evec_i\in\cV$.

We start by treating the case when $\xi=\zevec$.
In this case let $j_k=2\abs{x_k}_1$ or $j_k=2\abs{x_k}_1+1$, so that $\ell_k-j_k$ is even.
Observe that $j_k/\ell_k\to0$ and hence $\ell_k\ge j_k$ for $k$ large. Thus, one can make an admissible loop of length $\ell_k-j_k$ from $\zevec$ back to $\zevec$ and then take a path of length $j_k$ from $\zevec$ to $x_k$.
From \eqref{gpp-lim} we have $\ell_k^{-1}\zGpp_{\zevec,(\ell_k-j_k),\zevec}\to\zgpp(\zevec)$.
If $j_k$ is bounded then so is $\abs{x_k}_1$ and we have $\ell_k^{-1}\zGpp_{\zevec,(j_k),x_k}\to0$.  On the other hand, if $j_k\to\infty$ along some subsequence, then along this subsequence, and 
for $k$ large, we have $j_k/3\le\abs{x_k}_1\le2j_k/3$ and, applying \eqref{eqn:aux0000}, we then get 
	\[\zGpp_{\zevec,(\ell_k),x_k}\le\zGpp_{\zevec,(\ell_k-j_k),\zevec}+\zGpp_{\zevec,(j_k),x_k}\le\zGpp_{\zevec,(\ell_k-j_k),\zevec}+\kappa j_k,\]
for $k$ large enough. Dividing by $\ell_k$ and taking $k\to\infty$ we deduce that 
	\[\varlimsup_{k\to\infty}\ell_k^{-1}\zGpp_{\zevec,(\ell_k),x_k}\le\zgpp(\zevec).\]
The continuity of $\zgpp$ at $\zevec$ implies then that the $\varlimsup$ on the left-hand side of \eqref{shape2} is $0$. For the rest of the proof we can and will assume that $\xi\neq\zevec$.\medskip

 Define $\gamma\in(0,\infty)$ as in  \eqref{def:gamma}. 
 Let  $\delta$ be a rational in $(0,\gamma/2)$.
Choose $\beta_i^\pm$, $i\in[d]$, so that for $\sigg\in\{-,+\}$, when $\alpha_i^\sig=0$ we have $\beta_i^\sig=0$ and when $\alpha_i^\sig>0$ we have
 $\beta_i^\sig\in[\delta,1]\cap\Q$ such that $\abs{\alpha_i^\sig-\beta_i^\sig}\le\delta$ and overall we have
	\[\sum_{i=1}^d(\beta_i^++\beta_i^-)\le1\quad\text{and}\quad(1-2\delta\gamma^{-1})\sum_{i\in[d]}(\beta_i^+-\beta_i^-)\evec_i\ne\sum_{i\in[d]}(\alpha_i^+-\alpha_i^-)\evec_i.\]   
This is possible since $\xi\ne\zevec$ and therefore $\alpha_i^\sig>0$ for some $i\in[d]$ and $\sigg\in\{-,+\}$. 
Let $\zeta=\sum_{i=1}^d (\beta_i^+-\beta_i^-)\evec_i$ and choose $\delta$ small enough so that $\zeta\in\inter\Uset$.  
Note that  
	\[\alpha_i^\sig-(1-2\delta\gamma^{-1})\beta_i^\sig\ge0\quad\text{for all }i\in[d]\text{ and }\sig\in\{-,+\}.\]
Indeed, this clearly holds when $\alpha_i^\sig=0$ and when $\alpha_i^\sig>0$ we have 
	\[\alpha_i^\sig-(1-2\delta\gamma^{-1})\beta_i^\sig\ge\alpha_i^\sig-(1-2\delta\gamma^{-1})(\alpha_i^\sig+\delta)\ge\delta.\]
The above two observations imply that
	\[\zeta'=\frac{\sum_{i\in[d]}\Bigl((\alpha_i^+-\alpha_i^-)-(1-2\delta\gamma^{-1})(\beta_i^+-\beta_i^-)\Bigr)\evec_i}{\delta+\sum_{i\in[d]}\Bigl((\alpha_i^++\alpha_i^-)-(1-2\delta\gamma^{-1})(\beta_i^++\beta_i^-)\Bigr)}\in\inter\Uset\setminus\{\zevec\}.\]
We can then find rational $\epsilon<\rho$ in $(0,1)$ such that $\epsilon<\abs{\zeta'}_1<\rho$.

Let $\ell\in\N$ be such that $\ell\beta_i^+,\ell\beta_i^-\in\Z_+$ for $i\in[d]$ and take $\underline n_k$ such that  
	\[m(\underline n_k)\le(1-2\delta\gamma^{-1})\ell_k/\ell\le m(\underline n_k+1),\] 
for the sequence $m(n)$ in Lemma \ref{lm:aux0101} corresponding to $x=\ell\zeta$ and to the above choice of $\epsilon$ and $\rho$.
Abbreviate $\underline m_k=m(\underline n_k)$ and observe that if $\alpha_i^+>0$ then
	\begin{align}\label{eqn:x>m.zeta}
	\lim_{k\to\infty}\ell_k^{-1}(a_{i,k}^+-\underline m_k\ell\beta_i^+)= \alpha_i^+-(1-2\delta\gamma^{-1})\beta_i^+\ge\delta.
	\end{align}
Then for  large $k$  
	\begin{align}\label{a>m}
	a_{i,k}^+-\underline m_k\ell\beta_i^+\ge0
	\end{align}
This inequality is trivial if $\alpha_i^+=\beta_i^+=0$.
The same computation works with minus sign superscripts. 
This implies that $x_k$ is accessible from $\underline m_k\ell\zeta$  in 
	\[\underline j_k=\sum_{i=1}^d(a_{i,k}^+-\underline m_k\ell\beta_i^+)+\sum_{i=1}^d(a_{i,k}^--\underline m_k\ell\beta_i^-)\]
$\range$-steps and $\fl{\delta\ell_k}$ $\zevec$-steps. 
Note that
	\[\lim_{k\to\infty}\underline j_k/\ell_k=\sum_{i=1}^d(\alpha_i^+-(1-2\delta\gamma^{-1})\beta_i^+)+\sum_{i=1}^d(\alpha_i^--(1-2\delta\gamma^{-1})\beta_i^-)\le(2d+2\gamma^{-1})\delta.\]
As a consequence,
	\[\lim_{k\to\infty}\frac{x_k-\underline m_k\ell\zeta}{\fl{\delta\ell_k}+\underline j_k}=\zeta'\]   
and one can then apply \eqref{whatever}. Then, as in the proof of \eqref{shape1}, using subadditivity 
then taking $k\to\infty$ and then  $\delta\to0$ and using the continuity of $\zgpp$ on $\inter\Uset$ gives 
	\[\varlimsup_{k\to\infty}\frac{\zGpp_{\zevec,(\ell_k),x_k}}{\ell_k}\le  \zgpp(\xi).\]
Another use of the continuity of  $\zgpp$  completes the proof of \eqref{shape2}.  
\end{proof}

\begin{proof}[Proof of Theorem \ref{thm:shape-G2}]
Apply  Theorem \ref{thm:shape-G} with $\cV=\{\xi\in\Uset:\abs{\xi}_1\le1/(1+\alpha)\}$.
\end{proof}


\section{Peierls argument} \label{a:Pei}  

This appendix follows the ideas of    \cite{Gri-Kes-84, Ber-Kes-93}.
First we prove a general estimate and then specialize it to prove Lemma \ref{pei-lm1}. Let $d\in\N$. 
Tile  $\Z^d$ by    $N$-cubes  $S(\kvec)=N\kvec+[0,N)^d$  indexed by  $\kvec\in\Z^d$.    Each     $N$-cube  $S(\kvec)$ is colored  randomly  black or white in a shift-stationary manner.  Let $p=p(N)$ be the marginal probability that a particular cube is black  and assume that 
\be\label{pa:p(N)}  p(N)\to 1\quad\text{as}\quad N\to\infty.  \ee
  Assume finite range dependence:  there is a strictly  positive  integer constant $a_0$ such that 
\be\label{pa:ind4} 
\text{$\forall \uvec\in\Z^d$, the colors of the cubes $\{ S(\kvec): \kvec \in \uvec + a_0\Z^d\}$ are i.i.d.}
\ee
    There are $K_0=a_0^d$ distinct i.i.d.\ collections, indexed by   $\uvec\in\{0,1,\dotsc,a_0-1\}^d$. 

It may be desirable to let the separation of the cubes be a parameter.  For  a positive integer $a_1$ and $\uvec\in\{0,1,\dotsc,a_1-1\}^d$, define the collection $\cS_{a_1, \uvec}= \{ S(\kvec): \kvec \in \uvec + a_1\Z^d\}$ of cubes with lower left corners on the grid $\uvec + a_1\Z^d$.   For a given $a_1$,  $K_1=a_1^d$ is  the number of distinct collections $\cS_{a_1, \uvec}$ indexed by  $\uvec\in\{0,1,\dotsc,a_1-1\}^d$.    We always consider $a_1\ge a_0$ where $a_0$ is the fixed constant of the independence assumption \eqref{pa:ind4}. 

Let 
$  \bB(0,r)=\{ x\in\Z^d:  \abs{x}_1\le r \}  $ denote the $\ell^1$-ball (diamond) of radius $\fl{r}$ in $\Z^d$, with (inner) boundary  
$\partial \bB(0,r)=\{ x\in\Z^d:  \abs{x}_1=\fl r \}  $.  

\begin{lemma}\label{pa45lm}   Assume \eqref{pa:p(N)} and \eqref{pa:ind4}.    Let  $a_1\in \Z_{\ge a_0} $ and $K_1=a_1^d$.  Then there exists a  constant $N_0=N_0(d)$  such that  for $N\ge N_0$ and $n\ge 2(d+1)N $, 
\be\label{pa45} \begin{aligned} 
\P\bigl\{&\text{\,$\forall$  \!lattice path  $\pi$   from the origin to $\partial \bB(0, n)$  $\exists$ \!\!$\uvec\in([ 0,a_1-1] \cap\Z)^d$ such that  }\\
&\qquad 
\text{$\pi$ intersects at least  $\frac{n}{4NK_1}$ black cubes from   $\cS_{a_1,\uvec}$}\,\bigr\}
\ge 1\,-\, \exp\Bigl( -\,\frac{n}{2N}\Bigr) .
\end{aligned} \ee

\end{lemma}

To prove Lemma \ref{pa45lm} we  record a Bernoulli large deviation bound.  
 
 \begin{lemma}\label{pa48lm}   Assume   \eqref{pa:ind4} and let $p\in(0,1)$ be the marginal probability of a black cube.  
Then there exist constants $A(p, K, \delta)>0$ such that, for all integers $a_1\ge a_0$, $m\in\N$, and $\delta\in(0,p/K_1)$,   with $K_1=a_1^d$,   and for  any particular sequence $S(\kvec_1),\dotsc, S(\kvec_m)$ of distinct $N$-cubes,   the following estimate holds for some $\uvec$ determined by $\{S(\kvec_i)\}_{i=1}^m$: 
 \begin{align*} 
    \P\bigl\{ \text{
    $S(\kvec_1),\dotsc, S(\kvec_m)$ contains at least $m\delta$   black cubes from $\cS_{a_1, \uvec}$}\bigr\} 
    \ge 1-e^{-A(p, K_1, \delta)m}.  
 \end{align*}
Furthermore,  $\lim_{p\nearrow 1} A(p,K, \delta)=\infty$ for all $K\in\N$ and $\delta\in(0,p/K)$. 
 \end{lemma}
 \begin{proof}    Pick  $\uvec$  so that 
  $\cS_{a_1, \uvec}$ contains  at least $\ce{m/K_1}$ of the cubes $S(\kvec_1),\dotsc, S(\kvec_m)$.  Since these are colored independently and $\delta<p/K_1$,  basic large deviations gives 
 \begin{align*}
& \P(\text{at most $m\delta$ cubes among $\{S(\kvec_i)\}_{i=1}^m \cap \cS_{a_1, \uvec} $ are black}) \\
&\le \P(\text{at most $m\delta$ cubes among $\ce{m/K_1}$  independently colored cubes   are black}) \\
&\le \exp\Bigl\{ -\, \frac{m}{K_1} I_p(K_1\delta) \Bigr\}    =  e^{-A(p, K_1, \delta)m}
 \end{align*}
 where the last equality defines $A$ and  the well-known Cram\'er rate function \cite{Ras-Sep-15-ldp}  of  the Bernoulli$(p)$ distribution is 
 \[  I_p(s)= s\log\frac{s}{p} + (1-s)\log\frac{1-s}{1-p} \qquad \text{for } \ s\in[0, 1]. \]
To complete the proof, observe that 
\[
\lim_{p\nearrow 1} A(p,K, \delta)= \lim_{p\nearrow 1} \frac{1}{K} I_p(K\delta) 
=  \lim_{p\nearrow 1}\Bigl( \delta  \log\frac{K\delta}{p} + \frac{1-K\delta}{K}\log\frac{1-K\delta}{1-p} \Bigr)=\infty. 
\qedhere 
\]  
 \end{proof}  
 
 \smallskip 
 
 \begin{proof}[Proof of Lemma \ref{pa45lm}]  
 
 Consider for the moment a fixed  path $\pi$ from $0$ to a point $y$ such that $\abs{y}_1=n$.    Assume   $n>dN$ so that  $y\notin S(\mathbf 0)$.  
 
 For $j\in\Z_+$ let {\it level $j$} of $N$-cubes refer to the collection  $\cL_j=\{ S(\kvec) :  \abs{\kvec}_1=j\} $.  Since points $x=(x_1,\dotsc,x_d)\in S(\kvec)$ satisfy  
 \[  k_iN\le x_i\le k_iN+N-1   \quad \text{ for $i\in[d]$,}\]    
 level $j$ cubes are subsets  of  
 $\{ x:  Nj-d(N-1) \le \abs{x}_1 \le   Nj+d(N-1)  \}$.  
 
 To reach the  point $y$,    path $\pi$ must have entered and exited   at least one  $N$-cube at levels $0, 1, \dotsc, m_0$ where  $m_0$ satisfies 
 \[    
 Nm_0+d(N-1) < \abs{y}_1 \le   N(m_0+1)+d(N-1) . 
   \]  
 This calculation excludes  the cube that contains the endpoint $y$.   
From this
 \be\label{pa990}  
   m_0\ge   \frac{\abs{y}_1-d(N-1)}{N}-1  \ge  \frac{n}N -(d+1).  
   \ee

Consider the sequence of  $N$-cubes that  path $\pi$ intersects:   $S(\mathbf 0)=S(\kvec_0), S(\kvec_1), \dotsc, S(\kvec_{m_1})$, with the initial point $0\in S(\mathbf 0)=S(\kvec_0)$ and the final point $y\in S(\kvec_{m_1})$.   Remove loops from this sequence (if any), for example by the following procedure:
\begin{enumerate}
\item Let $i_0$ be the minimal index such that  $\kvec_{i_0}=\kvec_j$ for some $j>i_0$. Let $j_0$ be the maximal $j$ for $i_0$.  Then remove $S(\kvec_{i_0+1}),\dotsc, S(\kvec_{j_0})$.
\item Repeat  the same step on the remaining sequence $S(\kvec_0),  \dotsc, S(\kvec_{i_0}), S(\kvec_{j_0+1}),\dotsc, S(\kvec_{m_1})$, as long as loops remain. 
\end{enumerate} 

After loop removal  relabel the  sequence of remaining  cubes consecutively  to arrive at a new sequence  $S(\kvec_0), S(\kvec_1)\dotsc, S(\kvec_{m_2})$ of distinct $N$-cubes with $m_2\le m_1$ and still $0\in S(\mathbf 0)=S(\kvec_0)$ and  $y\in S(\kvec_{m_2})$.   This sequence takes  nearest-neighbor steps on the coarse-grained  lattice of $N$-cubes, in the sense that $\abs{\kvec_i-\kvec_{i-1}}_1=1$, because this property is preserved by the loop removal.    Since $\pi$ enters and leaves behind  at least one $N$-cube on each level $0,\dotsc,m_0$, we have the bound $m_2-1\ge m_0$. 

 We have now associated to each path $\pi$ a sequence of $m_0$ distinct $N$-cubes that $\pi$ both  enters from the outside  and exits again.    We apply Lemma \ref{pa48lm} to these sequences of cubes. 

Take  $a_1\ge a_0\ge 1$ and $K_1=a_1^d$   as in the statement of Lemma \ref{pa45lm}. 
 Let $\delta_0=(2K_1)^{-1}$.  
 Fix $N$ large enough so that $p=p(N)>\tfrac12=\delta_0 K_1$ and the constant given by Lemma \ref{pa45lm} satisfies \[ A(p, K_1, \delta_0) > \log 2d +1.\]    
 Consider
 $n\ge 2(d+1)N $  to guarantee that the rightmost expression in \eqref{pa990}  and thereby also $m_0$ is larger than $n/(2N)$.   Then  also  $m_0\delta_0\ge n/(4NK_1)$. 
By Lemma \ref{pa48lm}, 
\begin{align*}
&\P\bigl\{\text{$\forall\,$path  $\pi: 0 \to \partial \bB(0,n)$  $\exists \uvec$ such that $\pi$    enters and exits}  \\ 
&\qquad\qquad\qquad\qquad 
\text{at least  $\frac{n}{4NK_1}$ distinct black cubes from $\cS_{a_1,\uvec}$}\bigr\}\\[4pt] 
&\ge \P\bigl\{\text{every nearest-neighbor sequence of $m_0$  $N$-cubes starting at $S(\mathbf 0)$  }\\[2pt] 
&\qquad\qquad\qquad\qquad 
 \text{  contains at least  $m_0\delta_0$ black cubes from some $\cS_{a_1,\uvec}$}\bigr\}\\[4pt]
&\ge 1-  (2d)^{m_0}  e^{-A(p, K_1, \delta_0)m_0}  
\ge 1-e^{-m_0} 
\ge 1- e^{-n/(2N)}.
\end{align*}
%
This completes the proof of Lemma \ref{pa45lm}. 
\end{proof} 

\begin{proof}[Proof of Lemma \ref{pei-lm1}]  Surround each $N$-cube $S(\kvec)$ with $2d$ $N$-boxes so that each $d-1$ dimensional face of $S(\kvec)$ is directly opposite a large face of one of the $N$-boxes.  Precisely,  first put $S(\kvec)$ at the center of the  $3N$-cube   $T(\kvec)=N\kvec+[-N,2N]^d$ on $\Z^d$, and then define $2d$ $N$-boxes $B^{\pm j}(\kvec)=T(\kvec)\cap T(\kvec\pm 2\evec_j)$ for $j\in[d]$.   Any lattice path that enters $S(\kvec)$ and exits $T(\kvec)$ must cross in the sense of \eqref{cross4}  one of the $N$-boxes that surround $S(\kvec)$.  

Color $S(\kvec)$ black if all $2d$ $N$-boxes surrounding it are black.  The probability that $S(\kvec)$ is black can be made arbitrarily close to 1 by choosing $s_0$ and $N$ large enough and $\delta_0>0$ small enough in the definition  \eqref{1.bl1}--\eqref{1.bl2} of a black $N$-box.    The color of $S(\kvec)$ depends only on the edge variables in the union $\overline T(\kvec)$ of the $2d$   boxes $\overline B^{\,\pm j}(\kvec)$ enlarged as in \eqref{B6.3}. 
The separation of $a_0$ in \eqref{pa:ind4}  can be fixed large enough to guarantee that over $\kvec \in \uvec + a_0\Z^d$ the cubes $\overline T(\kvec)$ are pairwise disjoint. 


Apply Lemma \ref{pa45lm} with  $K_1=a_1^d=a_0^d$.   
Tighten the requirement $n\ge 2(d+1)N$ of  Lemma \ref{pa45lm} to    $n\ge 4dN$  to  guarantee that if a path $\pi$  intersects $S(\kvec)$ then it also intersects the complement of $T(\kvec)$. (If $\pi$ remains inside $T(\kvec)$ then the $\ell^1$-distance between the endpoints of $\pi$ is at most $3dN$ and  $\pi$ cannot connect the origin to $\partial\bB(0,n)$.)   Thus for every $S(\kvec)$ intersected by $\pi$, at least one of the  $N$-boxes surrounding $S(\kvec)$ is crossed by $\pi$ in the sense of \eqref{cross4}. 
In conclusion, on the event in \eqref{pa45} each   path from the origin to $\partial \bB(0,n)$ crosses  at least $\ce{n/(4NK_1)}=\ce{na_0^{-d}/(4N)}$ disjoint $N$-boxes.   Of these,  at least  $\ce{na_0^{-d}/(4N)}/K$  must come from some  particular collection $\cB_j$.   
Thus in Lemma \ref{pei-lm1} we can take $\delta_1=1/(4a_0^dNK)$, $n_1=4dN$ and $D_1=1/(2N)$. 
\end{proof}


\section{Convex analysis} \label{a:conv} 

\begin{lemma}\label{lm:co1}  Let $f$ be a proper convex function on $\R^d$ {\rm(}$-\infty<f\le\infty$ and $f$ is not identically $\infty${\rm)}  and $\xi\in\ri(\effdom f)$. Then the following statements are equivalent. 

 \begin{enumerate} [label=\rm(\alph{*}), ref=\rm\alph{*}]  \itemsep=3pt 
 
 \item\label{co1-a}  For some $b\in\R$, $\partial f(\xi)\subset\{h\in\R^d: h\cdot\xi=b\}$. 
 
 \item\label{co1-b} $f^*$ is constant over $\partial f(\xi)$. 
 
  \item\label{co1-c}  $t\mapsto f(t\xi)$ is differentiable at $t=1$. 
 
 \end{enumerate} 

\end{lemma} 

\begin{proof}  \eqref{co1-a}$\implies$\eqref{co1-b}.  For all $h\in\partial f(\xi)$,  
$f^*(h)=h\cdot\xi-f(\xi)=b-f(\xi)$. 

\smallskip 

 \eqref{co1-b}$\implies$\eqref{co1-a}.   Suppose $f^*(h)=s$ for all $h\in\partial f(\xi)$.  Then  for all $h\in\partial f(\xi)$,   $h\cdot\xi= f^*(h)+f(\xi)=s+f(\xi)$.   
 
\smallskip 

  \eqref{co1-c}$\implies$\eqref{co1-a}.   Let $b=(d/dt)f(t\xi)\vert_{t=1}$ and  $h\in\partial f(\xi)$. Then for all $\abs s\le\e$,  by convexity,  
  $f(\xi+s\xi) -f(\xi) \ge s\tsp h\cdot\xi$.  
   This says that $h\cdot\xi$ lies in the subdifferential of the function $t\mapsto  f(t\xi)$ at $t=1$, but by assumption this latter equals the singleton $\{b\}$.  

\smallskip 

 \eqref{co1-a}$\implies$\eqref{co1-c}.   The directional derivatives satisfy the following, where in both equations  the second equality comes from \cite[Thm.~23.4]{Roc-70}. 
 \begin{align*}
 f'(\xi;\xi) = \lim_{s\searrow0} \frac{f(\xi+s\xi) -f(\xi)}s = \sup\{ \xi\cdot h: h\in\partial f(\xi)\}  = b 
 \end{align*} 
 and 
  \begin{align*}
 f'(\xi;-\xi) = \lim_{s\searrow0} \frac{f(\xi-s\xi) -f(\xi)}s = \sup\{ -\xi\cdot h: h\in\partial f(\xi)\}  = -b . 
 \end{align*} 
 From this we see the equality of the left and right derivatives of  $\varphi(t)=f(t\xi)$ at $t=1$:
 \[  \varphi'(t-) =\lim_{t\nearrow0} \frac{f(\xi+t \xi) -f(\xi)}t =- f'(\xi;-\xi) =b 
 \]
 and 
  \[  \varphi'(t+) =\lim_{t\searrow0} \frac{f(\xi+t \xi) -f(\xi)}t =f'(\xi;\xi) =b .  
\qedhere \]
\end{proof}

\small

\medskip 

\bibliographystyle{plain}

\bibliography{firasbib2010}

\end{document}